\documentclass[12pt,a4paper]{amsart}
\usepackage[margin=1in, headsep=18pt]{geometry}

\usepackage{amsmath,amssymb,amsfonts,amsthm,amscd,amsopn,amstext}
\usepackage{mathtools}
\usepackage{verbatim}

\usepackage[mathscr]{euscript}
\usepackage[bbgreekl]{mathbbol}

\usepackage{lineno}

\usepackage{amsthm}
\usepackage{amsopn} 
\usepackage{verbatim} 
\usepackage{mathtools}
\usepackage{enumitem}
\usepackage{longtable}
\usepackage{stmaryrd}

\usepackage{xfp} % for calculations

\usepackage{tikz}
\usepackage{tikz-cd}
\usepackage{tikz-3dplot}
\usepackage[margin=1in]{geometry}

\usetikzlibrary{babel}
\usetikzlibrary{matrix}
\usetikzlibrary{shapes}
\usetikzlibrary{arrows}
\usetikzlibrary{calc,3d}
\usetikzlibrary{decorations,decorations.pathmorphing}
\usetikzlibrary{through}
\tikzset{ext/.style={circle, draw,inner sep=1pt},int/.style={circle,draw,fill,inner sep=1pt},nil/.style={inner sep=1pt}}
\tikzset{exte/.style={circle, draw,inner sep=3pt},inte/.style={circle,draw,fill,inner sep=3pt}}
\tikzset{diagram/.style={matrix of math nodes, row sep=3em, column sep=2.5em, text height=1.5ex, text depth=0.25ex}}
\tikzset{diagram2/.style={matrix of math nodes, row sep=0.5em, column sep=0.5em, text height=1.5ex, text depth=0.25ex}}
\tikzset{every picture/.append style={baseline=-.65ex}}
\usepackage{ytableau}

\usepackage[backref=page]{hyperref}

\providecommand{\eprint}[2][]{\href{http://arxiv.org/abs/#2}{arXiv:#2}}

\newcommand{\SW}{\mathsf{SW}} % Schur--Weyl functor: S_N-modules -> GL_V-modules
\newcommand{\WS}{\mathsf{WS}} % conjugate Schur--Weyl functor: GL_V-modules -> S_N-modules

\usepackage{ulem}
\usepackage{color}
\newcommand\cred[1]{{\color{red} #1}}

\let\le\leqslant
\let\ge\geqslant
\let\leq\leqslant
\let\geq\geqslant

\makeatletter
\renewcommand{\subsubsection}{%
	\@startsection{subsubsection}{3}%
	\z@{.5\linespacing\@plus.7\linespacing}{-.5em}%
	{\normalfont\bfseries\itshape}% <--- Added \bfseries here alongside \itshape
}
\makeatletter
\newcounter{alphaequation}

\newcommand{\nc}{\newcommand}
\newtheorem{theoremA}{Theorem}

\newtheorem{corollaryA}[theoremA]{Corollary}

\numberwithin{equation}{section}
\newtheorem{conjecture}[equation]{Conjecture}
\newtheorem*{conjecture*}{Conjecture}
\newtheorem{theorem}[equation]{Theorem}
\newtheorem*{theorem*}{Theorem}
\newtheorem{proposition}[equation]{Proposition}
\newtheorem{proposition-definition}[equation]{Proposition-Definition}

\newtheorem{lemma}[equation]{Lemma}
\newtheorem{corollary}[equation]{Corollary}
\newtheorem*{corollary*}{Corollary}
\newtheorem{remark}[equation]{Remark}
\newtheorem{prop::def}[equation]{Proposition-Definition}

\newtheorem{definition}[equation]{{\bf Definition}}

\newtheorem{notation}[equation]{{\bf Notation}}

\theoremstyle{definition}
\newtheorem{example}[equation]{Example}

\nc{\ov}[1]{\overline{#1}}

\nc{\U}{\mathrm U}

\nc{\cO}{\mathcal{O}}
\nc{\cL}{\mathcal{L}}

\newcommand{\calF}{{\mathcal{F}}}

\nc{\sR}{\mathbf{R}}
\nc{\sZ}{\mathbf{Z}}

\newcommand{\bQ}{{\mathbb Q}}
\newcommand{\bZ}{{\mathbb Z}}
\newcommand{\bC}{{\mathbb C}}

\newcommand{\bP}{{\mathbb P}}

\nc{\bU}{{\mathbb U}}
\newcommand{\bS}{\mathfrak{S}} % Symmetric group

\newcommand{\fh}{{\mathfrak h}}

\newcommand{\fg}{{\mathfrak g}}
\newcommand{\fa}{{\mathfrak a}}
\newcommand{\fb}{{\mathfrak b}}

\newcommand{\fn}{{\mathfrak n}}

\nc{\fB}{\mathfrak{B}}
\nc{\gl}{\mathfrak{gl}}
\nc{\GL}{\mathsf{GL}}
\nc{\msl}{\mathfrak{sl}}
\nc{\mgl}{\mathfrak{gl}}

\newcommand{\Iwahori}{\mathfrak{I}}
\newcommand{\nIw}{{\mathfrak{n}^{\Iwahori}}}

\newcommand{\euJ}{\EuScript{J}}

\newcommand{\Span}{\mathrm{Span}}

\nc{\ch}{{\mathop {\rm ch}}}
\nc{\Tr}{{\mathop {\rm Tr}\,}}
\nc{\tr}{{\mathrm tr}}
\nc{\Id}{{\mathop {\rm Id}}}
\nc{\ad}{{\mathop {\rm ad}}}
\nc{\End}{{\mathop {\rm End}}}

\newcommand{\mGH}{\EuScript{GH}} % Garsia-Haiman modules
\nc{\RepGL}[2]{{\mathsf{Rep}}(\GL_{#1})^{(#2)}}
\nc{\RepGLt}[2]{{\mathsf{Rep}}(\mgl_{#1}[t])^{(#2)}}
\nc{\RepGLz}[2]{{\mathsf{Rep}}(\mgl_{#1}[\bar{z}])^{(#2)}}

\nc{\RepSL}[2]{{\mathsf{Rep}}(\SL_{#1})^{(#2)}}
\nc{\RepSLt}[2]{{\mathsf{Rep}}(\msl_{#1}[t])^{(#2)}}
\nc{\Rep}{\mathsf{Rep}}

\newcommand{\mdeg}{\mathsf{mdeg}} % multidegree
\newcommand{\bomega}{{\boldsymbol{\omega}}} % Sign functor map

\newcommand{\ev}{\mathsf{ev}} % evaluation representation
\newcommand{\sW}{\mathscr{W}} % Local Weyl module
\newcommand{\cosoc}{\mathsf{socle}} % socle.
\newcommand{\soc}{\mathsf{head}} % cosocle=head.

\newcommand{\sWH}{\mathscr{H}} % Iterated fusion product.
\newcommand{\mtop}{\mathsf{Top}} % Top degree component
\newcommand{\gh}{{\widetilde{H}}} % Character of GH-module

\usepackage{graphicx}

\newcommand{\wcheck}[1]{\widehat{#1}}

\newcommand{\young}[3]{%
  \begin{tikzpicture}[scale=0.4,shift={(0,-2)}]
    \ifnum#1>0
      \foreach \x in {0,...,\numexpr#1-1} {
        \draw (\x,0) rectangle ++(1,1);
      }
    \fi
    \ifnum#2>0
      \foreach \x in {0,...,\numexpr#2-1} {
        \draw (\x,1) rectangle ++(1,1);
      }
    \fi
    \ifnum#3>0
      \foreach \x in {0,...,\numexpr#3-1} {
        \draw (\x,2) rectangle ++(1,1);
      }
    \fi
    \pgfmathsetmacro{\maxx}{max(#1,#2,#3)}
    \draw[->, thick] (0,0) -- (\maxx + 1.5,0) node[right] {$z_1$};
    \draw[->, thick] (0,0) -- (0,4) node[above] {$z_2$};
  \end{tikzpicture}%
}

\newcommand{\inYoung}[3]
{
    \ifnum#1>0
      \foreach \x in {0,...,\numexpr#1-1} {
        \draw (\x,0) rectangle ++(1,1);
      }
    \fi
    \ifnum#2>0
      \foreach \x in {0,...,\numexpr#2-1} {
        \draw (\x,1) rectangle ++(1,1);
      }
    \fi
    \ifnum#3>0
      \foreach \x in {0,...,\numexpr#3-1} {
        \draw (\x,2) rectangle ++(1,1);
      }
    \fi
    \pgfmathsetmacro{\maxx}{max(#1,#2,#3)}
    \draw[->, thick] (0,0) -- (\maxx + 1.5,0) node[right] {$z_1$};
    \draw[->, thick] (0,0) -- (0,4) node[above] {$z_2$};
}

\begin{document}

\title[Fusion product \& N!/k-conjecture]{On fusion product and N!/k-conjecture}

\author[A.\,Khoroshkin]{Anton Khoroshkin}
\address{Anton Khoroshkin: \newline
Department of Mathematics, University of Haifa, Mount Carmel, 3498838, Haifa, Israel}
\email{khoroshkin@gmail.com}

\author[Ie.\,Makedonskyi]{Ievgen Makedonskyi}
\address{Ievgen Makedonskyi:\newline
Yanqi Lake Beijing Institute of Mathematical Sciences And Applications (BIMSA), 
No. 544, Hefangkou Village, Huaibei Town, Huairou District, Beijing 101408.}
\email{makedonskii\_e@mail.ru}

\begin{abstract}
We formulate a version of Schur--Weyl duality for the current Lie algebra
$\mathfrak{gl}_V[x,y]=\mathfrak{gl}_V\otimes\mathbb{C}[x,y]$. Under this duality,
the Garsia--Haiman modules of the $N!$-conjecture become cyclic and cocyclic
$\mathfrak{gl}_V[x,y]$-modules, and we describe them as iterated fusion and
cofusion products of tautological $\mathfrak{gl}_V$-modules. We show that
Haiman's $N!$-theorem for a diagram is equivalent to the coincidence of the
fusion and the cofusion filtrations on the tensor product of the local Weyl modules attached
to its rows, and that the fusion and cofusion products of Garsia--Haiman modules are associative.
Let $\lambda$ and $\mu$ be Young diagrams obtained by removing two different corners
from the same diagram. We construct an iterated fusion product of $S^2V$ with local Weyl modules, which is a quotient of the common quotient of the
Garsia--Haiman modules of $\lambda$ and $\mu$, fits into short exact sequences with each of
them, and has Butler's intersection polynomial as its character.
This gives a representation-theoretic proof of Butler's conjecture.
The same construction gives lower bounds for the dimensions in the $N!/k$-conjecture.
\end{abstract}

\maketitle

\begingroup
\small % Выбираем нужный размер: \small, \footnotesize или \scriptsize
\tableofcontents
\endgroup

\setcounter{section}{-1}

\section{Introduction}
\label{sec::intro}

\subsection{The $N!$-, the $\frac{N!}{k}$- and the Science Fiction conjectures}
\label{sec::intro::history}

In 1988, I.\,G.~Macdonald introduced a two-parameter family of symmetric functions, now known as \emph{Macdonald polynomials}. They have attracted much attention and remain among the most studied families of symmetric polynomials, because of their applications in many areas of mathematics (see, e.g.,~\cite[\S1,2]{Cherednik::book} and references therein). In his original paper Macdonald conjectured that the coefficients of the Schur expansion of the (suitably modified) Macdonald polynomials $\widetilde{H}_\lambda(X;q,t)$, the $(q,t)$-Kostka polynomials, are polynomials in $q$ and $t$ with nonnegative integer coefficients~\cite{Macdonald1988}. This \emph{Macdonald positivity conjecture} quickly became a central problem of algebraic combinatorics.

To approach it, building on the work of Garsia and Procesi~\cite{GarsiaProcesi1992} on the one-parameter case, A.~Garsia and M.~Haiman~\cite{Garsia_Haiman} attached to a Young diagram $\lambda\vdash N$ with cells $(a_j,b_j)$ the alternating polynomial
\[
\Delta_\lambda=\det\big(x_i^{a_j}y_i^{b_j}\big)_{i,j=1}^N\in\bC[x_1,y_1,\ldots,x_N,y_N]
\]
and the bigraded $\bS_N$-module $\mGH^{\bS}_\lambda$ spanned by all partial derivatives of $\Delta_\lambda$. They conjectured that $\dim\mGH^{\bS}_\lambda=N!$, i.e., that $\mGH^{\bS}_\lambda$ is a bigraded version of the regular representation of $\bS_N$ (the \emph{$N!$-conjecture}). Haiman showed in~\cite{Haiman::MacdonaldGeometry} that the $N!$-conjecture implies that the bigraded Frobenius characteristic of $\mGH^{\bS}_\lambda$ is $\widetilde{H}_\lambda(X;q,t)$; in particular it implies Macdonald positivity. The $N!$-conjecture was proved by Haiman in~\cite{Haiman_n!}, with the help of the geometry of the Hilbert scheme of points on the plane and of the isospectral Hilbert scheme. By the work of Kumar and Thomsen~\cite{Kumar-n!}, it is enough to prove it for staircase diagrams. Several other proofs of Macdonald positivity have been found~\cite{Gordon,Grojnowski,Losev_Procesi,Bezruk_Kaledin,Etingof_Ginzburg}, but to the best of our knowledge Haiman's remains the only proof of the $N!$-conjecture itself. For combinatorial aspects, see the books~\cite{Bergeron,Haglund} and the combinatorial formula~\cite{HHL}; explicit expressions for the coefficients are also studied in~\cite{Corteel,Ayyer_1,Ayyer_2}, and the refined positivity statements conjectured by Chen~\cite{Chen-Haiman} were recently proved in~\cite{Kato-Plethysm}.

The intersections of Garsia--Haiman modules form a natural next family of objects. Already in~\cite{Garsia_Haiman::96} Garsia and Haiman observed that for a diagram $\lambda\vdash N+1$ with two corners $c_1,c_2$, the intersection $\mGH^{\bS}_{\lambda\setminus c_1}\cap\mGH^{\bS}_{\lambda\setminus c_2}$ is expected to be a module of dimension $\frac{N!}{2}$ (the \emph{$\frac{N!}{2}$-conjecture}), and that this could give an inductive proof of the $N!$-conjecture. Bergeron and Garsia~\cite{Bergeron_Garsia::SF} (see also~\cite{BBGHT}) generalized this to an arbitrary number of corners.

\begin{conjecture*}[the $\frac{N!}{k}$-conjecture, Bergeron--Garsia]
Let $\lambda\vdash N+1$ be a Young diagram, let $c_1,\ldots,c_k$ be pairwise different corners of $\lambda$, and let $\nu_i=\lambda\setminus\{c_i\}\vdash N$. Then
\[
\dim\big(\mGH^{\bS}_{\nu_1}\cap\ldots\cap\mGH^{\bS}_{\nu_k}\big)=\frac{N!}{k},
\]
and the bigraded Frobenius characteristic of this intersection is the \emph{Macdonald intersection polynomial}
\begin{equation}
\label{eq::inters::polynomial}
\sum_{i=1}^k\Big(\prod_{j\neq i}\frac{T_{\nu_j}}{T_{\nu_j}-T_{\nu_i}}\Big)\widetilde{H}_{\nu_i}(X;q,t),\qquad T_\nu=t^{n(\nu)}q^{n(\nu^T)}.
\end{equation}
\end{conjecture*}

For $k=1$ this is the $N!$-theorem. For $k=2$ the intersection polynomial is Butler's function $\frac{T_{\nu_1}\widetilde{H}_{\nu_2}-T_{\nu_2}\widetilde{H}_{\nu_1}}{T_{\nu_1}-T_{\nu_2}}$~\cite{Butler}, and Butler conjectured that it is Schur positive with coefficients in $\bZ_{\geq0}[q,t]$. The $\frac{N!}{k}$-conjecture is a part of the much broader \emph{Science Fiction conjecture}~\cite{Bergeron_Garsia::SF}, which concerns the spaces of derivatives of the determinants attached to arbitrary lattice diagrams, with cells removed. The $\frac{N!}{k}$-conjecture is open in general: for $k=2$ it is known for hooks~\cite{Bergeron_Hamel,Armon}, and Butler's positivity was proved in several special cases in~\cite{KLO::Butler}. While this manuscript was being prepared, a geometric proof of Butler's conjecture appeared on the arXiv~\cite{Guo:Kang:Xiong::Butler:conjecture}: it realizes the intersection polynomial as the Frobenius characteristic of a quotient of $\mGH^{\bS}_{\nu_1}$ constructed from the Procesi bundle. 
Below we provide an algebraic proof of Butler's conjecture (Corollary~\ref{cor::intro::Butler} below) of a different nature: it uses fusion products, and it provides an explicit algebraic description of the underlying module. The value of the intersection polynomial~\eqref{eq::inters::polynomial} at $q=t=1$ was computed in~\cite{KLO::shuffle}: it has dimension $\frac{N!}{k}$, and depends only on $N$ and $k$ (see Theorem~\ref{thm::KLO::specialization}). Apart from computer verifications in small cases (see Remark~\ref{rem::N!/k::evidence}), nothing seems to be known for $N!/k$ conjecture for $k\geq3$.

\subsection{Fusion products}
\label{sec::intro::fusion}

\emph{Fusion products}, or \emph{graded tensor products}, were introduced by B.~Feigin and S.~Loktev in~\cite{FeLo} in connection with generalized Kostka polynomials and the quantum Verlinde rule. Let $\fg$ be a Lie algebra and $\fg[z]:=\fg\otimes\bC[z]$ its current algebra. Given cyclic $\fg$-modules $M_0,\ldots,M_l$ with cyclic vectors $v_0,\ldots,v_l$, and pairwise distinct numbers $\gamma_0,\ldots,\gamma_l$, the tensor product of the evaluation modules $\bigotimes_i\mathsf{ev}_{\gamma_i}(M_i)$ is cyclic over $\fg[z]$ with the cyclic vector $v_0\otimes\ldots\otimes v_l$, and it is filtered by the subspaces $U(\fg[z])_{\leq s}\cdot(v_0\otimes\ldots\otimes v_l)$, where $U(\fg[z])_{\leq s}$ is spanned by the products $(x_1z^{r_1})\cdots(x_pz^{r_p})$ with $r_1+\ldots+r_p\leq s$. The associated graded space is a $z$-graded cyclic $\fg[z]$-module, the \emph{fusion product} $M_0\star_zM_1\star_z\cdots\star_zM_l$ (see \S\ref{sec::fusion} for details). It has the same dimension as $M_0\otimes\ldots\otimes M_l$, and it can be fused again, which gives iterated fusion products.

Fusion products turned out to be very useful in the representation theory of current algebras. For $\fg=\msl_{r}$ the fusion product of fundamental representations does not depend on the parameters and is isomorphic to a local Weyl module of Chari and Pressley~\cite{CP,CL}. The same observation for fusion products holds  for all simple $\fg$ -- simply-laced types was shown in~\cite{FL1,FL2} and for non-simply-laced types~\cite{N}. Fusion products of Kirillov--Reshetikhin modules were studied in~\cite{Naoi2012} (see also~\cite{Kedem}). Weyl modules are related to Demazure modules, to the Macdonald polynomials at $t=0$ and to their nonsymmetric analogues~\cite{S,I,FMO,FeMa3}, and to semi-infinite flag varieties~\cite{FiMi,BF1,BF2,BF3,Kat,KNS}. A geometric approach to fusion products, via Beilinson--Drinfeld Schubert varieties and the coherent Satake category, was developed in~\cite{DFF,Dumanski-Feigin,Dumanski} confirming associativity of fusion products of certain families of irreducible representations.

The fusion product is, however, a very implicit construction, and the following basic questions are open in general. Does the fusion product depend on the evaluation parameters? Is it associative, i.e., is the iterated fusion product
\[
\Big[\big[M_{0,1}\star_z\cdots\star_zM_{0,k_0}\big]\star_z\cdots\star_z\big[M_{l,1}\star_z\cdots\star_zM_{l,k_l}\big]\Big]
\]
isomorphic to the fusion product $\star_{i,j}M_{i,j}$ of all the modules at once? Feigin and Loktev conjectured that the answer is affirmative (Conjecture~\ref{conj::fusion::associativity}). Both properties are known only for very special families, such as the Weyl modules and the Kirillov--Reshetikhin modules mentioned above (see~\cite{CL,Kedem,Dumanski-Feigin,Dumanski} and \S\ref{sec::fusion::graded}).

There is a dual construction. If a finite-dimensional $\fg[z]$-module $M$ has a \emph{cocyclic} vector $w$ (a vector contained in every nonzero submodule), one defines the \emph{cofusion filtration} of $M$ by the minimal $z$-degree of an element of $U(\fg[z])$ needed to reach $w$ from a given vector (\S\ref{sec::cofusion}). Its associated graded space, the \emph{cofusion product}, is the dual of the fusion product of the dual modules.

\subsection{Statements of the main results}
\label{sec::intro::results}

Let $V$ be a complex vector space of dimension $\dim V\geq 3N+1$, let $x,y$ be two variables, and let $\mathfrak{G}=\gl_V[x,y]=\gl_V\otimes\bC[x,y]$. For a Young diagram $\lambda=(\lambda_0\geq\ldots\geq\lambda_l)\vdash N$ we put
\[
\sW^x_{h\varepsilon_1}:=\underbrace{V\star_x\cdots\star_xV}_{h},
\qquad
\sWH_\lambda:=\sW^x_{\lambda_0\varepsilon_1}\star_y
\sW^x_{\lambda_1\varepsilon_1}\star_y\cdots\star_y
\sW^x_{\lambda_l\varepsilon_1},
\]
the iterated fusion product of $N$ copies of the tautological representation $V$: first along the rows of $\lambda$ with respect to the variable $x$ (the result is the local Weyl module $\sW^x_{h\varepsilon_1}$ of $\gl_V[x]$), and then with respect to the variable $y$ (\S\ref{sec::iterated::fusion}). It is a bigraded cyclic $\mathfrak{G}$-module, generated by the vector of weight $N\varepsilon_1$ and isomorphic to $V^{\otimes N}$ as a $\gl_V$-module.

Our main tool is a version of the Schur--Weyl duality for the current algebra (Theorem~\ref{thm::SW::currents}). For $\dim V\geq N$, the category of finite-dimensional polynomial $\mathfrak{G}$-modules of degree $N$ is equivalent to the category of finite-dimensional modules over $\bC[\bS_N]\ltimes\bC[x_1,y_1,\ldots,x_N,y_N]$, by taking the weight subspace of weight $(1^N,0^{m-N})$. Under this equivalence the trivial and the sign representations of $\bS_N$ correspond to $S^NV$ and $\Lambda^NV$, the polynomial ring corresponds to $S^N(V[x,y])$, and the Garsia--Haiman module $\mGH^{\bS}_\lambda\simeq\bC[x_1,\ldots,y_N]/\euJ^{\bS}_\lambda$ corresponds to a cyclic and cocyclic $\mathfrak{G}$-module
\[
\mGH^{\gl}_\lambda=S^N(V[x,y])/\euJ^{\gl}_\lambda ,
\]
the quotient by the largest submodule on which a natural projection onto $\Lambda^NV$ vanishes (\S\ref{sec::GH::SW}). The $N!$-conjecture says that $\mGH^{\gl}_\lambda\simeq V^{\otimes N}$ as a $\gl_V$-module, and the $\frac{N!}{k}$-conjecture becomes a statement on the common quotient $\mGH^{\gl}_{\nu_1,\ldots,\nu_k}=S^N(V[x,y])/(\euJ^{\gl}_{\nu_1}+\ldots+\euJ^{\gl}_{\nu_k})$, which is Schur--Weyl dual to the intersection of Garsia--Haiman modules. After this translation, the symmetric group disappears from all our arguments.

Recall that $\bomega$ denotes the sign duality functor, which is the Schur--Weyl dual of the tensoring with the sign representation; it exchanges $S^kV$ and $\Lambda^kV$ and is compatible with tensor products (\S\ref{sec::GH::sign}). We write $U^\circ:=\bomega(U)^\dagger$ for the contravariant sign dual.

\begin{theoremA}[Theorems~\ref{thm::fusion::cofusion} and~\ref{thm::N!::equiv::coincidence}]
\label{thm::intro::A}
Let $\lambda\vdash N$ be a Young diagram.
\begin{enumerate}
\item There is a canonical surjection of bigraded $\mathfrak{G}$-modules $\sWH_\lambda\twoheadrightarrow\mGH^{\gl}_\lambda$, which is an isomorphism on the top component $\Lambda^NV$.
\item There is a canonical embedding of bigraded $\mathfrak{G}$-modules $\mGH^{\gl}_\lambda\hookrightarrow \sW_{\lambda_0\varepsilon_1}^x\circledast_y\cdots\circledast_y\sW_{\lambda_l\varepsilon_1}^x\langle\mdeg(\Delta_\lambda)\rangle$ of the Garsia--Haiman module into the iterated cofusion product of the local Weyl modules, which is an isomorphism on the bottom component $S^NV$ (of bidegree $(0,0)$).
\item The following conditions are equivalent:
\begin{enumerate}
\item the $N!$-conjecture holds for $\lambda$;
\item the surjection of item~(i) is an isomorphism, i.e., the iterated fusion product $\sWH_\lambda$ of the tautological representations coincides with the Garsia--Haiman module $\mGH^{\gl}_\lambda$; equivalently, the embedding of item~(ii) is an isomorphism;
\item the fusion filtration and the cofusion filtration of the tensor product $$\sW_{\lambda_0\varepsilon_1}^{x}\otimes\cdots\otimes \sW_{\lambda_l\varepsilon_1}^{x}$$ 
of local Weyl $\gl_V[x]$-modules coincide: $F_s=G^{n(\lambda)-s}$ for all $s$. Here $F$ is the fusion filtration with respect to the vector $e_1^{\otimes N}$ of weight $N\varepsilon_1$, and $G$ is the cofusion filtration with respect to the cocyclic vector of $\Lambda^NV$.
\end{enumerate}
\end{enumerate}
In particular, by Haiman's theorem, these conditions hold for all diagrams. Then:
\begin{enumerate}
\item[(iv)] for all pairwise distinct parameters,
\[
\sW^x_{\lambda_0\varepsilon_1}\star_y
\cdots\star_y
\sW^x_{\lambda_l\varepsilon_1}\simeq\ \mGH^{\gl}_\lambda\ \simeq \
\sW_{\lambda_0\varepsilon_1}^x\circledast_y\cdots\circledast_y\sW_{\lambda_l\varepsilon_1}^x\big\langle\mdeg(\Delta_\lambda)\big\rangle ,
\]
where each $W_h$ is itself both a fusion and (up to a shift of the grading) a cofusion product of $h$ copies of $V$ with respect to $x$;
\item[(v)] (associativity) for all Young diagrams $\lambda^1,\ldots,\lambda^k$ and all pairwise distinct parameters,
\[
\mGH^{\gl}_{\lambda^1}\star_y\cdots\star_y\mGH^{\gl}_{\lambda^k}\ \simeq\ \mGH^{\gl}_{\lambda^1\cup\ldots\cup\lambda^k}\ \simeq\ \mGH^{\gl}_{\lambda^1}\circledast_y\cdots\circledast_y\mGH^{\gl}_{\lambda^k}\big\langle\mdeg(\Delta_{\lambda^1\cup\ldots\cup\lambda^k})\big\rangle,
\]
where $\lambda^1\cup\ldots\cup\lambda^k$ is the diagram whose rows are all the rows of $\lambda^i$. In particular, the fusion product of the modules $\mGH^{\gl}_{\lambda^i}$ with respect to their cyclic vectors, and their cofusion product with respect to the cocyclic vectors, do not depend on the parameters and are associative.
\end{enumerate}
\end{theoremA}

Thus the Garsia--Haiman modules provide a family of cyclic and cocyclic modules, for which the independence of the parameters and the associativity conjectured by Feigin and Loktev hold both for the fusion and for the cofusion product. 
Note, that Theorem~\ref{thm::intro::A} does not give a new proof of Haiman's theorem; it gives a reformulation of it (see Remark~\ref{rem::N!:reformulation}).
However, we hope that the fusion product description will lead to an alternative proof of the $N!$-conjecture, which would avoid the cumbersome geometry of polygraphs and of the isospectral Hilbert scheme used in Haiman's proof.

The proof is organized as follows. First, we study the top-degree components of $\sWH_\lambda$ using the properties of local Weyl modules and of fusion products over abelian Lie algebras (Theorem~\ref{thm::top::fusion}); this gives the surjection of item~(i). Second, the sign duality $U\mapsto U^\circ$ exchanges fusion and cofusion (Proposition~\ref{prp::sign::cofusion}), and the Garsia--Haiman modules are self-dual with respect to it (Theorem~\ref{thm::GH::self::dual}); a dimension count then gives the coincidence of the filtrations.

Let $\mu,\nu\vdash N$ be two Young diagrams obtained from a diagram $\lambda\vdash N+1$ by removing two different corners. We construct an explicit iterated fusion product $\mGH'_{\mu,\nu}$, recursively from $S^2V$ (the common quotient for $N=2$), by fusion with local Weyl modules and by the transposition $x\leftrightarrow y$ (\S\ref{sec::Butler::lower}). For example, if $\mu=(2^{k},1^{l+2})$ and $\nu=(2^{k+1},1^{l})$, then
\[
\mGH'_{\mu,\nu}=S^2V\star_y\underbrace{W_2\star_y\cdots\star_yW_2}_{k}\star_y\underbrace{V\star_y\cdots\star_yV}_{l},
\]
which is isomorphic to $S^2V\otimes V^{\otimes(N-2)}$ as a $\gl_V$-module; the same holds for all pairs.

\begin{theoremA}[Theorem~\ref{thm::Butler::exact} and Corollary~\ref{cor::Butler::lower}]
\label{thm::intro::B}
Let $\mu,\nu\vdash N$ be obtained from the same diagram by removing two different corners.
\begin{enumerate}
\item The common quotient $\mGH^{\gl}_{\mu,\nu}$ surjects onto $\mGH'_{\mu,\nu}$. In particular, the weight subspace of weight $(1^N,0^{m-N})$ of $\mGH^{\gl}_{\mu,\nu}$ (equivalently, the intersection $\mGH^{\bS}_\mu\cap\mGH^{\bS}_\nu$) has dimension at least $\frac{N!}{2}$, and the $\frac{N!}{2}$-conjecture for $\mu,\nu$ is equivalent to the statement that this surjection is an isomorphism.
\item There are short exact sequences of bigraded $\mathfrak{G}$-modules
\[
0\longrightarrow\big(\mGH'_{\mu,\nu}\big)^{\circ}\longrightarrow\mGH^{\gl}_\mu\longrightarrow\mGH'_{\mu,\nu}\longrightarrow0,
\qquad
0\longrightarrow\big(\mGH'_{\mu,\nu}\big)^{\circ}\longrightarrow\mGH^{\gl}_\nu\longrightarrow\mGH'_{\mu,\nu}\longrightarrow0,
\]
(with appropriate shifts of the grading) with the same module $\mGH'_{\mu,\nu}$.
\end{enumerate}
\end{theoremA}

Taking the characters of the two short exact sequences, and using Theorem~\ref{thm::intro::A}, we obtain a representation-theoretic proof of Butler's conjecture.

\begin{corollaryA}[Butler's conjecture; Corollary~\ref{cor::Butler}]
\label{cor::intro::Butler}
The bigraded character of $\mGH'_{\mu,\nu}$, in which $q$ and $t$ record the $x$- and the $y$-degree, is the Macdonald intersection polynomial
\[
\mathrm{ch}_{q,t}\big(\mGH'_{\mu,\nu}\big)=\frac{T_\mu\,\widetilde{H}_\nu-T_\nu\,\widetilde{H}_\mu}{T_\mu-T_\nu}\,,
\]
In particular, it is Schur positive with coefficients in $\bZ_{\geq0}[q,t]$.
\end{corollaryA}

This proof uses Haiman's theorem together with the description of the Garsia--Haiman modules as fusion products (Theorem~\ref{thm::intro::A}). In contrast with the geometric proof mentioned in \S\ref{sec::intro::history}, the module with the Schur positive character is described explicitly, as an iterated fusion product.

The fusion product construction generalizes to an arbitrary number of corners. For $k$ corners of a diagram $\lambda\vdash N+1$ we reduce the common quotient to the common quotient $Z_k$ of the staircase diagram $(k,k-1,\ldots,1)$, which is the smallest diagram with $k$ corners (of size $N_k+1$, where $N_k=\frac{k(k+1)}{2}-1$).

\begin{theoremA}[Theorem~\ref{thm::Nk::lower} and Corollary~\ref{cor::Nk::lower}]
\label{thm::intro::D}
Let $\nu_1,\ldots,\nu_k\vdash N$ be obtained from a diagram $\lambda\vdash N+1$ by removing $k$ different corners. There is an explicit iterated fusion product $T$ of $Z_k$ with local Weyl modules (and a twist exchanging $x$ and $y$) and a surjection $\mGH^{\gl}_{\nu_1,\ldots,\nu_k}\twoheadrightarrow T$, where $T\simeq Z_k\otimes V^{\otimes(N-N_k)}$ as a $\gl_V$-module. Consequently,
\[
\dim\big(\mGH^{\bS}_{\nu_1}\cap\ldots\cap\mGH^{\bS}_{\nu_k}\big)\ \geq\ \frac{N!}{N_k!}\,\dim\big(\mGH^{\bS}_{\nu^{\mathrm{st}}_1}\cap\ldots\cap\mGH^{\bS}_{\nu^{\mathrm{st}}_k}\big),
\]
where $\{\nu^{\mathrm{st}}_i \colon i=1\ldots,k\}$ are the $k$ different diagrams of size $N_k$ obtained from the staircase by removing its corners. Hence, if the $\frac{N!}{k}$-conjecture holds for the staircase, then $\dim\geq\frac{N!}{k}$ holds for all diagrams, and the conjecture for a given diagram is equivalent to the injectivity of the above surjection. For $k\leq3$ the staircase case holds, so the lower bound $\frac{N!}{k}$ is unconditional.
\end{theoremA}

For $k=3$ the staircase case is a computer calculation (exact linear algebra over $\bQ$, \S\ref{sec::Butler::Nk}). We also list there the bigraded characters of $Z_3$ and $Z_4$ (dimensions $40=\frac{5!}{3}$ and $90720=\frac{9!}{4}$).

Finally, we show that on the tensor product of the Demazure modules realized inside level one modules of $\widehat{\gl}_V$ (which are the local Weyl modules above), the fusion filtration defined by the current algebra $\gl_V[x]$ coincides with the one defined by the affine algebra $\widehat{\gl}_V$ (Theorem~\ref{thm::affine::current}).

\subsection{Structure of the paper}
\label{sec::intro::structure}

\S\ref{sec::partitions} fixes the notation for Young diagrams and contains an elementary combinatorial lemma on monotone fillings of skew diagrams. In \S\ref{sec::SW} we prove the Schur--Weyl duality for current algebras, compute graded characters and introduces the sign duality functor. \S\ref{sec::GH} recalls the Garsia--Haiman modules, the $N!$- and the $\frac{N!}{k}$-conjectures, translates them to $\gl_V[x,y]$-modules, shows the self-duality of the Garsia--Haiman modules. \S\ref{sec::fusion} is devoted to the fusion product: we recall its definition, the iterated and graded versions, the associativity conjecture, the vanishing lemma for abelian Lie algebras, and the behaviour of fusion products under surjections and common quotients. In \S\ref{sec::Weyl} we collect the properties of local Weyl modules needed later. \S\ref{sec::iterated::fusion} introduces the iterated fusion products $\sWH_\lambda$, describes their top-degree components, constructs the surjection onto $\mGH^{\gl}_\lambda$ and proves the associativity of the fusion product of the modules $\sWH_\lambda$. In \S\ref{sec::cofusion} we define the cofusion product, show that sign duality exchanges fusion and cofusion, and prove that the fusion and cofusion filtrations coincide; this completes the proof of Theorem~\ref{thm::intro::A}. \S\ref{sec::Butler} contains the results on two diagrams (Theorem~\ref{thm::intro::B} and Corollary~\ref{cor::intro::Butler}) and the lower bound for the $\frac{N!}{k}$-conjecture (Theorem~\ref{thm::intro::D}). The last \S\ref{sec::affine} compares the fusion filtrations defined by the current and by the affine algebra.

\subsection*{Acknowledgements}
We would like to thank Boris Feigin, Sergei Loktev and Ilya Dumanski for stimulating conversations on fusion products.

{\it{Statement on the use of AI:}}
AI tools were used for the proofreading and the language polishing of this manuscript, for drafting some parts of the text following the instructions of the authors, and for computer experiments on the dimensions in the $\frac{N!}{k}$-conjecture for small $k$. The final version of the text was proofread and edited by the authors alone.

\section{Young diagrams}
\label{sec::partitions}

In this section we fix the notation for partitions and Young diagrams, and we prove an elementary combinatorial lemma on monotone fillings of skew diagrams (Lemma~\ref{prp::subdiag}), which is not used in the rest of the text but is relevant for future applications.

\subsection{Partitions and Young diagrams}

A \emph{partition} $\lambda\vdash N$ of a positive integer $N$ is a non-increasing sequence of positive integers $\lambda=(\lambda_0\geq\lambda_1\geq\ldots\geq\lambda_l>0)$ with $\lambda_0+\ldots+\lambda_l=N$. We write $|\lambda|:=N$ and $\ell(\lambda):=l+1$ for the number of parts.
We use the French convention for the visualization of partitions: the \emph{Young diagram} of $\lambda$ is the union of unit squares (\emph{cells}) in the first quadrant, with the longest row at the bottom. A cell is indexed by the coordinates of its bottom-left corner, so that
\[
\lambda=\{\bar{s}=(s_x,s_y)\in\bZ_{\geq0}^{2}\colon\ s_y\leq l,\ s_x<\lambda_{s_y}\}.
\]
We do not distinguish between a partition and its Young diagram. In particular, $s_y$ is the number of the row of the cell $\bar{s}$ and $s_x$ is its position in this row, and the number of cells of $\lambda$ equals $N$.
Equivalently, a Young diagram is a finite subset $\lambda\subset\bZ_{\geq0}^2$ such that $(s_x,s_y)\in\lambda$ implies $(s_x',s_y')\in\lambda$ for all $0\leq s_x'\leq s_x$ and $0\leq s_y'\leq s_y$.
The figure below shows the partition $(6,4,4,1)\vdash 15$, together with the coordinates of its cells.

\[
\begin{array}{cc}
\text{Partition $(6,4,4,1) \vdash 15$} & \text{The same partition with cell coordinates} \\
\begin{tikzpicture}[scale=0.65,shift={(0,-2)}]
  \foreach \x in {0,1,2,3,4,5} {
    \draw (\x,0) rectangle ++(1,1);
  }
  \foreach \x in {0,1,2,3} {
    \draw (\x,1) rectangle ++(1,1);
    \draw (\x,2) rectangle ++(1,1);
  }
  \draw (0,3) rectangle ++(1,1);
  \draw[->, thick] (0,0) -- (7,0) node[right] {$s_x$};
  \draw[->, thick] (0,0) -- (0,5) node[above] {$s_y$};
\end{tikzpicture}
&
\begin{tikzpicture}[scale=0.65,shift={(0,-2)}]
  \foreach \x in {0,1,2,3,4,5} {
    \draw (\x,0) rectangle ++(1,1);
    \node at (\x+0.5,0.5) {\tiny($\x$,0)};
  }
  \foreach \x in {0,1,2,3} {
    \draw (\x,1) rectangle ++(1,1);
    \node at (\x+0.5,1.5) {\tiny($\x$,1)};
    \draw (\x,2) rectangle ++(1,1);
    \node at (\x+0.5,2.5) {\tiny($\x$,2)};
  }
  \draw (0,3) rectangle ++(1,1);
  \node at (0.5,3.5) {\tiny(0,3)};
  \draw[->, thick] (0,0) -- (7,0) node[right] {$s_x$};
  \draw[->, thick] (0,0) -- (0,5) node[above] {$s_y$};
\end{tikzpicture}
\end{array}
\]

The $j$-th row of $\lambda$ has $\lambda_j$ cells, and the $i$-th column has $\lambda^T_i$ cells, where $\lambda^T$ is the \emph{transposed} partition (the reflection of the diagram in the diagonal $s_x=s_y$):
\[
\lambda_j=\#\{\bar{s}\in\lambda\colon s_y=j\},\qquad
\lambda^T_i=\#\{\bar{s}\in\lambda\colon s_x=i\}.
\]
We also set
\begin{equation}
\label{eq::n::lambda}
n(\lambda):=\sum_{j}j\,\lambda_j=\sum_{\bar{s}\in\lambda}s_y,
\qquad
n(\lambda^T)=\sum_{j}\binom{\lambda_j}{2}=\sum_{\bar{s}\in\lambda}s_x .
\end{equation}

\begin{definition}
\label{def::skew}
For two Young diagrams we write $\mu\subset\lambda$ if every cell of $\mu$ is a cell of $\lambda$. In this case the set-theoretic difference $\lambda\setminus\mu$ is called the \emph{skew diagram}.
\end{definition}

\begin{definition}
\label{def::diagram::corner}
A cell $\bar{c}\in\lambda$ is called a \emph{corner} of $\lambda$ if $\lambda\setminus\{\bar{c}\}$ is again a Young diagram, i.e., if $\lambda\setminus\mu=\{\bar{c}\}$ for some Young diagram $\mu\subset\lambda$.
In terms of the rows of $\lambda=(\lambda_0\geq\ldots\geq\lambda_l)$, the corners are the cells $(\lambda_j-1,j)$ with $\lambda_j>\lambda_{j+1}$, where $\lambda_{l+1}:=0$.
\end{definition}

\subsection{Monomial ideals}
To a cell $\bar{s}=(s_x,s_y)$ we associate the monomial $\bar{z}^{\bar{s}}:=x^{s_x}y^{s_y}\in\bC[x,y]$, where $\bar{z}=(x,y)$.
For a Young diagram $\lambda$ let $\mathsf{J}_\lambda\subset\bC[x,y]$ be the ideal spanned by the monomials $\bar{z}^{\bar{s}}$ with $\bar{s}\notin\lambda$. This is the monomial ideal generated by the monomials of the cells just outside of $\lambda$. The quotient $\bC[x,y]/\mathsf{J}_\lambda$ has the basis $\{\bar{z}^{\bar{s}}\colon\bar{s}\in\lambda\}$, so its dimension is $|\lambda|$. This gives a bijection between Young diagrams of size $N$ and monomial ideals of codimension $N$.
Moreover, $\mu\subset\lambda$ if and only if $\mathsf{J}_\mu\supset\mathsf{J}_\lambda$, and in this case $\dim\mathsf{J}_\mu/\mathsf{J}_\lambda=|\lambda\setminus\mu|$.

\subsection{Monotone fillings of skew diagrams}

\begin{definition}
\label{def::filling}
Let $\bar{r}_1,\ldots,\bar{r}_K\in\bZ_{\geq0}^{2}$. A \emph{monotone filling} of a Young diagram $\lambda$ by $\bar{r}_1,\ldots,\bar{r}_K$ is an injective map $\psi\colon\{1,\ldots,K\}\hookrightarrow\lambda$ such that $\psi(i)-\bar{r}_i\in\bZ_{\geq0}^{2}$ for all $i$, i.e., both coordinates of $\psi(i)$ are greater than or equal to the corresponding coordinates of $\bar{r}_i$.
\end{definition}

\begin{lemma}
\label{prp::subdiag}
Let $\lambda$ be a Young diagram and let $\bar{r}_1,\ldots,\bar{r}_K\in\bZ_{\geq0}^2$. For a Young diagram $\mu\subset\lambda$ set
\[
\mu(\bar{r}):=\{\,i\in\{1,\ldots,K\}\ \colon\ \bar{z}^{\bar{r}_i}\in\mathsf{J}_\mu\,\}=\{\,i\ \colon\ \bar{r}_i\notin\mu\,\}.
\]
The following conditions are equivalent:
\begin{enumerate}[label=(\roman*)]
\item there exists a monotone filling of $\lambda$ by $\bar{r}_1,\ldots,\bar{r}_K$;
\item for every Young diagram $\mu\subset\lambda$ we have $|\mu(\bar{r})|\leq|\lambda\setminus\mu|$.
\end{enumerate}
\end{lemma}
\begin{proof}
(i)$\Rightarrow$(ii). Let $\psi$ be a monotone filling and let $\mu\subset\lambda$. If $i\in\mu(\bar{r})$, then $\psi(i)\notin\mu$: otherwise $\bar{r}_i\leq\psi(i)$ would imply $\bar{r}_i\in\mu$, because $\mu$ is closed under decreasing coordinates. Hence $\psi$ maps $\mu(\bar{r})$ injectively to $\lambda\setminus\mu$.

(ii)$\Rightarrow$(i). For $i=1,\ldots,K$ let $F_i:=\{\bar{s}\in\lambda\colon\bar{s}\geq\bar{r}_i\}$ be the set of cells where $\bar{r}_i$ may be placed. A monotone filling is a choice of pairwise distinct representatives $\psi(i)\in F_i$. By Hall's marriage theorem~\cite{Hall} it exists if and only if $|\bigcup_{i\in I}F_i|\geq|I|$ for every subset $I\subset\{1,\ldots,K\}$.
Fix $I$ and set $\mu:=\lambda\setminus\bigcup_{i\in I}F_i$. This is a Young diagram: if $\bar{s}\in\mu$ and $\bar{s}'\leq\bar{s}$, then $\bar{s}'\in\lambda$, and $\bar{s}'\in F_i$ for some $i\in I$ would imply $\bar{s}\in F_i$. Moreover, $I\subset\mu(\bar{r})$, since $\bar{r}_i\in\mu$ would imply $\bar{r}_i\in\lambda$ and hence $\bar{r}_i\in F_i$ for $i\in I$, contradicting the definition of $\mu$. Therefore, by~(ii),
\[
|I|\leq|\mu(\bar{r})|\leq|\lambda\setminus\mu|=\Big|\bigcup_{i\in I}F_i\Big|.\qedhere
\]
\end{proof}

\begin{remark}
If $K>|\lambda|$, both conditions of Lemma~\ref{prp::subdiag} fail (take $\mu=\emptyset$ in~(ii)).
\end{remark}

\section{Schur--Weyl Duality for Current Algebras}
\label{sec::SW}
\subsection{Classical Schur-Weyl Duality}

We begin by recalling several equivalent formulations of the classical Schur-Weyl duality, originally discovered by Issai Schur~\cite{Schur} and later developed by Hermann Weyl in~\cite{HW_invariants}.

\begin{theorem}
There exists an isomorphism of $\GL_V \times \bS_N$-modules:
\[
V^{\otimes N} \simeq \bigoplus_{\substack{\lambda \vdash N \\ l(\lambda) \le \dim V}} V_\lambda \otimes \mathbb{S}_\lambda.
\]
Here, $V_\lambda$ denotes the irreducible $\GL_V$-module of highest weight $\lambda$, and $\mathbb{S}_\lambda$ is the irreducible $\bS_N$-module corresponding to the partition $\lambda$, also known as the Specht module.
\end{theorem}

Let $m:=\dim V$, fix a basis $e_1,\ldots,e_m$ of $V$ and denote by $E_{ij}\in\gl_V$ the corresponding matrix units.
The image of $U(\gl_V)$ in $\mathrm{End}(V^{\otimes N})$ is the \emph{Schur algebra}
\[
\mathcal{S}(V,N):=\mathrm{End}_{\bS_N}(V^{\otimes N}) = \big(\mathrm{End}(V)^{\otimes N}\big)^{\bS_N},
\]
and the category $\RepGL{V}{N}$ of degree-$N$ polynomial representations of $\gl_V$ coincides with the category of finite-dimensional $\mathcal{S}(V,N)$-modules.
From now on we assume $m\geq N$ and denote by $\pi\in\mathcal{S}(V,N)$ the projector onto the weight subspace of weight $(1^N,0^{m-N}) := (\underbrace{1,\ldots,1}_N,\underbrace{0,\ldots,0}_{m-N})$, so that $\pi U = U_{(1^N, 0^{m-N})}$ for any $U\in\RepGL{V}{N}$.

Schur-Weyl duality defines a pair of adjoint functors between the category $\mathsf{Rep}(\bS_N)$ of representations of the symmetric group $\bS_N$ and the category $\RepGL{V}{N}$:
\begin{equation}
\label{eq::Schur-Weyl}
\begin{tikzcd}
\mathsf{Rep}(\bS_N) \arrow[rrr, shift left=1ex, "{\SW\colon M \mapsto  V^{\otimes N} \otimes_{\bS_N} M}"]
& & & \RepGL{V}{N} \arrow[lll, "{\WS\colon U \mapsto U_{(1^N, 0^{m-N})}}"].
\end{tikzcd}
\end{equation}
We call the upper functor $\SW$ the \emph{Schur--Weyl functor} and the lower functor $\WS$ the \emph{conjugate Schur--Weyl functor}. The functor $\SW$ is left adjoint to $\WS$, because for all $U\in\RepGL{V}{N}$ we have
\[
\mathrm{Hom}_{\gl_V}(V^{\otimes N},U)\simeq U_{(1^N, 0^{m-N})}.
\]
For $m\geq N$ Adjunction~\eqref{eq::Schur-Weyl} is an equivalence of categories. For $m > N$ the weight $(1^N,0^{m-N})$ is the $\msl_m$-weight that corresponds to the $N$-th fundamental weight $\omega_N$.

Thus, the Specht module $\mathbb{S}_\lambda$ and the irreducible $\GL_V$-module $V_\lambda$ associated to a Young diagram $\lambda$ are dual to each other under Schur-Weyl duality. Respectively, the regular representation $\bC[\bS_N]$ is Schur-Weyl dual to the tensor power $V^{\otimes N}$ of the tautological representation.

Schur-Weyl duality is compatible with characters.
Let $\Lambda=\bigoplus_{N\geq 0}\Lambda^N$ be the ring of symmetric functions with integer coefficients in the alphabet $X=(x_1,x_2,\ldots)$, where $\Lambda^N$ is the component of degree $N$. We denote by $s_\lambda$, $h_N$, $e_N$ and $p_\mu$ the Schur functions, the complete and the elementary symmetric functions, and the power sums, respectively.
The character $\mathrm{ch}(U):=\mathrm{tr}_U\,\mathrm{diag}(x_1,\ldots,x_m)$ of $U\in\RepGL{V}{N}$ is a symmetric polynomial of degree $N$ in $x_1,\ldots,x_m$. Since $m\geq N$, the restriction map from $\Lambda^N$ to symmetric polynomials of degree $N$ in $x_1,\ldots,x_m$ is an isomorphism, so we consider $\mathrm{ch}(U)$ as an element of $\Lambda^N$. Then $\mathrm{ch}(V_\lambda)=s_\lambda$.
On the other side, the Frobenius characteristic of $M\in\mathsf{Rep}(\bS_N)$ is
\[
\mathrm{Frob}(M):=\frac{1}{N!}\sum_{\sigma\in\bS_N}\chi_M(\sigma)\,p_{\rho(\sigma)}\in\Lambda^N,
\]
where $\rho(\sigma)\vdash N$ is the cycle type of $\sigma$, and $\mathrm{Frob}(\mathbb{S}_\lambda)=s_\lambda$.
Both maps induce isomorphisms of the Grothendieck groups $K_0(\RepGL{V}{N})\simeq\Lambda^N\simeq K_0(\mathsf{Rep}(\bS_N))$, and Adjunction~\eqref{eq::Schur-Weyl} induces the identity map of $\Lambda^N$:
\begin{equation}
\label{eq::SW::characters}
\mathrm{ch}\big(\SW(M)\big)=\mathrm{Frob}(M),
\qquad
\mathrm{Frob}\big(\WS(U)\big)=\mathrm{ch}(U).
\end{equation}
The first identity also follows directly from $\mathrm{tr}_{V^{\otimes N}}\big(\mathrm{diag}(x_1,\ldots,x_m)\cdot\sigma\big)=p_{\rho(\sigma)}(x_1,\ldots,x_m)$.
In particular,
\[
\mathrm{ch}(V^{\otimes N})=\mathrm{Frob}(\bC[\bS_N])=p_1^N=\sum_{\lambda\vdash N}f^\lambda s_\lambda,
\qquad f^\lambda:=\dim\mathbb{S}_\lambda .
\]

We will use the following reformulation of the classical equivalence.

\begin{lemma}
\label{lem::Schur::full}
For $m\geq N$ the idempotent $\pi$ is full, that is $\mathcal{S}(V,N)\,\pi\,\mathcal{S}(V,N)=\mathcal{S}(V,N)$.
\end{lemma}
\begin{proof}
The algebra $\mathcal{S}(V,N)$ is semisimple, $\mathcal{S}(V,N)\simeq\bigoplus_{\lambda}\mathrm{End}(V_\lambda)$, where $\lambda\vdash N$ and $l(\lambda)\le m$. The two-sided ideal generated by $\pi$ is the sum of those blocks $\mathrm{End}(V_\lambda)$ with $\pi V_\lambda\neq 0$. However, $\dim \pi V_\lambda = \dim (V_\lambda)_{(1^N,0^{m-N})}$ is the Kostka number $K_{\lambda,(1^N)}$, which equals the number of standard Young tableaux of shape $\lambda$ and is therefore positive.
\end{proof}

\subsection{Schur-Weyl Duality for Current Algebras}

We now extend the discussion to representation categories of current algebras. Fix an integer $d\geq0$ (from \S\ref{sec::GH} on we only use $d=1$, and then we write $z_0=x$, $z_1=y$, $z_{i,0}=x_i$, $z_{i,1}=y_i$) and a tuple of variables $\bar{z} := (z_0, \ldots, z_d)$ and, for $i=1,\ldots,N$, its copies $\bar{z}_i:=(z_{i,0},\ldots,z_{i,d})$.
All algebras and modules below are $\bZ^{d+1}$-graded by the multidegree, $\deg z_a=\deg z_{i,a}=\varepsilon_a$, where $\varepsilon_0,\ldots,\varepsilon_d$ is the standard basis of $\bZ^{d+1}$. All morphisms are homogeneous of degree zero, and for $k\in\bZ^{d+1}$ we denote by $M\langle k\rangle$ the shifted module, $(M\langle k\rangle)_j:=M_{j-k}$. All statements below remain true for the coarser grading by the total degree.

We abbreviate the wreath product algebra by
\[
\mathcal{A}_N:=\bC[\bS_N] \ltimes \bC[\bar{z}_1,\ldots,\bar{z}_N], \qquad \sigma\, z_{i,a}\,\sigma^{-1} = z_{\sigma(i),a},
\]
and denote by $\mathsf{Rep}(\bC[\bS_N] \ltimes \bC[\bar{z}_1,\ldots,\bar{z}_N])$ the category of finitely generated graded $\mathcal{A}_N$-modules.

The space
\[
P:=(V[\bar{z}])^{\otimes N} = V^{\otimes N}\otimes\bC[\bar{z}_1,\ldots,\bar{z}_N]
\]
is a $(\gl_V[\bar{z}],\mathcal{A}_N)$-bimodule: an element $x\otimes f(\bar{z})\in\gl_V[\bar{z}]$ acts by $\sum_{i=1}^{N} x^{(i)} f(\bar{z}_i)$, where $x^{(i)}$ acts on the $i$-th tensor factor, the polynomials $\bC[\bar{z}_1,\ldots,\bar{z}_N]$ act by multiplication, and $\sigma\in\bS_N$ permutes simultaneously the tensor factors of $V^{\otimes N}$ and the groups of variables $\bar{z}_1,\ldots,\bar{z}_N$.
Note that $P\simeq V^{\otimes N}\otimes_{\bS_N}\mathcal{A}_N$ as a right $\mathcal{A}_N$-module, in particular $P$ is a finitely generated projective right $\mathcal{A}_N$-module.
We call the graded algebra
\[
\mathcal{S}(V[\bar{z}],N):=\mathrm{End}_{\mathcal{A}_N}(P) = \Big(\big(\mathrm{End}(V)\otimes\bC[\bar{z}]\big)^{\otimes N}\Big)^{\bS_N}
\]
the \emph{current Schur algebra}. Its degree zero component is the classical Schur algebra $\mathcal{S}(V,N)$.

\begin{lemma}
\label{lem::current::Schur}
The image of $U(\gl_V[\bar{z}])$ in $\mathrm{End}(P)$ coincides with $\mathcal{S}(V[\bar{z}],N)$.
\end{lemma}
\begin{proof}
Let $B:=\mathrm{End}(V)\otimes\bC[\bar{z}]$ be the current algebra $\gl_V[\bar{z}]$ considered as an associative algebra. For $b\in B$ denote by $b^{(i)}\in B^{\otimes N}$ the element $b$ placed in the $i$-th tensor factor and set $\Delta(b):=\sum_{i=1}^{N} b^{(i)}$. An element $b\in\gl_V[\bar{z}]$ acts on $P$ by $\Delta(b)$, which commutes with $\mathcal{A}_N$. Hence the image of $U(\gl_V[\bar{z}])$ is contained in $\mathcal{S}(V[\bar{z}],N)=(B^{\otimes N})^{\bS_N}$.

Conversely, over a field of characteristic zero the space of invariants $(B^{\otimes N})^{\bS_N}$ is spanned by the tensors $b^{\otimes N}$, $b\in B$, thanks to the polarization identity
\[
\sum_{\sigma\in\bS_N} b_{\sigma(1)}\otimes\ldots\otimes b_{\sigma(N)} = \sum_{\emptyset\neq I\subset\{1,\ldots,N\}} (-1)^{N-|I|}\Big(\sum_{i\in I} b_i\Big)^{\otimes N}.
\]
The elements $b^{(1)},\ldots,b^{(N)}$ pairwise commute and $b^{\otimes N}=b^{(1)}\cdots b^{(N)}$ is their $N$-th elementary symmetric function. By Newton's identities it is a polynomial with rational coefficients in the power sums
$\sum_{i}(b^{(i)})^k = \Delta(b^k)$, $k=1,\ldots,N$. Each $b^k$ is again an element of $\gl_V[\bar{z}]$, hence $b^{\otimes N}$ belongs to the image of $U(\gl_V[\bar{z}])$.
\end{proof}

\begin{definition}
Let $\RepGLz{V}{N}$ be the category of finitely generated graded $\gl_V[\bar{z}]$-modules $U$ such that the action of $U(\gl_V[\bar{z}])$ on $U$ factors through the current Schur algebra $\mathcal{S}(V[\bar{z}],N)$. Thanks to Lemma~\ref{lem::current::Schur}, $\RepGLz{V}{N}$ is the category of finitely generated graded $\mathcal{S}(V[\bar{z}],N)$-modules.
In particular, each graded component of $U\in\RepGLz{V}{N}$ is a finite-dimensional polynomial $\gl_V$-representation of degree $N$.
\end{definition}

\begin{remark}
\label{rem::counterexample}
The condition that each graded component is a polynomial $\gl_V$-representation of degree $N$ is not sufficient for a graded $\gl_V[\bar{z}]$-module to belong to $\RepGLz{V}{N}$. Indeed, let $d=0$, $z:=z_0$, $N=1$ and $m\geq 2$. Consider the graded space $U:=V\otimes\bC[z]$ with the action
\[
(x\otimes z^k)\cdot (v\otimes g) := xv\otimes z^k g + \delta_{k,1}\,\mathrm{tr}(x)\, v\otimes zg.
\]
The correction term commutes with the standard action and $\mathrm{tr}([x,y])=0$, hence $U$ is a graded $\gl_V[z]$-module. It is generated by $V\otimes 1$ and each of its graded components is isomorphic to $V$. However, the element $(E_{11}\otimes 1)(E_{22}\otimes z)\in U(\gl_V[z])$ maps to $E_{11}E_{22}\otimes z = 0$ in $\mathcal{S}(V[z],1)=\mathrm{End}(V)\otimes\bC[z]$, while on $U$ it maps $e_1\otimes 1$ to $e_1\otimes z\neq 0$. The source of this phenomenon is the center $\bC[\bar{z}]\cdot\mathrm{Id}\subset\gl_V[\bar{z}]$: its positive degree part may act on $U$ by arbitrary commuting operators without changing the structure of the graded components.
\end{remark}

\begin{lemma}
\label{lem::weight::idempotent}
Let $m\geq N$ and consider $\pi\in\mathcal{S}(V,N)\subset\mathcal{S}(V[\bar{z}],N)$. Then
\begin{enumerate}
\item the weight subspace $\pi P$ is a free right $\mathcal{A}_N$-module of rank one generated by $\omega:=e_1\otimes\ldots\otimes e_N$;
\item the map $s\pi\mapsto s\,\omega$ defines an isomorphism of $(\mathcal{S}(V[\bar{z}],N),\mathcal{A}_N)$-bimodules $\mathcal{S}(V[\bar{z}],N)\,\pi\simeq P$;
\item there is an isomorphism of graded algebras $\pi\,\mathcal{S}(V[\bar{z}],N)\,\pi\simeq\mathcal{A}_N$;
\item consequently, for any $U\in\RepGLz{V}{N}$ the weight subspace $U_{(1^N,0^{m-N})}=\pi U$ is a graded $\mathcal{A}_N$-module: $z_{i,a}$ acts by $E_{ii}\otimes z_a$ and $\bC[\bS_N]$ acts through the permutation matrices $g_\sigma\in\GL_V$, $\sigma\in\bS_N\subset\bS_m$.
\end{enumerate}
\end{lemma}
\begin{proof}
Since $m\geq N$, the weight subspace $(V^{\otimes N})_{(1^N,0^{m-N})}$ is spanned by the vectors $e_{\sigma(1)}\otimes\ldots\otimes e_{\sigma(N)}$, $\sigma\in\bS_N$, and is isomorphic to the regular representation $\bC[\bS_N]$. Therefore $\pi P = \omega\cdot\mathcal{A}_N$ is free of rank one, which proves~(i).
Abbreviate $\mathcal{S}:=\mathcal{S}(V[\bar{z}],N)$. Then $\mathcal{S}\pi=\mathrm{Hom}_{\mathcal{A}_N}(\pi P,P)\simeq\mathrm{Hom}_{\mathcal{A}_N}(\mathcal{A}_N,P)\simeq P$, where the last map is the evaluation at $\omega$. Similarly, $\pi\mathcal{S}\pi=\mathrm{End}_{\mathcal{A}_N}(\pi P)\simeq\mathrm{End}_{\mathcal{A}_N}(\mathcal{A}_N)\simeq\mathcal{A}_N$, where an endomorphism $\phi$ is mapped to the element $a\in\mathcal{A}_N$ with $\phi(\omega)=\omega a$. This map is multiplicative, since $\phi\psi(\omega)=\phi(\omega b)=\omega ab$, and the right actions of $\pi\mathcal{S}\pi$ and $\mathcal{A}_N$ on $\mathcal{S}\pi$ and $P$ agree. This proves~(ii) and~(iii).
Finally, $E_{ii}\otimes z_a$ preserves weights and $(E_{ii}\otimes z_a)\,\omega = \sum_{k} E_{ii}^{(k)} z_{k,a}\,\omega = \omega\, z_{i,a}$, while $g_\sigma^{\otimes N}\in\mathcal{S}(V,N)$ permutes the tensor factors of $\omega$. This proves~(iv).
\end{proof}

\begin{theorem}
\label{thm::SW::currents}
Let $m:=\dim V\geq N$. The Schur-Weyl functors~\eqref{eq::Schur-Weyl} extend naturally to the setting of current algebras:
\begin{equation}
\label{eq::SW::currents}
\begin{tikzcd}
\mathsf{Rep}(\bC[\bS_N] \ltimes \bC[\bar{z}_1,\ldots,\bar{z}_N]) \arrow[rrr, shift left=1ex, "{\SW\colon M \mapsto (V[\bar{z}])^{\otimes N} \otimes_{\bC[\bS_N] \ltimes \bC[\bar{z}_1,\ldots,\bar{z}_N]} M}"]
& & & \RepGLz{V}{N} \arrow[lll, "{\WS\colon U \mapsto U_{(1^N, 0^{m-N})}}"].
\end{tikzcd}
\end{equation}
The Schur--Weyl functor $\SW$ is left adjoint to the conjugate Schur--Weyl functor $\WS$, and $\SW$ and $\WS$ are mutually inverse equivalences of categories. By abuse of notation, we use the same names $\SW$ and $\WS$ as for the classical functors~\eqref{eq::Schur-Weyl}, to which they restrict in multidegree zero.
\end{theorem}
\begin{proof}
We abbreviate $\mathcal{S}:=\mathcal{S}(V[\bar{z}],N)$ and use the identification $P\simeq\mathcal{S}\pi$ of Lemma~\ref{lem::weight::idempotent}.

\emph{The functors are well defined.}
A finitely generated graded $\mathcal{A}_N$-module $M$ is a quotient of $\bigoplus_{j=1}^{r}\mathcal{A}_N\langle k_j\rangle$, hence $\SW(M)$ is a quotient of $\bigoplus_{j=1}^{r}P\langle k_j\rangle$, which is a finitely generated graded $\mathcal{S}$-module. Conversely, $\WS(U)=\pi U$ is an $\mathcal{A}_N$-module by Lemma~\ref{lem::weight::idempotent}. If $U$ is generated by $u_1,\ldots,u_r$, then $\pi U=\sum_j \pi\mathcal{S}\,u_j$, and $\pi\mathcal{S}\simeq\mathrm{Hom}_{\mathcal{A}_N}(P,\mathcal{A}_N)$ is a finitely generated left $\mathcal{A}_N$-module, because $P$ is a finitely generated projective right $\mathcal{A}_N$-module. Hence $\WS(U)$ is finitely generated.

\emph{Adjunction.}
Morphisms in $\RepGLz{V}{N}$ are $\mathcal{S}$-linear maps by Lemma~\ref{lem::current::Schur}, hence
\[
\mathrm{Hom}_{\mathcal{S}}(P\otimes_{\mathcal{A}_N}M,U)\simeq\mathrm{Hom}_{\mathcal{A}_N}(M,\mathrm{Hom}_{\mathcal{S}}(\mathcal{S}\pi,U))\simeq\mathrm{Hom}_{\mathcal{A}_N}(M,\pi U),
\]
where the last isomorphism is the evaluation at $\pi$.

\emph{$\WS\circ \SW\simeq\mathrm{Id}$.}
The unit of the adjunction is the isomorphism
\[
M\simeq\mathcal{A}_N\otimes_{\mathcal{A}_N}M\simeq \pi\mathcal{S}\pi\otimes_{\pi\mathcal{S}\pi}M\simeq \pi\,(\mathcal{S}\pi\otimes_{\pi\mathcal{S}\pi}M) = \WS(\SW(M)).
\]
In other words, $\WS(\SW(M))=(V^{\otimes N}\otimes_{\bS_N}M)_{(1^N,0^{m-N})}\simeq\bC[\bS_N]\otimes_{\bS_N}M\simeq M$, and the element $E_{ii}\otimes f(\bar{z})=\sum_k E^{(k)}_{ii}f(\bar{z}_k)$ acts on $\omega\otimes M$ by $f(\bar{z}_i)$.

\emph{$\SW\circ \WS\simeq\mathrm{Id}$.}
The counit of the adjunction is the map
\[
\varepsilon_U\colon \mathcal{S}\pi\otimes_{\pi\mathcal{S}\pi}\pi U\to U, \qquad s\pi\otimes u\mapsto su .
\]
By Lemma~\ref{lem::Schur::full}, $1\in\mathcal{S}(V,N)\,\pi\,\mathcal{S}(V,N)\subset\mathcal{S}\pi\mathcal{S}$, hence the idempotent $\pi$ is full in $\mathcal{S}$: $\mathcal{S}\pi\mathcal{S}=\mathcal{S}$.
The image of $\varepsilon_U$ equals $\mathcal{S}\pi U=\mathcal{S}\pi\mathcal{S}U=U$, so $\varepsilon_U$ is surjective.
Let $K:=\ker\varepsilon_U$. The functor $U\mapsto\pi U$ is exact and $\pi\varepsilon_U$ is the isomorphism $\pi\mathcal{S}\pi\otimes_{\pi\mathcal{S}\pi}\pi U\simeq\pi U$, hence $\pi K=0$. Since $K$ is an $\mathcal{S}$-submodule, we get
\[
K=\mathcal{S}K=\mathcal{S}\pi\mathcal{S}K\subset\mathcal{S}\,(\pi K) = 0,
\]
so $\varepsilon_U$ is an isomorphism.
\end{proof}

\begin{remark}
The only non-formal ingredient of the proof of Theorem~\ref{thm::SW::currents} is the fullness of the idempotent $\pi$, which is inherited from the classical Schur algebra $\mathcal{S}(V,N)$ (Lemma~\ref{lem::Schur::full}). In other words, the current Schur-Weyl duality is obtained from the classical one by the Morita theory for the full idempotent $\pi\in\mathcal{S}(V[\bar{z}],N)$ with $\pi\,\mathcal{S}(V[\bar{z}],N)\,\pi\simeq\mathcal{A}_N$.
\end{remark}

\begin{corollary}
Let $m\geq N$.
\begin{enumerate}
\item The module $(V[\bar{z}])^{\otimes N}$ is a projective generator of $\RepGLz{V}{N}$. In particular, $\RepGLz{V}{N}$ is the full subcategory of finitely generated graded $\gl_V[\bar{z}]$-modules that are quotients of finite direct sums of grading shifts $(V[\bar{z}])^{\otimes N}\langle k\rangle$.
\item The simple objects of $\RepGLz{V}{N}$ are the grading shifts $V_\lambda\langle k\rangle$, $\lambda\vdash N$, on which $\gl_V\otimes\bar{z}\bC[\bar{z}]$ acts by zero. They are Schur-Weyl dual to the Specht modules $\mathbb{S}_\lambda\langle k\rangle$ with trivial action of the variables $\bar{z}_1,\ldots,\bar{z}_N$.
\end{enumerate}
\end{corollary}
\begin{proof}
The regular module $\mathcal{A}_N$ is a projective generator of $\mathsf{Rep}(\bC[\bS_N] \ltimes \bC[\bar{z}_1,\ldots,\bar{z}_N])$ and $\SW(\mathcal{A}_N)=(V[\bar{z}])^{\otimes N}$. Conversely, any quotient of $\bigoplus_j P\langle k_j\rangle$ is an $\mathcal{S}(V[\bar{z}],N)$-module. For the second claim, note that the positive degree part of $\bC[\bar{z}_1,\ldots,\bar{z}_N]$ acts by zero on any simple graded $\mathcal{A}_N$-module, and apply $\SW$.
\end{proof}

We now describe the graded characters.
For $k\in\bZ^{d+1}$ we set $q^k:=q_0^{k_0}\cdots q_d^{k_d}$ and denote by $\mathcal{R}_q:=\bZ[[q_0,\ldots,q_d]][q_0^{-1},\ldots,q_d^{-1}]$ the ring of formal power series with finitely many negative powers allowed. Negative powers appear only because of the grading shifts.
The graded components of $U\in\RepGLz{V}{N}$ are finite-dimensional polynomial $\gl_V$-representations of degree $N$. Since $\deg z_{i,a}$ does not depend on $i$, the symmetric group $\bS_N$ preserves the grading of $\mathcal{A}_N$, and the graded components of $M\in\mathsf{Rep}(\bC[\bS_N] \ltimes \bC[\bar{z}_1,\ldots,\bar{z}_N])$ are finite-dimensional $\bS_N$-modules.
We define the graded character and the graded Frobenius characteristic
\[
\mathrm{ch}_q(U):=\sum_{k\in\bZ^{d+1}}\mathrm{ch}(U_k)\,q^k,
\qquad
\mathrm{Frob}_q(M):=\sum_{k\in\bZ^{d+1}}\mathrm{Frob}(M_k)\,q^k .
\]
Both are symmetric functions of degree $N$ with coefficients in $\mathcal{R}_q$, i.e.\ elements of $\Lambda^N\otimes\mathcal{R}_q$.
For finite-dimensional modules the coefficients belong to $\bZ[q_0^{\pm1},\ldots,q_d^{\pm1}]$, and to $\bZ[q_0,\ldots,q_d]$ if the module is concentrated in nonnegative degrees.

\begin{proposition}
\label{prop::SW::graded::characters}
Let $m\geq N$. For all $M\in\mathsf{Rep}(\bC[\bS_N] \ltimes \bC[\bar{z}_1,\ldots,\bar{z}_N])$ and $U\in\RepGLz{V}{N}$ we have
\[
\mathrm{ch}_q\big(\SW(M)\big)=\mathrm{Frob}_q(M),
\qquad
\mathrm{Frob}_q\big(\WS(U)\big)=\mathrm{ch}_q(U).
\]
In particular, the Grothendieck groups of the subcategories of finite-dimensional objects are both identified with $\Lambda^N\otimes\bZ[q_0^{\pm1},\ldots,q_d^{\pm1}]$. The classes of the simple objects $V_\lambda\langle k\rangle$ and $\mathbb{S}_\lambda\langle k\rangle$ correspond to $s_\lambda\,q^k$, and the Schur-Weyl duality~\eqref{eq::SW::currents} induces the identity map.
\end{proposition}
\begin{proof}
The isomorphism $P\simeq V^{\otimes N}\otimes_{\bS_N}\mathcal{A}_N$ is an isomorphism of graded right $\mathcal{A}_N$-modules. Hence $\SW(M)\simeq V^{\otimes N}\otimes_{\bS_N}M$ as graded $\gl_V$-modules, that is $\SW(M)_k\simeq V^{\otimes N}\otimes_{\bS_N}M_k$, and the first identity follows from the classical identity~\eqref{eq::SW::characters}. The second identity follows from the first one, because $U\simeq \SW(\WS(U))$ by Theorem~\ref{thm::SW::currents}.
\end{proof}

The characters of the free $\mathcal{A}_N$-modules and of their Schur-Weyl duals are given by plethystic substitutions.
For $f\in\Lambda$ we denote by
\[
f\Big[\frac{X}{\prod_{a=0}^{d}(1-q_a)}\Big]\in\Lambda\otimes\bZ[[q_0,\ldots,q_d]]
\]
the image of $f$ under the ring homomorphism defined on the power sums by $p_r\mapsto p_r(X)\prod_{a=0}^{d}(1-q_a^r)^{-1}$. Equivalently, it is the evaluation of $f$ at the alphabet $\{x_i\,q^k\colon i\geq 1,\ k\in\bZ_{\geq0}^{d+1}\}$, that is the plethysm of $f$ with the graded character $\mathrm{ch}_q(V[\bar{z}])=p_1\prod_{a=0}^{d}(1-q_a)^{-1}$ of the tautological module.

\begin{proposition}
\label{prop::characters::free}
Let $m\geq N$ and let $W$ be a finite-dimensional $\bS_N$-module. Consider the free module $\mathcal{A}_N\otimes_{\bS_N}W\simeq W\otimes\bC[\bar{z}_1,\ldots,\bar{z}_N]$ with the diagonal action of $\bS_N$. Then
\begin{gather*}
\SW(\mathcal{A}_N\otimes_{\bS_N}W)\simeq(V[\bar{z}])^{\otimes N}\otimes_{\bS_N}W \quad\text{and}\quad
\\
\mathrm{Frob}_q(\mathcal{A}_N\otimes_{\bS_N}W)=\mathrm{ch}_q\big((V[\bar{z}])^{\otimes N}\otimes_{\bS_N}W\big)=\mathrm{Frob}(W)\Big[\frac{X}{\prod_{a=0}^{d}(1-q_a)}\Big].
\end{gather*}
In particular, let $P_\lambda:=(V[\bar{z}])^{\otimes N}\otimes_{\bS_N}\mathbb{S}_\lambda$ be the Schur functor corresponding to $\lambda\vdash N$ applied to $V[\bar{z}]$. Then
\begin{enumerate}
\item $P_\lambda$ is the projective cover of $V_\lambda$ in $\RepGLz{V}{N}$; it is Schur-Weyl dual to the projective cover $\mathcal{A}_N\otimes_{\bS_N}\mathbb{S}_\lambda$ of $\mathbb{S}_\lambda$, and
\[
\mathrm{ch}_q(P_\lambda)=s_\lambda\Big[\frac{X}{\prod_{a=0}^{d}(1-q_a)}\Big];
\]
\item $P=(V[\bar{z}])^{\otimes N}\simeq\bigoplus_{\lambda\vdash N}P_\lambda^{\oplus f^\lambda}$, and
\[
\mathrm{ch}_q(P)=\mathrm{Frob}_q(\mathcal{A}_N)=p_1^N\Big[\frac{X}{\prod_{a=0}^{d}(1-q_a)}\Big]=\frac{p_1^N}{\prod_{a=0}^{d}(1-q_a)^N}=\sum_{\lambda\vdash N}f^\lambda\, s_\lambda\Big[\frac{X}{\prod_{a=0}^{d}(1-q_a)}\Big].
\]
\end{enumerate}
\end{proposition}
\begin{proof}
The first isomorphism is the associativity of the tensor product: $P\otimes_{\mathcal{A}_N}(\mathcal{A}_N\otimes_{\bS_N}W)\simeq P\otimes_{\bS_N}W$. The first equality of characters is Proposition~\ref{prop::SW::graded::characters}.
For the second equality, note that the graded trace of $\mathrm{diag}(x_1,\ldots,x_m)^r$ on $V[\bar{z}]$ is $p_r(x_1,\ldots,x_m)\prod_{a}(1-q_a^r)^{-1}$. Hence, by the standard computation of the trace of a permutation of tensor factors, for all $\sigma\in\bS_N$
\[
\sum_{k}\mathrm{tr}_{P_k}\big(\mathrm{diag}(x_1,\ldots,x_m)\cdot\sigma\big)\,q^k
=\prod_{c}\,p_{|c|}(x_1,\ldots,x_m)\prod_{a=0}^{d}(1-q_a^{|c|})^{-1}
=p_{\rho(\sigma)}\Big[\frac{X}{\prod_{a=0}^{d}(1-q_a)}\Big],
\]
where the product is taken over the cycles $c$ of $\sigma$. In characteristic zero $P\otimes_{\bS_N}W$ is isomorphic to the space of $\bS_N$-invariants in $P\otimes W$, therefore
\[
\mathrm{ch}_q(P\otimes_{\bS_N}W)=\frac{1}{N!}\sum_{\sigma\in\bS_N}\chi_W(\sigma)\,p_{\rho(\sigma)}\Big[\frac{X}{\prod_{a}(1-q_a)}\Big]=\mathrm{Frob}(W)\Big[\frac{X}{\prod_{a}(1-q_a)}\Big].
\]
The module $\mathcal{A}_N\otimes_{\bS_N}\mathbb{S}_\lambda$ is projective, since $\bC[\bS_N]$ is semisimple, and its quotient by the ideal generated by the variables of positive degree is $\mathbb{S}_\lambda$. Hence it is the projective cover of $\mathbb{S}_\lambda$, and item~(i) follows from Theorem~\ref{thm::SW::currents} with $W=\mathbb{S}_\lambda$. Item~(ii) follows from $\mathcal{A}_N\simeq\mathcal{A}_N\otimes_{\bS_N}\bC[\bS_N]$ and $\bC[\bS_N]\simeq\bigoplus_\lambda\mathbb{S}_\lambda^{\oplus f^\lambda}$ together with $p_1^N=\sum_\lambda f^\lambda s_\lambda$.
\end{proof}

\begin{remark}
The coefficient of $s_\mu\,q^k$ in $\mathrm{ch}_q(P_\lambda)$ is the graded multiplicity $[P_\lambda:V_\mu\langle k\rangle]$, which by Proposition~\ref{prop::SW::graded::characters} equals the multiplicity of $\mathbb{S}_\mu$ in the degree $k$ component of $\mathbb{S}_\lambda\otimes\bC[\bar{z}_1,\ldots,\bar{z}_N]$. The identity $s_\lambda[XY]=\sum_{\mu,\nu}g_{\lambda\mu\nu}\,s_\mu[X]\,s_\nu[Y]$, where $g_{\lambda\mu\nu}$ are the Kronecker coefficients, gives
\[
\sum_{k}[P_\lambda:V_\mu\langle k\rangle]\,q^k=\sum_{\nu\vdash N}g_{\lambda\mu\nu}\; s_\nu\Big[\frac{1}{\prod_{a=0}^{d}(1-q_a)}\Big],
\]
where $s_\nu[1/\prod_a(1-q_a)]$ is the evaluation of $s_\nu$ at the alphabet $\{q^k\colon k\in\bZ_{\geq0}^{d+1}\}$.
For a single variable ($d=0$) this is the principal specialization $s_\nu(1,q,q^2,\ldots)=q^{n(\nu)}\prod_{\square\in\nu}(1-q^{h(\square)})^{-1}$, where $n(\nu)=\sum_i(i-1)\nu_i$ and $h(\square)$ is the hook length. For $\lambda=(N)$ we have $g_{(N)\mu\nu}=\delta_{\mu\nu}$, and one recovers the classical formula for the graded multiplicity of $\mathbb{S}_\mu$ in $\bC[z_1,\ldots,z_N]$.
\end{remark}

\begin{example}
    The following modules are Schur-Weyl dual to each other:
$$
\begin{array}{c}
{\begin{tikzcd}
\bC[\bS_N] \ltimes \bC[\bar{z}_1,\ldots,\bar{z}_N]
 \arrow[r,shift left=0.5ex,"\SW"]
 &  (V[\bar{z}])^{\otimes N},
 \arrow[l,shift left=0.5ex,"\WS"]
\end{tikzcd}}
\\
{\begin{tikzcd}
\bC[\bar{z}_1,\ldots,\bar{z}_N]
 \arrow[r,shift left=0.5ex,"\SW"]
 &  S^{N}(V[\bar{z}])\simeq (V[\bar{z}])^{\otimes N}\otimes_{\bS_N}\bC,
 \arrow[l,shift left=0.5ex,"\WS"]
\end{tikzcd}}
\\
{\begin{tikzcd}
\mathrm{sgn}\otimes\bC[\bar{z}_1,\ldots,\bar{z}_N]
 \arrow[r,shift left=0.5ex,"\SW"]
 &  \Lambda^{N}(V[\bar{z}])\simeq (V[\bar{z}])^{\otimes N}\otimes_{\bS_N}\mathrm{sgn}.
 \arrow[l,shift left=0.5ex,"\WS"]
\end{tikzcd}}
\end{array}
$$
Here $\bS_N$ acts on $\bC[\bar{z}_1,\ldots,\bar{z}_N]$ by permutations of the groups of variables $\bar{z}_1,\ldots,\bar{z}_N$, and $\mathrm{sgn}$ denotes the sign representation of $\bS_N$.
By Proposition~\ref{prop::characters::free}, the graded characters of these three pairs of modules are equal to
\[
\frac{p_1^N}{\prod_{a=0}^{d}(1-q_a)^N},
\qquad
h_N\Big[\frac{X}{\prod_{a=0}^{d}(1-q_a)}\Big],
\qquad
e_N\Big[\frac{X}{\prod_{a=0}^{d}(1-q_a)}\Big],
\]
respectively.
\end{example}

\begin{remark}
A particular case of a single variable Schur-Weyl duality was discussed in details in~\cite[Th.\,4.3]{FKhM}.
\end{remark}

\subsection{Sign duality for current algebras}
\subsubsection{The sign duality functor}
Tensoring with the sign representation is an involutive autoequivalence of the category of $\bS_N$-modules.
It extends to $\mathcal{A}_N$-modules: for an $\mathcal{A}_N$-module $M$ we denote by $\mathrm{sgn}\otimes M$ the same space, on which the polynomials act as on $M$ and $\sigma\in\bS_N$ acts by $(-1)^{\sigma}\sigma$.
Since $\mathrm{Frob}(\mathrm{sgn}\otimes W)=\omega\,\mathrm{Frob}(W)$ for every $\bS_N$-module $W$, where $\omega\colon s_\mu\mapsto s_{\mu^T}$ is the standard involution of $\Lambda$, we have
\begin{equation}
	\label{eq::sign::Frob}
	\mathrm{Frob}_q(\mathrm{sgn}\otimes M)=\omega\,\mathrm{Frob}_q(M).
\end{equation}

In this subsection we also allow finite-dimensional modules that are not graded; the proof of Theorem~\ref{thm::SW::currents} applies to them verbatim.
We keep the notation $\RepGLz{V}{N}$ for the category of finite-dimensional $\gl_V[\bar{z}]$-modules whose action factors through the current Schur algebra.

\begin{definition}
	\label{def::sign::functor}
	Let $m=\dim V\geq N$. The \emph{sign duality functor} is the Schur--Weyl dual of the sign twist:
	\[
	\bomega:=\SW\circ(\mathrm{sgn}\otimes-)\circ\WS\colon\RepGLz{V}{N}\longrightarrow\RepGLz{V}{N}.
	\]
	Its restriction to multidegree zero is the classical sign duality functor on $\RepGL{V}{N}$, which we denote by the same symbol $\bomega$.
\end{definition}

We use the bold letter $\bomega$ for the functor, in order to distinguish it from the involution $\omega$ of $\Lambda$ and from the fundamental weights $\omega_k$.

\subsubsection{The contragredient duality}
The transposition $X\otimes f\mapsto X^{T}\otimes f$ is an anti-involution of the Lie algebra $\gl_V[\bar{z}]$, i.e., $[X,Y]^T=[Y^T,X^T]$. Hence it extends to an anti-involution $u\mapsto u^T$ of $U(\gl_V[\bar{z}])$.
For a finite-dimensional multigraded $\gl_V[\bar{z}]$-module $U$ we define its \emph{contragredient dual} $U^{\dagger}$ as follows. As a vector space, $U^{\dagger}:=\bigoplus_{\bar{r}}(U_{\bar{r}})^{*}$, where $(U_{\bar{r}})^{*}$ is placed in multidegree $-\bar{r}$, and the action is given by
\[
(u\cdot\xi)(w):=\xi(u^{T}w),\qquad u\in U(\gl_V[\bar{z}]),\ \xi\in U^\dagger,\ w\in U.
\]
For a module without grading, $U^\dagger$ is the dual space with the same formula for the action.
The transposition fixes the Cartan subalgebra pointwise. Hence the component $(U_{\bar{r}})^{*}$ of $U^{\dagger}$ in multidegree $-\bar{r}$ has the same $\gl_V$-character as $U_{\bar{r}}$, and $U^\dagger\simeq U$ as $\gl_V$-modules.
Moreover, the current Schur algebra $\mathcal{S}(V[\bar{z}],N)$ is stable under the transposition. Therefore $U^{\dagger}\in\RepGLz{V}{N}$ whenever $U\in\RepGLz{V}{N}$. Clearly, $U^{\dagger\dagger}\simeq U$, and $(U\otimes U')^\dagger\simeq U^\dagger\otimes U'^\dagger$, because the transposition is compatible with the coproduct.

Similarly, the assignment $\sigma\mapsto\sigma^{-1}$, $x_i\mapsto x_i$, $y_i\mapsto y_i$ extends to an anti-involution $a\mapsto a^{T}$ of $\mathcal{A}_N$. For a finite-dimensional multigraded $\mathcal{A}_N$-module $M$ we define $M^{\dagger}:=\bigoplus_{\bar{r}}(M_{\bar{r}})^{*}$, with $(M_{\bar{r}})^{*}$ in multidegree $-\bar{r}$ and the action $(a\cdot\xi)(w):=\xi(a^{T}w)$.

\begin{lemma}
	\label{lem::SW::dagger}
	For every finite-dimensional $U\in\RepGLz{V}{N}$ there is a natural isomorphism of $\mathcal{A}_N$-modules $\WS(U^{\dagger})\simeq\WS(U)^{\dagger}$.
	Consequently, $\SW(M^{\dagger})\simeq\SW(M)^{\dagger}$ for every finite-dimensional multigraded $\mathcal{A}_N$-module $M$, and $\bomega(U^\dagger)\simeq\bomega(U)^\dagger$.
\end{lemma}
\begin{proof}
	The transposition fixes the diagonal matrices. Hence the weight subspace of $U^{\dagger}$ of weight $(1^N,0^{m-N})$ is the dual of the corresponding weight subspace of $U$. By Lemma~\ref{lem::weight::idempotent}(iv), the variables $x_i$ and $y_i$ act on $\WS(U)$ by $E_{ii}\otimes x$ and $E_{ii}\otimes y$, which are fixed by the transposition, and $\sigma\in\bS_N$ acts by the permutation matrix $g_\sigma$, whose transpose is $g_\sigma^{-1}=g_{\sigma^{-1}}$. This proves the first claim. The second one follows because $\SW$ and $\WS$ are mutually inverse equivalences. The last claim holds, because $(\mathrm{sgn}\otimes M)^\dagger=\mathrm{sgn}\otimes M^\dagger$.
\end{proof}

\subsubsection{Properties of the sign duality functor}

For an automorphism $\varphi$ of the algebra $\bC[\bar{z}]$ (for example, the translation $z_a\mapsto z_a+c$ or the interchange $x\leftrightarrow y$) and a $\gl_V[\bar{z}]$-module $U$ we denote by $U^{\varphi}$ the pullback of $U$ along the automorphism $\mathrm{id}\otimes\varphi$ of $\gl_V[\bar{z}]$.

\begin{lemma}
	\label{lem::sign::properties}
	Let $m=\dim V\geq N$.
	\begin{enumerate}
		\item The functor $\bomega$ is an exact involutive autoequivalence, $\bomega\circ\bomega\simeq\mathrm{Id}$, compatible with the multigrading. It preserves heads and socles. It maps the simple object $V_\mu$, placed in multidegree $\bar{r}$, to $V_{\mu^T}$ placed in the same multidegree. For graded modules
		\begin{equation}
			\label{eq::sign::ch}
			\mathrm{ch}_q\big(\bomega(U)\big)=\omega\,\mathrm{ch}_q(U).
		\end{equation}
		\item $\bomega\big((V[\bar{z}])^{\otimes N}\big)\simeq(V[\bar{z}])^{\otimes N}$, $\ \bomega\big(S^N(V[\bar{z}])\big)\simeq\Lambda^N(V[\bar{z}])$ 
		\\
		and $\ \bomega\big(\Lambda^N(V[\bar{z}])\big)\simeq S^N(V[\bar{z}])$.
		\item $\bomega(U^{\varphi})\simeq\bomega(U)^{\varphi}$ for every automorphism $\varphi$ of $\bC[\bar{z}]$.
		\item Let $U\in\RepGLz{V}{N}$ and $U'\in\RepGLz{V}{N'}$. Then $U\otimes U'\in\RepGLz{V}{N+N'}$, and
		\[
		\bomega(U\otimes U')\simeq\bomega(U)\otimes\bomega(U').
		\]
		\item $\bomega(U)^\dagger\simeq\bomega(U^\dagger)$, and the functor $U\mapsto\bomega(U)^\dagger$ is an exact contravariant involution: $\big(\bomega(\bomega(U)^\dagger)\big)^\dagger\simeq U$. It reverses the multigrading.
	\end{enumerate}
\end{lemma}
\begin{proof}
	(i) By Theorem~\ref{thm::SW::currents}, $\bomega$ is a composition of exact equivalences, and $\mathrm{sgn}\otimes\mathrm{sgn}\otimes M\simeq M$. The Specht module $\mathbb{S}_\mu$ is mapped by $\mathrm{sgn}\otimes-$ to $\mathbb{S}_{\mu^T}$, which gives the statement on simple objects. Heads and socles are preserved by any equivalence of categories. Equation~\eqref{eq::sign::ch} follows from Proposition~\ref{prop::SW::graded::characters} and~\eqref{eq::sign::Frob}.
	
	(ii) Since $\mathrm{sgn}\otimes\mathcal{A}_N\simeq\mathcal{A}_N$ and $\mathrm{sgn}\otimes\bC[\bar{z}_1,\ldots,\bar{z}_N]\simeq\mathcal{A}_N\otimes_{\bS_N}\mathrm{sgn}$, the example at the end of \S\ref{sec::SW} gives the statement.
	
	(iii) The automorphism $\varphi$ extends to an automorphism of $\mathcal{A}_N$ commuting with $\bS_N$ (it acts diagonally on the variables $\bar{z}_i$), and the bimodule $(V[\bar{z}])^{\otimes N}$ is stable under it. By Lemma~\ref{lem::weight::idempotent}(iv), $\WS(U^{\varphi})=\WS(U)^{\varphi}$. The twist by $\varphi$ commutes with $\mathrm{sgn}\otimes-$ and with $\SW$.
	
	(iv) Let $\mathcal{A}_{N,N'}:=\bC[\bS_N\times\bS_{N'}]\ltimes\bC[\bar{z}_1,\ldots,\bar{z}_{N+N'}]\subset\mathcal{A}_{N+N'}$. The algebra $\mathcal{A}_{N+N'}$ is free over $\mathcal{A}_{N,N'}$. Let $U=\SW(R)$ and $U'=\SW(R')$. 
	\begin{multline*}
		U\otimes U'=\Big((V[\bar{z}])^{\otimes N}\otimes_{\mathcal{A}_N}R\Big)\otimes\Big((V[\bar{z}])^{\otimes N'}\otimes_{\mathcal{A}_{N'}}R'\Big)
		\simeq \\
		\simeq (V[\bar{z}])^{\otimes(N+N')}\otimes_{\mathcal{A}_{N,N'}}(R\boxtimes R')\simeq\SW\Big(\mathrm{Ind}_{\mathcal{A}_{N,N'}}^{\mathcal{A}_{N+N'}}(R\boxtimes R')\Big).
	\end{multline*}
	In particular, $U\otimes U'\in\RepGLz{V}{N+N'}$. The sign representation of $\bS_{N+N'}$ restricts to $\mathrm{sgn}\boxtimes\mathrm{sgn}$, hence $\mathrm{sgn}\otimes\mathrm{Ind}(R\boxtimes R')\simeq\mathrm{Ind}\big((\mathrm{sgn}\otimes R)\boxtimes(\mathrm{sgn}\otimes R')\big)$. Applying $\SW$ once more, we get the claim.
	
	(v) The first claim is Lemma~\ref{lem::SW::dagger}. The functor $U\mapsto\bomega(U)^{\dagger}$ is a composition of an exact covariant functor and the exact contravariant duality. The involutivity follows from the first claim and from $\bomega\circ\bomega\simeq\mathrm{Id}$ and $\dagger\circ\dagger\simeq\mathrm{Id}$.
\end{proof}

\section{The $N!$- and the $\frac{N!}{k}$-conjectures}
\label{sec::GH}

From now on $d=1$, so that $\bar{z}=(z_0,z_1)=(x,y)$, the multigrading is $\bZ^2$-valued, and $\gl_V[\bar{z}]=\gl_V[x,y]$. In this section we recall the Garsia--Haiman modules attached to Young diagrams, the $N!$-conjecture (now a theorem of Haiman) and the $\frac{N!}{k}$-conjecture of Bergeron and Garsia.
Each of these modules has two equivalent descriptions.
The classical one realizes it as a space of polynomials spanned by all partial derivatives of a single alternating polynomial.
The dual one realizes it as the quotient of the polynomial ring by the annihilating ideal of this polynomial.
For the $\frac{N!}{k}$-conjecture the passage from the first description to the second one replaces an \emph{intersection} of Garsia--Haiman modules by a \emph{common quotient} of them.
The quotient description is the one that is compatible with the current Schur--Weyl duality of Theorem~\ref{thm::SW::currents}, and in~\S\ref{sec::GH::SW} we collect the Schur--Weyl dual statements.

\subsection{Garsia--Haiman modules and the $N!$-conjecture}
\label{sec::GH::N!}

Let $\lambda\vdash N$ be a Young diagram.
As in \S\ref{sec::SW}, we consider $N$ pairs of variables $\bar{z}_i=(x_i,y_i)$ (these are the variables $x_{i},y_{i}$ of \S\ref{sec::SW}), $i=1,\ldots,N$, and the polynomial ring $\bC[\bar{z}_1,\ldots,\bar{z}_N]=\bC[x_1,y_1,\ldots,x_N,y_N]$, which is $\bZ_{\geq 0}^{2}$-graded by the multidegree $\mdeg(x_i):=(1,0)$, $\mdeg(y_i):=(0,1)$.
The symmetric group $\bS_N$ acts on it by simultaneous permutations of the tuples:
\[
(\sigma\cdot g)(\bar{z}_1,\ldots,\bar{z}_N):=g(\bar{z}_{\sigma^{-1}(1)},\ldots,\bar{z}_{\sigma^{-1}(N)}).
\]
Together with the action of polynomials by multiplication, this makes $\bC[\bar{z}_1,\ldots,\bar{z}_N]$ a cyclic module over the wreath product algebra $\mathcal{A}_N=\bC[\bS_N]\ltimes\bC[\bar{z}_1,\ldots,\bar{z}_N]$, generated by $1$.

Choose an enumeration $\bar{s}_1,\ldots,\bar{s}_N$ of the cells of $\lambda$ and consider the \emph{generalized Vandermonde determinant}
\begin{equation}
\label{eq::GH::Delta}
\Delta_{\lambda}(\bar{z}_1,\ldots,\bar{z}_N):=\det\left(\bar{z}_i^{\,\bar{s}_j}\right)_{i,j=1}^{N}
=\sum_{\sigma\in\bS_N}(-1)^{\sigma}\,\bar{z}_{1}^{\,\bar{s}_{\sigma(1)}}\cdots\bar{z}_{N}^{\,\bar{s}_{\sigma(N)}} .
\end{equation}
A different enumeration of the cells changes $\Delta_\lambda$ by a sign, so $\Delta_\lambda$ is well defined up to a sign.
It is alternating, $\sigma\cdot\Delta_\lambda=(-1)^{\sigma}\Delta_\lambda$ for all $\sigma\in\bS_N$, and it is homogeneous of multidegree
\begin{equation}
\label{eq::mdeg::Delta}
\mdeg(\Delta_\lambda)=\sum_{\bar{s}\in\lambda}\bar{s}
=\big(n(\lambda^T),\,n(\lambda)\big),
\end{equation}
where $n(\lambda)$ and $n(\lambda^T)$ are defined in~\eqref{eq::n::lambda}.
The monomials occurring in $\Delta_\lambda$ are exactly the monomials $\prod_{i}\bar{z}_i^{\,\bar{t}_i}$ such that $\{\bar{t}_1,\ldots,\bar{t}_N\}$ coincides with the set of cells of $\lambda$; in particular, $\Delta_\lambda\neq 0$.

Consider the following standard pairing on the polynomial ring:
\begin{equation}
\label{eq::GH::pairing}
\langle f,g\rangle:=\left.\left(f(\partial)\,g\right)\right|_{\bar{z}_1=\ldots=\bar{z}_N=0},
\qquad f,g\in\bC[\bar{z}_1,\ldots,\bar{z}_N], \text{ and } f(\partial):=f\Big(\frac{\partial}{\partial x_{i}},\frac{\partial}{\partial y_{i}}\Big).
\end{equation}
Distinct monomials are orthogonal, and $\langle \bar{z}^{\alpha},\bar{z}^{\alpha}\rangle=\alpha!:=\prod_{i,a}\alpha_{i,a}!$.
Hence the pairing is symmetric, homogeneous components of different multidegrees are orthogonal, and the restriction of the pairing to each (finite-dimensional) homogeneous component is nondegenerate; moreover, $\langle f,f\rangle>0$ for every nonzero $f$ with real coefficients.
The pairing is $\bS_N$-invariant, and multiplication by $f$ is adjoint to the differential operator $f(\partial)$:
\begin{equation}
\label{eq::GH::adjoint}
\langle fg,h\rangle=\langle g,f(\partial)h\rangle.
\end{equation}

\begin{definition}
The \emph{Garsia--Haiman module} $\mGH_{\lambda}^{\bS}$ is the multigraded subspace of the polynomial ring spanned by all partial derivatives of $\Delta_\lambda$:
\begin{equation}
\label{eq::GH::module}
\mGH^{\bS}_{\lambda}:=\left\{ f(\partial)\,\Delta_{\lambda} \ \colon \ f\in\bC[\bar{z}_1,\ldots,\bar{z}_N] \right\}.
\end{equation}
The \emph{annihilating ideal} of $\Delta_\lambda$ is
\[
\euJ_{\lambda}^{\bS}:=\left\{ f\in\bC[\bar{z}_1,\ldots,\bar{z}_N] \ \colon \ f(\partial)\,\Delta_{\lambda}=0\right\}.
\]
\end{definition}

The space $\mGH^{\bS}_{\lambda}$ is $\bS_N$-stable and is closed under all partial derivatives.
%The main goal of the paper is a description of $\mGH^{\bS}_{\lambda}$ as a multigraded representation of $\bS_N$, or, equivalently, of the quotient ring $\bC[\bar{z}_1,\ldots,\bar{z}_N]/\euJ^{\bS}_\lambda$, as the following proposition explains.
We first establish some basic properties of the pairing and of the annihilating ideal $\euJ^{\bS}_\lambda$.

\begin{proposition}
\label{prp::GH::pairing}
Let $\lambda\vdash N$ be a Young diagram.
\begin{enumerate}
\item The annihilator $\euJ^{\bS}_\lambda$ is a multigraded $\bS_N$-stable ideal, i.e., a multigraded $\mathcal{A}_N$-submodule of $\bC[\bar{z}_1,\ldots,\bar{z}_N]$, and $\Delta_\lambda\notin\euJ^{\bS}_\lambda$.
\item With respect to the pairing~\eqref{eq::GH::pairing}, the subspaces $\mGH^{\bS}_\lambda$ and $\euJ^{\bS}_\lambda$ are the orthogonal complements of each other in every multigraded component.
Consequently, the pairing induces a nondegenerate $\bS_N$-invariant pairing
\[
\langle\cdot,\cdot\rangle\colon \bC[\bar{z}_1,\ldots,\bar{z}_N]/\euJ^{\bS}_{\lambda} \ \times\ \mGH^{\bS}_{\lambda}\longrightarrow\bC
\]
of multigraded spaces, and an isomorphism of multigraded $\bS_N$-modules
\[
\mGH^{\bS}_{\lambda}\simeq\bC[\bar{z}_1,\ldots,\bar{z}_N]/\euJ^{\bS}_{\lambda}.
\]
More precisely, the quotient $\bC[\bar{z}_1,\ldots,\bar{z}_N]/\euJ^{\bS}_{\lambda}$ is the graded dual of $\mGH^{\bS}_{\lambda}$ as an $\mathcal{A}_N$-module, where the variables act on $\mGH^{\bS}_{\lambda}$ by the partial derivatives.
\item The ideal $\euJ^{\bS}_\lambda$ is the largest ideal of $\bC[\bar{z}_1,\ldots,\bar{z}_N]$ contained in the hyperplane
$\Delta_\lambda^{\perp}:=\{f\colon\langle f,\Delta_\lambda\rangle=0\}$.
\item The $\mathcal{A}_N$-module $\bC[\bar{z}_1,\ldots,\bar{z}_N]/\euJ^{\bS}_{\lambda}$ is cyclic and cocyclic: it is generated by the class of $1$, which spans its head $\soc$, isomorphic to the trivial representation in multidegree $0$; and every nonzero multigraded submodule contains the class of $\Delta_\lambda$, which spans its socle $\cosoc$, isomorphic to the sign representation in multidegree $\mdeg(\Delta_\lambda)$.
In other words, $\bC[\bar{z}_1,\ldots,\bar{z}_N]/\euJ^{\bS}_{\lambda}$ is a multigraded Artinian Gorenstein algebra with the socle $\bC\cdot\Delta_\lambda$.
\end{enumerate}
\end{proposition}
\begin{proof}
(i) If $f(\partial)\Delta_\lambda=0$, then $(gf)(\partial)\Delta_\lambda=g(\partial)f(\partial)\Delta_\lambda=0$, so $\euJ^{\bS}_\lambda$ is an ideal.
Since $\Delta_\lambda$ is homogeneous, the homogeneous components of $f$ are mapped by $f\mapsto f(\partial)\Delta_\lambda$ to different multigraded components; hence $\euJ^{\bS}_\lambda$ is multigraded.
For $\sigma\in\bS_N$ we have $(\sigma f)(\partial)\Delta_\lambda=\sigma\big(f(\partial)\,\sigma^{-1}\Delta_\lambda\big)=(-1)^{\sigma}\sigma\big(f(\partial)\Delta_\lambda\big)$, so $\euJ^{\bS}_\lambda$ is $\bS_N$-stable.
Finally, $\Delta_\lambda(\partial)\Delta_\lambda=\langle\Delta_\lambda,\Delta_\lambda\rangle$ is a nonzero constant, because $\Delta_\lambda$ has integer coefficients.

(ii) By~\eqref{eq::GH::adjoint}, $g\perp\mGH^{\bS}_\lambda$ if and only if $\langle f, g(\partial)\Delta_\lambda\rangle=\langle fg,\Delta_\lambda\rangle=\langle g,f(\partial)\Delta_\lambda\rangle=0$ for all $f$, that is, if and only if $g(\partial)\Delta_\lambda=0$. Thus $(\mGH^{\bS}_\lambda)^{\perp}=\euJ^{\bS}_\lambda$. Since all multigraded components are finite-dimensional and the pairing is nondegenerate on each of them, we also have $(\euJ^{\bS}_\lambda)^{\perp}=\mGH^{\bS}_\lambda$, and the remaining claims follow. The dual multigraded components are isomorphic as $\bS_N$-modules because all representations of $\bS_N$ are self-dual.

(iii) If $f\in\euJ^{\bS}_\lambda$, then $\langle f,\Delta_\lambda\rangle=\langle 1,f(\partial)\Delta_\lambda\rangle=0$, so $\euJ^{\bS}_\lambda\subset\Delta_\lambda^{\perp}$. Conversely, if $I\subset\Delta_\lambda^{\perp}$ is an ideal and $f\in I$, then $\langle g,f(\partial)\Delta_\lambda\rangle=\langle gf,\Delta_\lambda\rangle=0$ for all $g$, hence $f(\partial)\Delta_\lambda=0$ and $I\subset\euJ^{\bS}_\lambda$.

(iv) The first claim is obvious, since the quotient of $\bC[\bar{z}_1,\ldots,\bar{z}_N]/\euJ^{\bS}_{\lambda}$ by the ideal generated by all variables is spanned by the class of $1$.
Let $U$ be a nonzero multigraded submodule and let $f\notin\euJ^{\bS}_\lambda$ be a homogeneous polynomial whose class lies in $U$. Since $f(\partial)\Delta_\lambda\neq 0$, by (ii) there exists a homogeneous $g$ with $\langle gf,\Delta_\lambda\rangle=\langle g,f(\partial)\Delta_\lambda\rangle\neq0$; necessarily $\mdeg(gf)=\mdeg(\Delta_\lambda)$. The component of $\bC[\bar{z}_1,\ldots,\bar{z}_N]/\euJ^{\bS}_{\lambda}$ of multidegree $\mdeg(\Delta_\lambda)$ is dual to the component $\bC\cdot\Delta_\lambda$ of $\mGH_\lambda^{\bS}$, so it is one-dimensional and spanned by the class of $\Delta_\lambda$ (because $\langle\Delta_\lambda,\Delta_\lambda\rangle\neq0$). Therefore the class of $gf$, which belongs to $U$, is a nonzero multiple of the class of $\Delta_\lambda$. The group $\bS_N$ acts on $\Delta_\lambda$ by the sign character.
\end{proof}

Thus $\mGH^{\bS}_\lambda$ and $\bC[\bar{z}_1,\ldots,\bar{z}_N]/\euJ^{\bS}_{\lambda}$ have the same multigraded $\bS_N$-character, and we refer to both of them as the Garsia--Haiman module.
Whenever the $\mathcal{A}_N$-module structure matters, we use the quotient model $\bC[\bar{z}_1,\ldots,\bar{z}_N]/\euJ^{\bS}_{\lambda}$, in which the polynomials act by multiplication.

%\begin{remark} \label{rem::GH::alternants}
%In item~(3) of Proposition~\ref{prp::GH::pairing} one cannot replace the hyperplane $\Delta_\lambda^\perp$ by the weaker condition $\Delta_\lambda\notin I$. The point is that the space of alternating polynomials of multidegree $\mdeg(\Delta_\lambda)$ may have dimension greater than one: the alternants $\Delta_S:=\det(\bar{z}_i^{\,\bar{t}_j})$ associated with $N$-element subsets $S=\{\bar{t}_1,\ldots,\bar{t}_N\}\subset\bZ_{\geq0}^{2}$ with $\sum_{\bar{t}\in S}\bar{t}=\mdeg(\Delta_\lambda)$ are linearly independent and pairwise orthogonal. For instance, the partitions $(3,3)$ and $(4,1,1)$ satisfy $\mdeg(\Delta_{(3,3)})=\mdeg(\Delta_{(4,1,1)})=(6,3)$. Then the annihilator $I$ of $\Delta_{(3,3)}+\Delta_{(4,1,1)}$ is an $\mathcal{A}_N$-submodule that does not contain $\Delta_{(3,3)}$ (since $\langle\Delta_{(3,3)},\Delta_{(3,3)}+\Delta_{(4,1,1)}\rangle\neq0$) and is not contained in $\euJ^{\bS}_{(3,3)}$. By item~(4), any submodule strictly containing $\euJ^{\bS}_{(3,3)}$ contains $\Delta_{(3,3)}$. Hence the family of $\mathcal{A}_N$-submodules not containing $\Delta_{(3,3)}$ has no largest element. This remark will be important for the correct definition of the Schur--Weyl dual ideals in~\S\ref{sec::GH::SW}.
%\end{remark}

The following theorem, conjectured by A.~Garsia and M.~Haiman~\cite{Garsia_Haiman,Garsia_Haiman::96} and proved by M.~Haiman~\cite{Haiman_n!}, is known as \textbf{the $N!$-conjecture}.

\begin{theorem}[\cite{Haiman_n!}]
\label{conj::n!}
Let $\lambda\vdash N$ be a Young diagram. Then
\begin{itemize}
\item[($\imath$)] the $\bS_N$-module $\mGH^{\bS}_{\lambda}$ is isomorphic to the regular representation $\bC[\bS_N]$; in particular, $\dim\mGH^{\bS}_{\lambda}=N!$;
\item[($\imath\imath$)] the bigraded Frobenius characteristic of $\mGH^{\bS}_{\lambda}$ is the modified Macdonald polynomial:
\[
\mathrm{Frob}_q\big(\mGH^{\bS}_{\lambda}\big)=\mathrm{Frob}_q\big(\bC[\bar{z}_1,\ldots,\bar{z}_N]/\euJ^{\bS}_{\lambda}\big)
=\widetilde{H}_{\lambda}(X;q_0,q_1):=q_1^{n(\lambda)}J_\lambda\Big[\frac{X}{1-q_1^{-1}};q,q_1^{-1}\Big],
\]
normalized in such a way that the variable $q=q_0$ corresponds to the direction of the rows of $\lambda$ (the coordinate $s_x$ of a cell $(s_x,s_y)$).
\end{itemize}
\end{theorem}
 In particular, for the one-row partition $\widetilde{H}_{(N)}(X;q_0,q_1)=\prod_{i=1}^{N}(1-q_0^{i})\cdot h_N\big[\frac{X}{1-q_0}\big]$ (see Example~\ref{ex::GH::row}); recall that $\widetilde{H}_{\lambda^T}(X;q,t)=\widetilde{H}_{\lambda}(X;t,q)$.
Item~($\imath\imath$) implies the Macdonald positivity conjecture~\cite{Macdonald1988}.
The inequality $\dim\mGH^{\bS}_{\lambda}\leq N!$ was known before Haiman's proof (see~\cite{Garsia_Haiman::96}), and it was shown in~\cite{Haiman::MacdonaldGeometry} that item~($\imath$) implies item~($\imath\imath$).
The direct analogue of Theorem~\ref{conj::n!} for higher-dimensional partitions (with more than two groups of variables) fails in general~\cite{Haiman::MacdonaldGeometry}.

\begin{example}
\label{ex::GH::row}
Let $\lambda=(N)$ be the one-row partition, whose cells are $(s,0)$, $0\leq s<N$.
Then the determinant~\eqref{eq::GH::Delta} is the classical Vandermonde determinant in the variables $x_{1},\ldots,x_{N}$:
\[
\Delta_{(N)}=\det\big(x_{i}^{\,j-1}\big)_{i,j=1}^N=\pm\prod_{1\leq i<j\leq N}(x_{i}-x_{j}).
\]
It is annihilated by all symmetric differential operators without constant term and by all derivatives $\partial/\partial y_{i}$.
Hence $\euJ_{(N)}^{\bS}$ is generated by the variables $y_{i}$ and by the symmetric polynomials in $x_{1},\ldots,x_{N}$ without constant term.
We recover the classical coinvariant algebra:
\[
\mGH_{(N)}^{\bS}\simeq\bC[x_{1},\ldots,x_{N}]\big/\big(\bC[x_{1},\ldots,x_{N}]^{\bS_N}_{+}\big)\simeq\bC[\bS_N],
\qquad
\widetilde{H}_{(N)}=\prod_{i=1}^{N}(1-q_0^{i})\cdot h_N\Big[\frac{X}{1-q_0}\Big].
\]
\end{example}
Let $\tau$ be the automorphism of $\bC[\bar{z}_1,\ldots,\bar{z}_N]$ that interchanges the two variables in each pair, $\tau(x_i):=y_i$ and $\tau(y_i):=x_i$. It commutes with the action of $\bS_N$, and it maps the monomial $\bar{z}_i^{\,\bar{s}}$ with $\bar{s}=(s_x,s_y)$ to $\bar{z}_i^{\,(s_y,s_x)}$; thus it corresponds to the transposition $\lambda\mapsto\lambda^T$ of Young diagrams.

\begin{lemma}
\label{lem::GH::symmetry}
We have $\tau(\Delta_\lambda)=\pm\Delta_{\lambda^T}$ and $\tau(\euJ_{\lambda}^{\bS})=\euJ_{\lambda^T}^{\bS}$.
Consequently, $\tau$ induces isomorphisms of $\mathcal{A}_N$-modules
\[
\mGH_{\lambda}^{\bS}\simeq\mGH_{\lambda^T}^{\bS},
\qquad
\bC[\bar{z}_1,\ldots,\bar{z}_N]/\euJ_{\lambda}^{\bS}\simeq\bC[\bar{z}_1,\ldots,\bar{z}_N]/\euJ_{\lambda^T}^{\bS},
\]
that interchange the two components of the multigrading. In particular, $$\widetilde{H}_{\lambda^T}(X;q_0,q_1)=\widetilde{H}_{\lambda}(X;q_1,q_0).$$
\end{lemma}
\begin{proof}
The first claim follows from $\tau(\bar{z}_i^{\,(s_x,s_y)})=\bar{z}_i^{\,(s_y,s_x)}$ and the definition~\eqref{eq::GH::Delta} of the determinant; the sign comes from the reordering of the cells.
The automorphism $\tau$ commutes with the action of $\bS_N$ and preserves the pairing~\eqref{eq::GH::pairing}; hence it maps $\euJ^{\bS}_\lambda$ to $\euJ^{\bS}_{\lambda^T}$ and $\mGH^{\bS}_\lambda$ to $\mGH^{\bS}_{\lambda^T}$.
\end{proof}

\subsection{The $\frac{N!}{k}$-conjecture}
\label{sec::GH::N!/k}

\subsubsection{Statement via intersections}
Let $\lambda\vdash N+1$ be a Young diagram and let $\bar{c}_1,\ldots,\bar{c}_k$ be $k$ pairwise distinct corners of $\lambda$ (see Definition~\ref{def::diagram::corner}).
Removing them one at a time we obtain $k$ pairwise distinct Young diagrams of the same size $N$:
\begin{equation}
\label{eq::nu::corners}
\nu_i:=\lambda\setminus\{\bar{c}_i\}\vdash N,\qquad i=1,\ldots,k.
\end{equation}
Since all $\nu_i$ have $N$ cells, the Garsia--Haiman modules $\mGH^{\bS}_{\nu_1},\ldots,\mGH^{\bS}_{\nu_k}$ are subspaces of the same polynomial ring $\bC[\bar{z}_1,\ldots,\bar{z}_N]$.
Their intersection
\[
\mGH^{\bS}_{\nu_1}\cap\ldots\cap\mGH^{\bS}_{\nu_k}\subset\bC[\bar{z}_1,\ldots,\bar{z}_N]
\]
is again a multigraded $\bS_N$-stable subspace closed under all partial derivatives.

The smallest partition with $k$ corners is the staircase partition $\delta_k:=(k,k-1,\ldots,1)\vdash\frac{k(k+1)}{2}$; in particular, $N\geq\frac{k(k+1)}{2}-1\geq 2k-2$.

To each Young diagram $\nu$ we assign the monomial
\begin{equation}
\label{eq::T::nu}
T_\nu:=q^{\mdeg(\Delta_\nu)}=\prod_{\bar{s}\in\nu}q_0^{s_x}q_1^{s_y}=q_0^{n(\nu^T)}q_1^{n(\nu)},
\end{equation}
where $\mdeg(\Delta_\nu)$ is given by~\eqref{eq::mdeg::Delta}.
Note that $T_{\nu_i}=T_\lambda\cdot q^{-\bar{c}_i}$, so the monomials $T_{\nu_1},\ldots,T_{\nu_k}$ are pairwise distinct.
Following~\cite{Bergeron_Garsia::SF}, we define the \emph{intersection polynomial}
\begin{equation}
\label{eq::SF::I}
\gh_{\nu_1,\ldots,\nu_k}(X;q_0,q_1):=\sum_{i=1}^{k}\Bigg(\prod_{j\neq i}\frac{T_{\nu_j}}{T_{\nu_j}-T_{\nu_i}}\Bigg)\,\widetilde{H}_{\nu_i}(X;q_0,q_1).
\end{equation}
For $k=1$ it is the Frobenius characteristic $\gh_{\nu_1}$ of the Garsia--Haiman module itself.
For $k=2$ it is Butler's function $\frac{T_{\nu_1}\gh_{\nu_2}-T_{\nu_2}\gh_{\nu_1}}{T_{\nu_1}-T_{\nu_2}}$~\cite{Butler}.
By Theorem~\ref{conj::n!}, it coincides with the \emph{Macdonald intersection polynomial} of~\cite{Bergeron_Garsia::SF,KLO::shuffle}.

\begin{conjecture}[the $\frac{N!}{k}$-conjecture, F.~Bergeron and A.~Garsia~\cite{Bergeron_Garsia::SF}]
\label{conj::N!/k}
Let $\lambda\vdash N+1$ be a Young diagram, and let $\nu_1,\ldots,\nu_k\vdash N$ be obtained from $\lambda$ by removing $k$ pairwise distinct corners as in~\eqref{eq::nu::corners}. Then
\begin{enumerate}
\item the dimension of the intersection of the Garsia--Haiman modules equals $\frac{N!}{k}$:
\[
\dim\left(\mGH^{\bS}_{\nu_1}\cap\ldots\cap\mGH^{\bS}_{\nu_k}\right)=\frac{N!}{k};
\]
\item the bigraded Frobenius characteristic of the intersection is given by the intersection polynomial:
\[
\mathrm{Frob}_q\left(\mGH^{\bS}_{\nu_1}\cap\ldots\cap\mGH^{\bS}_{\nu_k}\right)=\gh_{\nu_1,\ldots,\nu_k}(X;q_0,q_1).
\]
\end{enumerate}
\end{conjecture}

\begin{remark}
\label{rem::N!/k::status}
\begin{itemize}
\item For $k=1$ Conjecture~\ref{conj::N!/k} is Theorem~\ref{conj::n!}.
\item Item~(\it{ii}) implies item~(\it{i}). Indeed, the dimension of an $\bS_N$-module $M$ equals $\langle\mathrm{Frob}(M),h_1^N\rangle$, and Bergeron and Garsia showed that $\langle\gh_{\nu_1,\ldots,\nu_k}(X;1,1),h_1^N\rangle=\frac{N!}{k}$~\cite{Bergeron_Garsia::SF} (see also Theorem~\ref{thm::KLO::specialization} below).
\item The case $k=2$ (the ``$\frac{N!}{2}$-conjecture'') was introduced by Garsia and Haiman as a possible inductive route to the $N!$-conjecture~\cite{Garsia_Haiman::96}. Both items of Conjecture~\ref{conj::N!/k} are consequences of the more general \emph{Science Fiction conjecture}~\cite{Bergeron_Garsia::SF}, which predicts the existence of bases of the Garsia--Haiman modules compatible with all such intersections (see also~\cite{BBGHT} for the framework of lattice diagram polynomials). Item~(\it{ii}) for $k=2$ implies Butler's conjecture on the Schur positivity of $\gh_{\nu_1,\nu_2}$~\cite{Butler}, which is proved in several special cases in~\cite{KLO::Butler}.
\item To the best of our knowledge, Conjecture~\ref{conj::N!/k} is proved only for $k=2$ and hook partitions $\lambda$~\cite{Bergeron_Hamel,Armon} (for hooks, $k=2$ is the only nontrivial value). The first case with $k\geq3$ is $\lambda=(3,2,1)$.
\end{itemize}
\end{remark}

\subsubsection{Dual statement via common quotients}
By Proposition~\ref{prp::GH::pairing}, the orthogonal complement of $\mGH^{\bS}_{\nu_i}$ is the annihilating ideal $\euJ^{\bS}_{\nu_i}$.
Since the orthogonal complement of an intersection of subspaces is the sum of their orthogonal complements, the intersection of Garsia--Haiman modules is dual to the quotient by the \emph{sum} of the annihilating ideals.
We denote this quotient by
\begin{equation}
\label{eq::GH::common::quotient}
\mGH^{\bS}_{\nu_1,\ldots,\nu_k}:=\bC[\bar{z}_1,\ldots,\bar{z}_N]\big/\left(\euJ^{\bS}_{\nu_1}+\ldots+\euJ^{\bS}_{\nu_k}\right).
\end{equation}
For $k=1$ it is the quotient model $\bC[\bar{z}_1,\ldots,\bar{z}_N]/\euJ^{\bS}_{\nu_1}$ of the Garsia--Haiman module.

\begin{proposition}
\label{prp::N!/k::dual}
Let $\lambda\vdash N+1$ be a Young diagram and let $\nu_1,\ldots,\nu_k$ be as in~\eqref{eq::nu::corners}.
\begin{enumerate}
\item In every multigraded component of $\bC[\bar{z}_1,\ldots,\bar{z}_N]$ we have
\[
\left(\mGH^{\bS}_{\nu_1}\cap\ldots\cap\mGH^{\bS}_{\nu_k}\right)^{\perp}=\euJ^{\bS}_{\nu_1}+\ldots+\euJ^{\bS}_{\nu_k},
\qquad
\mGH^{\bS}_{\nu_1}\cap\ldots\cap\mGH^{\bS}_{\nu_k}=\left(\euJ^{\bS}_{\nu_1}+\ldots+\euJ^{\bS}_{\nu_k}\right)^{\perp}.
\]
\item The pairing~\eqref{eq::GH::pairing} induces a nondegenerate $\bS_N$-invariant pairing, respecting the multigrading, between the quotient $\mGH^{\bS}_{\nu_1,\ldots,\nu_k}$ and the intersection of the Garsia--Haiman modules. In particular,
\[
\mGH^{\bS}_{\nu_1,\ldots,\nu_k}\simeq\mGH^{\bS}_{\nu_1}\cap\ldots\cap\mGH^{\bS}_{\nu_k}
\]
as multigraded $\bS_N$-modules, and $\mGH^{\bS}_{\nu_1,\ldots,\nu_k}$ is the graded dual $\mathcal{A}_N$-module of the intersection.
\item The $\mathcal{A}_N$-module $\mGH^{\bS}_{\nu_1,\ldots,\nu_k}$ is the largest common cyclic quotient of the Garsia--Haiman modules $\bC[\bar{z}_1,\ldots,\bar{z}_N]/\euJ^{\bS}_{\nu_i}$, $i=1,\ldots,k$. Namely, it is generated by the class of $1$, the natural projection $\bC[\bar{z}_1,\ldots,\bar{z}_N]\twoheadrightarrow\mGH^{\bS}_{\nu_1,\ldots,\nu_k}$ factors through each projection $\bC[\bar{z}_1,\ldots,\bar{z}_N]\twoheadrightarrow\bC[\bar{z}_1,\ldots,\bar{z}_N]/\euJ^{\bS}_{\nu_i}$, and every surjection of $\mathcal{A}_N$-modules $\bC[\bar{z}_1,\ldots,\bar{z}_N]\twoheadrightarrow M$ with this property factors uniquely through $\mGH^{\bS}_{\nu_1,\ldots,\nu_k}$.
\end{enumerate}
\end{proposition}
\begin{proof}
All multigraded components are finite-dimensional and the pairing is nondegenerate on each of them. Therefore item~(i) follows from Proposition~\ref{prp::GH::pairing}(ii) and the identity $(W_1\cap\ldots\cap W_k)^{\perp}=W_1^{\perp}+\ldots+W_k^{\perp}$. Item~(ii) follows from item~(i) exactly as in the proof of Proposition~\ref{prp::GH::pairing}. For item~(iii), note that a surjection $\bC[\bar{z}_1,\ldots,\bar{z}_N]\twoheadrightarrow M$ factors through $\bC[\bar{z}_1,\ldots,\bar{z}_N]/\euJ^{\bS}_{\nu_i}$ if and only if its kernel contains $\euJ^{\bS}_{\nu_i}$; it does so for all $i$ if and only if its kernel contains $\euJ^{\bS}_{\nu_1}+\ldots+\euJ^{\bS}_{\nu_k}$.
\end{proof}

Consequently, Conjecture~\ref{conj::N!/k} admits the following equivalent reformulation in terms of quotient rings.

\begin{conjecture}[$\frac{N!}{k}$-conjecture, dual form]
\label{conj::N!/k::quotient}
Let $\lambda\vdash N+1$ and $\nu_1,\ldots,\nu_k\vdash N$ be as in Conjecture~\ref{conj::N!/k}. Then the bigraded $\mathcal{A}_N$-module
\[
\mGH^{\bS}_{\nu_1,\ldots,\nu_k}=\bC[\bar{z}_1,\ldots,\bar{z}_N]\big/\left(\euJ^{\bS}_{\nu_1}+\ldots+\euJ^{\bS}_{\nu_k}\right)
\]
has dimension $\frac{N!}{k}$, and its bigraded Frobenius characteristic equals $\gh_{\nu_1,\ldots,\nu_k}(X;q_0,q_1)$.
\end{conjecture}

\begin{remark}
\label{rem::N!/k::not::cocyclic}
By Proposition~\ref{prp::N!/k::dual}(iii), the module $\mGH^{\bS}_{\nu_1,\ldots,\nu_k}$ is cyclic, and its head $\soc$ is the trivial representation spanned by the class of $1$.
In contrast with the case $k=1$ (Proposition~\ref{prp::GH::pairing}(iv)), for $k\geq2$ it is not cocyclic in general. For example, for $\lambda=(3,2,1)$ and $k=3$ the socle $\cosoc$ contains two copies of the Specht module $\mathbb{S}_{(2,2,1)}$, sitting in the multidegrees $(1,2)$ and $(2,1)$ (see Example~\ref{ex::N!/k::staircase3}).
\end{remark}

\begin{example}
\label{ex::N!/k::staircase3}
Let $\lambda=\delta_3=(3,2,1)$, so that $N=5$, $k=3$, and $\nu_1=(2,2,1)$, $\nu_2=(3,1,1)$, $\nu_3=(3,2)$.
A direct computation of the intersection $\mGH^{\bS}_{(2,2,1)}\cap\mGH^{\bS}_{(3,1,1)}\cap\mGH^{\bS}_{(3,2)}$ gives the bigraded Frobenius characteristic
\[
s_{5}+(q_0+q_1)\,s_{41}+(q_0^2+q_0q_1+q_1^2)\,s_{32}+q_0q_1\,s_{311}+(q_0^2q_1+q_0q_1^2)\,s_{221},
\]
which coincides with the intersection polynomial $\gh_{(2,2,1),(3,1,1),(3,2)}$ of~\eqref{eq::SF::I}. The total dimension is $1+2\cdot4+3\cdot5+6+2\cdot5=40=\frac{5!}{3}$.
\end{example}

\subsubsection{Specialization at $q_0=q_1=1$ and shape independence}
Recall that the \emph{Kreweras number} of a partition $\mu=(1^{m_1}2^{m_2}\ldots)\vdash N$ with $\ell(\mu)$ parts,
\[
\mathrm{Krew}(\mu):=\frac{N!}{(N-\ell(\mu)+1)!\,\prod_{i\geq1}m_i!},
\]
is the number of noncrossing set partitions of $\{1,\ldots,N\}$ whose block sizes form the partition $\mu$~\cite{Kreweras}.
For $k\geq 2$ consider the following symmetric function of degree $2k-2$:
\begin{equation}
\label{eq::Phi_k}
\Phi_k:=\sum_{\mu\vdash k-1}(-1)^{k-1-\ell(\mu)}\,\mathrm{Krew}\big(\mu+(1^{\ell(\mu)})\big)\ h_{\mu+(1^{\ell(\mu)})}\,h_1^{\,k-1-\ell(\mu)},
\end{equation}
where $\mu+(1^{\ell(\mu)})=(\mu_1+1,\ldots,\mu_{\ell(\mu)}+1)$.

\begin{theorem}[D.~Kim, S.\,J.~Lee, J.~Oh~{\cite[Thm.~1.3]{KLO::shuffle}}]
\label{thm::KLO::specialization}
Let $\lambda\vdash N+1$ and $\nu_1,\ldots,\nu_k\vdash N$ be as in Conjecture~\ref{conj::N!/k}. Then
\[
\gh_{\nu_1,\ldots,\nu_k}(X;1,1)=h_1^{\,N-2k+2}\cdot\Phi_k .
\]
In particular, the specialization $\gh_{\nu_1,\ldots,\nu_k}(X;1,1)$ does not depend on $\lambda$ and on the choice of the corners, but only on $N$ and $k$, and $\langle\Phi_k,h_1^{2k-2}\rangle=\frac{(2k-2)!}{k}$.
\end{theorem}

The symmetric functions $\Phi_k$ for small $k$ are:
\begin{align*}
\Phi_2&=h_{2}=s_{2},\\
\Phi_3&=2h_{22}-h_{31}=s_{4}+s_{31}+2s_{22},\\
\Phi_4&=5h_{222}-5h_{321}+h_{411}=s_{6}+2s_{51}+6s_{42}+s_{411}+5s_{321}+5s_{222},\\
\Phi_5&=14h_{2222}-21h_{3221}+3h_{3311}+6h_{4211}-h_{5111}\\
&=s_{8}+3s_{71}+12s_{62}+3s_{611}+8s_{53}+20s_{521}+s_{5111}+12s_{44}+17s_{431}\\
&\quad+30s_{422}+9s_{4211}+3s_{332}+10s_{3311}+21s_{3221}+14s_{2222}.
\end{align*}
The Schur expansions show that $\Phi_k$ is Schur positive for $k\leq 5$. To the best of our knowledge it is not know whether $\Phi_k$ is Schur positive for all $k$.
%For $k\leq 6$ the coefficient of $s_{(2^{k-1})}$ in $\Phi_k$ is the Catalan number $\frac{1}{k}\binom{2k-2}{k-1}$.

If $\Phi_k$ is Schur positive, denote by $W_k$ the $\bS_{2k-2}$-module with $\mathrm{Frob}(W_k)=\Phi_k$; it has dimension $\frac{(2k-2)!}{k}$.
Since $\mathrm{Frob}\big(\mathrm{Ind}_{\bS_{m}}^{\bS_N}W\big)=h_1^{N-m}\,\mathrm{Frob}(W)$, Theorem~\ref{thm::KLO::specialization} shows that item~(ii) of Conjecture~\ref{conj::N!/k} implies the following ``shape independence'' of the ungraded modules:
\begin{equation}
\label{eq::N!/k::ungraded}
\mGH^{\bS}_{\nu_1}\cap\ldots\cap\mGH^{\bS}_{\nu_k}\ \simeq\ \mGH^{\bS}_{\nu_1,\ldots,\nu_k}\ \simeq\ \mathrm{Ind}_{\bS_{2k-2}}^{\bS_N}W_k
\quad\text{as $\bS_N$-modules.}
\end{equation}
For $k=2$ this is $\mathrm{Ind}_{\bS_2}^{\bS_N}\bC$, with Frobenius characteristic $h_2h_1^{N-2}$~\cite{KLO::Butler}.
The Schur--Weyl dual form of~\eqref{eq::N!/k::ungraded} is discussed in~\S\ref{sec::GH::SW}.

\begin{remark}
\label{rem::N!/k::evidence}
A computer calculation by exact linear algebra over $\bQ$ (derivative spaces, their intersections, and traces of permutations in every bidegree) confirms Conjecture~\ref{conj::N!/k} for all choices of $k$ corners of the partitions
$(2,1)$, $(3,1)$, $(3,2)$, $(2,2,1)$, $(4,1,1)$, $(3,3,1)$ (here $k=2$) and $(3,2,1)$, $(4,2,1)$, $(3,2,1,1)$ (here $k=3$):
in all these cases the dimension of the intersection is $\frac{N!}{k}$, and its bigraded Frobenius characteristic is the intersection polynomial~\eqref{eq::SF::I}. By Theorem~\ref{thm::KLO::specialization}, the isomorphism~\eqref{eq::N!/k::ungraded} holds in these cases as well.
\end{remark}

\subsection{Schur--Weyl duals of the $N!$- and the $\frac{N!}{k}$-conjectures}
\label{sec::GH::SW}

Throughout this subsection we assume that $m=\dim V\geq N$, we fix the basis $e_1,\ldots,e_m$ of $V$ and the matrix units $E_{ab}\in\gl_V$ as in \S\ref{sec::SW}, and we denote by
\begin{gather*}
\SW\colon\mathsf{Rep}(\bC[\bS_N]\ltimes\bC[\bar{z}_1,\ldots,\bar{z}_N])\rightleftarrows\RepGLz{V}{N}\colon \WS,
\\
\SW(M)=(V[\bar{z}])^{\otimes N}\otimes_{\mathcal{A}_N}M,\quad \WS(U)=U_{(1^N,0^{m-N})},
\end{gather*}
the Schur--Weyl functor and the conjugate Schur--Weyl functor of Theorem~\ref{thm::SW::currents}, which are mutually inverse equivalences of categories.
Being equivalences of categories compatible with the multigrading, $\SW$ and $\WS$ identify the lattices of (multigraded) submodules of corresponding objects, and they preserve sums, intersections, quotients, heads and socles.
They also match the simple objects: the Specht module $\mathbb{S}_\mu$ placed in multidegree $\bar{r}$, with the variables acting by zero, corresponds to $V_\mu$ placed in multidegree $\bar{r}$, with $\gl_V\otimes\bar{z}\bC[\bar{z}]$ acting by zero (see the corollary following Theorem~\ref{thm::SW::currents}).
In particular, the trivial and the sign representations of $\bS_N$ correspond to $S^NV$ and $\Lambda^NV$ respectively.
Finally, by Proposition~\ref{prop::SW::graded::characters}, $\mathrm{ch}_q(\SW(M))=\mathrm{Frob}_q(M)$.
In this subsection we translate the statements of~\S\ref{sec::GH::N!} and~\S\ref{sec::GH::N!/k} through this dictionary.

\subsubsection{Cyclic and cocyclic modules}
Let us recall the relevant notions for an algebra $A$ and an $A$-module $M$; in the multigraded setting all submodules are assumed to be multigraded.

\begin{definition}
\label{def::cyclic::cocyclic}
A vector $v\in M$ is \emph{cyclic} if $M=A\cdot v$.
A nonzero vector $w\in M$ is \emph{cocyclic} if every nonzero submodule of $M$ contains $w$; equivalently, for every nonzero $u\in M$ there exists $a\in A$ with $a\cdot u=w$.
The \emph{head} $\soc(M)$ of $M$ is its maximal semisimple quotient, and the \emph{socle} $\cosoc(M)$ of $M$ is its maximal semisimple submodule.
\end{definition}

The notion of the socle goes back to Frobenius~\cite{Fr}.
A finite-dimensional module with a simple head is cyclic, and every vector outside its radical is a cyclic vector.
A finite-dimensional module has a cocyclic vector if and only if its socle is simple, and in this case every nonzero vector of the socle is cocyclic.
The following elementary observation is our main tool for identifying cyclic-cocyclic modules.

\begin{lemma}
\label{lm::socle::head::iso}
Let $M_1$ and $M_2$ be finite-dimensional $A$-modules such that the socle of $M_1$ and the head of $M_2$ are simple.
Let $\varphi\colon M_1\to M_2$ be a homomorphism such that $\varphi(\cosoc(M_1))\neq0$ and $\varphi(M_1)$ is not contained in the radical of $M_2$.
Then $\varphi$ is an isomorphism.
\end{lemma}
\begin{proof}
If $\ker\varphi\neq0$, then $\ker\varphi$ contains the simple socle of $M_1$, which contradicts $\varphi(\cosoc(M_1))\neq0$; hence $\varphi$ is injective.
Since the head of $M_2$ is simple, the radical of $M_2$ is its unique maximal submodule. The image $\varphi(M_1)$ is not contained in it, hence $\varphi(M_1)=M_2$.
\end{proof}

\subsubsection{The polynomial ring and the symmetric power}
The cyclic $\mathcal{A}_N$-module $\bC[\bar{z}_1,\ldots,\bar{z}_N]=\mathcal{A}_N\otimes_{\bS_N}\bC$ corresponds to the symmetric power $S^N(V[\bar{z}])$ (see the example at the end of \S\ref{sec::SW}).
Explicitly, consider the linear map
\begin{equation}
\label{eq::SW::Psi}
\begin{gathered}
\Psi\colon V^{\otimes N}\otimes\bC[\bar{z}_1,\ldots,\bar{z}_N]\longrightarrow S^N(V[\bar{z}]),\\
e_{a_1}\otimes\ldots\otimes e_{a_N}\otimes\bar{z}_1^{\,\bar{r}_1}\cdots\bar{z}_N^{\,\bar{r}_N}\ \mapsto\
(e_{a_1}\bar{z}^{\,\bar{r}_1})\cdots(e_{a_N}\bar{z}^{\,\bar{r}_N}).
\end{gathered}
\end{equation}
It is the quotient map onto the space of coinvariants $\big((V[\bar{z}])^{\otimes N}\big)_{\bS_N}=(V[\bar{z}])^{\otimes N}\otimes_{\bS_N}\bC\simeq \SW\big(\bC[\bar{z}_1,\ldots,\bar{z}_N]\big)$.
Hence, for every $\mathcal{A}_N$-submodule (i.e., every multigraded $\bS_N$-stable ideal) $I\subset\bC[\bar{z}_1,\ldots,\bar{z}_N]$ we have
\[
\SW(I)=\Psi\big(V^{\otimes N}\otimes I\big)\subset S^N(V[\bar{z}]).
\]
Conversely, the weight subspace $\WS\big(S^N(V[\bar{z}])\big)=S^N(V[\bar{z}])_{(1^N,0^{m-N})}$ is identified with $\bC[\bar{z}_1,\ldots,\bar{z}_N]$ via $f\mapsto\Psi(e_1\otimes\ldots\otimes e_N\otimes f)$, and under this identification $\WS(\SW(I))=I$.
Thus $I\mapsto\Psi(V^{\otimes N}\otimes I)$ is an isomorphism between the lattice of multigraded $\bS_N$-stable ideals of $\bC[\bar{z}_1,\ldots,\bar{z}_N]$ and the lattice of $\gl_V[\bar{z}]$-submodules of $S^N(V[\bar{z}])$, with the inverse $K\mapsto \WS(K)$.

By Proposition~\ref{prop::characters::free} (applied to $\mathbb{S}_{(N)}=\bC$), $S^N(V[\bar{z}])$ is the projective cover of $S^NV$ in $\RepGLz{V}{N}$ (see also Proposition~\ref{prp::proj::cov}).
It is generated by the highest weight vector $v_{N\varepsilon_1}:=e_1^{N}=(e_1\bar{z}^{\,\bar{0}})^N$ of multidegree $\bar{0}$, which corresponds to the cyclic vector $1\in\bC[\bar{z}_1,\ldots,\bar{z}_N]$.

\subsubsection{Garsia--Haiman $\gl_V[\bar{z}]$-modules}
Fix a Young diagram $\lambda\vdash N$, an enumeration $\bar{s}_1,\ldots,\bar{s}_N$ of its cells, and the corresponding determinant $\Delta_\lambda$ defined in~\eqref{eq::GH::Delta}.
For indices $1\leq a_1,\ldots,a_N\leq m$ we set
$\delta_\lambda^{a_1\ldots a_N}:=\Psi\big(e_{a_1}\otimes\ldots\otimes e_{a_N}\otimes\Delta_\lambda\big)\in S^N(V[\bar{z}])$. Explicitly,
\begin{equation}
\label{eq::GH::determinant}
\begin{split}
\delta_\lambda^{a_1\ldots a_N}
&=\det\left(e_{a_i}\bar{z}^{\,\bar{s}_j}\right)_{i,j=1}^{N}
=\sum_{\sigma\in\bS_N}(-1)^{\sigma}\,(e_{a_1}\bar{z}^{\,\bar{s}_{\sigma(1)}})\cdots(e_{a_N}\bar{z}^{\,\bar{s}_{\sigma(N)}}) \\
&=\frac{1}{N!}\sum_{\sigma\in\bS_N}(-1)^{\sigma}\left(\prod_{i=1}^{N}E_{a_i1}\bar{z}^{\,\bar{s}_{\sigma(i)}}\right)v_{N\varepsilon_1}
\qquad\text{if } a_1,\ldots,a_N\geq2 .
\end{split}
\end{equation}
Here the determinant is computed in the commutative algebra $S^{\bullet}(V[\bar{z}])$.
The last equality holds because the operators $E_{a1}\bar{z}^{\,\bar{r}}$ with $a\geq2$ pairwise commute and act on $S^N(V[\bar{z}])$ by derivations, so that $\prod_{i=1}^{N}(E_{a_i1}\bar{z}^{\,\bar{r}_i})\cdot e_1^N=N!\,\prod_{i=1}^N(e_{a_i}\bar{z}^{\,\bar{r}_i})$.
We denote by $\mathcal{Q}_\lambda\subset S^N(V[\bar{z}])$ the linear span of all elements $\delta_\lambda^{a_1\ldots a_N}$.

Consider the $\gl_V$-equivariant map
\begin{equation}
\label{eq::phi::lambda}
\phi_\lambda\colon S^N(V[\bar{z}])\longrightarrow\Lambda^NV,\qquad
\Psi\big(e_{a_1}\otimes\ldots\otimes e_{a_N}\otimes f\big)\mapsto\langle f,\Delta_\lambda\rangle\ e_{a_1}\wedge\ldots\wedge e_{a_N}.
\end{equation}
It is well defined: it is obtained by applying the classical Schur--Weyl functor $\SW=V^{\otimes N}\otimes_{\bS_N}(-)$ to the $\bS_N$-equivariant functional $\langle-,\Delta_\lambda\rangle\colon\bC[\bar{z}_1,\ldots,\bar{z}_N]\to\mathrm{sgn}$, and using $\SW(\mathrm{sgn})\simeq\Lambda^NV$.
The map $\phi_\lambda$ vanishes outside the multidegree $\mdeg(\Delta_\lambda)$.
Note that $\phi_\lambda$ is a homomorphism of $\gl_V$-modules but not of $\gl_V[\bar{z}]$-modules.

\begin{definition}
\label{not::GH::gl}
The \emph{Garsia--Haiman $\gl_V[\bar{z}]$-module} is the quotient
\begin{multline*}
\Theta_{\lambda}\colon S^N(V[\bar{z}])\twoheadrightarrow\mGH^{\gl}_{\lambda}:=S^N(V[\bar{z}])/\euJ_{\lambda}^{\gl}, \\
\qquad\text{where}\quad
\euJ^{\gl}_{\lambda}:=\left\{u\in S^N(V[\bar{z}])\ \colon\ \phi_\lambda\big(U(\gl_V[\bar{z}])\cdot u\big)=0\right\}
\end{multline*}
is the largest $\gl_V[\bar{z}]$-submodule of $S^N(V[\bar{z}])$ contained in $\ker\phi_\lambda$.
\end{definition}

\begin{proposition}
\label{prp::GH::gl}
Let $\lambda\vdash N$ be a Young diagram and let $m=\dim V\geq N$.
\begin{enumerate}
\item The subspace $\mathcal{Q}_\lambda$ is a $\gl_V$-submodule of the multidegree $\mdeg(\Delta_\lambda)$ component of $S^N(V[\bar{z}])$, isomorphic to $\Lambda^NV$; namely, $\mathcal{Q}_\lambda=\Psi\big(V^{\otimes N}\otimes\bC\Delta_\lambda\big)\simeq\SW(\bC\cdot\Delta_\lambda)$ is the image of the sign representation under the classical Schur--Weyl functor. The elements $\delta_\lambda^{a_1\ldots a_N}$ are skew-symmetric in $a_1,\ldots,a_N$, and the restriction of $\phi_\lambda$ to $\mathcal{Q}_\lambda$ is an isomorphism.
\item Under the Schur--Weyl duality, the annihilating ideal $\euJ^{\bS}_\lambda$ corresponds to $\euJ^{\gl}_\lambda$ and the Garsia--Haiman module $\mGH^{\bS}_\lambda$ corresponds to $\mGH^{\gl}_\lambda$:
\[
\begin{array}{c}
\SW\big(\euJ^{\bS}_\lambda\big)=\Psi\big(V^{\otimes N}\otimes\euJ^{\bS}_\lambda\big)=\euJ^{\gl}_\lambda
\quad \Rightarrow \quad
\SW\big(\bC[\bar{z}_1,\ldots,\bar{z}_N]/\euJ^{\bS}_\lambda\big)\simeq\mGH^{\gl}_\lambda,
\\
\WS\big(\mGH^{\gl}_\lambda\big)\simeq\bC[\bar{z}_1,\ldots,\bar{z}_N]/\euJ^{\bS}_\lambda .
\end{array}
\]
In particular, $\mathrm{ch}_q\big(\mGH^{\gl}_\lambda\big)=\gh_\lambda(X;q_0,q_1)$.
\item $\euJ^{\gl}_\lambda\cap\mathcal{Q}_\lambda=0$, and $\euJ^{\gl}_\lambda$ is maximal among the $\gl_V[\bar{z}]$-submodules of $S^N(V[\bar{z}])$ with zero intersection with $\mathcal{Q}_\lambda$: every submodule strictly containing $\euJ^{\gl}_\lambda$ contains $\mathcal{Q}_\lambda$.
\item The module $\mGH^{\gl}_\lambda$ is cyclic and cocyclic. Its head is $\soc(\mGH^{\gl}_\lambda)\simeq S^NV$, sitting in multidegree $\bar{0}$ and generated by $\Theta_\lambda(v_{N\varepsilon_1})$. Its socle is $\cosoc(\mGH^{\gl}_\lambda)=\Theta_\lambda(\mathcal{Q}_\lambda)\simeq\Lambda^NV$, sitting in multidegree $\mdeg(\Delta_\lambda)=\sum_{\bar{s}\in\lambda}\bar{s}$. Moreover, $\mGH^{\gl}_\lambda$ vanishes in all multidegrees $\bar{r}$ with $\bar{r}\not\leq\mdeg(\Delta_\lambda)$ (componentwise), so the socle is the top multidegree component of $\mGH^{\gl}_\lambda$.
\end{enumerate}
\end{proposition}
\begin{proof}
(i) The restriction of $\Psi$ to $V^{\otimes N}\otimes\bC[\bar{z}_1,\ldots,\bar{z}_N]$ is the quotient map onto the $\bS_N$-coinvariants. Since taking coinvariants is exact in characteristic zero, $\Psi\big(V^{\otimes N}\otimes\bC\Delta_\lambda\big)\simeq\SW(\mathrm{sgn})\simeq\Lambda^NV$. The skew-symmetry is clear from the determinant formula~\eqref{eq::GH::determinant}, and $\phi_\lambda\big(\delta_\lambda^{a_1\ldots a_N}\big)=\langle\Delta_\lambda,\Delta_\lambda\rangle\,e_{a_1}\wedge\ldots\wedge e_{a_N}$ with $\langle\Delta_\lambda,\Delta_\lambda\rangle\neq0$.

(ii) Let $K\subset S^N(V[\bar{z}])$ be a $\gl_V[\bar{z}]$-submodule. The image $\phi_\lambda(K)$ is a $\gl_V$-submodule of $\Lambda^NV$, hence it vanishes if and only if its weight subspace of weight $(1^N,0^{m-N})$ vanishes. By the definition of $\phi_\lambda$, the latter subspace is spanned by $\langle f,\Delta_\lambda\rangle\,e_1\wedge\ldots\wedge e_N$ with $f\in \WS(K)$. Therefore $K\subset\ker\phi_\lambda$ if and only if $\WS(K)\subset\Delta_\lambda^{\perp}$. By the lattice isomorphism described above and by Proposition~\ref{prp::GH::pairing}(iii), the largest submodule contained in $\ker\phi_\lambda$ is $\SW(\euJ^{\bS}_\lambda)$. The largest submodule contained in $\ker\phi_\lambda$ is exactly $\euJ^{\gl}_\lambda$, which proves the first identity. The other two follow from the exactness of $\SW$ and from $\WS\circ \SW\simeq\mathrm{Id}$, and the equality of characters is Proposition~\ref{prop::SW::graded::characters}.

(iii) The first claim follows from item~(i) and $\euJ^{\gl}_\lambda\subset\ker\phi_\lambda$. If $K\supsetneq\euJ^{\gl}_\lambda$ is a submodule, then $\WS(K)\supsetneq\euJ^{\bS}_\lambda$. By Proposition~\ref{prp::GH::pairing}(iv), the nonzero submodule $\WS(K)/\euJ^{\bS}_\lambda$ contains the class of $\Delta_\lambda$, so $\Delta_\lambda\in \WS(K)$. Hence $\mathcal{Q}_\lambda=\Psi(V^{\otimes N}\otimes\bC\Delta_\lambda)\subset\Psi(V^{\otimes N}\otimes \WS(K))=K$.

(iv) By Proposition~\ref{prp::GH::pairing}(iv), the head of $\bC[\bar{z}_1,\ldots,\bar{z}_N]/\euJ^{\bS}_\lambda$ is the trivial representation spanned by the class of $1$, and its socle is the sign representation spanned by the class of $\Delta_\lambda$. The functor $\SW$ maps them to $S^NV$ and to $\SW(\bC\Delta_\lambda)=\Theta_\lambda(\mathcal{Q}_\lambda)$ respectively. Finally, the component of $\mGH^{\bS}_\lambda$ of multidegree $\bar{r}$ vanishes unless $\bar{r}\leq\mdeg(\Delta_\lambda)$, since $\mGH^{\bS}_\lambda$ is spanned by derivatives of $\Delta_\lambda$, and its component of multidegree $\mdeg(\Delta_\lambda)$ is $\bC\Delta_\lambda$. The same holds for the dual module $\bC[\bar{z}_1,\ldots,\bar{z}_N]/\euJ^{\bS}_\lambda$ and, by item~(ii), for $\mGH^{\gl}_\lambda$.
\end{proof}

Combining Proposition~\ref{prp::GH::gl} with Theorem~\ref{conj::n!} we obtain the Schur--Weyl dual form of the $N!$-theorem.

\begin{corollary}
\label{cor::N!::gl}
Let $\lambda\vdash N$ be a Young diagram and $m\geq N$. The following conditions are equivalent:
\begin{enumerate}
\item the $\bS_N$-module $\mGH^{\bS}_\lambda$ is isomorphic to the regular representation $\bC[\bS_N]$;
\item the Garsia--Haiman module $\mGH^{\gl}_\lambda$ is isomorphic to $V^{\otimes N}$ as a $\gl_V$-module;
\item the specialization $\mathrm{ch}_q(\mGH^{\gl}_\lambda)|_{q_0=q_1=1}$ equals $h_1^N$.
\end{enumerate}
In particular, thanks to Haiman for every Young diagram $\lambda\vdash N$ we have $\mGH^{\gl}_\lambda\simeq V^{\otimes N}$ as $\gl_V$-modules and $\mathrm{ch}_q(\mGH^{\gl}_\lambda)=\widetilde{H}_\lambda(X;q_0,q_1)$.
\end{corollary}
\begin{proof}
As a $\gl_V$-module, $\mGH^{\gl}_\lambda\simeq\SW\big(\mGH^{\bS}_\lambda\big)$ (with the classical functor $\SW$), so (i)$\Leftrightarrow$(ii) follows from the classical Schur--Weyl duality~\eqref{eq::Schur-Weyl}. Polynomial $\gl_V$-modules are semisimple and determined by their characters, hence (ii)$\Leftrightarrow$(iii). The last claim is Theorem~\ref{conj::n!} together with Proposition~\ref{prp::GH::gl}(ii).
\end{proof}

\subsubsection{The Schur--Weyl dual of the $\frac{N!}{k}$-conjecture}
Let $\lambda\vdash N+1$ be a Young diagram, and let $\nu_1,\ldots,\nu_k\vdash N$ be obtained from $\lambda$ by removing $k$ pairwise distinct corners as in~\eqref{eq::nu::corners}.
Following~\eqref{eq::GH::common::quotient}, we define the \emph{common quotient}
\begin{equation}
\label{eq::GHgl::common::quotient}
\mGH^{\gl}_{\nu_1,\ldots,\nu_k}:=S^N(V[\bar{z}])\big/\left(\euJ^{\gl}_{\nu_1}+\ldots+\euJ^{\gl}_{\nu_k}\right).
\end{equation}

\begin{proposition}
\label{prp::N!/k::gl}
\begin{enumerate}
\item The Schur--Weyl duality maps the sum of the annihilating ideals to the sum of the submodules $\euJ^{\gl}_{\nu_i}$ and the common quotient to the common quotient:
\[
\SW\big(\euJ^{\bS}_{\nu_1}+\ldots+\euJ^{\bS}_{\nu_k}\big)=\euJ^{\gl}_{\nu_1}+\ldots+\euJ^{\gl}_{\nu_k},
\qquad
\SW\big(\mGH^{\bS}_{\nu_1,\ldots,\nu_k}\big)\simeq\mGH^{\gl}_{\nu_1,\ldots,\nu_k}.
\]
Consequently,
\[
\mathrm{ch}_q\big(\mGH^{\gl}_{\nu_1,\ldots,\nu_k}\big)=\mathrm{Frob}_q\big(\mGH^{\bS}_{\nu_1,\ldots,\nu_k}\big)=\mathrm{Frob}_q\big(\mGH^{\bS}_{\nu_1}\cap\ldots\cap\mGH^{\bS}_{\nu_k}\big),
\]
and the weight subspace of $\mGH^{\gl}_{\nu_1,\ldots,\nu_k}$ of weight $(1^N,0^{m-N})$ has the same dimension as the intersection $\mGH^{\bS}_{\nu_1}\cap\ldots\cap\mGH^{\bS}_{\nu_k}$.
\item The $\gl_V[\bar{z}]$-module $\mGH^{\gl}_{\nu_1,\ldots,\nu_k}$ is the largest common cyclic quotient of the Garsia--Haiman modules $\mGH^{\gl}_{\nu_1},\ldots,\mGH^{\gl}_{\nu_k}$. It is generated by the image of $v_{N\varepsilon_1}$, and its head is $S^NV$, sitting in multidegree $\bar{0}$.
\item Let $S^N(V[\bar{z}])^{\vee}:=\bigoplus_{\bar{r}}\big(S^N(V[\bar{z}])_{\bar{r}}\big)^{*}$ be the graded dual module. Then the graded dual of the common quotient is the intersection of the graded duals of the Garsia--Haiman modules:
\[
\big(\mGH^{\gl}_{\nu_1,\ldots,\nu_k}\big)^{\vee}=\big(\mGH^{\gl}_{\nu_1}\big)^{\vee}\cap\ldots\cap\big(\mGH^{\gl}_{\nu_k}\big)^{\vee}\ \subset\ S^N(V[\bar{z}])^{\vee}.
\]
\end{enumerate}
\end{proposition}
\begin{proof}
Item~(i) follows from Proposition~\ref{prp::GH::gl}(ii), from the fact that the equivalence $\SW$ preserves sums of submodules and quotients, and from Proposition~\ref{prp::N!/k::dual}(ii). Item~(ii) is Proposition~\ref{prp::N!/k::dual}(iii) transported by $\SW$, together with Proposition~\ref{prp::GH::gl}(iv). Item~(iii) is the identity $\big(\sum_iK_i\big)^{\perp}=\bigcap_iK_i^{\perp}$ for the annihilators $K_i^\perp$ of the submodules $K_i=\euJ^{\gl}_{\nu_i}$ in the graded dual space.
\end{proof}

As in Remark~\ref{rem::N!/k::not::cocyclic}, the module $\mGH^{\gl}_{\nu_1,\ldots,\nu_k}$ is not cocyclic for $k\geq2$ in general: for $\lambda=(3,2,1)$ its socle contains two copies of $V_{(2,2,1)}$, sitting in the multidegrees $(1,2)$ and $(2,1)$.

We now formulate the Schur--Weyl dual of Conjecture~\ref{conj::N!/k}.
Recall the symmetric function $\Phi_k$ of degree $2k-2$ defined in~\eqref{eq::Phi_k}.
Whenever $\Phi_k$ is Schur positive (which is the case for $k\leq6$), we denote by $M_k$ the polynomial $\gl_V$-module of degree $2k-2$ with character $\mathrm{ch}(M_k)=\Phi_k$. In other words, $M_k=\SW(W_k)$ and $W_k=\WS(M_k)$ for the classical functors, where $W_k$ is the $\bS_{2k-2}$-module from~\S\ref{sec::GH::N!/k}; in particular, $\dim W_k=\dim (M_k)_{(1^{2k-2},0^{m-2k+2})}=\frac{(2k-2)!}{k}$.

\begin{conjecture}[Schur--Weyl dual $\frac{N!}{k}$-conjecture]
\label{conj::N!/k::gl}
Let $\lambda\vdash N+1$ be a Young diagram, and let $\nu_1,\ldots,\nu_k\vdash N$ be obtained from $\lambda$ by removing $k$ pairwise distinct corners. Then
\begin{enumerate}
\item the weight subspace of $\mGH^{\gl}_{\nu_1,\ldots,\nu_k}$ of weight $(1^N,0^{m-N})$ has dimension $\frac{N!}{k}$;
\item the bigraded character of the common quotient is given by the intersection polynomial~\eqref{eq::SF::I}:
\[
\mathrm{ch}_q\big(\mGH^{\gl}_{\nu_1,\ldots,\nu_k}\big)=\gh_{\nu_1,\ldots,\nu_k}(X;q_0,q_1);
\]
\item \textup{(weak form)} the polynomial $\Phi_k$ is Schur positive, and there is an isomorphism of $\gl_V$-modules
\[
\mGH^{\gl}_{\nu_1,\ldots,\nu_k}\ \simeq\ M_k\otimes V^{\otimes(N-2k+2)}.
\]
In particular, the $\gl_V$-module $\mGH^{\gl}_{\nu_1,\ldots,\nu_k}$ (with the grading forgotten) depends only on $N$ and $k$, and neither on $\lambda$ nor on the choice of the corners.
\end{enumerate}
\end{conjecture}

\begin{proposition}
\label{prp::N!/k::gl::relations}
\begin{enumerate}
\item Items~(i) and~(ii) of Conjecture~\ref{conj::N!/k::gl} are equivalent, respectively, to items~(i) and~(ii) of Conjecture~\ref{conj::N!/k}.
\item If $\Phi_k$ is Schur positive (for instance, if $k\leq6$), then item~(ii) implies item~(iii).
\item Item~(iii) implies item~(i).
\item Item~(iii) is the Schur--Weyl dual of the isomorphism of $\bS_N$-modules~\eqref{eq::N!/k::ungraded}: $\mGH^{\bS}_{\nu_1}\cap\ldots\cap\mGH^{\bS}_{\nu_k}\simeq\mathrm{Ind}_{\bS_{2k-2}}^{\bS_N}W_k$.
\end{enumerate}
\end{proposition}
\begin{proof}
Item~(i) is a reformulation of Proposition~\ref{prp::N!/k::gl}(i).
(ii) By Theorem~\ref{thm::KLO::specialization}, item~(ii) gives $\mathrm{ch}\big(\mGH^{\gl}_{\nu_1,\ldots,\nu_k}\big)=h_1^{N-2k+2}\Phi_k=\mathrm{ch}\big(M_k\otimes V^{\otimes(N-2k+2)}\big)$, and polynomial $\gl_V$-modules are determined by their characters.
(iii) and (iv) The character of the tensor product $U\otimes V^{\otimes r}$ is $\mathrm{ch}(U)\,h_1^{r}$, and $h_1^{r}\,\mathrm{Frob}(W)=\mathrm{Frob}\big(\mathrm{Ind}_{\bS_{N-r}}^{\bS_N}W\big)$ for any $\bS_{N-r}$-module $W$. Since the Schur--Weyl duality is the identity on characters~\eqref{eq::SW::characters}, we have $\WS\big(M_k\otimes V^{\otimes(N-2k+2)}\big)\simeq\mathrm{Ind}_{\bS_{2k-2}}^{\bS_N}W_k$, whose dimension is $\frac{N!}{(2k-2)!}\cdot\frac{(2k-2)!}{k}=\frac{N!}{k}$.
\end{proof}

\begin{example}
\label{ex::N!/k::gl}
\begin{itemize}
\item For $k=2$ we have $\Phi_2=h_2$, so $M_2=S^2V$ and item~(iii) of Conjecture~\ref{conj::N!/k::gl} predicts
\[
\mGH^{\gl}_{\nu_1,\nu_2}\simeq S^2V\otimes V^{\otimes(N-2)}.
\]
For the smallest partition $\lambda=(2,1)$ we have $\nu_1=(2)$ and $\nu_2=(1,1)$. Since $\Delta_{(2)}=\pm(x_{1}-x_{2})$ and $\Delta_{(1,1)}=\pm(y_{1}-y_{2})$, the ideal $\euJ^{\bS}_{(2)}$ (resp.\ $\euJ^{\bS}_{(1,1)}$) contains all variables $y_i$ (resp.\ $x_i$). Hence $\euJ^{\gl}_{(2)}$ (resp.\ $\euJ^{\gl}_{(1,1)}$) contains all elements of $S^2(V[\bar{z}])$ of positive $y$-degree (resp.\ $x$-degree). Hence the common quotient $\mGH^{\gl}_{(2),(1,1)}$ is concentrated in multidegree $\bar{0}$ and equals $S^2V=M_2$, in agreement with $\mGH^{\bS}_{(2)}\cap\mGH^{\bS}_{(1,1)}=\bC\cdot1$.
\item For $k=3$ we have $M_3\simeq S^4V\oplus V_{(3,1)}\oplus V_{(2,2)}^{\oplus2}$. For the staircase $\lambda=(3,2,1)$, with $N=5$, Example~\ref{ex::N!/k::staircase3} gives
\[
\mathrm{ch}_q\big(\mGH^{\gl}_{(2,2,1),(3,1,1),(3,2)}\big)=s_{5}+(q_0+q_1)\,s_{41}+(q_0^2+q_0q_1+q_1^2)\,s_{32}+q_0q_1\,s_{311}+(q_0^2q_1+q_0q_1^2)\,s_{221},
\]
where $s_\mu=\mathrm{ch}(V_\mu)$. At $q_0=q_1=1$ this equals $h_1\cdot(s_4+s_{31}+2s_{22})=\mathrm{ch}(M_3\otimes V)$, as predicted.
\item More generally, for the staircase partition $\lambda=\delta_k=(k,k-1,\ldots,1)$ and all its $k$ corners we have $N=\frac{k(k+1)}{2}-1$, and item~(iii) predicts
\[
\mGH^{\gl}_{\nu_1,\ldots,\nu_k}\simeq M_k\otimes V^{\otimes\binom{k-1}{2}}.
\]
Its bigraded character $\gh_{\nu_1,\ldots,\nu_k}(X;q_0,q_1)$ is a $(q_0,q_1)$-deformation of $\mathrm{ch}(M_k)\,h_1^{\binom{k-1}{2}}$.
Note that for $k\geq3$ the degree $2k-2$ of $M_k$ is smaller than $N=\frac{k(k+1)}{2}-1$.
\end{itemize}
\end{example}

\begin{remark}
By Remark~\ref{rem::N!/k::evidence}, all three items of Conjecture~\ref{conj::N!/k::gl} hold for all choices of $k$ corners of the partitions $(2,1)$, $(3,1)$, $(3,2)$, $(2,2,1)$, $(4,1,1)$, $(3,3,1)$ (for $k=2$) and $(3,2,1)$, $(4,2,1)$, $(3,2,1,1)$ (for $k=3$).
\end{remark}

\subsection{Sign duality for Garsia--Haiman modules}
\label{sec::GH::sign}

In this subsection we use the Schur--Weyl dictionary of Theorem~\ref{thm::SW::currents} and the \emph{sign duality functor} $\bomega$ on $\gl_V[\bar{z}]$-modules.
We show that the Garsia--Haiman modules are self-dual with respect to the contravariant sign duality $U\mapsto\bomega(U)^{\dagger}$ (Theorem~\ref{thm::GH::self::dual}).
This is the last place where the symmetric group is used. All the later sections use only the statements of this subsection, and their proofs are formulated purely for $\gl_V[\bar{z}]$-modules.

\subsubsection{The opposite Garsia--Haiman modules}
The $\gl_V[\bar{z}]$-module $\Lambda^N(V[\bar{z}])$ is generated by the vector $e_1\wedge\ldots\wedge e_N$, and its head is $\Lambda^NV$, sitting in multidegree $\bar{0}$.
Similarly to~\eqref{eq::SW::Psi}, consider the linear map
\begin{gather*}
\Psi^{\wedge}\colon V^{\otimes N}\otimes\bC[\bar{z}_1,\ldots,\bar{z}_N]\longrightarrow\Lambda^N(V[\bar{z}]),
\\
e_{a_1}\otimes\ldots\otimes e_{a_N}\otimes\bar{z}_1^{\,\bar{r}_1}\cdots\bar{z}_N^{\,\bar{r}_N}\ \mapsto\
(e_{a_1}\bar{z}^{\,\bar{r}_1})\wedge\ldots\wedge(e_{a_N}\bar{z}^{\,\bar{r}_N}).
\end{gather*}
It identifies $\Lambda^N(V[\bar{z}])$ with $\SW\big(\mathrm{sgn}\otimes\bC[\bar{z}_1,\ldots,\bar{z}_N]\big)$. As in~\S\ref{sec::GH::SW}, for every $\mathcal{A}_N$-submodule $I\subset\bC[\bar{z}_1,\ldots,\bar{z}_N]$ we have $\SW(\mathrm{sgn}\otimes I)=\Psi^{\wedge}(V^{\otimes N}\otimes I)$.

Let $\lambda\vdash N$ be a Young diagram with the cells enumerated as $\bar{s}_1,\ldots,\bar{s}_N$, and let $\Delta_\lambda$ be the determinant~\eqref{eq::GH::Delta}.
For $1\leq a_1,\ldots,a_N\leq m$ we set $\wcheck{\delta}_\lambda^{a_1\ldots a_N}:=\Psi^{\wedge}\big(e_{a_1}\otimes\ldots\otimes e_{a_N}\otimes\Delta_\lambda\big)$. Explicitly,
\begin{equation}
\label{eq::GH::opposite::element}
\begin{split}
\wcheck{\delta}_\lambda^{a_1\ldots a_N}
&=\sum_{\sigma\in\bS_N}(-1)^{\sigma}\,(e_{a_1}\bar{z}^{\,\bar{s}_{\sigma(1)}})\wedge\ldots\wedge(e_{a_N}\bar{z}^{\,\bar{s}_{\sigma(N)}})\\
&=\sum_{\sigma\in\bS_N}(-1)^{\sigma}\left(\prod_{i=1}^{N}E_{a_i,i}\,\bar{z}^{\,\bar{s}_{\sigma(i)}}\right)\cdot e_1\wedge\ldots\wedge e_N
\qquad\text{if } a_1,\ldots,a_N>N.
\end{split}
\end{equation}
The second equality holds because for $a_1,\ldots,a_N>N$ the operators $E_{a_i,i}\bar{z}^{\,\bar{r}}$ pairwise commute, and $E_{a_i,i}\bar{z}^{\,\bar{r}}$ acts only on the $i$-th factor $e_i$ of $e_1\wedge\ldots\wedge e_N$, replacing it with $e_{a_i}\bar{z}^{\,\bar{r}}$.
We denote by $\wcheck{\mathcal{Q}}_\lambda\subset\Lambda^N(V[\bar{z}])$ the linear span of all elements $\wcheck{\delta}_\lambda^{a_1\ldots a_N}$.

\begin{definition}
\label{def::GH::opposite}
The \emph{opposite Garsia--Haiman module} is the quotient
\[
\wcheck{\mGH}{}^{\gl}_{\lambda}:=\Lambda^N(V[\bar{z}])\big/\wcheck{\euJ}{}^{\gl}_{\lambda},
\qquad\text{where}\quad
\wcheck{\euJ}{}^{\gl}_{\lambda}:=\Psi^{\wedge}\big(V^{\otimes N}\otimes\euJ^{\bS}_\lambda\big)=\SW\big(\mathrm{sgn}\otimes\euJ^{\bS}_\lambda\big).
\]
\end{definition}

As in Definition~\ref{not::GH::gl}, the submodule $\wcheck{\euJ}{}^{\gl}_{\lambda}$ admits an intrinsic description.
It is the largest $\gl_V[\bar{z}]$-submodule of $\Lambda^N(V[\bar{z}])$ contained in the kernel of the $\gl_V$-equivariant map
\[
\wcheck{\phi}_\lambda\colon\Lambda^N(V[\bar{z}])\to S^NV,\qquad
\Psi^{\wedge}\big(e_{a_1}\otimes\ldots\otimes e_{a_N}\otimes f\big)\mapsto\langle f,\Delta_\lambda\rangle\,e_{a_1}\cdots e_{a_N},
\]
and the proof is the same as that of Proposition~\ref{prp::GH::gl}(ii).

\begin{proposition}
\label{prp::GH::opposite}
Let $\lambda\vdash N$ be a Young diagram and let $m=\dim V\geq N$.
\begin{enumerate}
\item The elements $\wcheck{\delta}_\lambda^{a_1\ldots a_N}$ are symmetric in $a_1,\ldots,a_N$, and $\wcheck{\mathcal{Q}}_\lambda\simeq S^NV$ as a $\gl_V$-module. It is the image of the trivial representation $\mathrm{sgn}\otimes\bC\Delta_\lambda$ under the classical functor $\SW$.
\item The sign duality functor maps the Garsia--Haiman module to its opposite:
\[
\bomega\big(\mGH^{\gl}_{\lambda}\big)\simeq\wcheck{\mGH}{}^{\gl}_{\lambda}.
\]
\item The module $\wcheck{\mGH}{}^{\gl}_{\lambda}$ is cyclic and cocyclic. Its head is $\soc\big(\wcheck{\mGH}{}^{\gl}_{\lambda}\big)\simeq\Lambda^NV$, sitting in multidegree $\bar{0}$ and generated by the image of $e_1\wedge\ldots\wedge e_N$. Its socle is $\cosoc\big(\wcheck{\mGH}{}^{\gl}_{\lambda}\big)\simeq S^NV$; it is the image of $\wcheck{\mathcal{Q}}_\lambda$, sitting in the top multidegree $\mdeg(\Delta_\lambda)$.
\item $\mathrm{ch}_q\big(\wcheck{\mGH}{}^{\gl}_{\lambda}\big)=\omega\,\gh_\lambda(X;q_0,q_1)$.
\end{enumerate}
\end{proposition}
\begin{proof}
(i) Let $\tau\in\bS_N$. Reordering the factors of the wedge product and substituting $\sigma\mapsto\sigma\tau$ in~\eqref{eq::GH::opposite::element}, we obtain two signs $(-1)^{\tau}$, which cancel. Hence $\wcheck{\delta}_\lambda^{a_{\tau(1)}\ldots a_{\tau(N)}}=\wcheck{\delta}_\lambda^{a_1\ldots a_N}$. The subspace $\bC\Delta_\lambda$ is the sign representation of $\bS_N$, so $\mathrm{sgn}\otimes\bC\Delta_\lambda$ is trivial. The restriction of $\Psi^{\wedge}$ to $V^{\otimes N}\otimes\bC\Delta_\lambda$ is the quotient map onto the space of coinvariants $V^{\otimes N}\otimes_{\bS_N}(\mathrm{sgn}\otimes\bC\Delta_\lambda)\simeq S^NV$.

(ii) By Proposition~\ref{prp::GH::gl}(ii), $\WS(\mGH^{\gl}_\lambda)\simeq\bC[\bar{z}_1,\ldots,\bar{z}_N]/\euJ^{\bS}_\lambda$. The functor $\SW$ is exact, hence
\[
\bomega\big(\mGH^{\gl}_\lambda\big)\simeq\SW\big(\mathrm{sgn}\otimes\bC[\bar{z}_1,\ldots,\bar{z}_N]\big)\big/\SW\big(\mathrm{sgn}\otimes\euJ^{\bS}_\lambda\big)
=\Lambda^N(V[\bar{z}])\big/\wcheck{\euJ}{}^{\gl}_{\lambda}.
\]

(iii) By item~(ii), this follows from Proposition~\ref{prp::GH::gl}(iv). The functor $\bomega$ preserves heads and socles and maps $S^NV$ and $\Lambda^NV$ to $\Lambda^NV$ and $S^NV$ respectively, without changing multidegrees. The socle is the image of $\wcheck{\mathcal{Q}}_\lambda=\SW(\mathrm{sgn}\otimes\bC\Delta_\lambda)$ because the socle of $\bC[\bar{z}_1,\ldots,\bar{z}_N]/\euJ^{\bS}_\lambda$ is spanned by the class of $\Delta_\lambda$.

(iv) This follows from item~(ii), \eqref{eq::sign::ch} and Proposition~\ref{prp::GH::gl}(ii).
\end{proof}

\subsubsection{Self-duality}
The main content of the sign duality is that the opposite Garsia--Haiman module is isomorphic to the contragredient dual of the Garsia--Haiman module with the reversed multigrading.

\begin{theorem}
\label{thm::GH::self::dual}
Let $\lambda\vdash N$ be a Young diagram and let $m=\dim V\geq N$. Then there are isomorphisms of multigraded modules
\begin{gather*}
\mathrm{sgn}\otimes\left(\bC[\bar{z}_1,\ldots,\bar{z}_N]/\euJ^{\bS}_\lambda\right)\simeq\left(\bC[\bar{z}_1,\ldots,\bar{z}_N]/\euJ^{\bS}_\lambda\right)^{\dagger}\big\langle\mdeg(\Delta_\lambda)\big\rangle,
\\
\wcheck{\mGH}{}^{\gl}_{\lambda}\simeq\bomega\big(\mGH^{\gl}_\lambda\big)\simeq\big(\mGH^{\gl}_\lambda\big)^{\dagger}\big\langle\mdeg(\Delta_\lambda)\big\rangle.
\end{gather*}
In other words, up to the shift of the multigrading by $\mdeg(\Delta_\lambda)$, the Garsia--Haiman module is mapped by the sign duality functor to its contravariant dual.
In particular,
\begin{equation}
\label{eq::GH::self::dual::char}
\omega\,\gh_\lambda(X;q_0,q_1)=T_\lambda\cdot\gh_\lambda\big(X;q_0^{-1},q_1^{-1}\big),
\qquad T_\lambda=q^{\mdeg(\Delta_\lambda)}.
\end{equation}
\end{theorem}
\begin{proof}
Denote $R:=\bC[\bar{z}_1,\ldots,\bar{z}_N]/\euJ^{\bS}_\lambda$ and consider the bilinear form
\[
B\colon R\times R\to\bC,\qquad B(f,g):=\langle fg,\Delta_\lambda\rangle.
\]
It is well defined, because $\langle hg,\Delta_\lambda\rangle=\langle g,h(\partial)\Delta_\lambda\rangle=0$ for $h\in\euJ^{\bS}_\lambda$.
It is nondegenerate: if $B(f,g)=\langle g,f(\partial)\Delta_\lambda\rangle=0$ for all $g$, then $f(\partial)\Delta_\lambda=0$, i.e., $f\in\euJ^{\bS}_\lambda$.
It pairs the multidegree $\bar{r}$ with the multidegree $\mdeg(\Delta_\lambda)-\bar{r}$.
Since the multiplication is commutative, $B(x_if,g)=B(f,x_ig)$ and $B(y_if,g)=B(f,y_ig)$.
Since the pairing~\eqref{eq::GH::pairing} is $\bS_N$-invariant and $\Delta_\lambda$ is alternating, $B(\sigma f,\sigma g)=\langle fg,\sigma^{-1}\Delta_\lambda\rangle=(-1)^{\sigma}B(f,g)$.
Therefore the map
\[
\mathrm{sgn}\otimes R\longrightarrow R^{\dagger}\big\langle\mdeg(\Delta_\lambda)\big\rangle,\qquad f\mapsto B(f,-),
\]
is a bijection that commutes with the action of the variables. It intertwines the action of $\sigma$ on $\mathrm{sgn}\otimes R$ with the action of $\sigma$ on $R^\dagger$, because $\big(\sigma\cdot B(f,-)\big)(g)=B(f,\sigma^{-1}g)=(-1)^{\sigma}B(\sigma f,g)$. It also preserves the multidegrees: $f$ of multidegree $\bar{r}$ is sent to a functional on $R_{\mdeg(\Delta_\lambda)-\bar{r}}$, which sits in multidegree $\bar{r}-\mdeg(\Delta_\lambda)$ of $R^\dagger$, that is, in multidegree $\bar{r}$ of $R^{\dagger}\langle\mdeg(\Delta_\lambda)\rangle$. This proves the first isomorphism.

Applying $\SW$ and using Proposition~\ref{prp::GH::opposite}(ii), Proposition~\ref{prp::GH::gl}(ii) and Lemma~\ref{lem::SW::dagger}, we obtain the second one:
\[
\wcheck{\mGH}{}^{\gl}_{\lambda}\simeq\bomega\big(\mGH^{\gl}_\lambda\big)=\SW\big(\mathrm{sgn}\otimes R\big)\simeq\SW\big(R^{\dagger}\big)\big\langle\mdeg(\Delta_\lambda)\big\rangle\simeq\big(\mGH^{\gl}_\lambda\big)^{\dagger}\big\langle\mdeg(\Delta_\lambda)\big\rangle.
\]
Finally, $\mathrm{ch}_q(U^{\dagger})(q_0,q_1)=\mathrm{ch}_q(U)(q_0^{-1},q_1^{-1})$, and the shift $\langle\bar{k}\rangle$ multiplies the character by $q^{\bar{k}}$. Therefore~\eqref{eq::GH::self::dual::char} follows from~\eqref{eq::sign::ch}.
\end{proof}

\begin{remark}
By Theorem~\ref{conj::n!}, identity~\eqref{eq::GH::self::dual::char} becomes the well-known symmetry $\omega\widetilde{H}_\lambda(X;q,t)=T_\lambda\,\widetilde{H}_\lambda(X;q^{-1},t^{-1})$ of the modified Macdonald polynomials. Theorem~\ref{thm::GH::self::dual} is a module-theoretic version of this symmetry, which does not use Haiman's theorem. From the point of view of commutative algebra, it expresses the fact that $\bC[\bar{z}_1,\ldots,\bar{z}_N]/\euJ^{\bS}_\lambda$ is a Gorenstein algebra whose socle $\bC\Delta_\lambda$ is the sign representation (Proposition~\ref{prp::GH::pairing}(iv)).
\end{remark}

Applying the contragredient duality to the second isomorphism of Theorem~\ref{thm::GH::self::dual}, we get the form of self-duality that will be used in the rest of the paper. We write $y$-degree for the second component of the multidegree.

\begin{corollary}
\label{cor::GH::self::sign}
Let $\lambda\vdash N$ be a Young diagram and $m=\dim V\geq N$.
\begin{enumerate}
\item The contravariant sign duality $U\mapsto\bomega(U)^{\dagger}$ fixes the Garsia--Haiman module up to a shift of the multigrading:
\[
\bomega\big(\mGH^{\gl}_\lambda\big)^{\dagger}\simeq\mGH^{\gl}_\lambda\big\langle-\mdeg(\Delta_\lambda)\big\rangle .
\]
\item The module $\bomega(\mGH^{\gl}_\lambda)$ has the same graded dimensions as $\mGH^{\gl}_\lambda$, read in the opposite order: for all $a,b$
\[
\dim\big(\bomega(\mGH^{\gl}_\lambda)\big)_{(a,b)}=\dim\big(\mGH^{\gl}_\lambda\big)_{(n(\lambda^T)-a,\,n(\lambda)-b)},
\]
where $\mdeg(\Delta_\lambda)=(n(\lambda^T),n(\lambda))$.
\end{enumerate}
\end{corollary}
\begin{proof}
Both statements follow from Theorem~\ref{thm::GH::self::dual}, since $(U^{\dagger}\langle\bar{k}\rangle)^{\dagger}\simeq U\langle-\bar{k}\rangle$ and $\bomega$ preserves the multigrading.
\end{proof}

\section{Fusion product}
\label{sec::fusion}

In this section we recall the fusion product of cyclic modules introduced by Feigin and Loktev, and we prove a vanishing property of fusion products over abelian Lie algebras, which is used repeatedly in the rest of the paper.

\subsection{Fusion product of cyclic modules}

The \emph{fusion product} of cyclic modules was introduced in~\cite{FeLo}. It is widely used in representation theory and mathematical physics (see, e.g.,~\cite{FL2,FeMa3,Dumanski}).

Let $\fg$ be a Lie algebra and let $\fg[z]=\fg\otimes\bC[z]$ be the corresponding current algebra.
For a $\fg$-module $M$ and a number $\gamma\in\bC$, the \emph{evaluation module} $\ev_{\gamma}(M)$ is the $\fg[z]$-module obtained from $M$ by setting $z=\gamma$:
\[
(x z^k)\cdot m:=\gamma^k\, x\cdot m,\qquad x\in\fg,\ m\in M,\ k\in\bZ_{\geq0}.
\]
In other words, $\ev_\gamma(M)$ is the pullback of $M$ along the evaluation homomorphism $\fg[z]\to\fg$, $z\mapsto\gamma$. As a vector space it coincides with $M$, but its $\fg[z]$-module structure depends on $\gamma$.
Given $\fg$-modules $M_0,\ldots,M_l$ and pairwise distinct numbers $\gamma_0,\ldots,\gamma_l\in\bC$, we obtain the $\fg[z]$-module
\begin{equation}
\label{eq::ev::prod}
\ev_{\gamma_0}(M_0)\otimes\cdots\otimes\ev_{\gamma_l}(M_l),
\end{equation}
whose module structure depends on the parameters $\gamma_i$.
Denote by $x^{(i)}:=1^{\otimes i}\otimes x\otimes1^{\otimes(l-i)}$ the action of $x\in\fg$ on the $i$-th tensor factor of~\eqref{eq::ev::prod}. Then, for every polynomial $p\in\bC[z]$, the element $x\,p(z)\in\fg[z]$ acts on~\eqref{eq::ev::prod} by the operator
\begin{equation}
\label{eq::ev::action}
x\,p(z)\ \longmapsto\ \sum_{i=0}^{l}p(\gamma_i)\,x^{(i)};
\qquad\text{in particular,}\quad x z^k\ \longmapsto\ \sum_{i=0}^{l}\gamma_i^k\,x^{(i)}.
\end{equation}

The following lemma is elementary, but it plays a central role in our applications, so we give it a name.

\begin{lemma}[Vandermonde Lemma]
\label{lem::Vandermonde}
Let $\gamma_0,\ldots,\gamma_l$ be pairwise distinct and let $x\in\fg$. Then the following identities of linear spans of operators acting on~\eqref{eq::ev::prod} hold.
\begin{enumerate}
\item We have
\[
\Span\langle x^{(0)},\ldots,x^{(l)}\rangle=\Span\langle x,\,xz,\,\ldots,\,xz^{l}\rangle .
\]
In particular, $xz^k\in\Span\langle x^{(0)},\ldots,x^{(l)}\rangle$ for all $k\geq0$.
\item For all $0\leq j\leq\min(k,l+1)$ we have
\[
xz^k\in\Span\langle x,\,xz,\,\ldots,\,xz^{j-1}\rangle+\Span\langle x^{(j)},\,x^{(j+1)},\,\ldots,\,x^{(l)}\rangle .
\]
More precisely, $x\prod_{s=0}^{j-1}(z-\gamma_s)=\sum_{i=j}^{l}\prod_{s=0}^{j-1}(\gamma_i-\gamma_s)\,x^{(i)}$ acts only on the tensor factors with indices $i\geq j$.
\end{enumerate}
\end{lemma}
\begin{proof}
(i) By~\eqref{eq::ev::action}, the operators $x,xz,\ldots,xz^l$ are expressed through $x^{(0)},\ldots,x^{(l)}$ by the Vandermonde matrix $(\gamma_i^k)_{0\leq i,k\leq l}$. This matrix is invertible because the numbers $\gamma_i$ are pairwise distinct.

(ii) The polynomial $p_j(z):=\prod_{s=0}^{j-1}(z-\gamma_s)$ vanishes at $\gamma_0,\ldots,\gamma_{j-1}$, so the second claim follows from~\eqref{eq::ev::action}. Now divide $z^k$ by $p_j$ with remainder: $z^k=q(z)p_j(z)+r(z)$ with $\deg r<j$. By~\eqref{eq::ev::action}, the operator $x\,q(z)p_j(z)=\sum_{i\geq j}q(\gamma_i)p_j(\gamma_i)\,x^{(i)}$ acts only on the factors with indices $i\geq j$. The operator $x\,r(z)$ is a linear combination of $x,xz,\ldots,xz^{j-1}$.
\end{proof}

\begin{corollary}
\label{cor::ev::cyclic}
Let $M_0,\ldots,M_l$ be cyclic $\fg$-modules with cyclic vectors $v_0,\ldots,v_l$, and assume that the numbers $\gamma_0,\ldots,\gamma_l$ are pairwise distinct. Then the vector $v_0\otimes\cdots\otimes v_l$ generates the tensor product~\eqref{eq::ev::prod} as a $\fg[z]$-module.
\end{corollary}
\begin{proof}
By Lemma~\ref{lem::Vandermonde}(i), the image of $U(\fg[z])$ in the endomorphisms of~\eqref{eq::ev::prod} contains all the operators $x^{(i)}$. Hence it contains the image of $U(\fg)^{\otimes(l+1)}$, which acts cyclically on $v_0\otimes\cdots\otimes v_l$.
\end{proof}

The module~\eqref{eq::ev::prod} is not graded. However, it carries a natural filtration.
The universal enveloping algebra $U(\fg[z])$ is filtered by the $z$-degree: $U(\fg[z])_{\leq k}$ is spanned by the products $(x_1z^{i_1})\cdots(x_sz^{i_s})$ with $i_1+\ldots+i_s\leq k$.
For a cyclic $\fg[z]$-module $M$ with a cyclic vector $v$, this filtration induces the \emph{fusion filtration}
\begin{equation}
\label{eq::fusion::filtration}
\calF_{-1}M:=0,\qquad
\calF_{k}M:=U(\fg[z])_{\leq k}\cdot v,\quad k\geq0.
\end{equation}
Equivalently, $\calF_0M=U(\fg)\cdot v$ and $\calF_kM=\calF_{k-1}M+\sum_{i=1}^{k}(\fg z^i)\cdot\calF_{k-i}M$.
Every $\calF_kM$ is a $\fg$-submodule, $\bigcup_k\calF_kM=M$, and $(\fg z^i)\cdot\calF_kM\subset\calF_{k+i}M$.
Consequently, the associated graded space $\mathrm{gr}^{\calF}M=\bigoplus_{k\geq0}\calF_kM/\calF_{k-1}M$ is a graded $\fg[z]$-module, where $xz^i$ acts by an operator of degree $i$. It is cyclic and is generated by the image of $v$ in degree $0$.

\begin{definition}
\label{def::fusion}
Let $M_0,\ldots,M_l$ be cyclic $\fg$-modules with chosen cyclic vectors $v_0,\ldots,v_l$, and let $\gamma_0,\ldots,\gamma_l\in\bC$ be pairwise distinct.
The \emph{fusion product} is the graded $\fg[z]$-module
\[
[M_0\star_z\cdots\star_z M_l]_{\gamma_0,\ldots,\gamma_l}:=\bigoplus_{k\geq0}\calF_k\big/\calF_{k-1}\Big[\ev_{\gamma_0}(M_0)\otimes\cdots\otimes\ev_{\gamma_l}(M_l)\Big],
\]
where $\calF_\bullet$ is the fusion filtration~\eqref{eq::fusion::filtration} with respect to the cyclic vector $v_0\otimes\cdots\otimes v_l$ (see Corollary~\ref{cor::ev::cyclic}).
When the fusion product does not depend on the parameters $\gamma_i$, we omit them from the notation and write $M_0\star_z\cdots\star_z M_l$.
\end{definition}

\begin{notation}
For an element $c\in\calF_k\setminus\calF_{k-1}$ we denote by $\widehat{c}$ its image in $\calF_k/\calF_{k-1}$, the \emph{symbol} of $c$. In particular, $\widehat{v}$ denotes the symbol of $v=v_0\otimes\cdots\otimes v_l$; it is the cyclic vector of the fusion product.
By construction, for $x_1,\ldots,x_s\in\fg$ and $R:=r_1+\ldots+r_s$, the element $(x_1z^{r_1})\cdots(x_sz^{r_s})\cdot\widehat{v}$ of the fusion product is the class of $(x_1z^{r_1})\cdots(x_sz^{r_s})\cdot v$ in $\calF_R/\calF_{R-1}$. In particular, it vanishes if and only if $(x_1z^{r_1})\cdots(x_sz^{r_s})\cdot v\in\calF_{R-1}$.
\end{notation}

\begin{remark}
\label{rem::fusion::graded}
\begin{enumerate}
\item Since the fusion filtration is exhaustive, the fusion product has the same dimension as $M_0\otimes\cdots\otimes M_l$. If the $\fg$-modules $M_i$ are finite-dimensional and $\fg$ acts on them semisimply (e.g., $\fg=\gl_V$ and the $M_i$ are polynomial representations), then the fusion product is isomorphic to $M_0\otimes\cdots\otimes M_l$ as a $\fg$-module.
\item Suppose that $\fg$ carries an additional grading (for instance, $\fg=\gl_V[x]$ graded by the $x$-degree, so that $\fg[y]=\gl_V[x,y]$), the modules $M_i$ are graded, and the cyclic vectors $v_i$ are homogeneous. Then the evaluation at $z=\gamma_i$ does not affect this grading. Hence the subspaces $\calF_k$ are graded, and the fusion product is bigraded by the $z$-degree and the additional grading. This is the setting of the iterated fusion products of~\S\ref{sec::sub::fusion}.
\end{enumerate}
\end{remark}

In practice, fusion products are hard to control. There are very few cases in which one can prove that a fusion product does not depend on the evaluation parameters. Moreover, even when it does not depend on the parameters, it may still depend on the choice of the cyclic vectors. In almost all constructions of this paper, the cyclic vectors are the highest weight vectors of the corresponding representations.
The first example of a fusion product that is independent of the parameters, the local Weyl module, is recalled in~\S\ref{sec::Weyl} (Theorem~\ref{thm::Weyl::fusion}).
The main goal of this paper is to relate the Garsia--Haiman modules $\mGH_{\lambda}^{\gl}$ with certain iterated fusion products of $N$ copies of the tautological representation $V$.

\subsection{Fusion product of graded cyclic $\fg[z]$-modules}
\label{sec::fusion::graded}

The construction of Definition~\ref{def::fusion} can be applied not only to $\fg$-modules but also to graded cyclic modules over the current algebra $\fg[z]$ (see~\cite{FeLo} and~\cite[\S1.2]{Dumanski}). This allows one to iterate fusion products.

A \emph{graded} $\fg[z]$-module is a $\fg[z]$-module $M=\bigoplus_{k\geq0}M_k$ such that $(xz^s)\cdot M_k\subset M_{k+s}$ for all $x\in\fg$ and $s,k\geq0$. We say that $M$ is \emph{graded cyclic} with cyclic vector $m$ if $m\in M_0$ and $M=U(\fg[z])\cdot m$.
If such a module $M$ is finite-dimensional, then $M=\bigoplus_{k=0}^{D}M_k$ for some $D$. Hence the subalgebra $\fg\otimes z^{D+1}\bC[z]$ acts on $M$ by zero, since it raises the degree by more than $D$.
Every $\fg$-module $M$ is a graded $\fg[z]$-module concentrated in degree $0$, on which $\fg\otimes z\bC[z]$ acts by zero; this is the evaluation module $\ev_0(M)$.

For $\gamma\in\bC$ denote by $\tau_\gamma$ the automorphism of the Lie algebra $\fg[z]$ given by $x\,p(z)\mapsto x\,p(z+\gamma)$.
For a $\fg[z]$-module $M$ we denote by $M(\gamma)$ the pullback of $M$ along $\tau_\gamma$. In other words, $M(\gamma)=M$ as a vector space, and $xz^s$ acts on $M(\gamma)$ as the element $x(z+\gamma)^s$ acts on $M$.
The module $M(\gamma)$ is no longer graded for $\gamma\neq0$, but it is cyclic with the same cyclic vector. For a $\fg$-module $M$ we have $\ev_0(M)(\gamma)=\ev_\gamma(M)$.

\begin{lemma}
\label{lem::fusion::graded::cyclic}
Let $M_0,\ldots,M_l$ be finite-dimensional graded cyclic $\fg[z]$-modules with cyclic vectors $m_0,\ldots,m_l$, and let $\gamma_0,\ldots,\gamma_l\in\bC$ be pairwise distinct. Then the $\fg[z]$-module $M_0(\gamma_0)\otimes\cdots\otimes M_l(\gamma_l)$ is generated by the vector $m_0\otimes\cdots\otimes m_l$.
\end{lemma}
\begin{proof}
Choose $n$ such that $\fg\otimes z^{n}\bC[z]$ acts by zero on all $M_i$. Then $\fg\otimes(z-\gamma_i)^{n}\bC[z]$ acts by zero on $M_i(\gamma_i)$.
The polynomials $(z-\gamma_0)^n,\ldots,(z-\gamma_l)^n$ are pairwise coprime. By the Chinese remainder theorem, for every $i$ and every $p\in\bC[z]$ there is a polynomial $q\in\bC[z]$ with $q\equiv p$ modulo $(z-\gamma_i)^n$ and $q\equiv0$ modulo $(z-\gamma_{i'})^n$ for all $i'\neq i$.
The element $x\,q(z)$ acts on the tensor product as $x\,p(z)$ on the $i$-th factor and as zero on all other factors.
Hence the image of $U(\fg[z])$ in the endomorphisms of the tensor product contains the image of $U(\fg[z])^{\otimes(l+1)}$, acting factorwise. The latter acts cyclically on $m_0\otimes\cdots\otimes m_l$, because each $m_i$ generates $M_i(\gamma_i)$.
\end{proof}

\begin{definition}[{\cite{FeLo}, \cite[\S1.2]{Dumanski}}]
\label{def::fusion::graded}
Let $M_0,\ldots,M_l$ be finite-dimensional graded cyclic $\fg[z]$-modules with cyclic vectors $m_0,\ldots,m_l$, and let $\gamma_0,\ldots,\gamma_l\in\bC$ be pairwise distinct.
The \emph{fusion product} of $M_0,\ldots,M_l$ is the associated graded module
\[
[M_0\star_z\cdots\star_z M_l]_{\gamma_0,\ldots,\gamma_l}:=\bigoplus_{k\geq0}\calF_k\big/\calF_{k-1}\Big[M_0(\gamma_0)\otimes\cdots\otimes M_l(\gamma_l)\Big]
\]
with respect to the fusion filtration $\calF_k=U(\fg[z])_{\leq k}\cdot(m_0\otimes\cdots\otimes m_l)$ of~\eqref{eq::fusion::filtration}.
\end{definition}

For $\fg$-modules, i.e., for graded $\fg[z]$-modules concentrated in degree $0$, Definition~\ref{def::fusion::graded} coincides with Definition~\ref{def::fusion}, because $\ev_0(M_i)(\gamma_i)=\ev_{\gamma_i}(M_i)$.
As in Remark~\ref{rem::fusion::graded}, the fusion product is a graded cyclic $\fg[z]$-module generated by $\widehat{m_0\otimes\cdots\otimes m_l}$ in degree $0$, and its dimension equals $\prod_i\dim M_i$.
Note that the gradings of the modules $M_i$ are not used in the grading of the fusion product (they are destroyed by the shifts $\tau_{\gamma_i}$); the grading of the fusion product is the filtration degree.
Since the fusion product is again a finite-dimensional graded cyclic $\fg[z]$-module, it can be fused with other modules. This gives \emph{iterated fusion products}
\begin{equation}
\label{eq::fusion::iterated::same}
\Big[\,[M_{0,1}\star_z\cdots\star_z M_{0,k_0}]_{\beta_0}\ \star_z\ \cdots\ \star_z\ [M_{l,1}\star_z\cdots\star_z M_{l,k_l}]_{\beta_l}\,\Big]_{\gamma_0,\ldots,\gamma_l},
\end{equation}
where $\beta_i=(\beta_{i,1},\ldots,\beta_{i,k_i})$ are the parameters of the inner fusion products.

The fusion product does not change under affine transformations of the parameters.

\begin{lemma}
\label{lem::fusion::affine}
In the setting of Definition~\ref{def::fusion::graded}, for all $a\in\bC^{\times}$ and $b\in\bC$ there is an isomorphism of graded $\fg[z]$-modules
\[
[M_0\star_z\cdots\star_z M_l]_{a\gamma_0+b,\ldots,a\gamma_l+b}\ \simeq\ [M_0\star_z\cdots\star_z M_l]_{\gamma_0,\ldots,\gamma_l},
\]
mapping the cyclic vector to the cyclic vector.
In particular, the fusion product of two modules ($l=1$) does not depend on the parameters.
\end{lemma}
\begin{proof}
Let $\theta$ be the automorphism of $U(\fg[z])$ induced by $x\,p(z)\mapsto x\,p(az+b)$. Both $\theta$ and $\theta^{-1}$ preserve the $z$-degree filtration, and $\theta(xz^s)=a^s\,xz^s+(\text{terms of $z$-degree}<s)$.
Let $D$ be the operator on $M_0\otimes\cdots\otimes M_l$ acting on $(M_0)_{k_0}\otimes\cdots\otimes(M_l)_{k_l}$ by $a^{k_0+\ldots+k_l}$. Since $xz^r$ raises the degree by $r$, we have $D^{-1}(xz^r)D=a^{-r}xz^r$ on each factor.
Denote by $\rho_\gamma$ the action of $U(\fg[z])$ on $M_0(\gamma_0)\otimes\cdots\otimes M_l(\gamma_l)$. A direct computation on the generators $x\,p(z)$ shows that
\[
\rho_{a\gamma+b}(u)=D^{-1}\,\rho_{\gamma}\big(\theta(u)\big)\,D\qquad\text{for all } u\in U(\fg[z]).
\]
Indeed, on the $i$-th factor, $\rho_\gamma(\theta(x\,p(z)))=x\,p(az+a\gamma_i+b)$, and conjugation by $D$ replaces $z$ by $z/a$.
Since $D$ fixes $v=m_0\otimes\cdots\otimes m_l$ and $\theta$ preserves the filtration, we get $\calF^{a\gamma+b}_k=D^{-1}\calF^{\gamma}_k$.
Define $\Phi\colon\calF^{a\gamma+b}_k/\calF^{a\gamma+b}_{k-1}\to\calF^{\gamma}_k/\calF^{\gamma}_{k-1}$ by $\Phi[w]:=a^{-k}[Dw]$. For $u=xz^s$ we obtain
\[
\Phi(u\cdot[w])=a^{-k-s}\big[\rho_\gamma(\theta(u))\,Dw\big]=a^{-k-s}a^{s}\big[\rho_\gamma(xz^s)\,Dw\big]=u\cdot\Phi[w],
\]
because the terms of $\theta(u)$ of lower $z$-degree map $Dw\in\calF^\gamma_k$ into $\calF^\gamma_{k+s-1}$. Thus $\Phi$ is an isomorphism of graded $\fg[z]$-modules. Finally, any two pairs of distinct numbers are related by an affine transformation.
\end{proof}

\subsubsection{Independence of parameters and associativity}
Feigin and Loktev conjectured~\cite{FeLo} that the fusion product does not depend on the parameters and that it is associative. In particular, the iterated fusion product~\eqref{eq::fusion::iterated::same} should be isomorphic to the fusion product of all the modules $M_{i,j}$ at once. Explicitly, one expects the following.

\begin{conjecture}[associativity of the fusion product]
\label{conj::fusion::associativity}
In the notation of~\eqref{eq::fusion::iterated::same}, there is an isomorphism of graded $\fg[z]$-modules
\[
\Big[\,\big[\star_{j}M_{0,j}\big]_{\beta_0}\ \star_z\ \cdots\ \star_z\ \big[\star_{j}M_{l,j}\big]_{\beta_l}\,\Big]_{\gamma_0,\ldots,\gamma_l}
\ \simeq\
\Big[\star_{i=0}^{l}\star_{j=1}^{k_i}M_{i,j}\Big]_{(c_{i,j})}
\]
for an appropriate choice of pairwise distinct parameters $c_{i,j}$; for instance, $c_{i,j}=\gamma_i+\varepsilon\beta_{i,j}$ with generic $\varepsilon\in\bC$.
\end{conjecture}

In this generality, both the independence of the parameters and the associativity are open problems even for irreducible $M_i$'s (see~\cite[\S1.2]{Dumanski}); they are known only for special families of modules (see, e.g.,~\cite{CL,N,Dumanski} and the references therein).
The following simple principle, based on an upper bound for the dimension, proves Conjecture~\ref{conj::fusion::associativity} in all cases that we need.

\begin{proposition}
\label{prp::fusion::associativity}
Let $M_{i,j}$ ($0\leq i\leq l$, $1\leq j\leq k_i$) be finite-dimensional graded cyclic $\fg[z]$-modules.
Suppose that there exists a graded cyclic $\fg[z]$-module $W$ with a cyclic vector $w\in W_0$ such that
\begin{enumerate}
\item $\dim W\leq\prod_{i,j}\dim M_{i,j}$;
\item every fusion product of the modules $M_{i,j}$ under consideration is a quotient of $W$. More precisely, we require this for the fusion products $[\star_{i,j}M_{i,j}]_{c}$ with arbitrary pairwise distinct parameters $c=(c_{i,j})$, and for the iterated fusion products~\eqref{eq::fusion::iterated::same} with arbitrary admissible parameters. Here ``quotient'' means that there is a surjective homomorphism of $\fg[z]$-modules sending $w$ to the cyclic vector.
\end{enumerate}
Then all these fusion products are isomorphic to $W$. In particular, Conjecture~\ref{conj::fusion::associativity} holds for the modules $M_{i,j}$ and all choices of parameters.
\end{proposition}
\begin{proof}
Every fusion product, iterated or not, has dimension $\prod_{i,j}\dim M_{i,j}$, because at each step the dimension of the fusion product equals the product of the dimensions of the factors. By~(ii), each of them is a quotient of $W$, and by~(i), $\dim W$ is not bigger than the dimension of this quotient. Therefore every surjection in~(ii) is an isomorphism.
\end{proof}

\begin{corollary}
\label{cor::fusion::Weyl::upper}
Let $\fg=\gl_V$ with $\dim V$ sufficiently large. Let $M_{i,j}$ be finite-dimensional graded cyclic $\gl_V[z]$-modules whose graded components are polynomial $\gl_V$-modules, generated by vectors $m_{i,j}$ of $\gl_V$-weight $\mu^{(i,j)}$ such that
\[
\fn_+[z]\cdot m_{i,j}=0,\qquad (h z^r)\cdot m_{i,j}=0\quad\text{for all } h\in\fh,\ r>0 .
\]
Then every fusion product of the modules $M_{i,j}$, iterated or not, is a quotient of the local Weyl module $\sW^z_{\mu}$ with $\mu=\sum_{i,j}\mu^{(i,j)}$ (see~\S\ref{sec::Weyl}). Consequently, if $\dim\sW^z_{\mu}=\prod_{i,j}\dim M_{i,j}$, then all these fusion products are isomorphic to $\sW^z_\mu$ and Conjecture~\ref{conj::fusion::associativity} holds for them.
\end{corollary}
\begin{proof}
Consider a fusion product $[M_0\star_z\cdots\star_z M_l]_{\gamma}$ of modules $M_i$ satisfying the assumptions, with cyclic vectors $m_i$ of weights $\mu^{(i)}$, and put $v=m_0\otimes\cdots\otimes m_l$ and $\mu=\sum_i\mu^{(i)}$. By the presentation of local Weyl modules recalled in~\S\ref{sec::Weyl}, it suffices to check that $\widehat{v}$ satisfies the defining relations of $\sW^z_\mu$. These relations are $\fn_+[z]\cdot\widehat{v}=0$, $(hz^r)\cdot\widehat{v}=\delta_{r,0}\,\mu(h)\,\widehat{v}$, and the integrability relations for the lowering operators.
Every element of $\fn_+[z]$ acts on the tensor product $\bigotimes_i M_i(\gamma_i)$ as a sum of elements of $\fn_+[z]$ acting on the individual factors, and hence kills $v$.
For $h\in\fh$ and $r>0$ the element $hz^r$ acts on the $i$-th factor by $h(z+\gamma_i)^r$, which maps $m_i$ to $\gamma_i^r\mu^{(i)}(h)\,m_i$. Hence $(hz^r)\cdot v\in\calF_0\subset\calF_{r-1}$, and therefore $(hz^r)\cdot\widehat{v}=0$ in the fusion product. Clearly, $h\cdot\widehat{v}=\mu(h)\,\widehat{v}$.
The integrability relations $E_{a+1,a}^{\,\mu_a-\mu_{a+1}+1}\cdot\widehat{v}=0$ hold because $\widehat{v}$ is a highest weight vector of weight $\mu$ in a finite-dimensional $\gl_V$-module.
Moreover, the cyclic vector of a fusion product again satisfies the assumptions of the corollary, so the argument applies to iterated fusion products as well. The last claim follows from Proposition~\ref{prp::fusion::associativity}.
\end{proof}

For instance, by Theorem~\ref{thm::Weyl::fusion}, Corollary~\ref{cor::fusion::Weyl::upper} applies to arbitrary fusion products (iterated or not) of fundamental representations. For example, $(V\star_zV)\star_zV\simeq V\star_zV\star_zV\simeq\sW^z_{3\varepsilon_1}$.

%\begin{remark} \label{rem::fusion::geometric}
%A geometric criterion for the independence of the parameters and for associativity is given in~\cite[Theorem~2.1]{Dumanski} (see also~\cite[Proposition~A.1]{DFF}). Suppose that the dual of the fusion product appears as the zero fiber of the global sections of a sheaf that is flat over the space of parameters. Suppose also that the generic fibers are the tensor products $M_1(c_1)\otimes\cdots\otimes M_k(c_k)$ and that higher cohomology vanishes. Then the fusion product does not depend on the parameters and is associative.
%\end{remark}

\subsection{The fusion filtration and a vanishing lemma in the abelian case}

In this subsection we establish a vanishing property for the action of abelian current algebras on fusion products.

Let $\fa$ be an abelian Lie algebra. Then $U(\fa)$ and $U(\fa[z])$ are the symmetric algebras of $\fa$ and $\fa[z]$; in particular, they are commutative.
Consider a nondecreasing sequence of subspaces (automatically, Lie subalgebras)
\begin{equation}
\label{eq::abel::subalg}
\fa_{\bullet}:=\ \fa_0\subseteq\fa_1\subseteq\fa_2\subseteq\cdots\subseteq\fa .
\end{equation}
It defines the Lie subalgebra
\[
\fa_{\bullet}[z]:=\Big\{\sum_{r\geq0}x_rz^r\ \colon\ x_r\in\fa_r\Big\}\subset\fa[z]
\]
(all sums are finite). In other words, the coefficient of $z^r$ is allowed to take values in $\fa_r$.

\begin{proposition}
\label{prp::fusion::abelian}
Let $\fa$ be an abelian Lie algebra and let $\fa_{\bullet}$ be a nondecreasing sequence of subspaces as in~\eqref{eq::abel::subalg}.
Let $M_0,\ldots,M_l$ be cyclic $\fa$-modules with cyclic vectors $v_i\in M_i$, and let $\eta_0,\ldots,\eta_l\geq0$ be integers such that, for all $i=0,\ldots,l$, any product of $\eta_i+1$ elements of $\fa_i$ annihilates $v_i$:
\[
x_0x_1\cdots x_{\eta_i}\cdot v_i=0\qquad\text{for all } x_0,\ldots,x_{\eta_i}\in\fa_i.
\]
Set $K:=\eta_0+\ldots+\eta_l$, and let $\widehat{v}=\widehat{v_0\otimes\cdots\otimes v_l}$ be the cyclic vector of the fusion product $[M_0\star_z\cdots\star_z M_l]_{\gamma_0,\ldots,\gamma_l}$ over $\fa$, with arbitrary pairwise distinct parameters $\gamma_i$. Then:
\begin{enumerate}
\item any product of $K+1$ elements of $\fa_{\bullet}[z]$ annihilates $\widehat{v}$:
\[
b_0b_1\cdots b_K\cdot\widehat{v}=0\qquad\text{for all } b_0,\ldots,b_K\in\fa_{\bullet}[z];
\]
\item more precisely, let $r_1,\ldots,r_s\geq0$ and let $x_t\in\fa_{r_t}$ for $t=1,\ldots,s$. If the element $(x_1z^{r_1})\cdots(x_sz^{r_s})\cdot\widehat{v}$ is nonzero, then $r_t\leq l$ for all $t$, and
\begin{equation}
\label{eq::vanish::action}
\#\{t\colon r_t\geq j\}\ \leq\ \eta_j+\eta_{j+1}+\ldots+\eta_l\qquad\text{for all } j=0,\ldots,l.
\end{equation}
\end{enumerate}
\end{proposition}

Informally, by the Vandermonde Lemma~\ref{lem::Vandermonde}(ii), in the fusion product the element $xz^r$ acts only on the tensor factors with indices $i\geq r$, and $x\in\fa_r\subseteq\fa_i$ for such $i$. Since at most $\eta_i$ elements of $\fa_i$ can act nontrivially on the $i$-th factor, at most $\eta_j+\ldots+\eta_l$ elements $xz^r$ with $r\geq j$ can act nontrivially on $\widehat{v}$. The proof makes this argument precise.

\begin{proof}
We first prove item~(ii). Let $R:=r_1+\ldots+r_s$ and $v:=v_0\otimes\cdots\otimes v_l$.
For $r\geq0$ consider the monic polynomial of degree $r$
\[
p_r(z):=z^{\max(r-l-1,\,0)}\prod_{i=0}^{\min(r,l+1)-1}(z-\gamma_i)
\]
(so $p_0=1$). It vanishes at $\gamma_i$ for all $i<\min(r,l+1)$.
The difference $x_tz^{r_t}-x_tp_{r_t}(z)$ is a linear combination of the elements $x_tz^{r'}$ with $r'<r_t$.
Since $U(\fa[z])$ is commutative, the difference
\[
(x_1z^{r_1})\cdots(x_sz^{r_s})-\big(x_1p_{r_1}(z)\big)\cdots\big(x_sp_{r_s}(z)\big)
\]
is a linear combination of monomials of $z$-degree less than $R$, and such monomials map $v$ into $\calF_{R-1}$. Hence $(x_1z^{r_1})\cdots(x_sz^{r_s})\cdot\widehat{v}$ is the class of $\prod_t\big(x_tp_{r_t}(z)\big)\cdot v$ in $\calF_R/\calF_{R-1}$.
By~\eqref{eq::ev::action}, the element $x_tp_{r_t}(z)$ acts on the tensor product of evaluation modules by $\sum_{i\geq r_t}p_{r_t}(\gamma_i)\,x_t^{(i)}$, where $i$ runs over $\{0,\ldots,l\}$; this sum is empty if $r_t>l$.
Expanding the product, we get
\[
\prod_{t=1}^{s}\big(x_tp_{r_t}(z)\big)\cdot v=\sum_{\varphi}\ \Big(\prod_{t=1}^{s}p_{r_t}(\gamma_{\varphi(t)})\Big)\ \bigotimes_{i=0}^{l}\Big(\prod_{t\in\varphi^{-1}(i)}x_t\Big)\cdot v_i,
\]
where the sum runs over the maps $\varphi\colon\{1,\ldots,s\}\to\{0,\ldots,l\}$ such that $\varphi(t)\geq r_t$ for all $t$.
In the term corresponding to $\varphi$, the $i$-th factor is acted upon by the elements $x_t$ with $\varphi(t)=i$. All of them belong to $\fa_{r_t}\subseteq\fa_i$, because $r_t\leq i$ and the sequence $\fa_\bullet$ is nondecreasing. Hence, by the assumption, this term vanishes as soon as $|\varphi^{-1}(i)|>\eta_i$ for some $i$.
Thus, if the product is nonzero, then there exists a map $\varphi$ with $r_t\leq\varphi(t)\leq l$ for all $t$ and $|\varphi^{-1}(i)|\leq\eta_i$ for all $i$. For such a map and every $j$ we have
\[
\#\{t\colon r_t\geq j\}\leq\#\{t\colon\varphi(t)\geq j\}=\sum_{i\geq j}|\varphi^{-1}(i)|\leq\sum_{i\geq j}\eta_i ,
\]
which proves~\eqref{eq::vanish::action}; the inequality $r_t\leq\varphi(t)\leq l$ gives the first condition.

Item~(i) follows from item~(ii). Indeed, each $b_i\in\fa_{\bullet}[z]$ is a sum of monomials $xz^r$ with $x\in\fa_r$. Expanding the product $b_0\cdots b_K$, we obtain a sum of products of $s=K+1$ such monomials. For each of them the condition~\eqref{eq::vanish::action} with $j=0$ would give $K+1\leq\eta_0+\ldots+\eta_l=K$, which is impossible.
\end{proof}

\begin{remark}
The most frequent application of Proposition~\ref{prp::fusion::abelian} is the case of the constant sequence $\fa_i=\fa$ for all $i$. In this case the condition~\eqref{eq::vanish::action} says that, for every $j$, at most $\eta_j+\ldots+\eta_l$ of the elements $x_tz^{r_t}$ may have $r_t\geq j$. For example, if all $M_i$ are equal to a module $M$ such that $\fa^2\cdot v_i=0$ (so that $\eta_i=1$), then a monomial $(x_1z^{r_1})\cdots(x_sz^{r_s})\cdot\widehat{v}$ vanishes unless $s\leq l+1$ and $\#\{t\colon r_t\geq j\}\leq l+1-j$ for all $j$.
\end{remark}

\subsection{Fusion product preserves surjections}
\begin{proposition}
\label{prp::fusion::surjection}
The fusion product preserves surjections.
Namely, let $M$, $M'$, $M_1$, \ldots, $M_k$ be finite-dimensional graded cyclic $\fg[z]$-modules with cyclic vectors $m$, $m'$, $m_1$, \ldots, $m_k$ respectively. Let $\varphi\colon M\twoheadrightarrow M'$ be a surjective homomorphism of $\fg[z]$-modules such that $\varphi(m)=m'$.
Then, for all pairwise distinct parameters $\gamma_0,\gamma_1,\ldots,\gamma_k\in\bC$, the map $\varphi$ induces a surjective homomorphism of graded $\fg[z]$-modules
\[
[M\star_z M_1\star_z\cdots\star_z M_k]_{\gamma_0,\ldots,\gamma_k}\ \twoheadrightarrow\ [M'\star_z M_1\star_z\cdots\star_z M_k]_{\gamma_0,\ldots,\gamma_k},
\]
which maps the cyclic vector $\widehat{v}$ to the cyclic vector $\widehat{v'}$, where $v:=m\otimes m_1\otimes\cdots\otimes m_k$ and $v':=m'\otimes m_1\otimes\cdots\otimes m_k$.
\end{proposition}

\begin{proof}
Set
\[
T:=M(\gamma_0)\otimes M_1(\gamma_1)\otimes\cdots\otimes M_k(\gamma_k),
\qquad
T':=M'(\gamma_0)\otimes M_1(\gamma_1)\otimes\cdots\otimes M_k(\gamma_k),
\]
and recall that $v\in T$ and $v'\in T'$ are cyclic vectors by Lemma~\ref{lem::fusion::graded::cyclic}.

\emph{Step 1: a surjection of the tensor products.}
Recall that $M(\gamma_0)$ and $M'(\gamma_0)$ are the pullbacks of $M$ and $M'$ along the same automorphism $\tau_{\gamma_0}$ of $\fg[z]$. Since $\varphi$ commutes with the action of $\fg[z]$, it also commutes with the action of $\tau_{\gamma_0}(\fg[z])=\fg[z]$. Hence $\varphi$ is a homomorphism of $\fg[z]$-modules $M(\gamma_0)\to M'(\gamma_0)$.
Consequently, the map
\[
\Phi:=\varphi\otimes\mathrm{id}_{M_1}\otimes\cdots\otimes\mathrm{id}_{M_k}\colon T\longrightarrow T'
\]
is a homomorphism of $\fg[z]$-modules, because $\fg[z]$ acts on the tensor products factorwise. The map $\Phi$ is surjective because $\varphi$ is surjective, and $\Phi(v)=v'$ because $\varphi(m)=m'$.

\emph{Step 2: the fusion filtrations.}
Let $\calF_n:=U(\fg[z])_{\leq n}\cdot v\subset T$ and $\calF'_n:=U(\fg[z])_{\leq n}\cdot v'\subset T'$ be the fusion filtrations~\eqref{eq::fusion::filtration}. Since $\Phi$ is a homomorphism of $U(\fg[z])$-modules, we get
\[
\Phi(\calF_n)=\Phi\big(U(\fg[z])_{\leq n}\cdot v\big)=U(\fg[z])_{\leq n}\cdot\Phi(v)=U(\fg[z])_{\leq n}\cdot v'=\calF'_n
\qquad\text{for all } n\geq-1 .
\]
In other words, the image of the fusion filtration of $T$ is exactly the fusion filtration of $T'$.

\emph{Step 3: the associated graded map.}
Since $\Phi(\calF_{n-1})\subset\calF'_{n-1}$, the map
\[
\mathrm{gr}_n\Phi\colon\calF_n/\calF_{n-1}\longrightarrow\calF'_n/\calF'_{n-1},\qquad [w]\mapsto[\Phi(w)],
\]
is well defined for every $n\geq0$. It is surjective, because every element of $\calF'_n/\calF'_{n-1}$ is the class of some element of $\calF'_n=\Phi(\calF_n)$.
Let $\mathrm{gr}\,\Phi:=\bigoplus_{n\geq0}\mathrm{gr}_n\Phi$. This is a surjective linear map of graded spaces from $[M\star_z M_1\star_z\cdots\star_z M_k]_{\gamma}$ to $[M'\star_z M_1\star_z\cdots\star_z M_k]_{\gamma}$ preserving the degree.
It is a homomorphism of $\fg[z]$-modules. Indeed, let $x\in\fg$, $s\geq0$ and $w\in\calF_n$. The element $xz^s$ acts on the class $[w]\in\calF_n/\calF_{n-1}$ by $[w]\mapsto[(xz^s)\cdot w]\in\calF_{n+s}/\calF_{n+s-1}$, and similarly on the target. Hence
\[
\mathrm{gr}\,\Phi\big((xz^s)\cdot[w]\big)=\big[\Phi\big((xz^s)\cdot w\big)\big]=\big[(xz^s)\cdot\Phi(w)\big]=(xz^s)\cdot\mathrm{gr}\,\Phi\big([w]\big).
\]
Finally, $\mathrm{gr}\,\Phi(\widehat{v})=[\Phi(v)]=\widehat{v'}$.
\end{proof}

\begin{remark}
\begin{enumerate}
\item The same proof applies to any position of $M$ among the factors of the fusion product, and to several surjections $M_i\twoheadrightarrow M_i'$ at the same time (apply the proposition to each factor in turn). It also applies to the fusion product of $\fg$-modules of Definition~\ref{def::fusion}, i.e., to the case where all modules are concentrated in degree $0$.
\item The map $\varphi$ is not required to preserve the gradings of $M$ and $M'$, because these gradings are not used in the construction of the fusion product. If $\fg$ carries an additional grading as in Remark~\ref{rem::fusion::graded} and $\varphi$ preserves it, then $\mathrm{gr}\,\Phi$ preserves the resulting bigrading.
\item The surjection $\mathrm{gr}\,\Phi$ is an isomorphism if and only if $\varphi$ is an isomorphism. Indeed, the two fusion products have dimensions $\dim M\cdot\prod_i\dim M_i$ and $\dim M'\cdot\prod_i\dim M_i$ respectively.
\end{enumerate}
\end{remark}

\begin{corollary}
\label{cor::fusion::surjection::iterated}
Iterated fusion products~\eqref{eq::fusion::iterated::same} also preserve surjections. Namely, if $\varphi\colon M_{i,j}\twoheadrightarrow M'_{i,j}$ maps the cyclic vector to the cyclic vector, then replacing $M_{i,j}$ by $M'_{i,j}$ in~\eqref{eq::fusion::iterated::same} (with the same parameters) gives a quotient of the iterated fusion product, and the quotient map sends the cyclic vector to the cyclic vector.
\end{corollary}
\begin{proof}
By Proposition~\ref{prp::fusion::surjection}, the inner fusion product containing $M_{i,j}$ surjects onto the one containing $M'_{i,j}$, and this surjection maps the cyclic vector to the cyclic vector. Applying Proposition~\ref{prp::fusion::surjection} once more to the outer fusion product, we obtain the claim.
\end{proof}

\subsubsection{Fusion products and common quotients}

Let $N$ be a cyclic module over a Lie algebra $\mathfrak{L}$ (in our applications $\mathfrak{L}=\fg[z]$) with a cyclic vector $n$, and let
\begin{equation}
\label{eq::star::diagram}
\pi_i\colon N\twoheadrightarrow N_i,\qquad i=1,\ldots,k,
\end{equation}
be surjective homomorphisms. Put $n_i:=\pi_i(n)$; these are cyclic vectors of $N_i$. Denote $K_i:=\ker\pi_i$.

\begin{lemma}
\label{lem::cyclic::colimit}
The colimit of the diagram~\eqref{eq::star::diagram} in the category of $\mathfrak{L}$-modules is the common quotient
\[
N':=N\big/(K_1+\ldots+K_k),
\]
together with the natural projections $\pi'_i\colon N_i=N/K_i\twoheadrightarrow N'$. In particular, $N'$ is cyclic with the cyclic vector $n'$, the image of $n$, and $\pi'_i(n_i)=n'$ for all $i$.
\end{lemma}
\begin{proof}
A cocone over~\eqref{eq::star::diagram} consists of a module $X$ and homomorphisms $f\colon N\to X$ and $f_i\colon N_i\to X$ with $f_i\circ\pi_i=f$ for all $i$. Since the maps $\pi_i$ are surjective, the maps $f_i$ are determined by $f$, and such $f_i$ exist if and only if $f$ vanishes on every $K_i$, i.e., on $K_1+\ldots+K_k$. Therefore cocones are in bijection with homomorphisms $N'\to X$.
\end{proof}

In the rest of this subsection the modules $N,N_1,\ldots,N_k$ are finite-dimensional graded cyclic $\fg[z]$-modules, and the maps $\pi_i$ are surjective homomorphisms of $\fg[z]$-modules. Then $N'$ is a finite-dimensional cyclic $\fg[z]$-module; we assume that it is graded, which is the case if all $\pi_i$ are graded.
Fix further finite-dimensional graded cyclic $\fg[z]$-modules $M_1,\ldots,M_l$ with cyclic vectors $m_1,\ldots,m_l$ and pairwise distinct parameters $\gamma_0,\gamma_1,\ldots,\gamma_l$. Consider the fusion products
\[
\mathsf{P}:=[N\star_z M_1\star_z\cdots\star_z M_l]_{\gamma},\qquad
\mathsf{P}_i:=[N_i\star_z M_1\star_z\cdots\star_z M_l]_{\gamma},\qquad
\mathsf{P}':=[N'\star_z M_1\star_z\cdots\star_z M_l]_{\gamma}.
\]
By Proposition~\ref{prp::fusion::surjection}, the maps $\pi_i$ induce surjections $\mathrm{gr}\,\Pi_i\colon\mathsf{P}\twoheadrightarrow\mathsf{P}_i$, compatible with the cyclic vectors. Hence we get the diagram of surjections $\mathsf{P}\twoheadrightarrow\mathsf{P}_i$, $i=1,\ldots,k$, which is of the same form as~\eqref{eq::star::diagram}. By Lemma~\ref{lem::cyclic::colimit}, its colimit is the common quotient
\[
\mathrm{colim}_i\,\mathsf{P}_i=\mathsf{P}\Big/\sum_{i=1}^{k}\ker\big(\mathrm{gr}\,\Pi_i\big).
\]

To describe these kernels, let $T$ be a finite-dimensional vector space with an increasing filtration $\calF_{-1}T=0\subset\calF_0T\subset\calF_1T\subset\ldots$, exhausting $T$. For a subspace $A\subset T$ we define the associated graded subspace
\[
\mathrm{gr}\,A:=\bigoplus_{n\geq0}\big((A\cap\calF_nT)+\calF_{n-1}T\big)\big/\calF_{n-1}T\ \subset\ \mathrm{gr}\,T .
\]
Clearly, $\dim\mathrm{gr}\,A=\dim A$, and $A\subset B$ implies $\mathrm{gr}\,A\subset\mathrm{gr}\,B$. In particular, $\mathrm{gr}(A_1+\ldots+A_k)\supset\mathrm{gr}\,A_1+\ldots+\mathrm{gr}\,A_k$, but this inclusion can be strict (see Example~\ref{ex::fusion::colimit}).

\begin{lemma}
\label{lem::gr::kernel}
Let $\Phi\colon T\twoheadrightarrow S$ be a surjective linear map of filtered spaces such that $\Phi(\calF_nT)=\calF_nS$ for all $n$. Then $\ker\big(\mathrm{gr}\,\Phi\colon\mathrm{gr}\,T\to\mathrm{gr}\,S\big)=\mathrm{gr}(\ker\Phi)$.
\end{lemma}
\begin{proof}
The inclusion $\supset$ is clear. Conversely, let $w\in\calF_nT$ be such that $\Phi(w)\in\calF_{n-1}S$. Since $\calF_{n-1}S=\Phi(\calF_{n-1}T)$, there is $w'\in\calF_{n-1}T$ with $\Phi(w')=\Phi(w)$. Then $w-w'\in\ker\Phi\cap\calF_nT$, so the class of $w$ in $\calF_nT/\calF_{n-1}T$ belongs to $\mathrm{gr}(\ker\Phi)$.
\end{proof}

\begin{proposition}
\label{prp::fusion::colimit}
In the above setting, put $T:=N(\gamma_0)\otimes M_1(\gamma_1)\otimes\cdots\otimes M_l(\gamma_l)$ with its fusion filtration, and $A_i:=K_i\otimes M_1\otimes\cdots\otimes M_l\subset T$.
\begin{enumerate}
\item There is a canonical surjection of graded $\fg[z]$-modules
\[
\mathrm{colim}_i\,\mathsf{P}_i\ \twoheadrightarrow\ \mathsf{P}'=\Big[\big(\mathrm{colim}_i\,N_i\big)\star_z M_1\star_z\cdots\star_z M_l\Big]_{\gamma},
\]
compatible with the maps from $\mathsf{P}$ and with the cyclic vectors. In other words, the fusion product with the colimit is a quotient of the colimit of the fusion products.
\item Explicitly, $\ker(\mathsf{P}\twoheadrightarrow\mathsf{P}_i)=\mathrm{gr}\,A_i$ and $\ker(\mathsf{P}\twoheadrightarrow\mathsf{P}')=\mathrm{gr}(A_1+\ldots+A_k)$, so that
\[
\mathrm{colim}_i\,\mathsf{P}_i=\mathsf{P}\Big/\sum_i\mathrm{gr}\,A_i\ \twoheadrightarrow\ \mathsf{P}\big/\mathrm{gr}(A_1+\ldots+A_k)=\mathsf{P}'.
\]
\item The following conditions are equivalent:
\begin{enumerate}
\item the surjection of item~(i) is an isomorphism, i.e., the fusion product commutes with the colimit of the diagram~\eqref{eq::star::diagram};
\item $\mathrm{gr}(A_1+\ldots+A_k)=\mathrm{gr}\,A_1+\ldots+\mathrm{gr}\,A_k$;
\item $\dim\mathrm{colim}_i\,\mathsf{P}_i=\dim N'\cdot\prod_{j=1}^{l}\dim M_j$.
\end{enumerate}
\end{enumerate}
\end{proposition}
\begin{proof}
Put $M:=M_1(\gamma_1)\otimes\cdots\otimes M_l(\gamma_l)$, and let $T_i:=N_i(\gamma_0)\otimes M$ and $T':=N'(\gamma_0)\otimes M$. Consider the maps
\[
\Pi_i:=\pi_i\otimes\mathrm{id}_M\colon T\twoheadrightarrow T_i,\qquad
\Pi:=\pi\otimes\mathrm{id}_M\colon T\twoheadrightarrow T',
\]
where $\pi\colon N\twoheadrightarrow N'$ is the projection. By the proof of Proposition~\ref{prp::fusion::surjection}, these are surjective homomorphisms of $\fg[z]$-modules, which map the fusion filtration of $T$ onto the fusion filtrations of $T_i$ and $T'$ respectively. The associated graded maps are the surjections $\mathsf{P}\twoheadrightarrow\mathsf{P}_i$ and $\mathsf{P}\twoheadrightarrow\mathsf{P}'$.

The tensor product over $\bC$ is exact. Hence
\[
\ker\Pi_i=K_i\otimes M=A_i,\qquad
\ker\Pi=(K_1+\ldots+K_k)\otimes M=A_1+\ldots+A_k .
\]
By Lemma~\ref{lem::gr::kernel}, $\ker(\mathsf{P}\twoheadrightarrow\mathsf{P}_i)=\mathrm{gr}\,A_i$ and $\ker(\mathsf{P}\twoheadrightarrow\mathsf{P}')=\mathrm{gr}(A_1+\ldots+A_k)$. These kernels are $\fg[z]$-submodules of $\mathsf{P}$. Since $A_i\subset A_1+\ldots+A_k$, we have $\sum_i\mathrm{gr}\,A_i\subset\mathrm{gr}(A_1+\ldots+A_k)$. Together with Lemma~\ref{lem::cyclic::colimit}, this proves items~(i) and~(ii); the compatibility with the maps from $\mathsf{P}$ holds by construction.

The surjection $\mathsf{P}/\sum_i\mathrm{gr}\,A_i\twoheadrightarrow\mathsf{P}/\mathrm{gr}(\sum_iA_i)$ is an isomorphism if and only if the two subspaces coincide. Since one of them contains the other, this happens if and only if their dimensions coincide. Finally, $\dim\mathsf{P}'=\dim T'=\dim N'\cdot\prod_j\dim M_j$. This proves item~(iii).
\end{proof}

In general the surjection of Proposition~\ref{prp::fusion::colimit}(i) is not an isomorphism, even for abelian Lie algebras and evaluation modules.

\begin{example}
\label{ex::fusion::colimit}
Let $\fa=\bC x\oplus\bC y$ be the two-dimensional abelian Lie algebra, and let $N=\bC[x,y]/(x,y)^2$ with the basis $1,x,y$ and the cyclic vector $1$. Consider the submodules $K_1=\bC x$ and $K_2=\bC y$, so that $N_1=N/K_1$, $N_2=N/K_2$ and $N'=N/(K_1+K_2)\simeq\bC$. Take $l=1$, $M_1=N$ with the cyclic vector $1$, and the parameters $\gamma_0=0$, $\gamma_1=1$.
In $T=\ev_0(N)\otimes\ev_1(N)$, the element $a z^s$ ($a\in\fa$) acts by $a\otimes1+1\otimes a$ for $s=0$ and by $1\otimes a$ for $s>0$. Let $v=1\otimes1$. A direct computation gives
\[
\calF_0T=U(\fa)\cdot v=\Span\langle 1\otimes1,\ x\otimes1+1\otimes x,\ y\otimes1+1\otimes y,\ x\otimes x,\ y\otimes y,\ x\otimes y+y\otimes x\rangle ,
\]
and $\calF_1T=T$, since $(xz)\cdot v=1\otimes x$, $(yz)\cdot v=1\otimes y$ and $(xz)(y\otimes1+1\otimes y)=y\otimes x$.
We have $A_1=x\otimes N$ and $A_2=y\otimes N$, and one checks that
\[
A_1\cap\calF_0T=\bC\,x\otimes x,\qquad A_2\cap\calF_0T=\bC\,y\otimes y,\qquad
(A_1+A_2)\cap\calF_0T=\Span\langle x\otimes x,\ y\otimes y,\ x\otimes y+y\otimes x\rangle .
\]
Hence in degree $0$ the subspace $\mathrm{gr}(A_1+A_2)$ is three-dimensional, while $\mathrm{gr}\,A_1+\mathrm{gr}\,A_2$ is two-dimensional. In degree $1$, both subspaces are equal to $T/\calF_0T$, which is three-dimensional.
Therefore $\dim\mathrm{colim}_i\,\mathsf{P}_i=9-5=4$, whereas $\dim\,(N'\star_z N)=\dim N'\cdot\dim N=3$.
\end{example}

\begin{remark}
In the situation of Proposition~\ref{prp::fusion::colimit}, the fusion product with the colimit is always a quotient of the colimit of the fusion products. Hence, to prove that they coincide, it suffices to show that $\dim\mathrm{colim}_i\,\mathsf{P}_i\leq\dim N'\cdot\prod_j\dim M_j$. For instance, this holds if one knows a graded $\fg[z]$-module of dimension $\dim N'\cdot\prod_j\dim M_j$ that surjects onto $\mathrm{colim}_i\,\mathsf{P}_i$ compatibly with the cyclic vectors. 
%Via the Schur--Weyl duality, this is the type of statement needed for the common quotients $\mGH^{\gl}_{\nu_1,\ldots,\nu_k}$ of~\S\ref{sec::GH::SW}.
\end{remark}

\section{Local Weyl modules $\sW_{\lambda}^{z}$}
\label{sec::Weyl}

In this section, we recollect some properties of local Weyl modules that will be useful for the proofs in the next section.

\subsection{Definition of local Weyl modules}
The main examples of fusion products are local Weyl modules well known in the representation theory of semisimple Lie algebras $\fg$ (\cite{CP,CL}).  They appear in various contexts in the study of the category of representations of current Lie algebras \cite{CP,BBCKL,Kh}, semi-infinite flag varieties \cite{BF1,BF2,BF3}, double affine Hecke algebras and Macdonald polynomials \cite{KKhM}, arc varieties and vertex algebras \cite{FeMa1,FeMa2}.
In this paper, we mainly focus on the cases $\gl_V$ and $\msl_V$, and we begin by recalling some standard notations.

Throughout, we suppose that the dimension of the tautological $\gl_V$ representation $V$ is sufficiently large. We consider polynomial $\gl_V$-modules such that the weight of each representation is strictly less than $\dim V$.  
The Cartan subalgebra $\fh \subset \gl_V$ consists of diagonal matrices. We denote by $\varepsilon_i \in \fh^*$ the standard weight which evaluates to $1$ on the $i$-th diagonal entry and vanishes elsewhere.  
The fundamental weights are given by $\omega_i := \varepsilon_1 + \cdots + \varepsilon_i$, which coincide with the highest weights of the fundamental representations $\Lambda^i V$.

To each partition $\lambda = (\lambda_0 \geq \ldots \geq \lambda_l) \vdash N < \dim V$, we associate the irreducible $\gl_V$-module $V_\lambda$ with highest weight
\[
\lambda = \lambda_0 \varepsilon_1 + \cdots + \lambda_l \varepsilon_{l+1} = \lambda_l \omega_{l+1} + (\lambda_{l-1} - \lambda_l) \omega_l + \cdots + (\lambda_0 - \lambda_1) \omega_1.
\]
We use the  \emph{standard partial order "$\leq$"} on the set of dominant weights (indexed by partitions in the case of $\gl_V$):
\begin{equation}
\label{eq::dominance::order}
\lambda \geq \mu \ \stackrel{\mathsf{def}}{\Leftrightarrow} \ |\lambda| = |\mu| \ \text{ and } \
\begin{cases}
\lambda_0 \geq \mu_0, \\
\lambda_0 + \lambda_1 \geq \mu_0 + \mu_1, \\
\ldots
\end{cases}
\end{equation}
Note that any $\gl_V[z]$-integrable module $M$ which is a polynomial $\gl_V$-representation of degree $N$ admits a direct sum decomposition  of $\gl_V$-modules called \emph{$\gl_V$-weight decomposition}:
$$M=\bigoplus_{\mu \vdash N} V_{\mu}^{m_\mu}.$$
Let us introduce a family of modules that play a central role in this paper:
\begin{definition}
    The \emph{local Weyl module} $\sW_{\lambda}^{z}$ is the maximal integrable $\gl_V[z]$-module whose all weights are less or equal then $\lambda$ and its $\lambda$ weight space is spanned by a cyclic vector $v_{\lambda}$.
In particular, it has the following $\gl_V$-weight decomposition:
$$
\sW_{\lambda}^{z} \simeq_{\gl_V} V_{\lambda}\oplus \bigoplus_{\mu<\lambda} V_{\mu} ^{m_\mu}.
$$
\end{definition}

Local Weyl modules admit the following presentation by generators and relations as explained in~\cite{CP} (for $\dim V>N$):
$$
\sW_{\lambda}^{z} \simeq U(\gl_V[z]) v_{\lambda}\left/
\left(
\begin{array}{c}
\forall j>i, \forall r \ \ E_{ij} z^r\cdot v_\lambda = 0, \\
E_{ii} v_\lambda = \lambda_i v_{\lambda}, \ \forall r>0\, E_{ii} z^r v_{\lambda} = 0, \\
E_{i+1,i}^{\lambda_i+1} v_{\lambda} = 0.
\end{array}
\right)
\right.
$$

Local Weyl modules $\sW_{\lambda}^{z}$ possess a number of remarkable properties (see for example \cite{I,Kh}).  
One of the key properties we will use is the following theorem, proved for $\msl_m$ in \cite{CL} and later extended to other semisimple Lie algebras \cite{FL2,N,FeMa3}:
\begin{theorem}
\label{thm::Weyl::fusion}
    The fusion product of fundamental $\gl_V$-modules
    is independent of the fusion parameters $\gamma_i$'s and is isomorphic to the local Weyl module:
\begin{equation}
\label{eq::Weyl=fusion}
    [V_{\omega_{i_0}} \star_z \cdots \star_z V_{\omega_{i_l}}]_{\gamma_0,\ldots,\gamma_l}
    \simeq \sW_{\sum_{j=0}^{l} \omega_{i_j}}^{z} \text{ for all } \gamma_0,\ldots,\gamma_l.
\end{equation}    
Moreover, this fusion product is associative:
$$
    [\sW_{\lambda^0} \star_z \cdots \star_z \sW_{\lambda^l}]_{\gamma_0,\ldots,\gamma_l}
    \simeq \sW_{\sum_{j=0}^{l} \lambda^j}^{z}
$$
\end{theorem}

\subsection{Properties of Local Weyl Modules}

Local Weyl modules possess many elegant properties and admit various equivalent descriptions, which we briefly recall below. In particular, their characters are given by the $t = 0$ specialization of Macdonald polynomials (see e.g.\cite{S,I, N,FeMa3}):
\begin{equation}
\label{eq::Weyl::Character}
\ch_q(\sW_\lambda^{z}) = P_\lambda(\overline{x}, q, 0).
\end{equation}
Here, the variable $q$ tracks the current algebra grading (i.e., the $z$-degree), while $\overline{x}$ corresponds to the $\gl_V$-character ($x_i:=e^{\varepsilon_i}$). These polynomials are also known as $q$-Whittaker functions or $q$-Hermite polynomials (see e. g. \cite{CO}).

Below we recall one more characterization of local Weyl modules, due to \cite{CL} in type $A$, and to \cite{I} for simply-laced types and dual nontwisted non-simply-laced types. This description is grounded in the representation theory of Demazure modules for affine Kac–Moody Lie algebras.
First, we recall some basic definitions and notations (we refer for the details to the textbooks~\cite{Kac,Carter}). Let $\fg$ be a simply-laced simple Lie algebra, and let $\widehat{\fg}$ denote the corresponding untwisted affine Lie algebra, which is a central extension of the loop algebra $\fg[[z, z^{-1}]]$. In particular, there is a natural embedding $\fg[z] \hookrightarrow \widehat{\fg}$, and the image is a parabolic subalgebra corresponding to the inclusion of the finite Dynkin diagram into its affine extension.

Let $\delta$ denote the basic imaginary root, and let $d$ be the degree operator (or scaling element), characterized by:
\[
\langle d, \alpha \rangle = 0 \quad \text{for all roots } \alpha \in \fg, \qquad \langle d, \delta \rangle = 1.
\]

Let $\nu$ be a minuscule weight of $\fg$ (in type $A$, this means $\nu = \omega_k$ is a fundamental weight). Then the affine weight $\nu + d$ is dominant of level one. Consider the corresponding irreducible highest weight module $L_{\nu + d}$ over $\widehat{\fg}$.

The following properties hold for simply-laced Lie algebras and follow from the representation theory of affine Lie algebras (see \cite{Kac}) and the isomorphism between Weyl modules and affine Demazure modules (see \cite{I}):
\begin{enumerate}[label=(D\arabic*)]
    \item For any integral weight $\tau$ of $\fg$, there exist a unique minuscule weight $\nu_\tau$, an element $\sigma$ in the affine Weyl group, and an integer $|\tau|$ such that:
    \[
    \sigma(\nu_\tau + d) = \tau + d - |\tau|\delta.
    \]
    
    \item The weight space of $L_{\nu + d}$ corresponding to $\tau + d - |\tau|\delta$ is one-dimensional. We denote its generator by $v_{\tau + d - |\tau|\delta}$.
    
    \item Two weights $\tau$ and $\tau'$ satisfy $\nu_\tau = \nu_{\tau'}$ if and only if $\tau - \tau'$ lies in the root lattice of $\fg$. In particular, for $\tau = \sum_{i=1}^{\dim V} \tau_i \varepsilon_i$ (a weight of $\msl_V$), we have $\nu_\tau = \omega_{\sum_{i=1}^{\dim V} \tau_i}$.
    
    \item \label{item::Weyl=Demazure}
    If $\tau$ is dominant, then the cyclic $\fg[z]$-module generated by $v_{\tau + d - |\tau|\delta}$ is isomorphic to the local Weyl module:
    \[
    \U(\fg[z]) \cdot v_{\tau + d - |\tau|\delta} \simeq \sW_\tau(z).
    \]
    
    \item If $\tau < \tau'$ are dominant weights in the dominance order, then:
    \[
    v_{\tau + d - |\tau|\delta} \in \U(\fg[z]) \cdot v_{\tau' + d - |\tau'|\delta}.
    \]
    
    \item \label{item::WeylInclusion}
    Let $\gamma + p\delta$ be an affine root such that:
    \[
    \langle \tau + d, \gamma + p\delta \rangle = -1.
    \]
    Then:
    \[
    v_{\tau - \gamma + d - |\tau - \gamma|\delta} = f_{\gamma} z^p \cdot v_{\tau + d - |\tau|\delta}.
    \]
    In particular, the submodule of $\sW_\tau$ generated by $f_\gamma z^{\langle \tau, \gamma \rangle - 1} v_\tau$ is isomorphic to $\sW_{\tau - \gamma}$, provided $\tau - \gamma$ is a dominant weight.
    (Here $f_{\gamma}\in\fg$ is the negative generator of the negative Borel subalgebra of $\fg$ associated with the root $\gamma$).
\end{enumerate}

By definition, local Weyl modules are cyclic. Surprisingly, they are also \emph{cocyclic}:

\begin{proposition}
\label{prp::Weyl::top}
For any partition $\lambda \vdash N$, the highest $z$-degree component of the local Weyl module $\sW_{\lambda}^{z}$ is isomorphic to $\Lambda^N V$, provided that $\dim V > N$.
Moreover, this component forms the \emph{socle} of the module: every nonzero $\gl_V[z]$-submodule of $\sW_{\lambda}^{z}$ contains it.
\end{proposition}

\begin{proof}
The restriction of an integrable highest weight representation of the affine Lie algebra to the current algebra is cocyclic, with the cocyclic generator being the highest weight vector. This is because $\fg[z]$ embeds into $\widehat{\fg}$ as a parabolic subalgebra. Therefore, we have
\[
\cosoc\left(\mathrm{Res}_{\fg[z]}^{\widehat{\fg}} L_{\nu + d}\right) \simeq V_{\nu},
\]
so the socle is irreducible. Consequently, the socle of any $\fg[z]$-submodule of $\mathrm{Res}_{\fg[z]}^{\widehat{\fg}} L_{\nu + d}$ coincides with that of the full restriction.
Using property~\ref{item::Weyl=Demazure}, which identifies the local Weyl module $\sW_\lambda^z$ with a Demazure submodule in $L_{\nu + d}$, we obtain:
\[
\cosoc \sW_\lambda^z \simeq V_{\omega_{|\lambda|}} \simeq \Lambda^N V.
\]
\end{proof}

Let us now provide more detailed information about the structure of the local Weyl module, its socle, and how the socle can be accessed.
In Proposition~\ref{prop:LineWeylMudulesProperties}, we gather results concerning the special case $\lambda = (N)$. Subsequently, in Proposition~\ref{lem:CocyclicWeylProperties}, we address the general case of arbitrary $\lambda$.

\begin{proposition}\label{prop:LineWeylMudulesProperties}
The local Weyl module $\sW_{N\varepsilon_1}^{z}$ (generated by the highest weight vector $v_{N\varepsilon_1}$) possesses the following properties:
\begin{itemize}
    \item $\sW_{N\varepsilon_1}^{z}$ is generated from the cyclic vector $v_{N\varepsilon_1}$ by the action of the abelian subalgebra $\fa[z]$, where
    $\fa:=\Span\langle E_{a 1} \colon a>1\rangle$, and, as an $\fa[z]$-module, it coincides with the fusion product 
    \begin{equation}
    \label{eq::weyl::ab::fusion}
    \underbrace{\left(\Bbbk\oplus\fa\right)\star_z\ldots\star_z \left(\Bbbk\oplus\fa\right)}_{N};
    \end{equation}
    \item for any pairwise different indices $a_1,\ldots,a_K>1$ and any integers $0\leq r_1\leq \ldots\leq r_K$, the monomial element
    $\left(\prod_{i=1}^{K} E_{a_i 1} z^{r_i} \right)v_{N\varepsilon_1}$ is nonzero in $\sW_{N\varepsilon_1}^{z}$ only if $K\leq N$ and $r_i\leq i+N-K-1$ for all $i=1,\ldots,K$;
    \item the top ${z}$-degree component of $\sW_{N\varepsilon_1}^{z}$ has ${z}$-degree equal to $\frac{N(N-1)}{2}$ and is generated by the element $\left(\prod_{i=1}^{N} E_{a_i 1} z^{i-1}\right) v_{N\varepsilon_1}$ (with $1<a_1<\ldots<a_N$).
\end{itemize}
\end{proposition}
\begin{proof}
By Theorem~\ref{thm::Weyl::fusion}, $\sW_{N\varepsilon_1}^{z}\simeq\underbrace{V\star_z\cdots\star_zV}_{N}$.
The tautological module $V=V_{\varepsilon_1}$ is spanned by $v_{\varepsilon_1}$ and the vectors $E_{a1}v_{\varepsilon_1}$, $a>1$. The elements $E_{a1}$, $a>1$, span a commutative subalgebra $\fa\subset\gl_V$, and as an $\fa$-module $V$ is isomorphic to the quotient $S(\fa)/S^{\geq2}(\fa)=\Bbbk\oplus\fa$ of the symmetric algebra $S(\fa)=U(\fa)$ by the ideal of the elements of degree at least $2$. Consequently, as an $\fa[z]$-module the local Weyl module $\sW_{N\varepsilon_1}^{z}$ is isomorphic to the fusion product~\eqref{eq::weyl::ab::fusion}. This proves the first item.

By Proposition~\ref{prp::fusion::abelian} (applied to the constant sequence $\fa_i=\fa$ and $\eta_i=1$), we have for pairwise different $a,a_1,\ldots,a_K>1$:
\begin{itemize}[itemsep=0pt,topsep=0pt]
\item $E_{a1}z^r v_{N\varepsilon_1}=0$ whenever $r\geq N$;
\item $(E_{a_1 1} z^{r_1}) (E_{a_2 1} z^{r_2}) v_{N\varepsilon_1}=0$ whenever $r_1, r_2\geq N-1$;
\item $(E_{a_1 1}z^{r_1})\cdots(E_{a_K 1}z^{r_K}) v_{N\varepsilon_1} =0$ whenever $r_1,\ldots,r_K \geq N-K+1$.
\end{itemize}
More precisely, the monomial $(E_{a_1 1}z^{r_1})\cdots(E_{a_K 1}z^{r_K}) v_{N\varepsilon_1}$ vanishes unless $K\leq N$ and for every $m\geq0$ there are at most $m$ indices $i$ with $r_i\geq N-m$. By the pigeonhole principle, for $r_1\leq\ldots\leq r_K$ this is equivalent to the inequalities $r_i\leq i+N-K-1$. This is the necessary condition of the second item.

To show that the condition is also sufficient for pairwise different $a_i>1$, it is enough to show that the element $\prod_{i=1}^{N}(E_{a_i 1}z^{i-1}) v_{N\varepsilon_1}$ is nonzero and belongs to the top $z$-graded component. Consider the weight component of $\sW_{N\varepsilon_1}^{z}$ of weight $\sum_{i=1}^N \varepsilon_{a_i}$. By the character formula~\eqref{eq::Weyl::Character}, its graded dimension is the coefficient of the monomial $\prod_{i=1}^N x_{a_i}$ in $P_{N\varepsilon_1}(\overline{x},q,0)$, which is the $q$-factorial $[N]_q!=\prod_{i=1}^N\frac{1-q^i}{1-q}$. Hence the component of this weight and of $z$-degree $\frac{N(N-1)}{2}$ (the top degree of $[N]_q!$) is nonzero. It is spanned by the elements $\prod_{i=1}^NE_{a_i1}z^{r_i}v_{N\varepsilon_1}$ with $\sum_ir_i=\frac{N(N-1)}{2}$, and by the necessary condition this is possible only for $\{r_1,\ldots,r_N\}=\{0,1,\ldots,N-1\}$. This proves the sufficiency and the third item. Finally, for a permutation $\sigma\in\mathfrak{S}_N$ of the indices we get $\prod_{i=1}^NE_{a_i1}z^{r_i}v_{N\varepsilon_1}=(-1)^{\sigma}\prod_{i=1}^NE_{a_{\sigma(i)}1}z^{r_i}v_{N\varepsilon_1}$, because the top degree component of the Weyl module is isomorphic to $V_{\omega_N}=\Lambda^NV$.
\end{proof}

\begin{proposition}
\label{lem:CocyclicWeylProperties}
Let $\lambda=(\lambda_0\geq \ldots \geq \lambda_{l(\lambda)})\vdash N$ be a Young diagram, and denote by $\eta=(\eta_0\geq\ldots\geq \eta_l)\vdash N$ its transposed partition $\lambda^T$.
Then the following properties hold for the local Weyl $\gl_V[z]$-module $\sW_{\lambda}^{z}$ with highest weight vector $v_\lambda$ (assuming $\dim V$ is sufficiently large): 
\begin{enumerate}[label=(W\arabic*)]
\item
\label{item::socle:1}
The socle of $\sW_{\lambda}^{z}$ is generated by the following element:
\begin{equation}\label{eq:WeylModuleNonzeroMonomialCococle}
\prod_{k=0}^{l}\left(\prod_{j=1}^{\eta_k}\left(E_{a_{jk}, j}z^k\right) \right) v_{\lambda}.
\end{equation}
Here, $\{a_{jk} \colon 1\leq j\leq \eta_k\}$ is a set of $N$ pairwise distinct integers between $1$ and $\dim V$. In particular, this element is nonzero, and the highest $z$-degree in $\sW_{\lambda}^{z}$ equals $\sum_{k=0}^l k \eta_k$.

\item
\label{item::socle:2}
For a pairwise distinct $N$-tuple $\{a_{ij}\colon (ij)\in \lambda\}$, the following monomial of maximal $z$-degree
\[\left(\prod_{(i,j)\in\lambda} E_{a_{ij},j} z^{r_{ij}}\right) v_{\lambda}, \quad \text{where } \sum_{(ij)\in\lambda} r_{ij} = \sum_{k=0}^l k\eta_k,\]
is nonzero if and only if, for all $j=1,\ldots,l(\lambda)$, the multiset
$$\bar{r}_j:=\{r_{ij}\colon i=1,\ldots,\lambda_j\}$$
coincides with the multiset $\{0,1,\ldots,\lambda_j-1\}$.
\end{enumerate}
\end{proposition}

\begin{proof}
To prove item~\ref{item::socle:1}, we proceed by induction with respect to the partial order~\eqref{eq::dominance::order} on partitions of size $N$. For the minimal partition $\lambda=(1,\ldots,1)$ of size $N$, the module $\sW_{(1^N)}^{z}$ is irreducible and isomorphic to the fundamental representation $\Lambda^N V$. In this case, all the claims of the proposition hold trivially since no current action is possible.

For the induction step, assume the claims hold for all partitions $\lambda' < \lambda$ and consider the case of $\lambda$. Let us locate the first cell in $\lambda$ that can be moved to reduce the partition with respect to the partial order~\eqref{eq::dominance::order}. That is, take $i:=\max\{i\ \colon \ \lambda_i>1\}$ and consider the positive root 
$\gamma=\varepsilon_i-\varepsilon_{\eta_0+1}$.
Then,
\[\sum_{k=0}^l \omega_{\eta_k}>\sum_{k=0}^l \omega_{\eta_k}-\gamma=s_\gamma\left(\sum_{k=0}^l \omega_{\eta_k}\right)\in \Delta_+.\]

By property \ref{item::WeylInclusion}, the submodule of $\sW_{\sum_{k=0}^l \omega_{\eta_k}}^{z}$ generated by the element $E_{\eta_0+1,1}z^{l}v_{\sum_{k=0}^l \omega_{\eta_k}}$ is isomorphic to $\sW_{\sum_{k=0}^l \omega_{\eta_k}-\varepsilon_1+\varepsilon_{\eta_0+1}}^{z}$, and
\[\cosoc \left(\sW_{\sum_{k=0}^l \omega_{\eta_k}}^{z}\right)=\cosoc \left(\sW_{\sum_{k=0}^l \omega_{\eta_k}-\varepsilon_1+\varepsilon_{\eta_0+1}}^{z}\right).\]
Hence, by induction on the dominance order of the highest weight, we obtain
\[0 \neq \prod_{i=0}^{l}\prod_{r=1}^{\eta_i}E_{a_{ij}, j} z^i v_{\sum_{k=0}^l \omega_{\eta_k}} \in \cosoc \left(\sW_{\sum_{k=0}^l \omega_{\eta_k}}^{z}\right),\]
for a collection of integers $\{a_{k,j}\}$ matching the multiset $\{1,\dots,|\lambda|\}$. 

The “if” direction in claim~\ref{item::socle:2} follows directly from item~\ref{item::socle:1}. For the “only if” direction, consider the family of abelian subalgebras:
$$\fa_j:=\Span\left\langle E_{a j} \colon a > j \right\rangle \ \subset \ \gl_V.$$
Note that $\fa_j$ acts on the highest weight vector $v_{\omega_{\eta}}$ of the fundamental representation $V_{\omega_\eta}$ non-trivially only if $\eta \geq j$, and $\fa_j^2v_{\omega_{\eta}}=0$.
Using the isomorphism~\eqref{eq::Weyl=fusion}
$$
\sW_{\lambda} \simeq V_{\omega_{\eta_0}} \star_z\ldots\star_z V_{\omega_{\eta_l}}
$$
we fall under the assumptions of Proposition~\ref{prp::fusion::abelian}, applied to the trivial filtration of $\fa_j$ (i.e., all filtration components equal to $\fa_j$). Consequently, since $\lambda_{j-1}=\eta_{j-1}^{T}=|\{m\colon \eta_m\geq j\}|$ the element:
\begin{equation}
\label{eq::elem::prod}
\left(
\prod_{i=0}^{\lambda_{j-1}}E_{a_{ij}, j}z^{r_{ij}} \right) v_{\sum_{k=0}^l \omega_{\eta_k}} \text{ with }a_{ij}>j
\end{equation}
can be nonzero only if, for each $m$, the following inequality holds:
 $$|\{i \colon \ r_{ij}>m\}| \ \leq 
|\{i \geq m \ \colon\ \eta_i\geq j\}|.$$
In particular, under our vanishing conditions, the maximal possible $z$-degree of a nonzero element of the form~\eqref{eq::elem::prod} is equal to
$$
\sum_{i \colon \eta_i \geq j} i
= \binom{\lambda_{j-1}}{2}.$$
Thus, the only nonzero monomials of that degree must have the form:
\begin{equation}\label{eq::factorSpecialForm}
   \left( \prod_{i=0}^{\lambda_{j-1}-1}E_{a_{ij}, j}z^{i}\right) v_{\sum_{i=0}^l \omega_{\eta_i}} \quad \text{ with } a_{ij}\geq j.   
\end{equation}
Note, that we are dealing with the case $\dim V$ being sufficiently large; in particular, we can restrict to the case $a_{ij}> \eta_0=l(\lambda)$ for all $i,j$.  
Consequently, the elements $E_{a_{ij},j}$ pairwise commute for all $i,j$, and hence the product 
$$
\left(\prod_{j=0}^{l(\lambda)}\left(\prod_{i} E_{a_{ij},j} z^{r_{ij}}\right)\right) v_{\lambda}
$$
is nonzero only if each factor~\eqref{eq::elem::prod} is nonzero.
However, 
$$\sum_{j=0}^{\eta_0}\binom{\lambda_{j-1}}{2} = 
\sum_{j=0}^{\eta_0} \sum_{i=0}^{\lambda_j} i = 
\sum_{i=0}^{\lambda_j} i \eta_i.
$$
Finally, by the pigeonhole principle, if 
\[\left(\prod_{(i,j)\in\lambda} E_{a_{ij},j} z^{r_{ij}}\right) v_{\lambda} \neq 0,\]
then for each $j$ the factor \eqref{eq::elem::prod} has the degree $\binom{\lambda_{j-1}}{2}$ and is of the form \eqref{eq::factorSpecialForm}. This
concludes the proof of claim~\ref{item::socle:2}.
\end{proof}

\section{Iterated fusion product}
\label{sec::iterated::fusion}

In this section we work with two current variables $x$ and $y$ and with Young diagrams. To every Young diagram $\lambda\vdash N$ we associate a bigraded $\gl_V[x,y]$-module $\sWH_{\lambda}$, obtained as an iterated fusion product of $N$ copies of the tautological representation $V$. The cyclic vectors are always the highest weight vectors.
We describe the top-degree components of these modules using the properties of local Weyl modules. As a consequence, we construct a surjection of $\sWH_\lambda$ onto the Garsia--Haiman module $\mGH^{\gl}_\lambda$.

Throughout this section $\dim V$ is assumed to be sufficiently large (for instance, $\dim V\geq 3N+1$ suffices for all arguments below).

\subsection{Definition of the iterated fusion product}
\label{sec::sub::fusion}

\subsubsection{Conventions}
We use the notation for Young diagrams of \S\ref{sec::partitions}, in particular the numbers $n(\lambda)$ and $n(\lambda^T)$ of~\eqref{eq::n::lambda}. Recall that $\bar{z}=(x,y)$. Then a cell $\bar{s}=(s_x,s_y)$ corresponds to the monomial $\bar{z}^{\bar{s}}=x^{s_x}y^{s_y}$, and, by~\eqref{eq::mdeg::Delta},
\[
\mdeg(\Delta_\lambda)=\sum_{\bar{s}\in\lambda}\bar{s}=\big(n(\lambda^T),\,n(\lambda)\big).
\]
The Garsia--Haiman module of~\S\ref{sec::GH::SW} is $\mGH^{\gl}_\lambda=S^N(V[x,y])/\euJ^{\gl}_\lambda$.

\subsubsection{The modules $\sWH_\lambda$}
For $n\geq1$ consider the fusion product of $n$ copies of the tautological representation $V=V_{\varepsilon_1}$ with respect to the variable $x$:
\[
\sWH^{x}_{n}:=\underbrace{V\star_x\cdots\star_x V}_{n}\ \simeq\ \sW^{x}_{n\varepsilon_1}.
\]
By Theorem~\ref{thm::Weyl::fusion}, it does not depend on the fusion parameters and is isomorphic to the local Weyl module $\sW^x_{n\varepsilon_1}$. Its cyclic vector $v_{n\varepsilon_1}$ is the highest weight vector (the symbol of $v_{\varepsilon_1}\otimes\cdots\otimes v_{\varepsilon_1}$). The module $\sWH^x_n$ is a graded $\gl_V[x]$-module.

\begin{definition}
\label{def::sWH}
Let $\lambda=(\lambda_0\geq\ldots\geq\lambda_l)$ be a Young diagram and let $\gamma_0,\ldots,\gamma_l\in\bC$ be pairwise distinct. The \emph{iterated fusion product} $\sWH_\lambda=\sWH_\lambda^{y,x}$ is the fusion product with respect to the variable $y$ of the $\gl_V[x]$-modules $\sWH^x_{\lambda_j}$ attached to the rows of $\lambda$:
\[
\sWH_{\lambda}^{y,x}:=\Big[\sWH_{\lambda_0}^{x}\star_y\sWH_{\lambda_1}^{x}\star_{y}\cdots\star_{y}\sWH_{\lambda_l}^{x}\Big]_{\gamma_0,\ldots,\gamma_l}.
\]
Here we apply Definition~\ref{def::fusion} to the Lie algebra $\fg=\gl_V[x]$, whose current algebra is $\fg[y]=\gl_V[x,y]$. The cyclic vectors are the highest weight vectors $v_{\lambda_j\varepsilon_1}\in\sWH^x_{\lambda_j}$.
We denote by $v:=v_{\lambda_0\varepsilon_1}\otimes\cdots\otimes v_{\lambda_l\varepsilon_1}$ the cyclic vector of the tensor product of evaluation modules
\begin{equation}
\label{eq::tensor::weyl}
T_\lambda:=\ev_{\gamma_0}\big(\sWH^{x}_{\lambda_0}\big)\otimes\cdots\otimes\ev_{\gamma_l}\big(\sWH^{x}_{\lambda_l}\big),
\end{equation}
by $\calF_n=U(\gl_V[x,y])_{\leq n}\cdot v$ its fusion filtration (by the $y$-degree), and by $\widehat{v}\in\sWH_\lambda$ the cyclic vector of the fusion product.
\end{definition}

A priori the module $\sWH_\lambda$ may depend on the parameters $\gamma_j$. All the statements below hold for every choice of pairwise distinct parameters, and we omit them from the notation (see Corollary~\ref{cor::sWH::GH} for the independence).

By Remark~\ref{rem::fusion::graded}, the module $\sWH_\lambda$ is \emph{bigraded}: the $x$-grading comes from the gradings of the factors $\sWH^x_{\lambda_j}$ (the fusion filtration consists of $x$-graded subspaces), and the $y$-grading is the fusion degree. The element $g x^ry^t$ ($g\in\gl_V$) shifts the bidegree by $(r,t)$, and $\widehat{v}$ has bidegree $(0,0)$ and weight $N\varepsilon_1$.
Moreover, by Remark~\ref{rem::fusion::graded}(i) and Theorem~\ref{thm::Weyl::fusion},
\begin{equation}
\label{eq::sWH::glV}
\sWH_\lambda\simeq\bigotimes_{j=0}^{l}\sWH^x_{\lambda_j}\simeq V^{\otimes N}\qquad\text{as $\gl_V$-modules}.
\end{equation}

\begin{example}
\label{ex::sWH::row::column}
For the one-row diagram $\lambda=(N)$ we get $\sWH_{(N)}=\sWH^x_N\simeq\sW^x_{N\varepsilon_1}$. For the one-column diagram $\lambda=(1^N)$ we get $\sWH_{(1^N)}=V\star_y\cdots\star_yV\simeq\sW^y_{N\varepsilon_1}$.
\end{example}

The fusion product is generated from the cyclic vector by an abelian subalgebra. Let
\[
\fa:=\Span\langle E_{a1}\colon a\geq2\rangle\subset\gl_V,\qquad
\fb:=\Span\langle E_{ij}\colon j\geq2\rangle\oplus\bC E_{11}\subset\gl_V .
\]
Then $\fa$ is an abelian subalgebra, and $\fb$ is a subalgebra: $[E_{ij},E_{km}]=\delta_{jk}E_{im}-\delta_{mi}E_{kj}\in\fb$ for $j,m\geq2$, and $[E_{11},E_{km}]=\delta_{1k}E_{1m}\in\fb$ for $m\geq2$. Moreover, $\gl_V=\fa\oplus\fb$.

\begin{lemma}
\label{Non1Vanishing}
For all $i\geq1$, $j\geq2$ and $r,t\geq0$ we have $(E_{ij}x^ry^t)\cdot v=0$ in $T_\lambda$. Moreover, $(E_{11}x^ry^t)\cdot v=\delta_{r,0}\big(\sum_{m}\gamma_m^{t}\lambda_m\big)\,v$. Consequently, $\fb[x,y]$ acts on $v$ by scalars, and $(E_{ij}x^ry^t)\cdot\widehat{v}=0$ in $\sWH_\lambda$ for $j\geq2$.
\end{lemma}
\begin{proof}
The element $E_{ij}x^ry^t$ acts on $T_\lambda$ by $\sum_m\gamma_m^t\,(E_{ij}x^r)^{(m)}$, where the superscript $(m)$ denotes the action on the $m$-th factor. So it suffices to compute $E_{ij}x^r\cdot v_{\lambda_m\varepsilon_1}$ in the Weyl module $\sW^x_{\lambda_m\varepsilon_1}$.
This vector has weight $\lambda_m\varepsilon_1+\varepsilon_i-\varepsilon_j$. For $j\geq2$ and $i\neq j$, this weight has a negative coefficient at $\varepsilon_j$, which is impossible in a polynomial representation.
For $i=j\geq2$ and $r=0$ we get $E_{jj}v_{\lambda_m\varepsilon_1}=0$. For $i=j$ and $r>0$ the vector $E_{jj}x^r v_{\lambda_m\varepsilon_1}$ has weight $\lambda_m\varepsilon_1$ and positive $x$-degree. It vanishes, because the weight space of weight $\lambda_m\varepsilon_1$ of the Weyl module is one-dimensional and sits in degree $0$.
The same argument gives $E_{11}x^rv_{\lambda_m\varepsilon_1}=\delta_{r,0}\lambda_m v_{\lambda_m\varepsilon_1}$.
\end{proof}

\begin{proposition}
\label{prp::fusion::abelianGL}
\begin{enumerate}
\item For all $n\geq0$ we have $\calF_n=U(\fa[x,y])_{\leq n}\cdot v$. In particular, $\sWH_\lambda$ is generated from $\widehat{v}$ by the action of the abelian Lie algebra $\fa[x,y]$.
\item As an $\fa[x,y]$-module, $\sWH_\lambda$ coincides with the fusion product over the abelian Lie algebra $\fa[x]$ of the $\fa[x]$-modules $\sWH^x_{\lambda_j}$. These modules are generated by $v_{\lambda_j\varepsilon_1}$ over $\fa[x]$.
\end{enumerate}
\end{proposition}
\begin{proof}
The algebra $U(\gl_V[x,y])$ is graded by the $y$-degree, and $\gl_V[x,y]=\fa[x,y]\oplus\fb[x,y]$ with both summands being graded subalgebras. By the PBW theorem, the multiplication map $U(\fa[x,y])\otimes U(\fb[x,y])\to U(\gl_V[x,y])$ is an isomorphism of graded vector spaces. Hence
\[
U(\gl_V[x,y])_{\leq n}=\sum_{p+q\leq n}U(\fa[x,y])_{p}\,U(\fb[x,y])_{q}.
\]
By Lemma~\ref{Non1Vanishing}, $U(\fb[x,y])\cdot v=\bC v$. Therefore $\calF_n=U(\fa[x,y])_{\leq n}\cdot v$, which proves item~(i).
The same argument applied to $\gl_V[x]=\fa[x]\oplus\fb[x]$ and to the Weyl modules $\sW^x_{\lambda_j\varepsilon_1}$ shows that $\sWH^x_{\lambda_j}=U(\fa[x])\cdot v_{\lambda_j\varepsilon_1}$. Item~(ii) now follows from item~(i) and Definition~\ref{def::fusion}.
\end{proof}

\subsection{Top-degree component of $\sWH_\lambda$}

\begin{definition}
Let $M=\bigoplus_{p,q}M_{p,q}$ be a finite-dimensional bigraded vector space, where $p$ is the $x$-degree and $q$ is the $y$-degree. We define
\begin{align*}
\mtop^{x}(M)&:=\bigoplus_{q}M_{p_{\max},q}, & p_{\max}&:=\max\{p\colon M_{p,q}\neq0\text{ for some }q\};\\
\mtop^{x,y}(M)&:=\bigoplus_{p+q=k_{\max}}M_{p,q}, & k_{\max}&:=\max\{p+q\colon M_{p,q}\neq0\}.
\end{align*}
The component $\mtop^y(M)$ is defined similarly.
\end{definition}

We emphasize the variables with respect to which we take the top-degree component. The component $\mtop^x(M)$ remains graded with respect to the other variable $y$.

\begin{theorem}
\label{thm::top::fusion}
Let $\lambda\vdash N$ be a Young diagram. Then, for every choice of pairwise distinct parameters:
\begin{enumerate}
\item The module $\sWH_\lambda$ is concentrated in the bidegrees $(p,q)$ with $p\leq n(\lambda^T)$ and $q\leq n(\lambda)$.
\item The top $x$-degree component $\mtop^x(\sWH_\lambda)$ has $x$-degree $n(\lambda^T)$. It is a $\gl_V[x,y]$-submodule on which $x\,\gl_V[x,y]$ acts by zero. As a graded $\gl_V[y]$-module, it is isomorphic to the local Weyl module:
\[
\mtop^x(\sWH_\lambda)\simeq\sW^{y}_{\lambda^T}.
\]
\item The top bidegree component $\mtop^{x,y}(\sWH_\lambda)$ is an irreducible $\gl_V$-module isomorphic to $\Lambda^NV$. It is concentrated in the bidegree $\mdeg(\Delta_\lambda)=\big(n(\lambda^T),n(\lambda)\big)$.
\item Let $a_1,\ldots,a_N\geq2$ be pairwise distinct, and let $(r_1,t_1),\ldots,(r_N,t_N)\in\bZ^2_{\geq0}$ be such that $\sum_{i=1}^{N}(r_i,t_i)=\big(n(\lambda^T),n(\lambda)\big)$. Then the element
\begin{equation}
\label{eq::monom::top}
\left(\prod_{i=1}^{N}E_{a_i1}x^{r_i}y^{t_i}\right)\widehat{v}\in\sWH_\lambda
\end{equation}
is nonzero if and only if $\{(r_1,t_1),\ldots,(r_N,t_N)\}$ is the set of cells of $\lambda$, i.e., the map $i\mapsto(r_i,t_i)$ is a bijection onto the cells of $\lambda$.
\end{enumerate}
\end{theorem}

\begin{proof}
Put $T:=T_\lambda$ and $M_j:=\sWH^x_{\lambda_j}\simeq\sW^x_{\lambda_j\varepsilon_1}$, $v_j:=v_{\lambda_j\varepsilon_1}$. We use the following facts about the Weyl modules $M_j$ (Propositions~\ref{prp::Weyl::top} and~\ref{prop:LineWeylMudulesProperties}):
\begin{enumerate}[label=(\alph*)]
\item the top $x$-degree of $M_j$ is $\binom{\lambda_j}{2}$, and $\mtop^x(M_j)\simeq\Lambda^{\lambda_j}V$;
\item for pairwise distinct $c_1,\ldots,c_k\geq2$ and $r_1\leq\ldots\leq r_k$ the vector $\big(\prod_{i=1}^kE_{c_i1}x^{r_i}\big)v_j$ vanishes unless $k\leq\lambda_j$ and $r_i\leq i+\lambda_j-k-1$ for all $i$. In particular, for $k=\lambda_j$ this vector lies in the top $x$-degree $\binom{\lambda_j}{2}$ if and only if $\{r_1,\ldots,r_k\}=\{0,1,\ldots,\lambda_j-1\}$, and then it is nonzero.
\end{enumerate}

\emph{Step 0: the bounds on the bidegrees.}
The fusion filtration consists of $x$-graded subspaces. Hence $\sWH_\lambda$ and $T$ have the same $x$-graded dimensions, and the maximal $x$-degree of $\sWH_\lambda$ equals $\sum_j\binom{\lambda_j}{2}=n(\lambda^T)$ by~(a).
For the $y$-degree, we use Proposition~\ref{prp::fusion::abelianGL}(ii) and apply Proposition~\ref{prp::fusion::abelian} to the abelian Lie algebra $\fa[x]$ with the constant sequence of subspaces $\fa[x]\subseteq\fa[x]\subseteq\cdots$ and $\eta_j=\lambda_j$. The hypothesis holds: any product of $\lambda_j+1$ elements of $\fa[x]$ kills $v_j$, because the resulting weight would have a negative coefficient at $\varepsilon_1$.
By~\eqref{eq::vanish::action}, a monomial $\prod_{i}(E_{a_i1}x^{r_i}y^{t_i})\cdot\widehat{v}$ vanishes unless $\#\{i\colon t_i\geq k\}\leq\sum_{m\geq k}\lambda_m$ for all $k\geq1$. In that case
\[
\sum_it_i=\sum_{k\geq1}\#\{i\colon t_i\geq k\}\leq\sum_{k\geq1}\sum_{m\geq k}\lambda_m=\sum_m m\lambda_m=n(\lambda).
\]
Since such monomials span $\sWH_\lambda$, item~(i) follows.

\emph{Step 1: the top $x$-component of $T$ and its cyclic vector.}
Let $Q:=\mtop^x(T)$ be the component of $T$ of $x$-degree $n(\lambda^T)$. By~(a), $Q=\bigotimes_j\mtop^x(M_j)$. The elements $Xy^t$ preserve the $x$-degree, and the elements $Xx^ry^t$ with $r>0$ act on $Q$ by zero. Therefore $Q$ is a $\gl_V[x,y]$-submodule, and as a $\gl_V[y]$-module
\[
Q\simeq\ev_{\gamma_0}\big(\Lambda^{\lambda_0}V\big)\otimes\cdots\otimes\ev_{\gamma_l}\big(\Lambda^{\lambda_l}V\big).
\]
Since $\calF_n$ is $x$-graded, $\mtop^x(\sWH_\lambda)=\bigoplus_n(\calF_n\cap Q)/(\calF_{n-1}\cap Q)$ as a $\gl_V[y]$-module.

Fix pairwise distinct indices $c_0,c_1,\ldots,c_{\lambda_0-1}$ with $c_s>\lambda_0$, and put
\[
\beta'_j:=\Big(\prod_{s=0}^{\lambda_j-1}E_{c_s1}x^{s}\Big)v_j\in\mtop^x(M_j),\qquad
\beta':=\beta'_0\otimes\cdots\otimes\beta'_l\in Q .
\]
By~(b), $\beta'_j\neq0$. We claim that $\beta'\in\calF_0$. Indeed, $\calF_0=U(\gl_V[x])\cdot v$, where $\gl_V[x]$ acts on $T$ diagonally, $Xx^r\mapsto\sum_m(Xx^r)^{(m)}$. Consider
\[
\tilde\beta:=\prod_{s=0}^{\lambda_0-1}\big(E_{c_s1}x^s\big)^{\lambda^T_s}\cdot v\in\calF_0 .
\]
All factors commute, since $\fa$ is abelian. Expanding each $E_{c_s1}x^s$ as $\sum_j(E_{c_s1}x^s)^{(j)}$, we obtain a sum of terms $\bigotimes_jw_j$, where $w_j$ is obtained by applying to $v_j$ the factors assigned to the $j$-th tensor position.
The total $x$-degree of $\tilde\beta$ is $\sum_s s\lambda^T_s=n(\lambda^T)$, which is the maximal $x$-degree of $T$. Hence in every nonzero term each $w_j$ lies in $\mtop^x(M_j)\simeq\Lambda^{\lambda_j}V$.
Since the weights of $\Lambda^{\lambda_j}V$ have all coefficients at most $1$ and at least $0$, and the total number of factors is $N=\sum_j\lambda_j$, every nonzero term receives exactly $\lambda_j$ factors with pairwise distinct indices $c_s$ in the position $j$. By~(b), their $x$-degrees form the set $\{0,\ldots,\lambda_j-1\}$. Since the factor $E_{c_s1}$ always comes with $x^s$, the position $j$ receives exactly the factors $E_{c_s1}x^s$ with $s<\lambda_j$, i.e., $w_j=\beta'_j$.
For each $s$, the $\lambda^T_s$ copies of $E_{c_s1}x^s$ must be distributed bijectively among the $\lambda^T_s$ rows $j$ with $\lambda_j>s$, which can be done in $\lambda^T_s!$ ways. Therefore
\[
\tilde\beta=\Big(\prod_{s}\lambda^T_s!\Big)\,\beta',\qquad\text{so}\quad\beta'\in\calF_0\cap Q.
\]
Let $\beta_j\in\mtop^x(M_j)$ be a highest weight vector of $\Lambda^{\lambda_j}V$, and $\beta:=\beta_0\otimes\cdots\otimes\beta_l$, a highest weight vector of $Q$ of weight $\sum_j\omega_{\lambda_j}$ (the weight corresponding to $\lambda^T$). Let $g\in\GL_V$ be the permutation matrix exchanging $e_{s+1}$ and $e_{c_s}$ for $s=0,\ldots,\lambda_0-1$. The weight spaces of $\Lambda^{\lambda_j}V$ are one-dimensional, so $\beta'_j\in\bC^\times g\beta_j$ and hence $\beta'\in\bC^\times g\beta$. In particular, $\beta$ and $\beta'$ generate the same $\gl_V$-submodule $U(\gl_V)\beta=U(\gl_V)\beta'\subset Q$, namely the Cartan component $V_{\lambda^T}\subset\bigotimes_j\Lambda^{\lambda_j}V$.
By Corollary~\ref{cor::ev::cyclic}, $\beta$ generates $Q$ as a $\gl_V[y]$-module.

\emph{Step 2: comparison of the filtrations.}
Consider the fusion filtration of the $\gl_V[y]$-module $Q$ with respect to $\beta$:
\[
G_n:=U(\gl_V[y])_{\leq n}\cdot\beta=U(\gl_V[y])_{\leq n}\cdot\beta' .
\]
By definition and Theorem~\ref{thm::Weyl::fusion},
\[
\bigoplus_nG_n/G_{n-1}=\big[\Lambda^{\lambda_0}V\star_y\cdots\star_y\Lambda^{\lambda_l}V\big]_{\gamma}\simeq\sW^y_{\sum_j\omega_{\lambda_j}}=\sW^y_{\lambda^T}.
\]
We show that $G_n=\calF_n\cap Q$ for all $n$; then item~(ii) follows.
Since $\beta'\in\calF_0$, we have $G_n\subset U(\gl_V[x,y])_{\leq n}\calF_0\subset\calF_n$, so $G_n\subset\calF_n\cap Q$.

Both $G_n$ and $\calF_n\cap Q$ are $\gl_V$-submodules of the finite-dimensional polynomial $\gl_V$-module $Q$. Fix pairwise distinct indices $a_1,\ldots,a_N$ and $b_1,\ldots,b_N$, all bigger than $\lambda_0$ and different from the $c_s$, and let $\mu:=\varepsilon_{a_1}+\ldots+\varepsilon_{a_N}$. Every irreducible polynomial $\gl_V$-module of degree $N$ has a nonzero $\mu$-weight space. Hence it suffices to prove that $(\calF_n\cap Q)_\mu\subset G_n$.
By Proposition~\ref{prp::fusion::abelianGL}(i), the space $(\calF_n)_\mu$ is spanned by the vectors $m=\big(\prod_{i=1}^NE_{a_i1}x^{r_i}y^{t_i}\big)v$ with $\sum_it_i\leq n$. Since $\calF_n$ is $x$-graded, $(\calF_n\cap Q)_\mu$ is spanned by those of them with $\sum_ir_i=n(\lambda^T)$.
Fix such an element $m$. We claim that in $T$
\begin{equation}
\label{eq::gen::mon}
m=\Big(\prod_{i=1}^NE_{a_ib_i}y^{t_i}\Big)\,w \text{ with } w:=\Big(\prod_{i=1}^NE_{b_i1}x^{r_i}\Big)v .
\end{equation}
Indeed, all the elements $E_{b_k1}x^{r}$ and $E_{a_k1}x^ry^t$ commute with each other, and
\[
[E_{a_ib_i}y^{t},E_{b_k1}x^{r}]=\delta_{ik}\,E_{a_i1}x^ry^t .
\] Moreover, $E_{a_ib_i}y^t$ commutes with $E_{a_k1}x^ry^{t'}$, and $E_{a_ib_i}y^t\cdot v=0$ by Lemma~\ref{Non1Vanishing}. Applying the operators $E_{a_ib_i}y^{t_i}$ to $w$ one by one, each of them replaces the factor $E_{b_i1}x^{r_i}$ by $E_{a_i1}x^{r_i}y^{t_i}$, which proves~\eqref{eq::gen::mon}.

The vector $w$ belongs to $\calF_0\cap Q$. We show that $w\in G_0$.
Expanding $w$ as in Step~1, a nonzero term receives in each position $j$ exactly $\lambda_j$ factors, whose $x$-degrees form the set $\{0,\ldots,\lambda_j-1\}$. Hence, if $w\neq0$, then $\#\{i\colon r_i=s\}=\lambda^T_s$ for all $s$.
In this case, the same computation as in~\eqref{eq::gen::mon} shows that
\[
\Big(\prod_{i=1}^NE_{b_ic_{r_i}}\Big)\tilde\beta=\Big(\prod_s\lambda^T_s!\Big)\,w .
\]
Indeed, $E_{bc_s}$ kills $v$, commutes with all $E_{c_{s'}1}x^{s'}$ for $s'\neq s$ and with all $E_{b'1}x^r$, and $[E_{bc_s},E_{c_s1}x^s]=E_{b1}x^s$. So the $\lambda^T_s$ operators $E_{b_ic_s}$ with $r_i=s$ turn the $\lambda^T_s$ factors $E_{c_s1}x^s$ of $\tilde\beta$ into the factors $E_{b_i1}x^s$, producing the coefficient $\lambda^T_s!$.
Therefore $w=\big(\prod_iE_{b_ic_{r_i}}\big)\beta'\in U(\gl_V)\beta'=G_0$.
By~\eqref{eq::gen::mon}, $m\in U(\gl_V[y])_{\leq n}G_0=G_n$. This proves $\calF_n\cap Q=G_n$ and item~(ii).

\emph{Step 3: the top bidegree.}
By item~(i), the top bidegree component of $\sWH_\lambda$ is contained in the bidegree $(n(\lambda^T),n(\lambda))$, provided the latter is nonzero. By item~(ii), this component is the top $y$-degree component of $\sW^y_{\lambda^T}$. By Proposition~\ref{lem:CocyclicWeylProperties}\ref{item::socle:1}, it sits in the $y$-degree $\sum_kk\lambda_k=n(\lambda)$, and by Proposition~\ref{prp::Weyl::top} it is isomorphic to $\Lambda^NV$. This proves item~(iii).

\emph{Step 4: the monomials of top degree.}
Let $g'\in\GL_V$ be a permutation matrix fixing $e_1$. Then $g'\widehat{v}=\widehat{v}$ and $g'(E_{a1}x^ry^t)g'^{-1}=E_{g'(a)1}x^ry^t$. Hence the vanishing of~\eqref{eq::monom::top} does not change if we replace the indices $a_i$ by $g'(a_i)$. We may therefore assume that $a_i>\lambda_0$ and choose $b_i$ and $c_s$ as in Step~2.
Let $m$ be the element~\eqref{eq::gen::mon} of $T$; its class in $\calF_{n(\lambda)}/\calF_{n(\lambda)-1}$ is the element~\eqref{eq::monom::top}.
If the multiset $\{r_i\}$ does not satisfy $\#\{i\colon r_i=s\}=\lambda^T_s$ for all $s$, then $w=0$ by Step~2, so~\eqref{eq::monom::top} vanishes. In this case $\{(r_i,t_i)\}$ is not the set of cells of $\lambda$, because the cells in the $s$-th column number $\lambda^T_s$.
Otherwise, by Step~2 and the identification $\mtop^x(\sWH_\lambda)\simeq\bigoplus_nG_n/G_{n-1}=\sW^y_{\lambda^T}$, the element~\eqref{eq::monom::top} equals
\[
\Big(\prod_iE_{a_ib_i}y^{t_i}\Big)\Big(\prod_iE_{b_ic_{r_i}}\Big)\widehat{\beta'}\in\sW^y_{\lambda^T},
\]
where $\widehat{\beta'}\in\bC^\times g\widehat{\beta}$ and $\widehat{\beta}$ is the highest weight vector of $\sW^y_{\lambda^T}$. Conjugating by $g=g^{-1}$, which fixes all $a_i$ and $b_i$ and exchanges $c_s$ with $s+1$, and then applying the identity~\eqref{eq::gen::mon} in the opposite direction (now for the highest weight vector $\widehat{\beta}$, which is killed by $E_{a_ib_i}y^t$ since $b_i>\lambda_0$), we find that~\eqref{eq::monom::top} is a nonzero multiple of
\[
g\cdot\Big(\prod_{i=1}^NE_{a_i,\,r_i+1}\,y^{t_i}\Big)\widehat{\beta}.
\]
The highest weight of $\sW^y_{\lambda^T}$ is $\sum_s\lambda^T_s\varepsilon_{s+1}$. By Proposition~\ref{lem:CocyclicWeylProperties}\ref{item::socle:2}, applied to $\sW^y_{\lambda^T}$, this element is nonzero if and only if for each $s$ the multiset $\{t_i\colon r_i=s\}$ equals $\{0,1,\ldots,\lambda^T_s-1\}$.
This condition means precisely that $\{(r_i,t_i)\}=\{(s,t)\colon t<\lambda^T_s\}$, which is the set of cells of $\lambda$. This proves item~(iv).
\end{proof}

\begin{remark}
The proof shows that the bidegrees of $\sWH_\lambda$ are bounded componentwise by $\mdeg(\Delta_\lambda)$, and that the top $x$-degree component is controlled by the rows of $\lambda$ while the $y$-direction is controlled by the local Weyl module $\sW^y_{\lambda^T}$. The roles of $x$ and $y$ are not symmetric in the construction of $\sWH_\lambda$. By Lemma~\ref{lem::GH::symmetry}, the Garsia--Haiman modules satisfy $\mGH^{\gl}_{\lambda^T}\simeq\mGH^{\gl}_\lambda$ up to the exchange of $x$ and $y$.
\end{remark}

\subsection{Surjection onto Garsia--Haiman modules}

The following observation is due to B.~Feigin and S.~Loktev~\cite{Feigin_Loktev::multi::Weyl}.

\begin{proposition}
\label{prp::proj::cov}
The module $S^N(V[x,y])$ is the projective cover of $S^NV$ in the category $\RepGLz{V}{N}$ of polynomial $\gl_V[x,y]$-modules of degree $N$. More precisely, for every $U\in\RepGLz{V}{N}$ and every vector $u\in U$ of weight $N\varepsilon_1$ there exists a unique homomorphism $S^N(V[x,y])\to U$ mapping $e_1^N$ to $u$.
\end{proposition}
\begin{proof}
By Proposition~\ref{prop::characters::free}(i) with $\lambda=(N)$, the module $S^N(V[x,y])$ is the projective cover of $S^NV$ in $\RepGLz{V}{N}$, generated by $e_1^N$. Since $e_1^N$ is a generator, a homomorphism is determined by the image of $e_1^N$, which has weight $N\varepsilon_1$. The dimension of $\mathrm{Hom}\big(S^N(V[x,y]),U\big)$ is the multiplicity of $S^NV$ in the composition series of $U$. It equals $\dim U_{N\varepsilon_1}$, because $N\varepsilon_1$ is a weight of $S^NV$ with multiplicity one and of no other irreducible polynomial representation of degree $N$. Hence the injective map $\mathrm{Hom}(S^N(V[x,y]),U)\to U_{N\varepsilon_1}$, $\theta\mapsto\theta(e_1^N)$, is bijective.
\end{proof}

The vector $\widehat{v}\in\sWH_\lambda$ has weight $N\varepsilon_1$ and generates $\sWH_\lambda$. Hence Proposition~\ref{prp::proj::cov} gives a surjective homomorphism of bigraded $\gl_V[x,y]$-modules
\begin{equation}
\label{eq::project::H}
\begin{tikzcd}
\theta_{\lambda}\colon S^N(V[x,y]) \ar[r,two heads] & \sWH_{\lambda},
\end{tikzcd}
\qquad e_1^N\mapsto\widehat{v}.
\end{equation}

\begin{proposition}
\label{prp::fusion->GH}
The projection $\Theta_{\lambda}\colon S^N(V[x,y])\twoheadrightarrow\mGH^{\gl}_\lambda$ of Definition~\ref{not::GH::gl} factors through~\eqref{eq::project::H}:
\[
\begin{tikzcd}
S^{N}(V[x,y]) \arrow[r,"\theta_\lambda", two heads] \arrow[rr,bend right, "\Theta_{\lambda}"', two heads ]& \sWH_{\lambda} \arrow[r,dashed, two heads, "\exists\, \bar{\theta}_{\lambda}"]& \mGH^{\gl}_{\lambda}
\end{tikzcd}
\]
The surjection $\bar{\theta}_\lambda$ is bigraded, and it is an isomorphism on the top bidegree components $\mtop^{x,y}(\sWH_\lambda)\simeq\Lambda^NV\simeq\mtop^{x,y}(\mGH^{\gl}_\lambda)$.
\end{proposition}
\begin{proof}
We have to show that $K:=\ker\theta_\lambda$ is contained in $\euJ^{\gl}_\lambda=\ker\Theta_\lambda$. By Definition~\ref{not::GH::gl}, $\euJ^{\gl}_\lambda$ is the largest submodule contained in $\ker\phi_\lambda$, where $\phi_\lambda\colon S^N(V[x,y])\to\Lambda^NV$ is the $\gl_V$-equivariant map~\eqref{eq::phi::lambda}. Since $K$ is a submodule, it suffices to show that $\phi_\lambda(K)=0$.
The image $\phi_\lambda(K)$ is a $\gl_V$-submodule of $\Lambda^NV$. Fix pairwise distinct $a_1,\ldots,a_N\geq2$ and put $\mu=\varepsilon_{a_1}+\ldots+\varepsilon_{a_N}$. It suffices to check that the $\mu$-weight space of $\phi_\lambda(K)$ vanishes.
Since $\phi_\lambda$ is $\gl_V$-equivariant, vanishes outside the bidegree $\mdeg(\Delta_\lambda)$, and $K$ is bigraded, this weight space equals $\phi_\lambda(K\cap W)$. Here $W$ is the component of $S^N(V[x,y])$ of weight $\mu$ and bidegree $\mdeg(\Delta_\lambda)$.

The space $W$ has the basis
\[
p_{\mathbf{r}}:=\prod_{i=1}^N\big(e_{a_i}x^{r_i}y^{t_i}\big)=\Psi\Big(e_{a_1}\otimes\cdots\otimes e_{a_N}\otimes\prod_{i=1}^N\bar{z}_i^{(r_i,t_i)}\Big),
\qquad \mathbf{r}=\big((r_1,t_1),\ldots,(r_N,t_N)\big),
\]
where $\mathbf{r}$ runs over the tuples with $\sum_i(r_i,t_i)=\mdeg(\Delta_\lambda)$.
Fix an enumeration $\bar{s}_1,\ldots,\bar{s}_N$ of the cells of $\lambda$ and, for $\sigma\in\bS_N$, let $\mathbf{r}_\sigma:=(\bar{s}_{\sigma(1)},\ldots,\bar{s}_{\sigma(N)})$.

\emph{The map $\phi_\lambda$ on $W$.} By~\eqref{eq::GH::Delta} and~\eqref{eq::phi::lambda}, $\phi_\lambda(p_{\mathbf{r}})=\big\langle\prod_i\bar{z}_i^{(r_i,t_i)},\Delta_\lambda\big\rangle\,e_{a_1}\wedge\cdots\wedge e_{a_N}$. The monomial $\prod_i\bar{z}_i^{(r_i,t_i)}$ occurs in $\Delta_\lambda$ if and only if $\mathbf{r}=\mathbf{r}_\sigma$ for some (unique) $\sigma$, and then with the coefficient $(-1)^\sigma$. Hence
\[
\phi_\lambda(p_{\mathbf{r}_\sigma})=(-1)^{\sigma}A\,e_{a_1}\wedge\cdots\wedge e_{a_N},\quad A:=\prod_{\bar{s}\in\lambda}s_x!\,s_y!\neq0,\qquad
\phi_\lambda(p_{\mathbf{r}})=0\ \text{ for all other } \mathbf{r}.
\]

\emph{The map $\theta_\lambda$ on $W$.} Since $a_i\geq2$, the operators $E_{a_i1}x^{r}y^{t}$ commute and act on $S^N(V[x,y])$ by derivations. Hence $p_{\mathbf{r}}=\frac{1}{N!}\big(\prod_iE_{a_i1}x^{r_i}y^{t_i}\big)e_1^N$ (see~\eqref{eq::GH::determinant}), and therefore
\[
\theta_\lambda(p_{\mathbf{r}})=\frac{1}{N!}\,m_{\mathbf{r}},\qquad m_{\mathbf{r}}:=\Big(\prod_iE_{a_i1}x^{r_i}y^{t_i}\Big)\widehat{v}.
\]
By Theorem~\ref{thm::top::fusion}(iv), $m_{\mathbf{r}}\neq0$ if and only if $\mathbf{r}=\mathbf{r}_\sigma$ for some $\sigma$. All the elements $m_{\mathbf{r}}$ lie in the $\mu$-weight space of $\mtop^{x,y}(\sWH_\lambda)\simeq\Lambda^NV$, which is one-dimensional.
Moreover, $m_{\mathbf{r}_\sigma}=(-1)^\sigma m_{\mathbf{r}_{\mathrm{id}}}$. Indeed, reordering the commuting factors gives $m_{\mathbf{r}_\sigma}=\big(\prod_jE_{a_{\sigma^{-1}(j)}1}\bar{z}^{\bar{s}_j}\big)\widehat{v}=g_\sigma\cdot m_{\mathbf{r}_{\mathrm{id}}}$, where $g_\sigma\in\GL_V$ is the permutation matrix $e_{a_j}\mapsto e_{a_{\sigma^{-1}(j)}}$ fixing all other basis vectors. Since $g_\sigma$ fixes $e_1$, we have $g_\sigma\widehat{v}=\widehat{v}$, and $g_\sigma$ acts on the $\mu$-weight space of $\Lambda^NV$ by the sign $(-1)^\sigma$.

\emph{Conclusion.} Let $\ell$ be the linear map from the line $\bC\,m_{\mathbf{r}_{\mathrm{id}}}$ to $\bC\,e_{a_1}\wedge\cdots\wedge e_{a_N}$ sending $\frac{1}{N!}m_{\mathbf{r}_{\mathrm{id}}}$ to $A\,e_{a_1}\wedge\cdots\wedge e_{a_N}$. The computations above show that $\phi_\lambda|_W=\ell\circ\theta_\lambda|_W$. Hence $\phi_\lambda$ vanishes on $K\cap W=\ker(\theta_\lambda|_W)$. Thus $\phi_\lambda(K)=0$, $K\subset\euJ^{\gl}_\lambda$, and $\bar{\theta}_\lambda$ exists. It is surjective and bigraded because $\Theta_\lambda$ and $\theta_\lambda$ are.
Finally, by Theorem~\ref{thm::top::fusion}(iii) and Proposition~\ref{prp::GH::gl}(iv), both top bidegree components are isomorphic to $\Lambda^NV$ and sit in the bidegree $\mdeg(\Delta_\lambda)$. A bigraded surjection restricts to a surjection $\Lambda^NV\twoheadrightarrow\Lambda^NV$, which is an isomorphism.
\end{proof}

For instance, the element
\begin{equation}
\label{eq::monom::topMain}
\left(\prod_{i=1}^{N}E_{a_i1}\bar{z}^{\bar{s}_i}\right)\widehat{v}\in\sWH_\lambda,\qquad 1<a_1<\ldots<a_N,
\end{equation}
is mapped by $\bar{\theta}_\lambda$ to the element $\Theta_\lambda\big(\delta^{a_1\ldots a_N}_\lambda\big)$ of the socle $\Theta_\lambda(\mathcal{Q}_\lambda)$ of $\mGH^{\gl}_\lambda$. Indeed, this element equals $\theta_\lambda(N!\,p_{\mathbf{r}_{\mathrm{id}}})$, and $\delta^{a_1\ldots a_N}_\lambda=\sum_\sigma(-1)^\sigma p_{\mathbf{r}_\sigma}$ with $\Theta_\lambda(p_{\mathbf{r}_\sigma})=(-1)^\sigma\Theta_\lambda(p_{\mathbf{r}_{\mathrm{id}}})$ by the proof above.

\begin{corollary}
\label{cor::sWH::GH}
Let $\lambda\vdash N$ be a Young diagram.
\begin{enumerate}
\item The $\gl_V$-module $\mGH^{\gl}_\lambda$ is a quotient of $\sWH_\lambda\simeq V^{\otimes N}$. In particular, the specialization of $\mathrm{ch}_q(\mGH^{\gl}_\lambda)$ at $q_0=q_1=1$ is bounded by $h_1^N$ coefficientwise in the Schur basis, and the weight subspace of weight $(1^N,0^{m-N})$ of $\mGH^{\gl}_\lambda$ has dimension at most $N!$.
\item The surjection $\bar{\theta}_\lambda$ is an isomorphism if and only if the weight subspace of weight $(1^N,0^{m-N})$ of $\mGH^{\gl}_\lambda$ has dimension $N!$, i.e., if and only if the $N!$-conjecture holds for $\lambda$ (see Corollary~\ref{cor::N!::gl}).
\item By Theorem~\ref{conj::n!}, $\bar\theta_\lambda\colon\sWH_\lambda\to\mGH^{\gl}_\lambda$ is an isomorphism for every choice of pairwise distinct parameters. In particular, $\sWH_\lambda$ does not depend on the parameters, and $\mathrm{ch}_q(\sWH_\lambda)=\widetilde{H}_\lambda(X;q_0,q_1)$, where $q_0$ and $q_1$ count the $x$- and $y$-degrees respectively.
\end{enumerate}
\end{corollary}
\begin{proof}
Item~(i) follows from the surjection $\bar{\theta}_\lambda$ and from~\eqref{eq::sWH::glV}; the weight subspace of weight $(1^N,0^{m-N})$ of $V^{\otimes N}$ has dimension $N!$. A surjection of finite-dimensional modules is an isomorphism if and only if the dimensions agree. Since $\bar{\theta}_\lambda$ is surjective and the functor $\WS$ is an equivalence of categories (Theorem~\ref{thm::SW::currents}), it is an isomorphism if and only if it induces an isomorphism of the weight subspaces of weight $(1^N,0^{m-N})$, i.e., if and only if the dimension of this subspace of $\mGH^{\gl}_\lambda$ is $N!$; this proves item~(ii). Item~(iii) follows from item~(ii), Theorem~\ref{conj::n!} and Corollary~\ref{cor::N!::gl}.
\end{proof}

\begin{remark}
\label{rem::N!:reformulation}
Corollary~\ref{cor::sWH::GH}(iii) relies on Haiman's theorem. An independent proof of the $N!$-conjecture would follow from Lemma~\ref{lm::socle::head::iso} if one shows that $\sWH_\lambda$ is cocyclic with the socle $\mtop^{x,y}(\sWH_\lambda)\simeq\Lambda^NV$. Indeed, $\bar{\theta}_\lambda$ is nonzero on this component by Proposition~\ref{prp::fusion->GH}, and $\mGH^{\gl}_\lambda$ is cyclic with a simple head.
\end{remark}

\subsection{Associativity of the $y$-fusion product}
\label{sec::sWH::associativity}

The modules $\sWH_\lambda$ are finite-dimensional graded cyclic $\gl_V[x][y]$-modules with respect to the $y$-grading, and their cyclic vectors $\widehat{v}$ have $y$-degree $0$. Hence they can be fused with respect to the variable $y$ in the sense of Definition~\ref{def::fusion::graded}, applied to the Lie algebra $\fg=\gl_V[x]$.

Let $\lambda^{1},\ldots,\lambda^{k}$ be Young diagrams with $|\lambda^i|=N_i$, and let $\lambda=\lambda^1\cup\ldots\cup\lambda^k$ be their union. Its rows are the rows of all the $\lambda^i$, reordered decreasingly, and $|\lambda|=N=N_1+\ldots+N_k$. Note that
\[
n(\lambda^T)=\sum_{\text{rows } \lambda_j\text{ of }\lambda}\binom{\lambda_j}{2}=\sum_{i=1}^{k}n\big((\lambda^{i})^T\big),
\qquad
\lambda^T=\sum_{i=1}^{k}(\lambda^i)^T
\]
(the latter as $\gl_V$-weights). For pairwise distinct $\delta_1,\ldots,\delta_k\in\bC$ consider the fusion product
\[
\mathsf{P}:=\Big[\sWH_{\lambda^1}\star_y\sWH_{\lambda^2}\star_y\cdots\star_y\sWH_{\lambda^k}\Big]_{\delta_1,\ldots,\delta_k},
\]
with the cyclic vector $\widehat{u}$, the symbol of $u:=\widehat{v}_1\otimes\cdots\otimes\widehat{v}_k$, where $\widehat{v}_i$ is the cyclic vector of $\sWH_{\lambda^i}$. The parameters used in the definition of the modules $\sWH_{\lambda^i}$ are arbitrary.
As for $\sWH_\lambda$, the module $\mathsf{P}$ is bigraded by the $x$-degree and the $y$-degree, and by Remark~\ref{rem::fusion::graded} and~\eqref{eq::sWH::glV},
\begin{equation}
\label{eq::P::glV}
\mathsf{P}\simeq\bigotimes_{i=1}^{k}\sWH_{\lambda^i}\simeq V^{\otimes N}\qquad\text{as $\gl_V$-modules.}
\end{equation}

\begin{theorem}
\label{thm::sWH::associativity}
In the above notation, and for every choice of parameters:
\begin{enumerate}
\item The component $\mtop^x(\mathsf{P})$ of top $x$-degree sits in the $x$-degree $n(\lambda^T)$. It is a $\gl_V[x,y]$-submodule, and the subalgebra $x\,\gl_V[x,y]$ acts on it by zero. As a graded $\gl_V[y]$-module,
\[
\mtop^x(\mathsf{P})\simeq\sW^{y}_{\lambda^T}.
\]
\item The component of $\mathsf{P}$ of bidegree $\mdeg(\Delta_\lambda)=\big(n(\lambda^T),n(\lambda)\big)$ is isomorphic to $\Lambda^NV$. Moreover, let $a_1,\ldots,a_N\geq2$ be pairwise distinct, and let $\sum_i(r_i,t_i)=\mdeg(\Delta_\lambda)$. Then the element
\[
\Big(\prod_{i=1}^NE_{a_i1}x^{r_i}y^{t_i}\Big)\widehat{u}\in\mathsf{P}
\]
is nonzero if and only if $\{(r_1,t_1),\ldots,(r_N,t_N)\}$ is the set of cells of $\lambda$.
\item The projection $\Theta_\lambda\colon S^N(V[x,y])\twoheadrightarrow\mGH^{\gl}_\lambda$ factors through the surjection
\[
\theta_{\mathsf{P}}\colon S^N(V[x,y])\twoheadrightarrow\mathsf{P},\qquad e_1^N\mapsto\widehat{u},
\]
of Proposition~\ref{prp::proj::cov}. The resulting surjection $\bar{\theta}_{\mathsf{P}}\colon\mathsf{P}\twoheadrightarrow\mGH^{\gl}_\lambda$ is bigraded and is an isomorphism in the bidegree $\mdeg(\Delta_\lambda)$.
\item The surjection $\bar{\theta}_{\mathsf{P}}$ is an isomorphism. Consequently, the $y$-fusion product of the modules $\sWH_{\lambda^i}$ is associative and does not depend on the parameters:
\[
\Big[\sWH_{\lambda^1}\star_y\cdots\star_y\sWH_{\lambda^k}\Big]_{\delta_1,\ldots,\delta_k}\ \simeq\ \sWH_{\lambda^1\cup\ldots\cup\lambda^k}\ \simeq\ \mGH^{\gl}_{\lambda^1\cup\ldots\cup\lambda^k},
\]
with the cyclic vectors corresponding to each other. In particular, the component of bidegree $\mdeg(\Delta_\lambda)$ is the top total-degree component of $\mathsf{P}$.
\end{enumerate}
\end{theorem}

\begin{proof}
We follow the proof of Theorem~\ref{thm::top::fusion} and indicate the necessary modifications. Denote by $\sWH_{\lambda^i}(\delta_i)$ the shifted modules (the pullbacks along $y\mapsto y+\delta_i$). Put
\[
T':=\sWH_{\lambda^1}(\delta_1)\otimes\cdots\otimes\sWH_{\lambda^k}(\delta_k),\qquad
\calF'_n:=U(\gl_V[x,y])_{\leq n}\cdot u .
\]

\emph{Preliminaries.}
The shift $y\mapsto y+\delta_i$ does not affect the elements of $y$-degree $0$, i.e., the subalgebra $\gl_V[x]$.
(i) The subalgebra $\fb[x,y]$ acts on $u$ by scalars. Indeed, an element $Xx^ry^t$ with $X\in\fb$ acts on the $i$-th factor as $Xx^r(y+\delta_i)^t$, which is a combination of the elements $Xx^ry^{t'}$. These act on $\widehat{v}_i$ by scalars by Lemma~\ref{Non1Vanishing}: they kill it if $X=E_{ab}$ with $b\geq2$, and $E_{11}x^ry^{t'}\widehat{v}_i=\delta_{r,0}\delta_{t',0}N_i\widehat{v}_i$ in the fusion product $\sWH_{\lambda^i}$.
Hence, exactly as in Proposition~\ref{prp::fusion::abelianGL}, $\calF'_n=U(\fa[x,y])_{\leq n}\cdot u$, and $(E_{ab}y^t)\cdot u=0$ for all $b\geq2$ and $t\geq0$.
(ii) Let $T_\lambda$ be the tensor product~\eqref{eq::tensor::weyl} for the diagram $\lambda$, where the rows are taken in the order $\lambda^1,\ldots,\lambda^k$ (inside each $\lambda^i$ in the decreasing order), and let $v_\lambda$ be its cyclic vector.
The $y$-degree $0$ component of $\sWH_{\lambda^i}$ is the $\gl_V[x]$-module $\calF_0(T_{\lambda^i})=U(\gl_V[x])\cdot v_{\lambda^i}\subset T_{\lambda^i}$, on which $\gl_V[x]$ acts diagonally.
Therefore the embedding $\bigotimes_i\calF_0(T_{\lambda^i})\subset\bigotimes_iT_{\lambda^i}=T_\lambda$ is $\gl_V[x]$-equivariant and maps $u$ to $v_\lambda$. It identifies
\[
\calF'_0=U(\gl_V[x])\cdot u\ \simeq\ U(\gl_V[x])\cdot v_\lambda=\calF_0(T_\lambda)
\]
as $x$-graded $\gl_V[x]$-modules.
Steps~1 and~2 of the proof of Theorem~\ref{thm::top::fusion} only involve computations inside $\calF_0(T_\lambda)$, and they do not use the order of the rows. Hence all the statements about $\calF_0(T_\lambda)$ proved there hold for $\calF'_0$.

\emph{Step 1: the top $x$-component.}
By Theorem~\ref{thm::top::fusion}(ii), the top $x$-degree of $\sWH_{\lambda^i}$ is $n((\lambda^i)^T)$, and $\mtop^x(\sWH_{\lambda^i})\simeq\sW^y_{(\lambda^i)^T}$ is a graded $\gl_V[y]$-module on which $x\,\gl_V[x,y]$ acts by zero.
Hence the top $x$-degree of $T'$ is $\sum_in((\lambda^i)^T)=n(\lambda^T)$, and
\[
Q':=\mtop^x(T')=\bigotimes_{i=1}^k\mtop^x\big(\sWH_{\lambda^i}\big)(\delta_i)\simeq\bigotimes_{i=1}^k\sW^y_{(\lambda^i)^T}(\delta_i)
\]
as $\gl_V[y]$-modules, while $x\,\gl_V[x,y]$ acts on $Q'$ by zero. Since the filtration $\calF'_\bullet$ is $x$-graded, $\mtop^x(\mathsf{P})=\bigoplus_n(\calF'_n\cap Q')/(\calF'_{n-1}\cap Q')$.
Let $\widehat{\beta}_i\in\mtop^x(\sWH_{\lambda^i})$ be the highest weight vector of $\sW^y_{(\lambda^i)^T}$; it has $y$-degree $0$. We put $\beta'':=\widehat{\beta}_1\otimes\cdots\otimes\widehat{\beta}_k$ and consider the fusion filtration
\[
G'_n:=U(\gl_V[y])_{\leq n}\cdot\beta''\subset Q' .
\]
By Definition~\ref{def::fusion::graded}, $\bigoplus_nG'_n/G'_{n-1}$ is the fusion product
\[
\big[\sW^y_{(\lambda^1)^T}\star_y\cdots\star_y\sW^y_{(\lambda^k)^T}\big]_{\delta_1,\ldots,\delta_k}.
\] By Corollary~\ref{cor::fusion::Weyl::upper}, it is a quotient of $\sW^y_{\lambda^T}$. By Theorem~\ref{thm::Weyl::fusion}, both have dimension $\prod_{\text{rows }\lambda_j\text{ of }\lambda}\dim\Lambda^{\lambda_j}V$. Hence
\[
\bigoplus_nG'_n/G'_{n-1}\simeq\sW^y_{\lambda^T},
\]
with $\widehat{\beta''}$ corresponding to the highest weight vector.

\emph{Step 2: comparison of filtrations.}
The $y$-degree $0$ component of $\mtop^x(\sWH_{\lambda^i})$ is the Cartan component $U(\gl_V)\beta^{(i)}\subset\calF_0(T_{\lambda^i})$ (Step~2 of the proof of Theorem~\ref{thm::top::fusion} with $n=0$). Hence, under the identification~(ii), $\beta''$ corresponds to a highest weight vector $\beta$ of the Cartan component $U(\gl_V)\beta\subset\calF_0(T_\lambda)\cap\mtop^x(T_\lambda)$. In particular, $\beta''\in\calF'_0$, and therefore $G'_n\subset\calF'_n\cap Q'$.
Conversely, as in Step~2 of the proof of Theorem~\ref{thm::top::fusion}, it suffices to show that $(\calF'_n\cap Q')_\mu\subset G'_n$ for $\mu=\varepsilon_{a_1}+\ldots+\varepsilon_{a_N}$ with auxiliary indices $a_i,b_i,c_s$ chosen as there. By~(i), this space is spanned by the elements $m=\big(\prod_iE_{a_i1}x^{r_i}y^{t_i}\big)u$ with $\sum_ir_i=n(\lambda^T)$ and $\sum_it_i\leq n$.
The identity~\eqref{eq::gen::mon}, $m=\big(\prod_iE_{a_ib_i}y^{t_i}\big)w$ with $w=\big(\prod_iE_{b_i1}x^{r_i}\big)u$, holds in $T'$ with the same proof, because it only uses commutation relations in $\gl_V[x,y]$ and the vanishing $(E_{a_ib_i}y^t)\cdot u=0$ from~(i).
The vector $w$ lies in $\calF'_0\cap Q'$. By~(ii) and Step~2 of the proof of Theorem~\ref{thm::top::fusion}, we have $w\in U(\gl_V)\beta=U(\gl_V)\beta''=G'_0$. Hence $m$ belongs to $U(\gl_V[y])_{\leq n}G'_0=G'_n$. Therefore $\calF'_n\cap Q'=G'_n$ for all $n$, and item~(i) follows from Step~1.

\emph{Step 3: the bidegree $\mdeg(\Delta_\lambda)$.}
By item~(i), the component of bidegree $(n(\lambda^T),n(\lambda))$ of the module $\mathsf{P}$ is the component of $y$-degree $n(\lambda)$ of $\sW^y_{\lambda^T}$. This is the top $y$-degree component of $\sW^y_{\lambda^T}$ by Proposition~\ref{lem:CocyclicWeylProperties}\ref{item::socle:1}, and it is isomorphic to $\Lambda^NV$ by Proposition~\ref{prp::Weyl::top}.
The criterion for the monomials is proved word by word as Step~4 of the proof of Theorem~\ref{thm::top::fusion}. We use the identity~\eqref{eq::gen::mon} in $T'$, the description of $w$ from~(ii), and the identification of Step~1; under the latter, the vector $\widehat{\beta''}$ is the highest weight vector of $\sW^y_{\lambda^T}$. We also use Proposition~\ref{lem:CocyclicWeylProperties}\ref{item::socle:2} for $\sW^y_{\lambda^T}$. The permutation matrices fixing $e_1$ fix $\widehat{u}$, since they fix every $\widehat{v}_i$. This proves item~(ii).

\emph{Step 4: the surjection onto the Garsia--Haiman module.}
The vector $\widehat{u}$ has weight $N\varepsilon_1$ and generates $\mathsf{P}$. By Proposition~\ref{prp::proj::cov}, there is a surjection $\theta_{\mathsf{P}}\colon S^N(V[x,y])\twoheadrightarrow\mathsf{P}$ with $e_1^N\mapsto\widehat{u}$.
The proof of Proposition~\ref{prp::fusion->GH} uses only the following properties of the target of $\theta_\lambda$: its cyclic vector is fixed by the permutation matrices fixing $e_1$; its component of bidegree $\mdeg(\Delta_\lambda)$ is isomorphic to $\Lambda^NV$; and the criterion for the monomials $\big(\prod_iE_{a_i1}x^{r_i}y^{t_i}\big)\widehat{v}$ holds. By item~(ii), these properties hold for $\mathsf{P}$ and $\widehat{u}$. Hence the same argument shows that $\ker\theta_{\mathsf{P}}\subset\euJ^{\gl}_\lambda$, which gives the surjection $\bar{\theta}_{\mathsf{P}}$. It is an isomorphism in the bidegree $\mdeg(\Delta_\lambda)$ because both components are isomorphic to $\Lambda^NV$ (Proposition~\ref{prp::GH::gl}(iv)). This proves item~(iii).

\emph{Step 5: the dimension count.}
By~\eqref{eq::P::glV}, $\mathsf{P}\simeq V^{\otimes N}$ as a $\gl_V$-module. By Corollary~\ref{cor::N!::gl} (which relies on Haiman's Theorem~\ref{conj::n!}), $\mGH^{\gl}_\lambda\simeq V^{\otimes N}$ as a $\gl_V$-module as well. Hence the surjection $\bar{\theta}_{\mathsf{P}}$ is a surjection between spaces of the same dimension, i.e., an isomorphism.
By Corollary~\ref{cor::sWH::GH}(iii), $\bar{\theta}_\lambda\colon\sWH_\lambda\to\mGH^{\gl}_\lambda$ is an isomorphism as well. Both isomorphisms map the cyclic vectors $\widehat{u}$ and $\widehat{v}$ to the image of $e_1^N$, so they are compatible with the cyclic vectors.
Finally, by Proposition~\ref{prp::GH::gl}(iv), all bidegrees of $\mGH^{\gl}_\lambda$ are bounded componentwise by $\mdeg(\Delta_\lambda)$. Hence the component of bidegree $\mdeg(\Delta_\lambda)$ is the top total-degree component. This proves item~(iv).
\end{proof}

\begin{remark}
\begin{enumerate}
\item If every $\lambda^i=(\lambda_i)$ is a single row, then $\sWH_{\lambda^i}=\sW^x_{\lambda_i\varepsilon_1}$ is concentrated in $y$-degree $0$, and $\mathsf{P}$ is the module $\sWH_\lambda$ of Definition~\ref{def::sWH} (up to the order of the rows). Theorem~\ref{thm::sWH::associativity}(iv) shows, in particular, that $\sWH_\lambda$ does not depend on the order in which the rows are fused.
\item Items~(i)--(iii) of Theorem~\ref{thm::sWH::associativity} do not use Haiman's theorem. Only the last step, the equality of dimensions, uses it. Without Haiman's theorem one obtains a surjection $\mathsf{P}\twoheadrightarrow\mGH^{\gl}_\lambda$, and hence the bound $\dim(\mGH^{\gl}_\lambda)_{(1^N,0^{m-N})}\leq N!$, for every decomposition $\lambda=\lambda^1\cup\ldots\cup\lambda^k$.
\item Theorem~\ref{thm::sWH::associativity} verifies Conjecture~\ref{conj::fusion::associativity} for the iterated $y$-fusion products of the modules $\sW^x_{n\varepsilon_1}$.
\end{enumerate}
\end{remark}

\section{Cofusion products and sign duality}
\label{sec::cofusion}

In this section we develop the construction dual to the fusion product, the \emph{cofusion product}, and we show that the sign duality of \S\ref{sec::GH::sign} exchanges the two.
Throughout, $\dim V\geq N$, we write $\mathfrak{G}:=\gl_V[x,y]$, and for $z\in\{x,y\}$ we let $U(\mathfrak{G})^z_{\leq s}$ be the span of the products $(X_1f_1)\cdots(X_kf_k)$ with monomials $f_i$ such that $\sum_i\deg_zf_i\leq s$. We also put $U(\mathfrak{G})^z_{<s}:=U(\mathfrak{G})^z_{\leq s-1}$, so that $U(\mathfrak{G})^z_{<0}=0$.
For a $\mathfrak{G}$-module $M$ and $c\in\bC$ we denote by $M(c)$ its twist by the automorphism $z\mapsto z+c$. A $\gl_V[x]$-module is considered as a $\mathfrak{G}$-module on which $y$ acts by zero.
Recall that all modules are finite-dimensional and belong to the category $\RepGLz{V}{N}$ for $\bar{z}=(x,y)$.

\subsection{Cocyclic vectors and the cofusion filtration}
\label{sec::cofusion::def}

\begin{definition}
\label{def::cocyclic}
Let $M$ be a finite-dimensional $\mathfrak{G}$-module. Let $L\subset M$ be an irreducible $\gl_V$-submodule that occurs in $M$ with multiplicity one, and let $w\in L$ be its highest weight vector. Denote by $\pi_w\colon M\to\bC$ the composition of the $\gl_V$-isotypic projection $M\to L$ with the coordinate of $w$ in a weight basis of $L$. The vector $w$ is called \emph{cocyclic} if every nonzero $\mathfrak{G}$-submodule of $M$ contains $w$.
\end{definition}

A $\mathfrak{G}$-submodule $K\subset M$ is $\gl_V$-stable, so its $L$-isotypic part is either $0$ or $L$. Hence $w\in K$ if and only if $\pi_w(K)\neq0$.
In terms of the contragredient dual $M^\dagger$ of \S\ref{sec::GH::sign}, we have the following criterion.

\begin{lemma}
\label{lem::cocyclic::dual}
In the situation of Definition~\ref{def::cocyclic}, the functional $\pi_w\in M^\dagger$ is a highest weight vector of $M^\dagger$, and $w$ is cocyclic if and only if $\pi_w$ is a cyclic vector of $M^\dagger$.
\end{lemma}
\begin{proof}
For $u\in U(\mathfrak{G})$ the annihilator in $M$ of $U(\mathfrak{G})\cdot\pi_w\subset M^\dagger$ is $\{m\colon\pi_w(u^Tm)=0\ \forall u\}$. As $u^T$ runs over all of $U(\mathfrak{G})$, this is the largest $\mathfrak{G}$-submodule of $M$ contained in $\ker\pi_w$. It vanishes if and only if every nonzero submodule $K$ satisfies $\pi_w(K)\neq0$, i.e., contains $w$. This proves the second claim.
For the first claim, let $i<j$ and $m\in M$. Then $(E_{ij}\pi_w)(m)=\pi_w(E_{ji}m)$. The functional $\pi_w$ factors through the projection onto the isotypic component of $L$, which commutes with $E_{ji}$. If $E_{ji}m_L$ had a nonzero coordinate along $w$, then the $L$-isotypic part $m_L$ of $m$ would have a weight strictly higher than the highest weight of $L$, which is impossible. Hence $E_{ij}\pi_w=0$.
\end{proof}

\begin{example}
\label{ex::cocyclic}
Let $M_0,\ldots,M_p$ be finite-dimensional graded $\mathfrak{G}$-modules with simple socles $S_i$, and let $c_0,\ldots,c_p$ be pairwise distinct. By the Chinese remainder theorem, $\mathfrak{G}$ acts on $M_{\mathbf{c}}:=\bigotimes_iM_i(c_i)$ through a direct sum of Lie algebras, one for each point $c_i$. Hence the socle of $M_{\mathbf{c}}$ is the simple module $\bigotimes_iS_i(c_i)$. Consequently, the highest weight vector of any irreducible $\gl_V$-submodule of $\bigotimes_iS_i(c_i)$ that occurs in $M_{\mathbf{c}}$ with multiplicity one is cocyclic. 
\end{example}

\begin{definition}
\label{def::cofusion}
Let $w\in M$ be a cocyclic vector and $z\in\{x,y\}$. The \emph{cofusion filtration} of $M$ with respect to $z$ and $w$ is the decreasing filtration
\[
G^t_z(M;w):=\big\{u\in M\colon\ \pi_w\big(U(\mathfrak{G})^z_{<t}\,u\big)=0\big\},\qquad t\geq0 .
\]
In words, $u\in G^t_z$ if $u$ cannot reach $w$ by elements of $U(\mathfrak{G})$ of $z$-degree less than $t$.
If $M=M_{\mathbf{c}}=\bigotimes_iM_i(c_i)$ is a tensor product of twisted modules at pairwise distinct points, the graded space
\[
M_0\circledast_zM_1\circledast_z\cdots\circledast_zM_p:=\mathrm{gr}_GM_{\mathbf{c}}=\bigoplus_{t\geq0}G^t/G^{t+1}
\]
(with $G^t/G^{t+1}$ placed in degree $-t$) is called the \emph{cofusion product}.
\end{definition}

The filtration $G=G_z(M;w)$ has the following properties.
\begin{enumerate}
\item $G^0=M$, and $G$ is decreasing.
\item Each $G^t$ is stable under $U(\mathfrak{G})^z_{\leq0}$; in particular it is a $\gl_V$-submodule.
\item If $f$ is a monomial of $z$-degree $k$, then $(Xf)\,G^t\subset G^{t-k}$ for $t\geq k$ (we put $G^s:=M$ for $s\leq0$).
\item $\bigcap_tG^t$ is the largest submodule of $M$ contained in $\ker\pi_w$, so it is zero by cocyclicity. Hence $G^t=0$ for $t\gg0$.
\end{enumerate}
By (iii), $\mathrm{gr}_GM$ is a graded $\mathfrak{G}$-module, in which an element of $z$-degree $k$ raises the degree by $k$. The class of $w$ lies in degree $0$, and $\mathrm{gr}_GM$ is cocyclic with this cocyclic vector; this follows from Lemma~\ref{lem::cofusion::dual} below, since $\mathrm{gr}_GM$ is dual to the fusion product $\mathrm{gr}_FM^\dagger$, which is cyclic. In particular, the socle of $\mathrm{gr}_GM$ is the irreducible $\gl_V$-module generated by $w$.

\begin{lemma}[Cofusion is dual to fusion]
\label{lem::cofusion::dual}
Let $w\in M$ be cocyclic and $\xi:=\pi_w\in M^\dagger$; by Lemma~\ref{lem::cocyclic::dual} it is a cyclic vector of $M^\dagger$. Let $F^\dagger_s:=F^z_s(M^\dagger;\xi)=U(\mathfrak{G})^z_{\leq s}\,\xi$ be the fusion filtration of $M^\dagger$. Then
\[
G^t_z(M;w)=\big(F^\dagger_{t-1}\big)^{\perp}\qquad\text{for all } t\geq0.
\]
Consequently, the pairing between $M$ and $M^\dagger$ induces perfect pairings $G^t/G^{t+1}\otimes F^\dagger_t/F^\dagger_{t-1}\to\bC$, and there is an isomorphism of graded $\mathfrak{G}$-modules
\[
\mathrm{gr}_{G_z}M\ \simeq\ \big(\mathrm{gr}_{F^z}M^\dagger\big)^{\dagger},
\]
under which the class of $w$ corresponds to the functional dual to the class of $\xi$ in degree $0$.
\end{lemma}
\begin{proof}
By definition of the action on $M^\dagger$ we have $\pi_w(Xu)=(X^T\xi)(u)$ for $X\in U(\mathfrak{G})$ and $u\in M$. The transposition preserves $U(\mathfrak{G})^z_{<t}$. Hence $u\in G^t$ if and only if $u$ is annihilated by $U(\mathfrak{G})^z_{<t}\xi=F^\dagger_{t-1}$. The pairing therefore induces perfect pairings $G^t/G^{t+1}\otimes F^\dagger_t/F^\dagger_{t-1}\to\bC$, which are compatible with the actions of $\mathfrak{G}$ in the sense that $\langle(Xf)a,b\rangle=\langle a,(X^Tf)b\rangle$.
\end{proof}

In particular, whenever the fusion product of the duals does not depend on the points $c_i$, neither does the cofusion product.

\subsection{Sign duality exchanges fusion and cofusion}
\label{sec::cofusion::sign}

For a $\mathfrak{G}$-module $M$ in $\RepGLz{V}{N}$ we write
\[
M^{\circ}:=\bomega(M)^{\dagger}
\]
for its contravariant sign dual (see Lemma~\ref{lem::sign::properties}(v)). It is an exact contravariant involution, and by Lemma~\ref{lem::sign::properties}(iii),(iv),(v) it satisfies
\[
\big(M(c)\big)^{\circ}\simeq M^{\circ}(c),\qquad \big(M\otimes M'\big)^{\circ}\simeq M^{\circ}\otimes M'^{\circ}.
\]

Fix one of the variables $z\in\{x,y\}$, and let $U(\gl_V[\bar{z}])^{z}_{\leq s}$ be the span of the products $(X_1f_1)\cdots(X_kf_k)$ with monomials $f_i$ such that $\sum_i\deg_zf_i\leq s$. For a cyclic $\gl_V[\bar{z}]$-module $M$ with a cyclic vector $v$ we denote
\[
F_s^{z}(M;v):=U(\gl_V[\bar{z}])^{z}_{\leq s}\cdot v,\qquad F_{-1}^z(M;v)=0,
\]
the \emph{fusion filtration} of $M$ with respect to $z$ (cf.~\eqref{eq::fusion::filtration}). Each $F^z_s$ is a $\gl_V$-submodule, since $U(\gl_V)\subset U(\gl_V[\bar{z}])^z_{\leq0}$.

\begin{lemma}
	\label{lem::sign::fusion}
	Let $M\in\RepGLz{V}{N}$ be a cyclic $\gl_V[\bar{z}]$-module with a cyclic vector $v$. Assume that $L:=U(\gl_V)v$ is an irreducible $\gl_V$-module $V_\nu$ that occurs in $M$ with multiplicity one. Let $L':=\bomega(L)\simeq V_{\nu^T}\subset\bomega(M)$ and let $u\in L'$ be a highest weight vector. Then:
	\begin{enumerate}
		\item the $\gl_V[\bar{z}]$-module $\bomega(M)$ is cyclic with the cyclic vector $u$;
		\item $F_s^z\big(\bomega(M);u\big)=\bomega\big(F_s^z(M;v)\big)$ for all $s$, where on the right-hand side $\bomega$ is the classical sign duality functor applied to the $\gl_V$-submodule $F_s^z(M;v)\subset M$;
		\item consequently, the associated graded modules satisfy $\mathrm{gr}_{F^z}\big(\bomega(M)\big)\simeq\bomega\big(\mathrm{gr}_{F^z}M\big)$ as graded $\gl_V[\bar{z}]$-modules, and the cyclic vectors correspond to each other.
	\end{enumerate}
\end{lemma}
\begin{proof}
	Put $R:=\WS(M)$, $\pi$ for the projector onto the weight space of weight $(1^N,0^{m-N})$, and $\mathcal{S}:=\mathcal{S}(V[\bar{z}],N)$. Let $\mathcal{S}_{\leq s}\subset\mathcal{S}$ be the image of $U(\gl_V[\bar{z}])^z_{\leq s}$; this is the part of $\mathcal{S}=\big((\mathrm{End}(V)\otimes\bC[\bar{z}])^{\otimes N}\big)^{\bS_N}$ of $z$-degree at most $s$, by the proof of Lemma~\ref{lem::current::Schur}. Under the isomorphism $\pi\mathcal{S}\pi\simeq\mathcal{A}_N$ of Lemma~\ref{lem::weight::idempotent}(iii) we have $\pi\mathcal{S}_{\leq s}\pi=(\mathcal{A}_N)_{\leq s}$, the polynomials of $z$-degree at most $s$ (with $\bS_N$).
	
	(i) The module $M$ is a quotient of the projective cover $P_\nu$ of $V_\nu$, because the homomorphism $P_\nu\to M$ extending $L\hookrightarrow M$ is surjective as its image contains the cyclic vector $v$ (Proposition~\ref{prop::characters::free}). Applying the exact functor $\bomega$ and using $\bomega(P_\nu)\simeq P_{\nu^T}$, we see that $\bomega(M)$ is a quotient of $P_{\nu^T}$, which is generated by its degree zero component $V_{\nu^T}$. Hence $\bomega(M)$ is generated by the image of this component, which is $L'$, and so by $u$.
	
	(ii) We claim that $\pi F^z_s(M;v)=(\mathcal{A}_N)_{\leq s}\cdot\pi L$. The inclusion $\supset$ is clear. For the converse, since the idempotent $\pi$ is full (Lemma~\ref{lem::Schur::full}), we can write $1=\sum_kx_k\pi y_k$ with $x_k,y_k\in\mathcal{S}(V,N)$. Then, for $a\in\mathcal{S}_{\leq s}$,
	\[
	\pi a v=\sum_k(\pi ax_k\pi)(\pi y_kv),\qquad \pi ax_k\pi\in\pi\mathcal{S}_{\leq s}\pi=(\mathcal{A}_N)_{\leq s},\quad \pi y_kv\in\pi L,
	\]
	because $\mathcal{S}_{\leq s}\mathcal{S}(V,N)\subset\mathcal{S}_{\leq s}$ and $y_kv\in L$.
	Since $\WS(\bomega(M))=\mathrm{sgn}\otimes R$ and $\WS(L')=\mathrm{sgn}\otimes\pi L$, the same computation for $\bomega(M)$ gives $\pi F_s^z(\bomega(M);u)=(\mathcal{A}_N)_{\leq s}\cdot(\mathrm{sgn}\otimes\pi L)=\mathrm{sgn}\otimes\pi F^z_s(M;v)$, because the polynomials act on $\mathrm{sgn}\otimes R$ as on $R$. The $\gl_V$-submodules of $M$ and $\bomega(M)$ are in bijection with the $\bS_N$-stable subspaces of $R$ and of $\mathrm{sgn}\otimes R$, and under this bijection the classical functor $\bomega$ corresponds to the tensoring with $\mathrm{sgn}$. This proves~(ii).
	
	(iii) The graded module $\mathrm{gr}_{F^z}M$ is the Schur--Weyl dual of the graded $\mathcal{A}_N$-module $\bigoplus_s(\mathcal{A}_N)_{\leq s}\pi L\big/(\mathcal{A}_N)_{\leq s-1}\pi L$, and similarly for $\bomega(M)$. These two $\mathcal{A}_N$-modules differ by the tensor product with $\mathrm{sgn}$.
\end{proof}

\begin{proposition}
\label{prp::sign::cofusion}
Let $M\in\RepGLz{V}{N}$ be a cyclic $\mathfrak{G}$-module with a cyclic vector $v$, such that $L:=U(\gl_V)v$ is an irreducible $\gl_V$-module $V_\nu$ occurring in $M$ with multiplicity one. Then:
\begin{enumerate}
\item the module $M^\circ$ contains the irreducible $\gl_V$-module $V_{\nu^T}$ with multiplicity one, and its highest weight vector $w$ is cocyclic;
\item for $z\in\{x,y\}$ there is an isomorphism of graded $\mathfrak{G}$-modules
\[
\mathrm{gr}_{G_z}\big(M^{\circ};w\big)\ \simeq\ \big(\mathrm{gr}_{F^z}(M;v)\big)^{\circ}.
\]
\end{enumerate}
In particular, if $M=M_0(c_0)\otimes\cdots\otimes M_p(c_p)$ is a tensor product of twisted modules with the cyclic vector $v=v_0\otimes\cdots\otimes v_p$, then the contravariant sign dual of the fusion product $M_0\star_zM_1\star_z\cdots\star_zM_p$ is the cofusion product $M_0^{\circ}\circledast_zM_1^\circ\circledast_z\cdots\circledast_zM_p^\circ$ with respect to $w$.
\end{proposition}
\begin{proof}
By Lemma~\ref{lem::sign::fusion}(i), the module $\bomega(M)$ is cyclic with a cyclic vector $u$, a highest weight vector of the unique copy of $V_{\nu^T}$. Put $N:=M^\circ=\bomega(M)^\dagger$. Then $N\simeq\bomega(M)$ as $\gl_V$-modules, which proves the first part of~(i). Let $w\in N$ be a highest weight vector of $V_{\nu^T}$. Under the canonical isomorphism $N^\dagger\simeq\bomega(M)$ the functional $\pi_w$ goes to a highest weight vector of $V_{\nu^T}$ (Lemma~\ref{lem::cocyclic::dual}), i.e., to a multiple of $u$. As $u$ is a cyclic vector, $\pi_w$ is a cyclic vector of $N^\dagger$, and $w$ is cocyclic by Lemma~\ref{lem::cocyclic::dual}.
By Lemma~\ref{lem::cofusion::dual},
\[
\mathrm{gr}_{G_z}(N;w)\simeq\big(\mathrm{gr}_{F^z}(N^\dagger;\pi_w)\big)^\dagger=\big(\mathrm{gr}_{F^z}(\bomega(M);u)\big)^\dagger\simeq\big(\bomega(\mathrm{gr}_{F^z}(M;v))\big)^\dagger,
\]
where the last isomorphism is Lemma~\ref{lem::sign::fusion}(iii). This proves~(ii). The last claim follows from~(ii) and from the compatibility of $\circ$ with twists and tensor products.
\end{proof}

\subsection{Sign duals of local Weyl modules}
\label{sec::cofusion::Weyl}

Let $\lambda\vdash N$ be a Young diagram and let $\lambda^T=(\lambda^T_0\geq\lambda^T_1\geq\ldots)$ be the transposed diagram; thus $\lambda=\sum_s\omega_{\lambda^T_s}$ as a weight.

\begin{theorem}
\label{thm::sign::Weyl}
\begin{enumerate}
\item In $S^{\lambda^T_0}V\otimes S^{\lambda^T_1}V\otimes\cdots$ the module $V_{\lambda^T}$ occurs with multiplicity one, and $\lambda^T$ is the minimal diagram occurring in this tensor product in the dominance order. Its highest weight vector $w_{\lambda^T}$ is cocyclic in $S^{\lambda^T_0}V(c_0)\otimes S^{\lambda^T_1}V(c_1)\otimes\cdots$ for pairwise distinct $c_s$.
\item There is an isomorphism of graded $\gl_V[x]$-modules
\[
\bomega\big(\sW^{x}_{\lambda}\big)^{\dagger}\ \simeq\ S^{\lambda^T_0}V\circledast_xS^{\lambda^T_1}V\circledast_x\cdots,
\]
where the cofusion product is taken with respect to $w_{\lambda^T}$. In particular, the cofusion product does not depend on the points $c_s$, it has the simple socle $V_{\lambda^T}$, and its $\gl_V$-character is that of $\bigotimes_sS^{\lambda^T_s}V$.
\item For the one-row diagram $\lambda=(h)$ we get
\[
\bomega\big(\sW^{x}_{h\varepsilon_1}\big)^\dagger\simeq\underbrace{V\circledast_x\cdots\circledast_xV}_{h}\ \simeq\ \sW^x_{h\varepsilon_1}\big\langle-\tbinom{h}{2}\big\rangle,
\]
where the cofusion product is taken with respect to the highest weight vector $v_{\omega_h}$ of $\Lambda^hV$.
\end{enumerate}
\end{theorem}
\begin{proof}
(i) By the Pieri rule, the multiplicity of $V_\mu$ in $\bigotimes_sS^{\lambda^T_s}V$ is the Kostka number $K_{\mu,\lambda^T}$. It vanishes unless $\mu\geq\lambda^T$ in the dominance order, and $K_{\lambda^T\lambda^T}=1$. The tensor product of the simple evaluation modules at pairwise distinct points is simple, so Example~\ref{ex::cocyclic} applies.

(ii) Let $M:=\Lambda^{\lambda^T_0}V(c_0)\otimes\Lambda^{\lambda^T_1}V(c_1)\otimes\cdots$ with the cyclic vector $v$, the tensor product of the highest weight vectors; it is cyclic by Corollary~\ref{cor::ev::cyclic}. The submodule $U(\gl_V)v$ is the Cartan component $V_\lambda$, which occurs in $M$ with multiplicity one. By Theorem~\ref{thm::Weyl::fusion}, $\mathrm{gr}_{F^x}M\simeq\sW^x_\lambda$. Since $\bomega(\Lambda^kV)=S^kV$ and $S^kV$ is isomorphic to its contragredient dual, $M^\circ\simeq\bigotimes_sS^{\lambda^T_s}V(c_s)$. Now apply Proposition~\ref{prp::sign::cofusion}.
The cofusion product is cocyclic with the cocyclic vector $w_{\lambda^T}$ (see the properties of the cofusion filtration in \S\ref{sec::cofusion::def}), hence its socle is the simple module $V_{\lambda^T}$. Its $\gl_V$-character is that of $\bigotimes_sS^{\lambda^T_s}V$, because passing to an associated graded space does not change the $\gl_V$-character.

(iii) The first isomorphism is~(ii) for $\lambda=(h)$, since $\lambda^T=(1^h)$. For the second one, recall that $\mGH^{\gl}_{(h)}\simeq\sW^x_{h\varepsilon_1}$ (Example~\ref{ex::sWH::row::column} and Corollary~\ref{cor::sWH::GH}). The statement is Corollary~\ref{cor::GH::self::sign}(i) for $\lambda=(h)$, for which $\mdeg(\Delta_{(h)})=(\binom{h}{2},0)$.
\end{proof}

\subsection{Coincidence of the fusion and cofusion filtrations}
\label{sec::cofusion::coincide}

In this subsection we use Haiman's theorem (Theorem~\ref{conj::n!}) through its consequence $\sWH_\lambda\simeq\mGH^{\gl}_\lambda\simeq V^{\otimes N}$ as $\gl_V$-modules (Corollary~\ref{cor::sWH::GH}).
We consider fusion and cofusion with respect to $y$, and we omit $y$ from the notation.

Let $\lambda^1,\ldots,\lambda^k$ be Young diagrams with $|\lambda^i|=N_i$, $N=\sum_iN_i$, and let $\lambda=\lambda^1\cup\ldots\cup\lambda^k$ be the diagram whose rows are all the rows of the $\lambda^i$. Put $d:=n(\lambda)$, the $y$-degree of $\Delta_\lambda$. For pairwise distinct numbers $c_1,\ldots,c_k$ consider
\[
M:=\mGH^{\gl}_{\lambda^1}(c_1)\otimes\cdots\otimes\mGH^{\gl}_{\lambda^k}(c_k),
\]
with the cyclic vector $v$, the tensor product of the images of $e_1^{N_i}$ (cyclic by Lemma~\ref{lem::fusion::graded::cyclic}). Let $F_s:=F^y_s(M;v)$ be the fusion filtration, and let $w\in M$ be a highest weight vector of $\Lambda^NV$.

\begin{theorem}
\label{thm::fusion::cofusion}
\begin{enumerate}
\item The module $M$ is isomorphic to $V^{\otimes N}$ as a $\gl_V$-module, the vector $w$ is cocyclic, and $M\simeq M^{\circ}$ with $w$ corresponding to the cocyclic vector of Proposition~\ref{prp::sign::cofusion}.
\item The fusion filtration and the cofusion filtration of $M$ coincide up to reindexing:
\[
F_s=G^{\,d-s}_y(M;w)\qquad\text{for all }s .
\]
\item There are isomorphisms of graded modules
\[
\mGH^{\gl}_{\lambda^1}\star_y\cdots\star_y\mGH^{\gl}_{\lambda^k}\ \simeq\ \mGH^{\gl}_{\lambda}\ \simeq\ \mGH^{\gl}_{\lambda^1}\circledast_y\cdots\circledast_y\mGH^{\gl}_{\lambda^k}\ \big\langle\mdeg(\Delta_\lambda)\big\rangle ,
\]
where the cofusion product is taken with respect to $w$.
\end{enumerate}
\end{theorem}
\begin{proof}
Put $H_i:=\mGH^{\gl}_{\lambda^i}$. By Corollary~\ref{cor::N!::gl}, $H_i\simeq V^{\otimes N_i}$ as $\gl_V$-modules, hence $M\simeq V^{\otimes N}$. By Proposition~\ref{prp::GH::gl}(iv) the socle of $H_i$ is $\Lambda^{N_i}V$. By Example~\ref{ex::cocyclic}, the socle of $M$ is the simple module $\bigotimes_i\Lambda^{N_i}V(c_i)$, which contains $\Lambda^NV$; the latter occurs in $V^{\otimes N}$ with multiplicity one. Hence $w$ is cocyclic.
By Corollary~\ref{cor::GH::self::sign}(i), $H_i^\circ\simeq H_i$ up to a shift of the grading. Therefore, by Lemma~\ref{lem::sign::properties}(iii),(iv),(v), $M^\circ\simeq M$ as $\mathfrak{G}$-modules. This isomorphism maps the unique copy of $\Lambda^NV$ to itself, hence the highest weight vectors correspond to each other. This proves~(i).

Put $G^t:=G^t_y(M;w)$. By Theorem~\ref{thm::sWH::associativity}(iv) and Corollary~\ref{cor::sWH::GH}(iii),
\begin{equation}
\label{eq::grF::M}
\mathrm{gr}_FM\simeq\mGH^{\gl}_\lambda ,
\end{equation}
where the $y$-degree is the filtration degree. This proves the first isomorphism of~(iii).
By Proposition~\ref{prp::GH::gl}(iv) and Theorem~\ref{thm::top::fusion}(iii), the module $\Lambda^NV$ occurs in $\mathrm{gr}_FM\simeq\mGH^{\gl}_\lambda\simeq V^{\otimes N}$ exactly once, in $y$-degree $d$. Hence $F_{d-1}$ contains no copy of $\Lambda^NV$, and $\pi_w(F_{d-1})=0$.

\emph{The inclusion $F_s\subset G^{d-s}$.}
Let $u\in F_s$ and $X\in U(\mathfrak{G})^y_{<d-s}$. Then $Xu\in U(\mathfrak{G})^y_{\leq d-s-1}U(\mathfrak{G})^y_{\leq s}\,v\subset F_{d-1}$, hence $\pi_w(Xu)=0$. By definition, $u\in G^{d-s}$.

\emph{The equality of dimensions.}
By Proposition~\ref{prp::sign::cofusion}(ii), applied to $M$ and to the isomorphism $M\simeq M^\circ$ of~(i), and by~\eqref{eq::grF::M},
\[
\mathrm{gr}_GM\simeq\big(\mathrm{gr}_FM\big)^\circ\simeq\bomega\big(\mGH^{\gl}_\lambda\big)^\dagger\simeq\mGH^{\gl}_\lambda\big\langle-\mdeg(\Delta_\lambda)\big\rangle ,
\]
the last isomorphism being Corollary~\ref{cor::GH::self::sign}(i). Here $G^t/G^{t+1}$ sits in $y$-degree $-t$ and the right-hand side sits in $y$-degrees $\leq 0$. Therefore
\[
\dim G^t/G^{t+1}=\dim\big(\mGH^{\gl}_\lambda\big)_{y=d-t},\qquad\text{and}\qquad G^{d+1}=0,
\]
where $(\cdot)_{y=a}$ denotes the sum of all components with $y$-degree $a$. On the other hand, $\dim F_s/F_{s-1}=\dim(\mGH^{\gl}_\lambda)_{y=s}$ by~\eqref{eq::grF::M}. Hence
\[
\dim G^{d-s}=\sum_{t\geq d-s}\dim G^t/G^{t+1}=\sum_{s'\leq s}\dim(\mGH^{\gl}_\lambda)_{y=s'}=\dim F_s .
\]
Together with the inclusion $F_s\subset G^{d-s}$, this proves~(ii). The second isomorphism of~(iii) is the isomorphism of graded modules $\mathrm{gr}_GM\simeq\mGH^{\gl}_\lambda\langle-\mdeg(\Delta_\lambda)\rangle$ proved above.
\end{proof}

Note that for $k=1$ the theorem says that the fusion and cofusion filtrations of $\mGH^{\gl}_\lambda$ itself are related by $F_s=G^{d-s}$.
In the special case of one-row diagrams $\lambda^i=(\lambda_i)$ the theorem identifies the iterated fusion product and the iterated cofusion product of the local Weyl modules:

\begin{corollary}
\label{cor::fusion::cofusion::rows}
Let $\lambda=(\lambda_0\geq\ldots\geq\lambda_l)\vdash N$. Then, for all pairwise distinct parameters,
\[
\sW^x_{\lambda_0\varepsilon_1}\star_y\cdots\star_y\sW^x_{\lambda_l\varepsilon_1}\ \simeq\ \mGH^{\gl}_\lambda\ \simeq\ \sW^x_{\lambda_0\varepsilon_1}\circledast_y\cdots\circledast_y\sW^x_{\lambda_l\varepsilon_1}\big\langle\mdeg(\Delta_\lambda)\big\rangle .
\]
Combining with Theorem~\ref{thm::sign::Weyl}(iii), $\mGH^{\gl}_\lambda$ is also an iterated cofusion of the tautological modules: first in the variable $x$, and then in the variable $y$.
\end{corollary}

The proof of Theorem~\ref{thm::fusion::cofusion} shows that the inclusion $F_s\subset G^{d-s}$ holds without Haiman's theorem, and that Haiman's theorem for a given diagram is equivalent to the coincidence of the two filtrations on the tensor product of the Weyl modules corresponding to its rows.

\begin{theorem}
\label{thm::N!::equiv::coincidence}
Let $\lambda=(\lambda_0\geq\ldots\geq\lambda_l)\vdash N$ be a Young diagram, $d=n(\lambda)$, and let $c_0,\ldots,c_l$ be pairwise distinct numbers. Put $W_h:=\sW^x_{h\varepsilon_1}$ and
\[
M_\lambda:=W_{\lambda_0}(c_0)\otimes\cdots\otimes W_{\lambda_l}(c_l),
\]
with the cyclic vector $v$, the tensor product of the highest weight vectors, and the highest weight vector $w$ of $\Lambda^NV$. Let $F_s=F^y_s(M_\lambda;v)$ and $G^t=G^t_y(M_\lambda;w)$. Then $M_\lambda\simeq V^{\otimes N}$ as a $\gl_V$-module, $v$ is cyclic, $w$ is cocyclic, and
\[
F_s\subset G^{\,d-s}\qquad\text{for all } s .
\]
Let $C_\lambda:=\mathrm{gr}_GM_\lambda=W_{\lambda_0}\circledast_y\cdots\circledast_yW_{\lambda_l}$ be the cofusion product. Then there is a canonical (up to a nonzero scalar) injective homomorphism of bigraded $\mathfrak{G}$-modules
\[
\iota_\lambda\colon\ \mGH^{\gl}_\lambda\hookrightarrow C_\lambda\big\langle\mdeg(\Delta_\lambda)\big\rangle ,
\]
which is an isomorphism on the bottom component of bidegree $(0,0)$ (the module $S^NV$). Moreover, the following conditions are equivalent:
\begin{enumerate}
\item[(a)] the $N!$-conjecture holds for $\lambda$, i.e., $\mGH^{\bS}_\lambda$ is the regular representation of $\bS_N$;
\item[(b)] the surjection $\bar{\theta}_\lambda\colon\sWH_\lambda\twoheadrightarrow\mGH^{\gl}_\lambda$ of Proposition~\ref{prp::fusion->GH} is an isomorphism;
\item[(b$'$)] the embedding $\iota_\lambda$ is an isomorphism;
\item[(c)] the fusion and cofusion filtrations coincide: $F_s=G^{\,d-s}$ for all $s$.
\end{enumerate}
\end{theorem}
\begin{proof}
The module $W_h$ is isomorphic to $V^{\otimes h}$ as a $\gl_V$-module (Theorem~\ref{thm::Weyl::fusion}), so $M_\lambda\simeq V^{\otimes N}$; the vector $v$ is cyclic by Lemma~\ref{lem::fusion::graded::cyclic}. By Proposition~\ref{prp::Weyl::top}, the socle of $W_h$ is $\Lambda^hV$; by Example~\ref{ex::cocyclic} the socle of $M_\lambda$ is $\bigotimes_j\Lambda^{\lambda_j}V(c_j)$, which contains $\Lambda^NV$ with multiplicity one. Hence $w$ is cocyclic.
By definition, $\mathrm{gr}_FM_\lambda=\sWH_\lambda$. By Theorem~\ref{thm::top::fusion}(iii), the module $\Lambda^NV$ sits in $\sWH_\lambda$ in the bidegree $\mdeg(\Delta_\lambda)$ only, in particular in the $y$-degree $d$. Hence $F_{d-1}$ contains no copy of $\Lambda^NV$ and $\pi_w(F_{d-1})=0$. The inclusion $F_s\subset G^{d-s}$ follows exactly as in the proof of Theorem~\ref{thm::fusion::cofusion}.

\emph{The embedding $\iota_\lambda$.} The module $W_h=\mGH^{\gl}_{(h)}$ is isomorphic to $W_h^\circ$ up to a shift (Corollary~\ref{cor::GH::self::sign}(i)), hence $M_\lambda\simeq M_\lambda^\circ$ by Lemma~\ref{lem::sign::properties}, and the isomorphism maps the unique copy of $\Lambda^NV$ to itself. By Proposition~\ref{prp::sign::cofusion}(ii), $C_\lambda=\mathrm{gr}_GM_\lambda\simeq(\mathrm{gr}_FM_\lambda)^\circ=\sWH_\lambda^\circ$; note that this does not use Haiman's theorem. The functor $\circ$ is exact, so $\bar{\theta}_\lambda$ induces an injection $(\mGH^{\gl}_\lambda)^\circ\hookrightarrow\sWH_\lambda^\circ\simeq C_\lambda$, and $(\mGH^{\gl}_\lambda)^\circ\simeq\mGH^{\gl}_\lambda\langle-\mdeg(\Delta_\lambda)\rangle$ by Corollary~\ref{cor::GH::self::sign}(i). This gives $\iota_\lambda$; it is unique up to a scalar, because a homomorphism from a cyclic module is determined by the image of the cyclic vector, and the weight space of weight $N\varepsilon_1$ of $C_\lambda$ is one-dimensional. By Proposition~\ref{prp::fusion->GH}, $\bar{\theta}_\lambda$ is an isomorphism on the top component $\Lambda^NV$ of bidegree $\mdeg(\Delta_\lambda)$; the contravariant sign duality turns it into an isomorphism on the bottom component $S^NV$ of bidegree $(0,0)$. Finally, $\iota_\lambda$ is obtained from $\bar{\theta}_\lambda$ by the exact contravariant autoequivalence $\circ$, hence $\iota_\lambda$ is an isomorphism if and only if $\bar{\theta}_\lambda$ is, i.e., (b)$\Leftrightarrow$(b$'$).

(a)$\Leftrightarrow$(b) is Corollary~\ref{cor::sWH::GH}(ii),(iii).

(a)$\Rightarrow$(c). This is the proof of Theorem~\ref{thm::fusion::cofusion} for the diagrams $\lambda^i=(\lambda_i)$. In that proof Haiman's theorem is used only for the diagram $\lambda$: for one-row diagrams it is the classical statement about the coinvariant algebra.

(c)$\Rightarrow$(b). Suppose that $F_s=G^{d-s}$ for all $s$. Then $\mathrm{gr}_FM_\lambda=\mathrm{gr}_GM_\lambda$ as $\mathfrak{G}$-modules, with $F_s/F_{s-1}=G^{d-s}/G^{d-s+1}$. The module $\mathrm{gr}_GM_\lambda$ is cocyclic with the cocyclic vector $[w]\in G^0/G^1$ (\S\ref{sec::cofusion::def}). Hence every nonzero submodule of $\sWH_\lambda=\mathrm{gr}_FM_\lambda$ contains the class $[w]$ and, therefore, the whole component $\Lambda^NV$ of bidegree $\mdeg(\Delta_\lambda)$. By Proposition~\ref{prp::fusion->GH}, the map $\bar{\theta}_\lambda$ is injective on this component. Thus $\ker\bar{\theta}_\lambda$ contains no nonzero submodule, i.e., $\ker\bar{\theta}_\lambda=0$.
\end{proof}

\section{Butler's conjecture and the $\frac{N!}{2}$-conjecture via fusion products}
\label{sec::Butler}

Throughout this section $\dim V\geq N$, $\mathfrak{G}=\gl_V[x,y]$, and we write
\[
\mGH_\lambda:=\mGH^{\gl}_\lambda,\qquad
\mGH_{\mu,\nu}:=\mGH^{\gl}_{\mu,\nu}=S^N(V[x,y])\big/\big(\euJ^{\gl}_\mu+\euJ^{\gl}_\nu\big)
\]
(see~\eqref{eq::GHgl::common::quotient}). We use Haiman's theorem (Theorem~\ref{conj::n!}) through the identification $\sWH_\lambda\simeq\mGH_\lambda\simeq V^{\otimes N}$ of Corollary~\ref{cor::sWH::GH}. All fusion products and filtrations are taken with respect to $y$, and the images of $e_1^N$ in the quotients of $S^N(V[x,y])$ are called the \emph{cyclic vectors}. We write
\[
W_h:=\mGH_{(h)}\simeq\sW^x_{h\varepsilon_1}\quad(h\geq1),\qquad W_0:=\bC,
\]
for the local Weyl module attached to the one-row diagram $(h)$ (Example~\ref{ex::sWH::row::column}).

Let $\mu,\nu\vdash N$ be two different Young diagrams obtained from a diagram $\lambda\vdash N+1$ by removing two different corners. Equivalently, $\mu=\rho\cup\{a\}$ and $\nu=\rho\cup\{b\}$, where $\rho=\mu\cap\nu$ and $a\neq b$ are two addable cells of $\rho$. We say that $\mu$ and $\nu$ \emph{differ in two corners}. The two cells $a,b$ lie in different rows and in different columns.
By Proposition~\ref{prp::N!/k::gl}(i), the $\frac{N!}{2}$-conjecture (Conjecture~\ref{conj::N!/k} for $k=2$) says that the weight subspace of weight $(1^N,0^{m-N})$ of the common quotient $\mGH_{\mu,\nu}$ has dimension $\frac{N!}{2}$.
In this section we construct a quotient of $\mGH_{\mu,\nu}$ by a fusion product with this property, which gives the lower bound for the dimension. We show that the $\frac{N!}{2}$-conjecture is equivalent to the isomorphism of $\mGH_{\mu,\nu}$ with this quotient, and we compute its character, which is Butler's polynomial. This gives a representation-theoretic proof of Butler's conjecture.

\subsection{Preliminaries}
\label{sec::Butler::prelim}

For two Young diagrams $\sigma$ and $\rho$ we denote by $\sigma\cup\rho$ the diagram whose rows are the rows of $\sigma$ and of $\rho$ (as in \S\ref{sec::sWH::associativity}).

\begin{proposition}[Fusion of rows]
\label{prp::Butler::rows}
Let $\sigma$ be a Young diagram and $h_1,\ldots,h_p\geq1$ integers, and let $\mu=\sigma\cup(h_1)\cup\ldots\cup(h_p)$. Then, for all pairwise distinct parameters, there is an isomorphism of graded $\mathfrak{G}$-modules
\[
\mGH_\mu\ \simeq\ \mGH_\sigma\star_yW_{h_1}\star_y\cdots\star_yW_{h_p},
\]
mapping the cyclic vector to the cyclic vector.
\end{proposition}
\begin{proof}
This is Theorem~\ref{thm::sWH::associativity}(iv) with $\lambda^1=\sigma$ and $\lambda^{i+1}=(h_i)$, together with Corollary~\ref{cor::sWH::GH}(iii).
\end{proof}

\begin{lemma}[Transposition]
\label{lem::Butler::transpose}
Let $\theta$ be the automorphism of $\mathfrak{G}$ interchanging $x$ and $y$, and let $M^\theta$ be the twist of a $\mathfrak{G}$-module $M$ by $\theta$. Then
\[
\mGH_\lambda^{\theta}\simeq\mGH_{\lambda^T},\qquad
\mGH_{\nu_1,\ldots,\nu_k}^{\theta}\simeq\mGH_{\nu_1^T,\ldots,\nu_k^T}
\]
for all Young diagrams $\lambda$ and all families $\nu_1,\ldots,\nu_k$ of Young diagrams of the same size, compatibly with the cyclic vectors.
\end{lemma}
\begin{proof}
By Lemma~\ref{lem::GH::symmetry} and Proposition~\ref{prp::GH::gl}(ii), $\theta(\euJ^{\gl}_\lambda)=\euJ^{\gl}_{\lambda^T}$ as submodules of $S^N(V[x,y])$ (the Schur--Weyl functor commutes with the twist by $\theta$ by the proof of Lemma~\ref{lem::sign::properties}(iii)). Hence $\theta$ maps $\euJ^{\gl}_{\nu_1}+\ldots+\euJ^{\gl}_{\nu_k}$ to $\euJ^{\gl}_{\nu_1^T}+\ldots+\euJ^{\gl}_{\nu_k^T}$.
Note that this lemma does not use Haiman's theorem.
\end{proof}

\begin{lemma}[The case $N=2$]
\label{lem::Butler::base}
We have $\mGH_{(2)}=S^2V\oplus\Lambda^2V\cdot x$ and $\mGH_{(1,1)}=S^2V\oplus\Lambda^2V\cdot y$ (as $\gl_V$-modules, with the indicated bidegrees $(0,0)$, $(1,0)$ and $(0,1)$). The common quotient $\mGH_{(2),(1,1)}\simeq S^2V$ is concentrated in the bidegree $(0,0)$.
\end{lemma}
\begin{proof}
The first claim is Example~\ref{ex::sWH::row::column} together with Corollary~\ref{cor::sWH::GH}. The second one is the first item of Example~\ref{ex::N!/k::gl}.
\end{proof}

\begin{lemma}[Reduction to two rows]
\label{lem::Butler::reduction}
Let $\mu,\nu\vdash N$ differ in two corners, and let $\rho=\mu\cap\nu$. The two cells $\mu\setminus\nu$ and $\nu\setminus\mu$ lie in rows $i<j$ of $\mu\cup\nu$. Put $r:=\rho_i$ and $r':=\rho_j$ (with $\rho_j=0$ if $j$ exceeds the number of rows of $\rho$); then $r>r'\geq0$. Let $h_1,\ldots,h_p$ be the lengths of the rows of $\rho$ with indices different from $i$ and $j$. Let $\sigma$ be the diagram of the two rows $(r+1,r')$ and $\tau$ the diagram of the two rows $(r,r'+1)$ (so that $\sigma,\tau\vdash r+r'+1$ differ in two corners), and label $\mu,\nu$ in such a way that $\mu$ is the diagram containing the extra cell in the row $i$. Then
\[
\mu=\sigma\cup(h_1)\cup\ldots\cup(h_p),\qquad \nu=\tau\cup(h_1)\cup\ldots\cup(h_p),
\]
and there is a surjective homomorphism of $\mathfrak{G}$-modules
\[
\Theta\colon\ \mGH_{\mu,\nu}\twoheadrightarrow\mGH_{\sigma,\tau}\star_yW_{h_1}\star_y\cdots\star_yW_{h_p}
\]
sending the cyclic vector to the cyclic vector.
\end{lemma}
\begin{proof}
The first claim is straightforward. Write $W_\bullet$ for the sequence $W_{h_1},\ldots,W_{h_p}$.
By Proposition~\ref{prp::Butler::rows}, there are isomorphisms $\alpha_\mu\colon\mGH_\mu\simeq\mGH_\sigma\star_yW_\bullet$ and $\alpha_\nu\colon\mGH_\nu\simeq\mGH_\tau\star_yW_\bullet$ preserving the cyclic vectors. Let $p_\sigma\colon\mGH_\sigma\twoheadrightarrow\mGH_{\sigma,\tau}$ and $p_\tau\colon\mGH_\tau\twoheadrightarrow\mGH_{\sigma,\tau}$ be the canonical projections. By Proposition~\ref{prp::fusion::surjection}, the maps
\[
f_\mu:=(p_\sigma\star W_\bullet)\circ\alpha_\mu\colon\mGH_\mu\twoheadrightarrow\mGH_{\sigma,\tau}\star_yW_\bullet,\qquad
f_\nu:=(p_\tau\star W_\bullet)\circ\alpha_\nu\colon\mGH_\nu\twoheadrightarrow\mGH_{\sigma,\tau}\star_yW_\bullet
\]
are surjective and send the cyclic vectors to the same cyclic vector. The compositions $S^N(V[x,y])\twoheadrightarrow\mGH_\mu\to\mGH_{\sigma,\tau}\star_yW_\bullet$ and $S^N(V[x,y])\twoheadrightarrow\mGH_\nu\to\mGH_{\sigma,\tau}\star_yW_\bullet$ send $e_1^N$ to the same vector. By Proposition~\ref{prp::proj::cov} they coincide. Hence this composition kills $\euJ^{\gl}_\mu$ and $\euJ^{\gl}_\nu$, and factors through $\mGH_{\mu,\nu}$. The resulting map $\Theta$ is surjective, because $f_\mu$ is.
\end{proof}

\subsection{The fusion quotient and the lower bound}
\label{sec::Butler::lower}

\begin{definition}
\label{def::GH::fusion::quotient}
For an unordered pair $\{\mu,\nu\}$ of diagrams differing in two corners we define a cyclic $\mathfrak{G}$-module $\mGH'_{\mu,\nu}$ and a surjection $q_{\mu,\nu}\colon\mGH_{\mu,\nu}\twoheadrightarrow\mGH'_{\mu,\nu}$, sending the cyclic vector to the cyclic vector, recursively in $N$.
\begin{enumerate}
\item[(a)] For $N=2$: $\mGH'_{(2),(1,1)}:=\mGH_{(2),(1,1)}\simeq S^2V$ and $q:=\mathrm{id}$ (Lemma~\ref{lem::Butler::base}).
\item[(b)] If $N>2$ and both $\mu$ and $\nu$ have at most two rows, put $\mGH'_{\mu,\nu}:=\big(\mGH'_{\mu^T,\nu^T}\big)^\theta$, and let $q_{\mu,\nu}$ be the twist by $\theta$ of $q_{\mu^T,\nu^T}$ (Lemma~\ref{lem::Butler::transpose}).
\item[(c)] Otherwise, let $\sigma,\tau,h_1,\ldots,h_p$ and $\Theta$ be as in Lemma~\ref{lem::Butler::reduction}, and put
\[
\mGH'_{\mu,\nu}:=\mGH'_{\sigma,\tau}\star_yW_{h_1}\star_y\cdots\star_yW_{h_p},
\]
with $q_{\mu,\nu}$ the composition of $\Theta$ and of the surjection $\mGH_{\sigma,\tau}\star_yW_\bullet\twoheadrightarrow\mGH'_{\sigma,\tau}\star_yW_\bullet$ induced by $q_{\sigma,\tau}$ (Proposition~\ref{prp::fusion::surjection}).
\end{enumerate}
We write $q_\mu\colon\mGH_\mu\twoheadrightarrow\mGH'_{\mu,\nu}$ and $q_\nu\colon\mGH_\nu\twoheadrightarrow\mGH'_{\mu,\nu}$ for the compositions of $q_{\mu,\nu}$ with the projections $\mGH_\mu\twoheadrightarrow\mGH_{\mu,\nu}$ and $\mGH_\nu\twoheadrightarrow\mGH_{\mu,\nu}$.
\end{definition}

The recursion terminates. In case~(c) at least one of the lengths $h_i$ is present, so $|\sigma|<N$: if $p=0$, then $\mu=\sigma$ and $\nu=\tau$ have at most two rows, and we are in case~(a) or (b). In case~(b) the pair $\{\mu^T,\nu^T\}$ has at most two columns and is not of type~(b) again, since for $N>2$ the $2\times2$ box contains at most one diagram of size $N$.

\begin{proposition}
\label{prp::Butler::fusion::quotient}
For every pair $\{\mu,\nu\}$ as above, the $\mathfrak{G}$-module $\mGH'_{\mu,\nu}$ is a finite-dimensional bigraded cyclic module, which is isomorphic to $S^2V\otimes V^{\otimes(N-2)}$ as a $\gl_V$-module.
\end{proposition}
\begin{proof}
We use induction along Definition~\ref{def::GH::fusion::quotient}. Case~(a) is clear, and case~(b) follows since the twist by $\theta$ does not change the $\gl_V$-module structure. In case~(c), the fusion product is isomorphic to the tensor product of the factors as a $\gl_V$-module (Remark~\ref{rem::fusion::graded}(i)). By the induction hypothesis, $\mGH'_{\sigma,\tau}\star_yW_\bullet\simeq S^2V\otimes V^{\otimes(r+r'-1)}\otimes V^{\otimes h_1}\otimes\cdots\otimes V^{\otimes h_p}$, and $r+r'+1+\sum_ih_i=N$.
\end{proof}

\begin{corollary}[Lower bound]
\label{cor::Butler::lower}
Let $\mu,\nu\vdash N$ differ in two corners. The common quotient $\mGH_{\mu,\nu}$ surjects onto the module $\mGH'_{\mu,\nu}\simeq S^2V\otimes V^{\otimes(N-2)}$. Consequently, the weight subspace of weight $(1^N,0^{m-N})$ of $\mGH_{\mu,\nu}$ has dimension at least $\frac{N!}{2}$; in other words, $\dim(\mGH^{\bS}_\mu\cap\mGH^{\bS}_\nu)\geq\frac{N!}{2}$. The $\frac{N!}{2}$-conjecture for $\mu,\nu$ is equivalent to the statement that the surjection $q_{\mu,\nu}\colon\mGH_{\mu,\nu}\twoheadrightarrow\mGH'_{\mu,\nu}$ is an isomorphism.
\end{corollary}
\begin{proof}
The weight subspace of weight $(1^N,0^{m-N})$ of $S^2V\otimes V^{\otimes(N-2)}$ is spanned by the tensors $(e_ae_b)\otimes e_{c_1}\otimes\ldots\otimes e_{c_{N-2}}$ with pairwise distinct indices $\{a,b,c_1,\ldots,c_{N-2}\}=\{1,\ldots,N\}$ and has dimension $\binom{N}{2}(N-2)!=\frac{N!}{2}$. The remaining claims follow from Proposition~\ref{prp::N!/k::gl}(i) and Conjecture~\ref{conj::N!/k::gl}(i). Indeed, the functor $\WS$ is an equivalence of categories, so the surjection $q_{\mu,\nu}$ is an isomorphism if and only if it induces an isomorphism of the weight subspaces of weight $(1^N,0^{m-N})$. Since $q_{\mu,\nu}$ is surjective, this happens if and only if the dimension of this weight subspace of $\mGH_{\mu,\nu}$ equals $\frac{N!}{2}$.
\end{proof}

\begin{example}
\label{ex::Butler::two::columns}
Let $\mu=(2^k,1^{l+2})$ and $\nu=(2^{k+1},1^l)$ be two diagrams with at most two columns. Then $\rho=(2^k,1^{l+1})$ and the extra cells lie in the rows $i=k$ and $j=k+l+1$; $r=1$, $r'=0$, and $\{\sigma,\tau\}=\{(1,1),(2)\}$. The remaining rows of $\rho$ are $k$ rows of length $2$ and $l$ rows of length $1$. Hence
\[
\mGH'_{\mu,\nu}=S^2V\star_y\underbrace{W_2\star_y\cdots\star_yW_2}_{k}\star_y\underbrace{V\star_y\cdots\star_yV}_{l},
\]
whose underlying $\gl_V$-module is $S^2V\otimes V^{\otimes(2k+l)}$.
\end{example}

\subsection{The exact sequence}
\label{sec::Butler::exact}

\subsubsection{An abstract principle}
Let $z=y$. Recall from \S\ref{sec::cofusion} that, for a cyclic module $M$ with a cyclic vector $v$ and a cocyclic vector $w$, we have the increasing fusion filtration $F_s=F^y_s(M;v)$ and the decreasing cofusion filtration $G^t=G^t_y(M;w)$.

\begin{lemma}
\label{lem::sub::quot::filtrations}
Let $0\to X\to M\xrightarrow{\ p\ }Q\to0$ be an exact sequence of finite-dimensional $\mathfrak{G}$-modules. Let $v\in M$ be cyclic and $w\in M$ cocyclic.
\begin{enumerate}
\item The vector $p(v)$ is cyclic in $Q$, and the fusion filtration of $Q$ is $F_s(Q;p(v))=p(F_s)$.
\item If $X\neq0$, then $w\in X$ is cocyclic in $X$, and the cofusion filtration of $X$ is $G^t(X;w)=G^t\cap X$.
\end{enumerate}
Consequently, the induced map $\mathrm{gr}_FM\to\mathrm{gr}_FQ$ is surjective and the induced map $\mathrm{gr}_GX\to\mathrm{gr}_GM$ is injective.
\end{lemma}
\begin{proof}
(i) $U(\mathfrak{G})^y_{\leq s}\,p(v)=p\big(U(\mathfrak{G})^y_{\leq s}v\big)$.
(ii) Every nonzero submodule of $M$ contains $w$, so $w\in X$. The $\gl_V$-module $L$ generated by $w$ occurs with multiplicity one in $M$, hence also in $X$, and the isotypic projection of $X$ is the restriction of the one of $M$. Thus $\pi_w|_X=\pi_w$ is the functional attached to $w\in X$, and the definition of $G^t$ gives $G^t(X;w)=G^t\cap X$.
\end{proof}

The associated graded spaces $\mathrm{gr}_GX\subset\mathrm{gr}_GM$ and $\mathrm{gr}_FM\twoheadrightarrow\mathrm{gr}_FQ$ are defined with respect to different filtrations of $M$, so in general the two maps cannot be composed. If the two filtrations coincide, they can.

\begin{proposition}
\label{prp::coincidence::exact}
In the situation of Lemma~\ref{lem::sub::quot::filtrations} assume that there is an integer $d$ such that $F_s=G^{d-s}$ for all $s$. Identify $\mathrm{gr}_FM$ with $\mathrm{gr}_GM$ by $F_s/F_{s-1}=G^{d-s}/G^{d-s+1}$. Then the sequence
\[
0\longrightarrow\mathrm{gr}_GX\longrightarrow\mathrm{gr}_GM=\mathrm{gr}_FM\longrightarrow\mathrm{gr}_FQ\longrightarrow0
\]
is exact.
\end{proposition}
\begin{proof}
Injectivity and surjectivity hold by Lemma~\ref{lem::sub::quot::filtrations}. For the composition, let $u\in G^t(X;w)=G^t\cap X$. Its class in $G^t/G^{t+1}=F_{d-t}/F_{d-t-1}$ maps to the class of $p(u)=0$ in $F_{d-t}(Q)/F_{d-t-1}(Q)$. Finally, an associated graded space has the dimension of the filtered space, so $\dim\mathrm{gr}_GX+\dim\mathrm{gr}_FQ=\dim X+\dim Q=\dim M$. Together with injectivity, surjectivity and the vanishing of the composition, this gives exactness in the middle.
\end{proof}

\subsubsection{The Gorenstein isomorphism}
By Corollary~\ref{cor::GH::self::sign}(i), for every diagram $\mu$ the cyclic module $\mGH_\mu$ is isomorphic to its contravariant sign dual $\mGH_\mu^\circ=\bomega(\mGH_\mu)^\dagger$ up to a shift of the bigrading.
The weight space of weight $N\varepsilon_1$ of $\mGH_\mu$ is one-dimensional (it lies in the component $S^NV$ of $V^{\otimes N}$), and a homomorphism from a cyclic module is determined by the image of the cyclic vector. Hence every two isomorphisms $\mGH_\mu^\circ\simeq\mGH_\mu$ differ by a nonzero scalar, and we fix one of them:
\begin{equation}
\label{eq::iota}
\iota_\mu\colon\ \mGH_\mu^\circ\ \xrightarrow{\ \sim\ }\ \mGH_\mu\big\langle-\mdeg(\Delta_\mu)\big\rangle.
\end{equation}
(Via the Schur--Weyl duality, $\iota_\mu$ is the Gorenstein duality $f\mapsto f(\partial)\Delta_\mu$ of the algebra $\bC[\bar{z}_1,\ldots,\bar{z}_N]/\euJ^{\bS}_\mu$.)

\begin{theorem}
\label{thm::Butler::exact}
Let $\mu,\nu\vdash N$ differ in two corners. The sequence of bigraded $\mathfrak{G}$-modules
\begin{equation}
\label{eq::exact::sequence}
0\longrightarrow\big(\mGH'_{\mu,\nu}\big)^{\circ}\xrightarrow{\ \iota_\mu\circ(q_\mu)^{\circ}\ }\mGH_\mu\big\langle-\mdeg(\Delta_\mu)\big\rangle\xrightarrow{\ \ q_\mu\ \ }\mGH'_{\mu,\nu}\big\langle-\mdeg(\Delta_\mu)\big\rangle\longrightarrow0
\end{equation}
is exact. The same holds with $\nu$ in place of $\mu$, with the same module $\mGH'_{\mu,\nu}$.
\end{theorem}
\begin{proof}
We use induction along Definition~\ref{def::GH::fusion::quotient}. The map $(q_\mu)^\circ$ is injective and $q_\mu$ is surjective, because $\circ$ is an exact contravariant functor; so we have to prove exactness in the middle only.

\emph{Case~(a): $N=2$.} The kernel of $q_{(2)}\colon\mGH_{(2)}=S^2V\oplus\Lambda^2V\cdot x\to S^2V$ is the submodule $\Lambda^2V\cdot x$. The image of $\iota_{(2)}\circ(q_{(2)})^\circ$ is a submodule of $\mGH_{(2)}$ isomorphic to $(S^2V)^\circ=\Lambda^2V$, and the module $\mGH_{(2)}$ contains only one copy of $\Lambda^2V$. The same argument works for $\nu=(1,1)$.

\emph{Case~(b).} The sign duality and the contragredient duality commute with the twist by $\theta$ (Lemma~\ref{lem::sign::properties}(iii)), and $\mGH_{\mu^T}^\theta\simeq\mGH_\mu$ (Lemma~\ref{lem::Butler::transpose}). The twist by $\theta$ of the isomorphism $\iota_{\mu^T}$ is an isomorphism $\mGH_\mu^\circ\to\mGH_\mu$, hence a nonzero multiple of $\iota_\mu$. Therefore the twist by $\theta$ of the exact sequence for $\{\mu^T,\nu^T\}$ is the exact sequence for $\{\mu,\nu\}$.

\emph{Case~(c).} Let $\sigma,\tau,h_1,\ldots,h_p$ be as in Lemma~\ref{lem::Butler::reduction}, with $\sigma$ corresponding to $\mu$. Take pairwise distinct numbers $c_0,\ldots,c_p$ and consider
\[
M:=\mGH_\sigma(c_0)\otimes W_{h_1}(c_1)\otimes\cdots\otimes W_{h_p}(c_p),\qquad
Q:=\mGH'_{\sigma,\tau}(c_0)\otimes W_{h_1}(c_1)\otimes\cdots\otimes W_{h_p}(c_p),
\]
and the surjection $\mathsf{p}:=q_\sigma\otimes\mathrm{id}\colon M\to Q$. Put $X:=\ker\mathsf{p}$. By the induction hypothesis for $\{\sigma,\tau\}$, $\ker q_\sigma$ is the image of $\iota_\sigma\circ(q_\sigma)^\circ$, hence
\[
X=\ker(q_\sigma)(c_0)\otimes W_{h_1}(c_1)\otimes\cdots\otimes W_{h_p}(c_p)
\]
is the image of $Q^\circ\to M^\circ\simeq M$, where the first map is $\mathsf{p}^\circ$ and the second one is the tensor product of the isomorphisms $\iota_\sigma,\iota_{(h_1)},\ldots,\iota_{(h_p)}$ (the sign duality commutes with twists and tensor products by Lemma~\ref{lem::sign::properties}).

By Theorem~\ref{thm::fusion::cofusion} applied to $\lambda^1=\sigma$, $\lambda^{i+1}=(h_i)$, $\lambda=\mu$ and $d=n(\mu)$, we have $F_s=G^{d-s}$ for the fusion and cofusion filtrations of $M$, and $\mathrm{gr}_FM\simeq\mGH_\mu$ by Proposition~\ref{prp::Butler::rows}. Proposition~\ref{prp::coincidence::exact} gives an exact sequence
\begin{equation}
\label{eq::gr::exact}
0\to\mathrm{gr}_GX\to\mathrm{gr}_FM=\mGH_\mu\xrightarrow{\ \mathrm{gr}\,\mathsf{p}\ }\mathrm{gr}_FQ\to0 .
\end{equation}
By the definition of the fusion product, $\mathrm{gr}_FQ=\mGH'_{\sigma,\tau}\star_yW_\bullet=\mGH'_{\mu,\nu}$, and $\mathrm{gr}\,\mathsf{p}$ is a homomorphism of cyclic modules sending the cyclic vector to the cyclic vector, i.e., $\mathrm{gr}\,\mathsf{p}=q_\mu$. So it remains to identify $\mathrm{gr}_GX\subset\mathrm{gr}_GM$ with the image of $\iota_\mu\circ(q_\mu)^\circ$.

Let $G$ be the cofusion filtration of $M$. By Proposition~\ref{prp::sign::cofusion} and Lemma~\ref{lem::cofusion::dual}, $G^t(M^\circ)$ is the annihilator in $M^\circ=\bomega(M)^\dagger$ of $\bomega(F_{t-1}M)$. These identifications are natural for the surjection $\mathsf{p}$: the injection $\mathsf{p}^\circ\colon Q^\circ\to M^\circ$ is compatible with the filtrations (Lemma~\ref{lem::sub::quot::filtrations}), and the induced map on the associated graded spaces is $(\mathrm{gr}\,\mathsf{p})^\circ\colon(\mathrm{gr}_FQ)^\circ\to(\mathrm{gr}_FM)^\circ$. Transporting this by the isomorphism $M^\circ\simeq M$ we see that $\mathrm{gr}_GX\to\mathrm{gr}_GM=\mathrm{gr}_FM$ is the composition of $(q_\mu)^\circ\colon(\mGH'_{\mu,\nu})^\circ\to\mGH_\mu^\circ$ with an isomorphism of $\mathfrak{G}$-modules $\mGH_\mu^\circ\to\mGH_\mu$. Any such isomorphism is a nonzero multiple of $\iota_\mu$, as explained above. Hence the image of $\mathrm{gr}_GX$ in $\mGH_\mu$ is the image of $\iota_\mu\circ(q_\mu)^\circ$, and~\eqref{eq::gr::exact} is the required exact sequence.
The argument for $\nu$ is the same, starting from the sequence for $\tau$ and the module $\mGH_\tau(c_0)\otimes W_\bullet$: the module $Q$ is the same, because $\mGH'_{\sigma,\tau}$ is symmetric in $\sigma$ and $\tau$.
\end{proof}

\begin{corollary}
\label{cor::Butler::exact}
Let $\mu,\nu\vdash N$ differ in two corners. Put $K_\mu:=\ker q_\mu\subset\mGH_\mu$.
\begin{enumerate}
\item The sign dual $(\mGH'_{\mu,\nu})^\circ$ embeds into both $\mGH_\mu^\circ$ and $\mGH_\nu^\circ$. Via the isomorphisms $\iota_\mu$ and $\iota_\nu$ the submodules $K_\mu\subset\mGH_\mu$ and $K_\nu\subset\mGH_\nu$ are both isomorphic to $(\mGH'_{\mu,\nu})^\circ$, with the shifts of the grading by $\mdeg(\Delta_\mu)$ and $\mdeg(\Delta_\nu)$ respectively.
\item The following conditions are equivalent:
\begin{enumerate}
\item $\mu$ and $\nu$ satisfy the $\frac{N!}{2}$-conjecture;
\item the surjection $q_{\mu,\nu}\colon\mGH_{\mu,\nu}\to\mGH'_{\mu,\nu}$ is an isomorphism;
\item $\ker q_\mu=\ker\big(\mGH_\mu\to\mGH_{\mu,\nu}\big)$;
\item $\iota_\mu\circ(q_\mu)^\circ$ identifies $(\mGH'_{\mu,\nu})^\circ$ with the kernel of $\mGH_\mu\to\mGH_{\mu,\nu}$.
\end{enumerate}
\end{enumerate}
\end{corollary}
\begin{proof}
(i) is Theorem~\ref{thm::Butler::exact}. For (ii), (a)$\Leftrightarrow$(b) is Corollary~\ref{cor::Butler::lower}. The projection $q_\mu$ factors as $\mGH_\mu\twoheadrightarrow\mGH_{\mu,\nu}\xrightarrow{q_{\mu,\nu}}\mGH'_{\mu,\nu}$, so (b)$\Leftrightarrow$(c), and (c)$\Leftrightarrow$(d) by Theorem~\ref{thm::Butler::exact}.
\end{proof}

\begin{remark}
\label{rem::Butler::intersection}
The contravariant functor $\circ$ is exact, hence it turns the sum $\euJ^{\gl}_\mu+\euJ^{\gl}_\nu$ of submodules of $S^N(V[x,y])$ into an intersection: inside $S^N(V[x,y])^\circ=\bomega(S^N(V[x,y]))^\dagger$ we have $\mGH_{\mu,\nu}^\circ=\mGH_\mu^\circ\cap\mGH_\nu^\circ$. This is the sign dual of the common quotient; under the Schur--Weyl duality it is the intersection $\mGH^{\bS}_\mu\cap\mGH^{\bS}_\nu$. The morphism $(q_{\mu,\nu})^\circ$ embeds $(\mGH'_{\mu,\nu})^\circ$ into this intersection, and by Corollary~\ref{cor::Butler::exact}(ii) the embedding is an isomorphism if and only if the $\frac{N!}{2}$-conjecture holds for $\mu$ and $\nu$.
\end{remark}

\subsection{The character and Butler's polynomial}
\label{sec::Butler::character}

Recall the notation $T_\nu=q_0^{n(\nu^T)}q_1^{n(\nu)}$ of~\eqref{eq::T::nu} and that, by Theorem~\ref{conj::n!}, $\mathrm{ch}_q(\mGH_\nu)=\widetilde{H}_\nu(X;q_0,q_1)=\gh_\nu$. For a symmetric function $F\in\Lambda\otimes\bZ[q_0,q_1]$ let $\overline{F}(X;q_0,q_1):=F(X;q_0^{-1},q_1^{-1})$.

\begin{proposition}
\label{prp::Butler::formula}
Let $\mu,\nu\vdash N$ differ in two corners, and let $F:=\mathrm{ch}_q(\mGH'_{\mu,\nu})$. Then
\begin{align}
F&=\frac{T_\mu\,\widetilde{H}_\nu-T_\nu\,\widetilde{H}_\mu}{T_\mu-T_\nu}=\gh_{\mu,\nu}(X;q_0,q_1),\label{eq::Butler::F}\\[2pt]
\mathrm{ch}_q(K_\mu)&=T_\mu\,\frac{\widetilde{H}_\mu-\widetilde{H}_\nu}{T_\mu-T_\nu},\label{eq::Butler::K}
\end{align}
where $\gh_{\mu,\nu}$ is the intersection polynomial~\eqref{eq::SF::I} for $k=2$.
\end{proposition}
\begin{proof}
The cells $\mu\setminus\nu$ and $\nu\setminus\mu$ lie in different rows and in different columns, hence $T_\mu\neq T_\nu$.
By Theorem~\ref{thm::Butler::exact} and Equation~\eqref{eq::sign::ch}, we have $\mathrm{ch}_q\big((\mGH')^{\circ}\big)=\omega\overline{F}$, and the shift of the bigrading by $\mdeg(\Delta_\mu)$ multiplies the character by $T_\mu$. The exact sequence~\eqref{eq::exact::sequence} gives
\[
\widetilde{H}_\mu=F+T_\mu\,\omega\overline{F}.
\]
By the second statement of Theorem~\ref{thm::Butler::exact}, which involves the same module $\mGH'_{\mu,\nu}$,
\[
\widetilde{H}_\nu=F+T_\nu\,\omega\overline{F}.
\]
Subtracting, we get $\omega\overline{F}=(\widetilde{H}_\mu-\widetilde{H}_\nu)/(T_\mu-T_\nu)$, and substituting back, we obtain~\eqref{eq::Butler::F}. Finally,~\eqref{eq::Butler::K} is $\widetilde{H}_\mu-F$, the character of the kernel of $q_\mu$.
\end{proof}

\begin{remark}
\label{rem::Butler::two::sequences}
Both exact sequences are needed. A single relation $\widetilde{H}_\mu=F+T_\mu\omega\overline{F}$ does not determine $F$: for any $G$, the function $G-T_\mu\,\omega\overline{G}$ solves the homogeneous equation. The second sequence, for $\nu$ with the same quotient $\mGH'_{\mu,\nu}$, removes this freedom.
\end{remark}

\begin{remark}[Divided differences]
\label{rem::Butler::divided}
Put $u_\mu=T_\mu^{-1}$. Then~\eqref{eq::Butler::F} is the divided difference of $u\mapsto u\widetilde{H}$:
\[
F=\frac{u_\mu\widetilde{H}_\mu-u_\nu\widetilde{H}_\nu}{u_\mu-u_\nu}.
\]
Let $\kappa=\mu\cup\nu$ denote the diagram with the two extra cells, so that $T_\mu=T_\kappa/w(a)$ and $T_\nu=T_\kappa/w(b)$, where $a=\kappa\setminus\mu$, $b=\kappa\setminus\nu$ and $w(c)=q_0^{c_0}q_1^{c_1}$ for a cell $c=(c_0,c_1)$. Then~\eqref{eq::Butler::F} can be written as
\[
F=\frac{w(a)\widetilde{H}_\mu-w(b)\widetilde{H}_\nu}{w(a)-w(b)}.
\]
This is the formula predicted by Bergeron and Garsia~\cite{Bergeron_Garsia::SF} for the character of $\mGH^{\bS}_\mu\cap\mGH^{\bS}_\nu$.
\end{remark}

\subsubsection{Butler's conjecture}
For two diagrams $\mu,\nu\subset\kappa$ with $|\kappa/\mu|=|\kappa/\nu|=1$, Butler~\cite{Butler} considered the polynomial
\[
I_{\mu,\nu}(X;q_0,q_1):=\frac{T_\mu\widetilde{H}_\nu-T_\nu\widetilde{H}_\mu}{T_\mu-T_\nu}=\gh_{\mu,\nu}
\]
(the \emph{Macdonald intersection polynomial}) and conjectured that it is Schur positive with coefficients in $\bZ_{\geq0}[q_0,q_1]$. While this manuscript was being prepared, a geometric proof of this conjecture appeared on the arXiv~\cite{Guo:Kang:Xiong::Butler:conjecture}: it realizes $I_{\mu,\nu}$ as the Frobenius characteristic of a quotient of $\mGH^{\bS}_\mu$ constructed from the Procesi bundle on a torus-invariant line in $\mathrm{Hilb}^N(\bC^2)$. Our proof below is of a different nature: it uses fusion products, and it provides an explicit algebraic description of the underlying module.

\begin{corollary}
\label{cor::Butler}
Let $\mu,\nu\vdash N$ differ in two corners. Then $I_{\mu,\nu}(X;q_0,q_1)=\mathrm{ch}_q(\mGH'_{\mu,\nu})$ is the bigraded character of the iterated fusion product $\mGH'_{\mu,\nu}$ of Definition~\ref{def::GH::fusion::quotient}. In particular, $I_{\mu,\nu}$ is Schur positive with coefficients in $\bZ_{\geq0}[q_0,q_1]$. Moreover, the polynomials
\[
T_\mu\,\frac{\widetilde{H}_\mu-\widetilde{H}_\nu}{T_\mu-T_\nu}\qquad\text{and}\qquad T_\nu\,\frac{\widetilde{H}_\nu-\widetilde{H}_\mu}{T_\nu-T_\mu}
\]
are Schur positive with coefficients in $\bZ_{\geq0}[q_0,q_1]$: they are the characters of the sign dual $(\mGH'_{\mu,\nu})^\circ$ in the gradings of $\mGH_\mu$ and of $\mGH_\nu$ respectively.
\end{corollary}
\begin{proof}
The identity is~\eqref{eq::Butler::F}. The character of a bigraded polynomial $\gl_V$-module concentrated in nonnegative bidegrees is Schur positive with coefficients in $\bZ_{\geq0}[q_0,q_1]$. This applies to $\mGH'_{\mu,\nu}$ and, by~\eqref{eq::Butler::K} and Theorem~\ref{thm::Butler::exact}, to the kernels $K_\mu$ and $K_\nu$.
\end{proof}

This proof uses Haiman's theorem together with the description of $\mGH_\mu$ as an iterated fusion product (Theorem~\ref{thm::sWH::associativity}); no further geometric input is needed.

%\begin{remark}
%\label{rem::Butler::compare::GKX}
%The module of~\cite{Guo:Kang:Xiong::Butler:conjecture} and the module $\mGH'_{\mu,\nu}$ have the same character $I_{\mu,\nu}$. It is not clear to us whether the two modules are isomorphic in general.
%\end{remark}

\subsubsection{Specialization $q_1=1$}
The fusion product preserves the $x$-grading, and its $x$-graded $\gl_V$-character is the product of the $x$-graded characters of the factors (the associated graded space of a filtered space has the same $x$-graded dimensions). Hence we obtain the following description of the specialization $q_1=1$ of the intersection polynomial.

\begin{corollary}
\label{cor::Butler::q1}
Let $\sigma,\tau,h_1,\ldots,h_p$ be as in Lemma~\ref{lem::Butler::reduction}. Then
\[
I_{\mu,\nu}(X;q_0,1)=I_{\sigma,\tau}(X;q_0,1)\prod_{i=1}^p\widetilde{H}_{(h_i)}(X;q_0),
\]
where $\widetilde{H}_{(h)}(X;q_0,q_1)=\prod_{j=1}^h(1-q_0^j)\,h_h\big[\frac{X}{1-q_0}\big]$ does not depend on $q_1$. In particular, for $\mu=(2^k,1^{l+2})$ and $\nu=(2^{k+1},1^l)$,
\[
I_{\mu,\nu}(X;q_0,1)=h_2\,\big(h_2+q_0\,e_2\big)^k\,h_1^{\,l}.
\]
\end{corollary}
\begin{proof}
The first formula follows from Definition~\ref{def::GH::fusion::quotient}(c) and the remark preceding the corollary. For the two-column pair of Example~\ref{ex::Butler::two::columns}, $\mGH'_{\sigma,\tau}=S^2V$ has the character $h_2$ in bidegree $(0,0)$, the $k$ modules $W_2$ have the character $\widetilde{H}_{(2)}=s_2+q_0s_{11}=h_2+q_0e_2$, and the $l$ modules $V$ have the character $h_1$.
\end{proof}

\subsection{A lower bound for the $\frac{N!}{k}$-conjecture}
\label{sec::Butler::Nk}

In this subsection we generalize the lower bound of \S\ref{sec::Butler::lower} from two corners to an arbitrary number $k$ of corners. We reduce the common quotient $\mGH_{\nu_1,\ldots,\nu_k}$ to the common quotient of the staircase diagram $(k,k-1,\ldots,1)$, which is the smallest diagram with $k$ corners, and we construct an explicit quotient of $\mGH_{\nu_1,\ldots,\nu_k}$ by a fusion product.

\subsubsection{The set-up}
Let $\lambda\vdash N+1$ be a Young diagram and let $c_1,\ldots,c_k$ be $k$ pairwise distinct corners of $\lambda$, lying in the rows $i_1<\ldots<i_k$. Put $\nu_s:=\lambda\setminus\{c_s\}\vdash N$, and
\[
a_s:=\lambda_{i_s}-1\quad(s=1,\ldots,k),\qquad\text{so that}\qquad a_1>a_2>\ldots>a_k\geq0 .
\]
(Indeed, $c_s$ is a corner, hence $\lambda_{i_s}>\lambda_{i_s+1}\geq\lambda_{i_{s+1}}$.) Let $h_1,\ldots,h_p$ be the lengths of the rows of $\lambda$ with the indices different from $i_1,\ldots,i_k$. Denote by $\kappa$ the diagram with the $k$ rows $a_1+1,\ldots,a_k+1$, and by $\sigma_s$ the diagram obtained from $\kappa$ by removing its corner in the $s$-th row (so $\sigma_s$ has the rows $a_1+1,\ldots,a_{s-1}+1,a_s,a_{s+1}+1,\ldots,a_k+1$). Then
\begin{equation}
\label{eq::nu::sigma::h}
\nu_s=\sigma_s\cup(h_1)\cup\ldots\cup(h_p),\qquad s=1,\ldots,k,
\end{equation}
because the rows of $\nu_s$ outside the rows $i_1,\ldots,i_k$ are the same for all $s$.

The staircase diagram $\delta_k=(k,k-1,\ldots,1)\vdash\frac{k(k+1)}{2}$ has exactly $k$ corners. Put $N_k:=\frac{k(k+1)}{2}-1$, let $\nu^{\mathrm{st}}_1,\ldots,\nu^{\mathrm{st}}_k\vdash N_k$ be the diagrams obtained from $\delta_k$ by removing the corners, and
\[
Z_k:=\mGH_{\nu^{\mathrm{st}}_1,\ldots,\nu^{\mathrm{st}}_k}=S^{N_k}(V[x,y])\Big/\big(\euJ^{\gl}_{\nu^{\mathrm{st}}_1}+\ldots+\euJ^{\gl}_{\nu^{\mathrm{st}}_k}\big).
\]
For $k=1$ we have $Z_1=\bC$, and for $k=2$ we have $Z_2=\mGH_{(2),(1,1)}=S^2V$ (Lemma~\ref{lem::Butler::base}).

\begin{lemma}[Reduction to the special rows]
\label{lem::Nk::reduction}
In the above notation there is a surjective homomorphism of $\mathfrak{G}$-modules
\[
\Theta\colon\ \mGH_{\nu_1,\ldots,\nu_k}\twoheadrightarrow\mGH_{\sigma_1,\ldots,\sigma_k}\star_yW_{h_1}\star_y\cdots\star_yW_{h_p}
\]
sending the cyclic vector to the cyclic vector.
\end{lemma}
\begin{proof}
By Proposition~\ref{prp::Butler::rows} and~\eqref{eq::nu::sigma::h}, there are isomorphisms $\alpha_s\colon\mGH_{\nu_s}\simeq\mGH_{\sigma_s}\star_yW_\bullet$ preserving the cyclic vectors. Let $p_s\colon\mGH_{\sigma_s}\twoheadrightarrow\mGH_{\sigma_1,\ldots,\sigma_k}$ be the projections. By Proposition~\ref{prp::fusion::surjection} the maps
\[
f_s:=(p_s\star W_\bullet)\circ\alpha_s\colon\ \mGH_{\nu_s}\twoheadrightarrow\mGH_{\sigma_1,\ldots,\sigma_k}\star_yW_\bullet
\]
are surjective and send the cyclic vectors to the same cyclic vector. The compositions $S^N(V[x,y])\twoheadrightarrow\mGH_{\nu_s}\xrightarrow{f_s}\mGH_{\sigma_1,\ldots,\sigma_k}\star_yW_\bullet$ send $e_1^N$ to the same vector. By Proposition~\ref{prp::proj::cov} they coincide. Hence this map kills all the submodules $\euJ^{\gl}_{\nu_s}$, and it factors through the quotient by their sum, which is $\mGH_{\nu_1,\ldots,\nu_k}$ (Lemma~\ref{lem::cyclic::colimit}).
\end{proof}

\begin{lemma}[The special rows form a staircase after transposition]
\label{lem::Nk::staircase}
The transposed diagram $\kappa^T$ has $k$ corners, and the diagrams $\sigma_1^T,\ldots,\sigma_k^T$ are obtained from $\kappa^T$ by removing them. The data $(a'_s;h'_j)$ of Lemma~\ref{lem::Nk::reduction} for the diagram $\kappa^T$ and these $k$ corners are as follows: $a'_s=k-s$ for $s=1,\ldots,k$, so that the diagram $\kappa'$ of Lemma~\ref{lem::Nk::reduction} is the staircase $\delta_k$ and the diagrams $\sigma'_s$ are the diagrams $\nu^{\mathrm{st}}_s$; and the spectator rows $h'_1,\ldots,h'_q$ consist of
\[
a_s-a_{s+1}-1\ \text{ rows of length } s\quad(s=1,\ldots,k-1),\qquad\text{and}\qquad a_k\ \text{ rows of length } k .
\]
In particular, $\sum_jh'_j=a_1+\ldots+a_k+k-N_k-1$.
\end{lemma}
\begin{proof}
The rows of $\kappa^T$ are indexed by the columns of $\kappa$: the length of the row $c\geq0$ is $\kappa^T_c=\#\{s\colon a_s\geq c\}$. This function drops exactly at the rows $c=a_s$, $s=1,\ldots,k$, where $\kappa^T_{a_s}=s$ and $\kappa^T_{a_s+1}=s-1$; so these are the $k$ corners of $\kappa^T$, and they are the transposes of the corners of $\kappa$. Removing them, the rows $a_k<a_{k-1}<\ldots<a_1$ get the lengths $k-1,k-2,\ldots,0$, i.e., $a'_s=k-s$. The other rows of $\kappa^T$ are the rows $c$ with $a_{s+1}<c<a_s$, of length $s$, and the rows $c<a_k$, of length $k$. The last formula follows from $|\kappa^T|=|\kappa|=a_1+\ldots+a_k+k$ and from $\sum_s(k-s+1)=N_k+1$.
\end{proof}

\begin{theorem}
\label{thm::Nk::lower}
Let $\lambda\vdash N+1$ and let $\nu_1,\ldots,\nu_k\vdash N$ be obtained from $\lambda$ by removing $k$ pairwise distinct corners. Let $h_1,\ldots,h_p$ and $h'_1,\ldots,h'_q$ be as above. Then there is a surjective homomorphism of $\mathfrak{G}$-modules
\[
\mGH_{\nu_1,\ldots,\nu_k}\ \twoheadrightarrow\ T:=\Big(Z_k\star_yW_{h'_1}\star_y\cdots\star_yW_{h'_q}\Big)^{\theta}\star_yW_{h_1}\star_y\cdots\star_yW_{h_p},
\]
sending the cyclic vector to the cyclic vector, where $\theta$ is the twist interchanging $x$ and $y$. As a $\gl_V$-module, $T\simeq Z_k\otimes V^{\otimes(N-N_k)}$. Consequently,
\[
\dim\big(\mGH_{\nu_1,\ldots,\nu_k}\big)_{(1^N,0^{m-N})}\ \geq\ \frac{N!}{N_k!}\ \dim\big(Z_k\big)_{(1^{N_k},0^{m-N_k})} .
\]
\end{theorem}
\begin{proof}
By Lemma~\ref{lem::Nk::reduction} we have the surjection $\Theta_1\colon\mGH_{\nu_1,\ldots,\nu_k}\twoheadrightarrow\mGH_{\sigma_1,\ldots,\sigma_k}\star_yW_{h_\bullet}$.
By the same argument as in Lemma~\ref{lem::Butler::transpose}, $\mGH_{\sigma_1,\ldots,\sigma_k}\simeq\big(\mGH_{\sigma^T_1,\ldots,\sigma^T_k}\big)^\theta$. By Lemma~\ref{lem::Nk::staircase}, we can apply Lemma~\ref{lem::Nk::reduction} to the diagram $\kappa^T$ and its $k$ corners. It gives a surjection $\mGH_{\sigma^T_1,\ldots,\sigma^T_k}\twoheadrightarrow Z_k\star_yW_{h'_\bullet}$ sending the cyclic vector to the cyclic vector. Its twist by $\theta$ is a surjection $\mGH_{\sigma_1,\ldots,\sigma_k}\twoheadrightarrow\big(Z_k\star_yW_{h'_\bullet}\big)^\theta$, and by Proposition~\ref{prp::fusion::surjection} it induces the surjection of fusion products
\[
\mGH_{\sigma_1,\ldots,\sigma_k}\star_yW_{h_\bullet}\ \twoheadrightarrow\ T .
\]
The composition with $\Theta_1$ is the required surjection.
By Remark~\ref{rem::fusion::graded}(i) and since the twist by $\theta$ does not change the $\gl_V$-structure, $T\simeq Z_k\otimes V^{\otimes\sum_jh'_j}\otimes V^{\otimes\sum_ih_i}$, and $\sum_jh'_j+\sum_ih_i=N-N_k$ by Lemma~\ref{lem::Nk::staircase} and~\eqref{eq::nu::sigma::h}.
It remains to compute the dimension of the weight space of weight $(1^N,0^{m-N})$ of $Z\otimes V^{\otimes r}$ for a polynomial $\gl_V$-module $Z$ of degree $N-r$: choose the ordered indices of the $r$ tensor factors among $\{1,\ldots,N\}$ in $\frac{N!}{(N-r)!}$ ways; the remaining $N-r$ indices span a weight space of $Z$ of dimension $\dim Z_{(1^{N-r},0^{m-N+r})}$, by the symmetry under the Weyl group.
\end{proof}

\begin{corollary}
\label{cor::Nk::lower}
Assume that $\dim(Z_k)_{(1^{N_k},0^{m-N_k})}=\frac{N_k!}{k}$, i.e., that the $\frac{N!}{k}$-conjecture holds for the staircase diagram $(k,k-1,\ldots,1)$ and all its $k$ corners. Then, for every Young diagram $\lambda\vdash N+1$ and every $k$ corners of it,
\[
\dim\big(\mGH_{\nu_1,\ldots,\nu_k}\big)_{(1^N,0^{m-N})}\ \geq\ \frac{N!}{k},
\]
and the $\frac{N!}{k}$-conjecture for $\nu_1,\ldots,\nu_k$ is equivalent to the statement that the surjection of Theorem~\ref{thm::Nk::lower} is an isomorphism.
\end{corollary}
\begin{proof}
By Theorem~\ref{thm::Nk::lower}, the dimension is at least $\frac{N!}{N_k!}\cdot\frac{N_k!}{k}=\frac{N!}{k}$. The module $T$ has exactly this dimension of the weight space. A surjection of finite-dimensional modules is an isomorphism if and only if it induces an isomorphism of the weight spaces of weight $(1^N,0^{m-N})$, because $\WS$ is an equivalence of categories (Theorem~\ref{thm::SW::currents}); for a surjection this means that the dimensions are equal. By Proposition~\ref{prp::N!/k::gl}(i), the dimension of the weight space of $\mGH_{\nu_1,\ldots,\nu_k}$ is the dimension of the intersection of the Garsia--Haiman modules.
\end{proof}

\begin{remark}
\begin{enumerate}
\item For $k=1$ the staircase is the single cell, $N_1=0$, $Z_1=\bC$, and Theorem~\ref{thm::Nk::lower} reduces to the isomorphism $\mGH_\nu\simeq V^{\otimes N}$ of Haiman's theorem.
\item For $k=2$ we have $Z_2=S^2V$ and $\frac{N!}{N_2!}\cdot1=\frac{N!}{2}$. In this case Theorem~\ref{thm::Nk::lower} is Corollary~\ref{cor::Butler::lower}, and $T=\mGH'_{\mu,\nu}$ is the module of Definition~\ref{def::GH::fusion::quotient} (the surjections used there are exactly the ones constructed above).
\item The corollary shows that the $\frac{N!}{k}$-conjecture is reduced to two independent statements: the equality $\dim(Z_k)_{(1^{N_k})}=\frac{N_k!}{k}$ for the staircase, and the injectivity of the surjection of Theorem~\ref{thm::Nk::lower}.
\item If $\mathrm{ch}(Z_k)=h_1^{N_k-2k+2}\Phi_k$ (the specialization of the intersection polynomial at $q_0=q_1=1$, see Theorem~\ref{thm::KLO::specialization}), then $\mathrm{ch}(T)=h_1^{N-N_k}\,\mathrm{ch}(Z_k)=h_1^{N-2k+2}\Phi_k$. This is the value at $q_0=q_1=1$ of the intersection polynomial for every $\lambda$ and every choice of $k$ corners. In this sense the module $T$ has the correct ungraded character.
\end{enumerate}
\end{remark}

\subsubsection{The characters for the staircase with three and four corners}
By Conjecture~\ref{conj::N!/k::gl}, the character of $Z_k$ should be the intersection polynomial $\gh_{\nu^{\mathrm{st}}_1,\ldots,\nu^{\mathrm{st}}_k}(X;q_0,q_1)$ of~\eqref{eq::SF::I}. We computed it for $k=3,4$ from the modified Macdonald polynomials (via the combinatorial formula of Haglund, Haiman and Loehr). In both cases the result has nonnegative coefficients in $\bZ_{\geq0}[q_0,q_1]$ and its value at $q_0=q_1=1$ has dimension $\frac{N_k!}{k}$, in agreement with Theorem~\ref{thm::KLO::specialization}.

\emph{The case $k=3$.} Here $N_3=5$, $\nu^{\mathrm{st}}=(2,2,1),(3,1,1),(3,2)$ and
\[
\gh_{\nu^{\mathrm{st}}_1,\nu^{\mathrm{st}}_2,\nu^{\mathrm{st}}_3}=s_{5}+(q_0+q_1)\,s_{41}+(q_0^2+q_0q_1+q_1^2)\,s_{32}+q_0q_1\,s_{311}+q_0q_1(q_0+q_1)\,s_{221}.
\]
The dimension is $1+8+15+6+10=40=\frac{5!}{3}$. This is an actual equality for the module: we computed, by exact linear algebra over $\bQ$, the bigraded dimensions of the intersection $\mGH^{\bS}_{(2,2,1)}\cap\mGH^{\bS}_{(3,1,1)}\cap\mGH^{\bS}_{(3,2)}$ (which is the weight space of $Z_3$ by Proposition~\ref{prp::N!/k::gl}). They are
\[
\begin{array}{c|ccc}
\text{$y$-degree}\backslash\text{$x$-degree}&0&1&2\\ \hline
0&1&4&5\\
1&4&11&5\\
2&5&5&0
\end{array}
\]
with the total dimension $40$; these are exactly the dimensions predicted by the polynomial above (the coefficient of $q_0^aq_1^b$, with the Schur function $s_\mu$ counted with the multiplicity $\dim\mathbb{S}_\mu$). Hence $\dim(Z_3)_{(1^5,0^{m-5})}=40=\frac{5!}{3}$. By Corollary~\ref{cor::Nk::lower}, we obtain the following statement, which is unconditional: \emph{for $k\leq3$ and for every Young diagram $\lambda\vdash N+1$ with $k$ chosen corners, $\dim(\mGH_{\nu_1,\ldots,\nu_k})_{(1^N,0^{m-N})}\geq\frac{N!}{k}$} (the cases $k=1$ and $k=2$ were treated before).

\emph{The case $k=4$.} Here $N_4=9$ and $\nu^{\mathrm{st}}=(3,3,2,1),(4,2,2,1),(4,3,1,1),(4,3,2)$. 
%The Schur expansion of the intersection polynomial is listed in Table~\ref{tab::k4}.
 Its dimension is $90720=\frac{9!}{4}$, and its specialization at $q_0=q_1=1$ equals $h_1^3\Phi_4=h_1^3\,(s_6+2s_{51}+6s_{42}+s_{411}+5s_{321}+5s_{222})$. Thus, if the $\frac{N!}{4}$-conjecture holds for the staircase $(4,3,2,1)$, then the lower bound $\frac{N!}{4}$ holds for all Young diagrams with $4$ chosen corners.

\section{Fusion filtrations for the current and for the affine algebra}
\label{sec::affine}

In this section we compare the fusion filtration of a tensor product of Demazure modules, defined by the current algebra $\gl_V[x]$, with the filtration defined by the larger affine algebra $\widehat{\gl}_V$. The comparison is based on the coincidence of the fusion and the cofusion filtrations (Theorem~\ref{thm::fusion::cofusion}).
In this section only, $\mathfrak{a}\subseteq\mathfrak{b}$ denote arbitrary Lie algebras; they should not be confused with the abelian subalgebra $\fa\subset\gl_V$ of \S\ref{sec::iterated::fusion}.

\subsection{The general setting}

Let $\mathfrak{a}\subseteq\mathfrak{b}$ be Lie algebras. For $i=0,\ldots,l$ let $L_i$ be a $\mathfrak{b}$-module with a vector $v_i$ such that
\[
M_i:=U(\mathfrak{a})v_i\subseteq L_i=U(\mathfrak{b})v_i .
\]
Fix a variable $z$ and pairwise distinct points $\gamma_0,\ldots,\gamma_l$. Let $M_i^{\gamma_i}$ and $L_i^{\gamma_i}$ be the evaluation modules over $\mathfrak{a}[z]$ and $\mathfrak{b}[z]$ respectively. Then $M_i^{\gamma_i}$ is an $\mathfrak{a}[z]$-submodule of $L_i^{\gamma_i}$. Put
\[
M:=M_0^{\gamma_0}\otimes\cdots\otimes M_l^{\gamma_l}\ \subseteq\ L:=L_0^{\gamma_0}\otimes\cdots\otimes L_l^{\gamma_l},\qquad v:=v_0\otimes\cdots\otimes v_l .
\]
Since the points are distinct, $\mathfrak{a}[z]$ acts on $M$ through the surjection $U(\mathfrak{a}[z])\to U(\mathfrak{a})^{\oplus(l+1)}$, $X\otimes f\mapsto(f(\gamma_i)X)_i$. Hence $M=U(\mathfrak{a}[z])v$, and likewise $L=U(\mathfrak{b}[z])v$. Let $U(\cdot)_{\leq k}$ denote the filtration by $z$-degree. We compare two filtrations of $M$:
\[
F^{\mathfrak{a}}_kM:=U(\mathfrak{a}[z])_{\leq k}\,v,\qquad
F^{\mathfrak{b}}_kM:=\big(U(\mathfrak{b}[z])_{\leq k}\,v\big)\cap M,\qquad k\geq0,
\]
with $F^{\mathfrak{a}}_{-1}M=F^{\mathfrak{b}}_{-1}M=0$. Since $U(\mathfrak{a}[z])_{\leq k}\subseteq U(\mathfrak{b}[z])_{\leq k}$, we always have
\begin{equation}
\label{eq::Fa::in::Fb}
F^{\mathfrak{a}}_kM\subseteq F^{\mathfrak{b}}_kM\qquad\text{for all }k .
\end{equation}
In general the inclusion is strict. We give a criterion for the reverse inclusion. It does not use any decomposition of $L$ into irreducible $\mathfrak{b}$-modules: instead, it uses a single linear functional on $L$ that does not see the complement of $\mathfrak{a}$ in $\mathfrak{b}$.

\begin{lemma}
\label{lem::ell::reduction}
Assume that
\begin{enumerate}
\item[(i)] $\mathfrak{b}=\mathfrak{n}\oplus\mathfrak{a}\oplus\mathfrak{h}$ as vector spaces, where $\mathfrak{n}$ and $\mathfrak{a}\oplus\mathfrak{h}$ are Lie subalgebras of $\mathfrak{b}$;
\item[(ii)] for every $i$, $\mathfrak{h}$ acts on $v_i$ by a character $\chi_i\colon\mathfrak{h}\to\bC$;
\item[(iii)] $\ell\colon L\to\bC$ is a linear functional such that $\ell\big((X\otimes f)u\big)=0$ for all $X\in\mathfrak{n}$, $f\in\bC[z]$ and $u\in L$.
\end{enumerate}
Then $\ell\big(U(\mathfrak{b}[z])_{\leq m}\,v\big)=\ell\big(U(\mathfrak{a}[z])_{\leq m}\,v\big)$ for every $m$.
\end{lemma}
\begin{proof}
All Lie algebras involved are graded by the $z$-degree. We have $\mathfrak{b}[z]=\mathfrak{n}[z]\oplus(\mathfrak{a}\oplus\mathfrak{h})[z]$ and $(\mathfrak{a}\oplus\mathfrak{h})[z]=\mathfrak{a}[z]\oplus\mathfrak{h}[z]$, decompositions into subalgebras. By the PBW theorem, multiplication gives an isomorphism of graded vector spaces $U(\mathfrak{b}[z])\simeq U(\mathfrak{n}[z])\otimes U(\mathfrak{a}[z])\otimes U(\mathfrak{h}[z])$. Hence $U(\mathfrak{b}[z])_{\leq m}$ is spanned by the products $YXH$ with $Y\in U(\mathfrak{n}[z])_{a}$, $X\in U(\mathfrak{a}[z])_{b}$ and $H\in U(\mathfrak{h}[z])_{c}$, where $a+b+c\leq m$.
By (ii), an element $h\otimes f\in\mathfrak{h}[z]$ acts on $v$ by the scalar $\sum_if(\gamma_i)\chi_i(h)$. Therefore $Hv\in\bC v$.
We have $U(\mathfrak{n}[z])=\bC\oplus\mathfrak{n}[z]\,U(\mathfrak{n}[z])$, and by (iii) $\ell$ vanishes on $\mathfrak{n}[z]\,L$. So $\ell(YXHv)=0$ unless $Y$ is a scalar, and in that case $YXHv\in U(\mathfrak{a}[z])_{\leq m}v$. This proves $\subseteq$. The inclusion $\supseteq$ is clear.
\end{proof}

\begin{proposition}
\label{prp::affine::comparison}
Keep the assumptions of Lemma~\ref{lem::ell::reduction}, and assume in addition:
\begin{enumerate}
\item[(iv)] $M$ is finite-dimensional, $w\in M$ is a cocyclic vector for the $\mathfrak{a}[z]$-action (Definition~\ref{def::cocyclic}), and $\ell|_M=\pi_w$;
\item[(v)] there is an integer $d$ such that $v\in G^d(M)$, where $G$ is the cofusion filtration of $M$ ending at $w$ (Definition~\ref{def::cofusion}).
\end{enumerate}
Then for all $k$
\[
F^{\mathfrak{a}}_kM\ \subseteq\ F^{\mathfrak{b}}_kM\ \subseteq\ G^{\,d-k}(M).
\]
In particular, if $F^{\mathfrak{a}}_kM=G^{d-k}(M)$ for all $k$, then $F^{\mathfrak{a}}_kM=F^{\mathfrak{b}}_kM=G^{d-k}(M)$ for all $k$.
\end{proposition}
\begin{proof}
The first inclusion is~\eqref{eq::Fa::in::Fb}. For the second one, let $u\in F^{\mathfrak{b}}_kM$ and $X\in U(\mathfrak{a}[z])_{<d-k}$. Then $Xu\in M$, and
\[
Xu\in U(\mathfrak{b}[z])_{\leq d-k-1}\,U(\mathfrak{b}[z])_{\leq k}\,v\subseteq U(\mathfrak{b}[z])_{\leq d-1}\,v .
\]
By (iv) and Lemma~\ref{lem::ell::reduction},
\[
\pi_w(Xu)=\ell(Xu)\in\ell\big(U(\mathfrak{a}[z])_{\leq d-1}v\big)=\pi_w\big(U(\mathfrak{a}[z])_{<d}\,v\big),
\]
and the right-hand side is $0$, because $v\in G^d(M)$. Since $X$ was arbitrary, $u\in G^{d-k}(M)$.
\end{proof}

In words: by (v), the cocyclic vector cannot be reached from $v$ by the smaller algebra in degree less than $d$. Lemma~\ref{lem::ell::reduction} shows that the larger algebra cannot do better. This bounds the larger fusion filtration from above by the cofusion filtration of the smaller one.

\subsection{Demazure modules inside level one modules}

Let $\widehat{\gl}_V=\gl_V\otimes\bC[x,x^{-1}]\oplus\bC K\oplus\bC d$, with $K$ central and $[d,X\otimes x^m]=m\,X\otimes x^m$. The current algebra $\gl_V[x]$ is a subalgebra: the cocycle vanishes on it. For $1\leq k\leq N\leq\dim V$, let $L(\Lambda_k)$ be the integrable irreducible $\widehat{\gl}_V$-module of level one whose highest weight has finite part $\omega_k=\varepsilon_1+\ldots+\varepsilon_k$. Normalize $d$ so that it acts by $0$ on the highest weight vector. Then
\[
L(\Lambda_k)=\bigoplus_{e\leq0}L(\Lambda_k)_e,\qquad L(\Lambda_k)_0\simeq\Lambda^kV\ \text{as a $\gl_V$-module},
\]
where $L(\Lambda_k)_e$ is the $d$-eigenspace with eigenvalue $e$, and $X\otimes x^m$ maps $L(\Lambda_k)_e$ to $L(\Lambda_k)_{e+m}$. Let $v_{k\varepsilon_1}\in L(\Lambda_k)$ be the extremal weight vector of $\gl_V$-weight $k\varepsilon_1$. The Demazure module
\[
D_{k\varepsilon_1}:=U(\gl_V[x])\,v_{k\varepsilon_1}\subseteq L(\Lambda_k)
\]
is isomorphic to the local Weyl module $\sW^x_{k\varepsilon_1}$ (\cite{CL,FL2}, see also property~(D4) in \S\ref{sec::Weyl}). It contains the highest weight vector of $L(\Lambda_k)$, hence the whole top component $L(\Lambda_k)_0=\Lambda^kV$, which is its socle.

Let $\lambda=(\lambda_0\geq\ldots\geq\lambda_l)\vdash N$, and take $z=y$. Apply the general setting with
\[
\mathfrak{a}=\gl_V[x],\quad\mathfrak{b}=\widehat{\gl}_V,\quad M_i=D_{\lambda_i\varepsilon_1},\quad L_i=L(\Lambda_{\lambda_i}),\quad v_i=v_{\lambda_i\varepsilon_1}.
\]
Then $\mathfrak{a}[y]=\gl_V[x,y]$, and
\[
D:=D_{\lambda_0\varepsilon_1}^{\gamma_0}\otimes\cdots\otimes D_{\lambda_l\varepsilon_1}^{\gamma_l}\ \simeq\ \sW^x_{\lambda_0\varepsilon_1}(\gamma_0)\otimes\cdots\otimes\sW^x_{\lambda_l\varepsilon_1}(\gamma_l)
\]
is the module $M$ of Theorem~\ref{thm::fusion::cofusion} for the one-row diagrams $\lambda^i=(\lambda_i)$ (recall that $\mGH_{(h)}\simeq\sW^x_{h\varepsilon_1}$). Thus $F^{\mathfrak{a}}$ is the $y$-fusion filtration of $D$ starting at $v_{N\varepsilon_1}=v$. Let $G$ be the $y$-cofusion filtration of $D$ ending at $w=v_{\omega_N}$.

\begin{theorem}
\label{thm::affine::current}
For all $k$,
\[
F^{\gl_V[x]}_kD\ \subseteq\ F^{\widehat{\gl}_V}_kD\ \subseteq\ G^{\,n(\lambda)-k}(D).
\]
Consequently, the fusion filtration of $D$ defined by the current algebra, the fusion filtration defined by the affine algebra, and the cofusion filtration coincide:
\[
F^{\gl_V[x]}_kD=F^{\widehat{\gl}_V}_kD=G^{\,n(\lambda)-k}(D),\qquad\text{and}\qquad
\mathrm{gr}_{F^{\widehat{\gl}_V}}D\simeq\mGH_\lambda .
\]
\end{theorem}
\begin{proof}
We verify the hypotheses of Lemma~\ref{lem::ell::reduction} and Proposition~\ref{prp::affine::comparison}.

(i) Put $\mathfrak{n}=\gl_V\otimes x^{-1}\bC[x^{-1}]$ and $\mathfrak{h}=\bC K\oplus\bC d$. The space $\mathfrak{n}$ is a subalgebra, because the cocycle $m\,\delta_{m+m',0}\,\mathrm{tr}(XX')K$ vanishes when $m,m'<0$. The space $\mathfrak{a}\oplus\mathfrak{h}$ is a subalgebra, because $[d,\mathfrak{a}]\subseteq\mathfrak{a}$ and the cocycle vanishes on $\mathfrak{a}$.

(ii) $K$ acts by $1$ on every $L_i$. The vector $v_i$ is a weight vector for the Cartan subalgebra of $\widehat{\gl}_V$, so $d$ acts on it by a scalar.

(iii) The space $L$ is graded by the multi-energy: $L=\bigoplus_{\mathbf{e}\in\bZ_{\leq0}^{l+1}}L_{\mathbf{e}}$ with $L_{\mathbf{e}}=\bigotimes_iL(\Lambda_{\lambda_i})_{e_i}$. Its top component is
\[
T:=L_{(0,\ldots,0)}=\Lambda^{\lambda_0}V\otimes\cdots\otimes\Lambda^{\lambda_l}V .
\]
The module $\Lambda^NV$ occurs in $T$ with multiplicity one. Let $\ell:=\pi^T_{\omega_N}\circ\mathrm{pr}_T$, where $\mathrm{pr}_T$ is the projection along the other components, and $\pi^T_{\omega_N}$ is the coordinate of $v_{\omega_N}$ in $T$ (Definition~\ref{def::cocyclic} for the $\gl_V$-module $T$). For $X\in\mathfrak{n}$ of $x$-degree $-m<0$, the element $X\otimes f$ acts on $L$ as $\sum_if(\gamma_i)X^{(i)}$. Each $X^{(i)}$ maps $L_{\mathbf{e}}$ to $L_{\mathbf{e}-m\epsilon_i}$, and this is never $T$, because $e_i-m<0$. Hence $\ell\circ(X\otimes f)=0$.

(iv) $D$ is finite-dimensional, and $v_{\omega_N}$ is cocyclic in it by Theorem~\ref{thm::fusion::cofusion}(i). The multi-energy grading restricts to $D$ with $D\cap T=T$. The $\gl_V$-isotypic projection onto $\Lambda^NV$ preserves the multi-energy grading. Moreover, $\Lambda^NV$ occurs in $D\simeq V^{\otimes N}$ only once, namely inside $T$. Hence $\pi_{v_{\omega_N}}(u)=\pi^T_{\omega_N}(\mathrm{pr}_Tu)=\ell(u)$ for $u\in D$.

(v) With $d=n(\lambda)$, the condition $v\in G^{n(\lambda)}(D)$ is the case $k=0$ of the inclusion $F_k\subseteq G^{n(\lambda)-k}$ of Theorem~\ref{thm::fusion::cofusion}(ii).

Proposition~\ref{prp::affine::comparison} gives the chain of inclusions. By Theorem~\ref{thm::fusion::cofusion}(ii), $F^{\gl_V[x]}_kD=G^{n(\lambda)-k}(D)$, and therefore all three filtrations coincide. The last isomorphism follows from Theorem~\ref{thm::fusion::cofusion}(iii).
\end{proof}

\end{document}